\documentclass{amsart}
\pdfoutput=1 
\usepackage{amssymb,amsmath,amsthm}
\usepackage[foot]{amsaddr}
\title{Effective Kan fibrations in simplicial sets}
\date{\today}
\author{Benno van den Berg and Eric Faber} 

\usepackage[toc,page]{appendix}
\usepackage{ast2}
\makeindex
\usepackage{subcaption} 
\usepackage{tikz}
\usepackage{calc}
\usetikzlibrary{cd}
\usetikzlibrary{matrix}
\usetikzlibrary{arrows.meta, arrows}
\usetikzlibrary{patterns}
\tikzset{
    position/.style args={#1:#2 from #3}{
        at=(#3.#1), anchor=#1+180, shift=(#1:#2)
    }
  }
\usepackage{enumerate}
\usepackage[margin=1.7in]{geometry}
\usepackage[
    backend=biber,
    style=alphabetic-verb,
    sortlocale=en_GB,
    natbib=true,
    url=false,
    isbn=false,
    doi=true,
    eprint=true
    ]{biblatex}
\usepackage{mathtools}

\usepackage{array}   
\newcolumntype{C}{>{$}c<{$}} 

\usepackage{caption, newfloat, xparse}
\newsavebox{\mybox}
\DeclareFloatingEnvironment[
  fileext=lob,
  listname={List of Boxes},
  name=Box,
  placement=t,
]{afloatingbox}
\NewDocumentEnvironment{textbox}{m}{
\begin{afloatingbox}
\centering
\begin{lrbox}{\mybox}
\begin{minipage}{\textwidth-3em}
}{
\end{minipage}
\end{lrbox}
\framebox[\textwidth]{\usebox{\mybox}}
\caption{#1}
\end{afloatingbox}
}

\newcommand{\varco}{\thinspace{} | \thinspace}
\newcommand{\co}{\thinspace{} \colon{} \thinspace}
\newcommand{\op}{\ensuremath{^{\mathrm{op}}}}
\newcommand{\cat}[1]{\ensuremath{\mathbf{#1}}}
\newcommand{\pullback}[2]{{}_{#1}\kern-\scriptspace{\times}_{#2}}
\newcommand{\E}{\ensuremath{\mathcal{E}}}
\newcommand{\dom}{\operatorname{dom}}
\newcommand{\cod}{\operatorname{cod}}
\newcommand{\ev}{\operatorname{ev}}
\newcommand{\llp}[1]{{}^\pitchfork{#1}}
\newcommand{\rlp}[1]{#1^{\pitchfork}}
\newcommand{\sq}[1]{\ensuremath{\mathbf{Sq}\left( #1 \right)}}
\newcommand{\cube}[1]{\ensuremath{\mathbf{Cube}\left( #1 \right)}}
\renewcommand{\coalg}[1]{#1\operatorname{-}\mathbb{C}\mathbf{oalg}}
\renewcommand{\alg}[1]{#1\operatorname{-}\mathbb{A}\mathbf{lg}}
\newcommand{\coalgcat}[1]{#1\operatorname{-}\mathbf{Coalg}}
\newcommand{\algcat}[1]{#1\operatorname{-}\mathbf{Alg}}
\newcommand{\coalgcofibr}[1]{#1\operatorname{-}\mathbf{coalg}}
\newcommand{\algfibr}[1]{#1\operatorname{-}\mathbf{alg}}
\newcommand{\freedbl}[1]{{#1}_{\text{dbl}}}

\newcommand{\C}{\ensuremath{\mathcal{C}}}
\newcommand{\Abb}{\ensuremath{\mathbb{A}}}
\renewcommand{\Bbb}{\ensuremath{\mathbb{B}}}
\newcommand{\Lbb}{\ensuremath{\mathbb{L}}}
\newcommand{\Rbb}{\ensuremath{\mathbb{R}}}
\newcommand{\Drl}[1]{{#1}^{\pitchfork\mskip-9.2mu\pitchfork}}
\newcommand{\Dll}[1]{{}^{\pitchfork\mskip-9.2mu\pitchfork}{#1}}

\renewcommand{\L}{\ensuremath{\mathcal{L}}}
\newcommand{\R}{\ensuremath{\mathcal{R}}}
\newcommand{\DBL}{\ensuremath{\mathbf{DBL}}}
\newcommand{\Cat}{\ensuremath{\mathbf{Cat}}}

\newcommand{\EEffRFib}{\ensuremath{\mathbb{E}\mathbf{ffFib}}}
\newcommand{\EffRFib}{\ensuremath{\mathbf{EffFib}}}
\newcommand{\effRFib}{\ensuremath{\mathbf{effFib}}}

\usepackage{todonotes}

\begin{document}

\begin{abstract}
  We introduce the notion of an effective Kan fibration, a new mathematical structure that can be used to study simplicial homotopy theory. Our main motivation is to make simplicial homotopy theory suitable for homotopy type theory. Effective Kan fibrations are maps of simplicial sets equipped with a structured collection of chosen lifts that satisfy certain non-trivial properties. This contrasts with the ordinary, unstructured notion of a Kan fibration.  We show that fundamental properties of Kan fibrations can be extended to explicit constructions on effective Kan fibrations. In particular, we give a constructive (explicit) proof showing that effective Kan fibrations are stable under push forward, or fibred exponentials. This is known to be impossible for ordinary Kan fibrations. We further show that effective Kan fibrations are local, or completely determined by their fibres above representables. We also give an (ineffective) proof saying that the maps which can be equipped with the structure of an effective Kan fibration are precisely the ordinary Kan fibrations. Hence implicitly, both notions still describe the same homotopy theory. By showing that the effective Kan fibrations combine all these properties, we solve an open problem in homotopy type theory. In this way our work provides a first step in giving a constructive account of Voevodsky's model of univalent type theory in simplicial sets.
\end{abstract}  

\maketitle

\tableofcontents

\section{Introduction}

\subsection{Main contribution and motivation}
This paper starts to redevelop the foundations of simplicial homotopy theory, in particular around the Kan-Quillen model structure on simplicial sets, in a more effective or ``structured'' style.  Our motivation comes from \emph{homotopy type theory}\index{homotopy type theory} (HoTT) and Voevodsky's construction of a model of HoTT in simplicial sets \cite{kapulkinlumsdaine18}, which relies heavily on the existence and properties of the Kan-Quillen model structure.

Type theory\index{type theory} refers to a family of formal systems which can act both as foundations for (constructive) mathematics and functional programming languages. Recently, it has become apparent that there exist many connections between type theory on the one hand and homotopy theory and higher category theory on the other. Besides Voevodsky's fundamental contributions, other key steps have been the groupoid model by Hoffman and Streicher \cite{hofmannstreicher98}, the interpretation of Martin-L\"of's identity types in categories equipped with a weak factorisation system \cite{awodeywarren09} and the proof that types in type theory carry the structure of an $\infty$-groupoid \cite{vdBerg-Garner-11,lumsdaine10}. As a result of these contributions, homotopy type theory has become an active area of research which keeps on developing at a quick pace, with implications for both type theory and homotopy theory. (For an overview, see the HoTT book~\cite{univalent13}.)

However, for type theory to fully benefit from the rich treasure chest of homotopy theory and higher category theory, a computational understanding of the relevant results from these areas is crucial. Indeed, we would like to think of type theory as a framework for computation. Then to fully exploit homotopy-theoretic ideas in this framework, one must be able to computationally reduce them. So a natural question is how constructive Voevodsky's model in simplicial sets is, or the proofs of the properties of the Kan-Quillen model structure on which it relies.

One fundamental obstacle with building Voevodsky's model in simplicial sets in a constructive framework was identified by Bezem, Coquand and Parmann \cite{bezem-coquand-parmann}. To interpret $\Pi$-types in simplicial sets, one uses that the category of simplicial sets is locally cartesian closed: that is, that the pullback functor along any map has a right adjoint, which we call \emph{pushforward}\index{pushforward}. Since the type families are interpreted as Kan fibrations in Voevodsky's model, we need to show that Kan fibrations are closed under pushforward along Kan fibrations. This is true classically, but as the authors of \cite{bezem-coquand-parmann} show, this result is unprovable constructively. We will refer to this as the BCP-obstruction and, given the importance of $\Pi$-types in type theory, it is quite a serious problem.

That this problem is not insurmountable was shown by Gambino and Sattler \cite{Gambino-Sattler}: the key idea here is to treat being a Kan fibration not as a property, but as structure. Inspired by the work in HoTT on cubical sets\index{cubical sets} (\cite{CCHM17} in particular), they define a structured notion of a \emph{uniform Kan fibration}\index{uniform Kan fibration} and give a constructive proof that uniform Kan fibrations are closed under pushforward. They also show that their definition is ``classically correct'' in that a map can be equipped with the structure of a uniform Kan fibration if and only if it has the right lifting property against the horn inclusions (is a Kan fibration in the usual sense).

In this paper we will introduce another solution to this problem: the \emph{effective Kan fibrations}\index{effective Kan fibration}. The reason for introducing a new solution is that Gambino and Sattler ran into trouble with another type constructor: universes. Indeed, the only known method for constructing universal fibrations in presheaf categories is via the Hofmann-Streicher construction \cite{hofmannstreicher97}, and this method can only be applied to notions of fibred structure which are \emph{local} (see \refdefi{localfibredstructure} below). The problem is that the uniform Kan fibration are not local, whereas we are able to show this for our notion of an effective Kan fibration. (That uniform Kan fibrations are not local was shown by Christian Sattler; his proof can be found in Appendix~\ref{sec:ongambinosattler} to this paper.)

So, to summarise, our main contribution is the introduction of the notion of an effective Kan fibration, a structured notion of Kan fibration for which we will prove the following results:
\begin{enumerate}
  \item Effective Kan fibrations are closed under pushforward\index{pushforward}.
  \item The notion of an effective Kan fibration is local and hence universal effective Kan fibrations exist.
  \item Effective Kan fibrations have the right lifting property against horn inclusions.
  \item A map which has the right lifting property against horn inclusions can be equipped with the structure of an effective Kan fibration.
\end{enumerate}
We will give constructive proofs of (1) -- (3), whereas the proof of (4) will necessarily be ineffective (due to the BCP-obstruction). As a result, the effective Kan fibrations are the first, and so far only, structured notion of fibration for which these properties have been shown.

Besides having a clear computational content, another advantage of constructive proofs is that they can be internalised to arbitrary Grothendieck toposes (not just $\Sets$). In fact, our arguments can be formalised in any locally cartesian closed category with finite colimits and a natural numbers object. For those who prefer to think in terms of set theory, our arguments can be performed in (a subsystem of) Aczel's constructive set theory {\bf CZF} (for which see \cite{aczelrathjen01}), which in turn is a subsystem of classical {\bf ZF}, Zermelo-Fraenkel set theory (without choice).

But however this may be, we feel that laying too great an emphasis on the metamathematical aspects of our work may be misleading. The task of reworking some of the fundamental concepts in simplicial homotopy theory in a more explicit or structured style is an interesting undertaking in itself, whatever one's foundational convictions, and we hope that any homotopy theorists reading this work will come to see it that way as well. Indeed, any mathematician who wishes to skip the occasional foundational aside on our part should feel free to do so, and can read this paper as just another piece of new mathematics.

\subsection{Related work} Besides the work of Gambino and Sattler we already mentioned, there are two strands of research with which our approach should be compared.

In response to the BCP-obstruction, most researchers in HoTT have abandoned simplicial sets and switched to cubical sets. In doing so people have managed to constructively prove the existence of a model structure and a model of HoTT in cubical sets. In addition, their cubical models can be seen as interpreting a cubical type theory, in which principles like univalence can be derived, and which enjoys (homotopy) canonicity (see \cite{CCHM17,CoquandHS19,huber19,BezemCH19}).

These are impressive results and our approach is by no means that far advanced. However, we still feel that analogous results for simplicial sets would be preferable: indeed, simplicial techniques pervade modern homotopy theory, much more than cubical approaches do, and in order to connect to most of the ongoing work in homotopy theory and higher category theory, a simplicial approach is more likely to be successful. In addition, it is at present not entirely clear whether any of the constructive model structures that people have developed in cubical sets model the world of homotopy types or $\infty$-groupoids.

The other approach with which our work should be compared is that of Gambino, Henry, Sattler and Szumi\l o, who, in face of the BCP-obstruction, decide to bite the bullet (see \cite{henry19,gambinohenry19,gambinoetal19}). Their starting point was the constructive proof by Simon Henry of the existence of the Kan-Quillen model structure on simplicial sets, using the standard definitions of the Kan and trivial Kan fibrations (having the right lifting property against the horn inclusions and boundary inclusions, respectively). Based on this work, Henry in collaboration with Gambino managed to construct a model of HoTT, modulo some tricky coherence problems. Their work has the advantage that it is based on the usual definitions of the (trivial) Kan fibrations, so in that sense it looks, at least at first glance, more familiar than our structured approach. In addition, their work is definitely more advanced than ours.

However, we still think that a structured approach looks more appealing. Due to the BCP-obstruction, they only have a weak form of $\Pi$-types. In comparision, our approach should give us genuine $\Pi$-types with definitional $\eta$- and $\beta$-rules. Also, it seems that to obtain a genuine model of homotopy type theory based on their work forces one to solve some quite difficult coherence problems, for which at present no solutions are known. In contrast, we expect that a more structured approach will be helpful in solving any coherence problems we would encounter if we would start to turn our work into a model of type theory.

\subsection{Fibrations as structure} So what is our notion of an effective Kan fibration? Before we answer that question, let us first discuss what, in general, we mean by a structured notion of fibration.

A common situation in homotopy theory is that we are working in some category \ct{E} equipped with a pullback stable class of fibrations; by pullback stability we mean that in a pullback square like
\begin{displaymath}
  \begin{tikzcd}
    Y' \ar[r] \ar[d, "p'"] & Y \ar[d, "p"] \\
    X' \ar[r] & X
  \end{tikzcd}
\end{displaymath}
the map $p'$ will be a fibration whenever $p$ is. If one conceptualises matters like this, being a fibration is a \emph{property} of a map; in this paper, however, we will think of being a fibration as \emph{structure}. In particular, the setting of a category \ct{E} with a pullback stable class of fibrations will be replaced by a \emph{structured notion of a fibration} or \emph{a notion of fibred structure}\index{notion of fibred structure} on \ct{E}, which we will define as a presheaf
\[ {\rm Fib}: (\ct{E}^\to_{\rm cart})^{op} \to Sets, \]
where $\ct{E}^\to_{\rm cart}$ denotes the category of arrows in \ct{E} and pullback squares between them. Given such a structured notion of fibration, an element $\sigma \in {\rm Fib}(p)$ will be called a \emph{fibration structure} on $p: Y \to X$; and if such an element $\sigma$ exists, we may call $p$ a \emph{fibration}\index{fibration}. These fibrations will form a pullback stable class, as before. Indeed, we can think of a pullback stable class as a degenerate notion of fibred structure where ${\rm Fib}(p)$ always contains at most one element, signalling whether the map $p$ is a fibration or not.

But in many examples being a fibration is quite naturally thought of as additional structure on a map. For instance, a common way of defining a class of fibrations is by saying that they are \emph{cofibrantly generated}\index{cofibrantly generated} by a class of maps $\smallmap{A}$. That is, a map $p: Y \to X$ is a fibration precisely when for any $m: B \to A \in \smallmap{A}$ and commutative square
\begin{displaymath}
  \begin{tikzcd}
    B \ar[r, "f"] \ar[d, "m"] & Y \ar[d, "p"] \\
    A \ar[r, "g"] \ar[ur, dotted] & X
  \end{tikzcd}
\end{displaymath}
there exists a dotted filler as shown making both triangles commute; one also says that $p$ has the \emph{right lifting property} against $\smallmap{A}$. In this situation we can define a structured notion of fibration ${\rm Fib}$ by declaring the elements of ${\rm Fib}(p)$ to be \emph{lifting structures} on $p$: that is, functions which assign to each square like the one above with $m \in \smallmap{A}$ a filler $\sigma_{m,f,g}: A \to Y$ making both triangles commute. Let us for the moment write this notion of fibred structure as ${\rm RLP}(\mathcal{A})$.

One thing which happens if one shifts to a structured style is that notions of fibration which are \emph{logically equivalent} as properties are no longer \emph{isomorphic} as structures. Take the trivial Kan fibrations in simplicial sets as an example. They can be defined as the maps which are cofibrantly generated by the monomorphisms; or as those which are cofibrantly generated by the monomorphisms $S \subseteq \Delta^n$ with representable codomain; or as those which are cofibrantly generated by the boundary inclusions $\partial \Delta^n \subseteq \Delta^n$. These may all be equivalent as properties, but as structures, they are all different. Indeed, there are ``forgetful'' morphisms of fibred structure (presheaves)
\[ {\rm RLP}({\rm monos}) \to {\rm RLP}({\rm sieves}) \to {\rm RLP}(\mbox {boundary inclusions}), \]
but they are not monomorphisms, let alone isomorphisms. So as structured notions of fibration they need to be carefully distinguished.

This leads us to another important point for this paper: one can try to repair this by imposing \emph{compatibility conditions}\index{compatibility condition} on the lifting structure (also known as \emph{uniformity conditions} in the literature). Indeed, in the way we have defined ${\rm RLP}(\smallmap{A})$ its elements $\sigma$ choose solutions for a class of lifting problems, but there are no conditions saying how these solutions should be related. For instance, suppose we have in simplicial sets a solid diagram of the form
\begin{displaymath}
  \begin{tikzcd}
    D \ar[d, "n"] \ar[r] & C \ar[d, "m"] \ar[r] & Y \ar[d, "p"] \\
    C \ar[r, "k"'] \ar[urr, dotted, "l_2", near start] & A \ar[r] \ar[ur, dotted, "l_1"'] & X
  \end{tikzcd}
\end{displaymath}
in which $p$ is a trivial Kan fibration and the left hand square is a pullback involving monomorphisms $n$ and $m$. Then any element in ${\rm RLP}({\rm monos})(p)$ must, among other things, choose dotted arrows $l_1$ and $l_2$ as shown; we could define a notion of fibred structure ${\rm RLP}_c({\rm monos})$ which would require that in such circumstances we must have $k.l_1 = l_2$. And if ${\rm RLP}_c({\rm sieves})$ would be the restriction of ${\rm RLP}({\rm sieves})$ to those lifting structures which in a situation like
\begin{displaymath}
  \begin{tikzcd}
    \alpha^* S \ar[d, "n"] \ar[r] & S \ar[d, "m"] \ar[r] & Y \ar[d, "p"] \\
    \Delta^m \ar[r, "\alpha"'] \ar[urr, dotted, "l_2", near start] & \Delta^n \ar[r] \ar[ur, dotted, "l_1"'] & X
  \end{tikzcd}
\end{displaymath}
would choose lifts $l_1$ and $l_2$ satisfying $\alpha.l_1 = l_2$, then the forgetful morphism
\[ {\rm RLP}_c({\rm monos}) \to {\rm RLP}_c({\rm sieves}) \]
would be an isomorphism of notions of fibred structure. The reader may wonder if any trivial Kan fibration can still be equipped with such a structure: that is, whether lifts against monos or sieves can always be chosen in such a way that these compatibility conditions are met. That is indeed the case (see \cite{Gambino-Sattler}).

Both ${\rm RLP}_c({\rm monos})$ and ${\rm RLP}_c({\rm sieves})$ are examples of right lifting structures defined by lifts against \emph{categories} rather than \emph{classes} of maps, and a rich variety of lifting structures is quite characteristic of our structured approach. Indeed, we will also consider \emph{double categorical} and even \emph{triple categorical} notions of lifting structure. To motivate this, let us consider the forgetful map
\[ {\rm RLP}({\rm monos}) \to {\rm RLP}(\mbox {boundary inclusions}). \]
This will not be a monomorphism even when we restrict to ${\rm RLP}_c({\rm monos})$, but there is a further (double-categorical) compatibility condition we could imagine imposing which would have this effect. Suppose we have a solid diagram in simplicial sets
\begin{displaymath}
  \begin{tikzcd}
    C \ar[r, "g"] \ar[d, "m"] & Y \ar[dd, "p"] \\
    B \ar[d, "n"] \\
    A \ar[r, "f"] & X
  \end{tikzcd}
\end{displaymath}
in which $p$ is a trivial Kan fibration and $m$ and $n$ are monomorphisms. Then a lifting structure $\sigma$ on $p$ will give rise to a filler $A \to Y$ in two different ways: we can use that monomorphisms are closed under composition and take $\sigma_{n.m,g,f}$. But we could also first construct a map $l: B \to Y$ by taking $l = \sigma_{m,g,f.n}$ and then use that to construct $\sigma_{n,l,g}$. A natural requirement would be that these two lifts should always coincide. If we write ${\rm RLP}_{dc}({\rm monos})$ for the notion of fibred structure where the lifts satisfy this condition on top of the previous one, then we will prove in this paper that
\[ {\rm RLP}_{dc}({\rm monos}) \to {\rm RLP}(\mbox {boundary inclusions}). \]
is a monomorphism of notions of fibred structure. We will also characterise the image of this map and show that every trivial Kan fibration can be equipped with such a double-categorical lifting structure. Indeed, with one further (constructive) twist, this will be our preferred structured notion of a trivial Kan fibration (an \emph{effective trivial Kan fibration}).

\subsection{Effective Kan fibrations} As said, the core of our paper is the definition of an effective Kan fibration, our preferred structured notion of a Kan fibration. To motivate this definition, let us recall the classical result (from \cite{gabrielzisman67}) that says that the Kan fibrations are cofibrantly generated by maps of the form $m \hat{\otimes} \partial_i$, where $m: A \to B$ is a cofibration, $\hat{\otimes}$ is the pushout-product and $\partial_i: 1 \to \mathbb{I}$ is one of the two endpoint inclusions into the interval $\mathbb{I} = \Delta^1$. This can be reformulated as follows: let us say that a map $p: Y \to X$ has the right lifting property against a commutative square
\begin{displaymath}
  \begin{tikzcd}
    D \ar[r] \ar[d] & B \ar[d] \\
    C \ar[r] & A
  \end{tikzcd}
\end{displaymath}
if for any solid diagram
\begin{displaymath}
  \begin{tikzcd}
    D \ar[r] \ar[d] & B \ar[d] \ar[r] & Y \ar[d, "p"] \\
    C \ar[r] \ar[urr, dotted] & A \ar[r] \ar[ur, dotted] & X
  \end{tikzcd}
\end{displaymath}
and dotted arrow $C \to Y$ making the diagram commute, there exists a dotted arrow $A \to Y$ making the whole picture commute. (Note that this is equivalent to having the right lifting property against the inscribed map from the pushout $B \coprod_D C$ to $A$.) Then a map is a Kan fibration if and only if it has the right lifting property against the left hand square in a double pullback diagram of the form
\begin{displaymath}
  \begin{tikzcd}
    A \ar[r, "{(1,\partial_i)}"] \ar[d, "m"] & A \times \mathbb{I} \ar[d, "m \times \mathbb{I}"] \ar[r, "\pi_1"] & A \ar[d, "m"] \\
    B \ar[r, "{(1, \partial_i)}"] & B \times \mathbb{I} \ar[r, "\pi_1"] & B.
  \end{tikzcd}
\end{displaymath}
The usual definition of a Kan fibration in terms of horn inclusions can also be stated as a lifting problem against a square, namely the left hand square in another double pullback diagram:
\begin{displaymath}
  \begin{tikzcd}
    \partial \Delta^n \ar[r] \ar[d] & s^* \partial \Delta^n \ar[d] \ar[r] & \partial \Delta^n \ar[d] \\
    \Delta^n \ar[r, "d"] & \Delta^{n+1} \ar[r, "s"] & \Delta^n,
  \end{tikzcd}
\end{displaymath}
where $s = s_i$ is one of the degeneracies and $d = d_i/d_{i+1}$ is one of its sections. (The reason being that the inclusion $\Delta^n \cup s^* \partial \Delta^n \to \Delta^{n+1}$ is the horn inclusion $\Lambda^{n+1}_{i/i+1}$.) We will call such left hand squares \emph{horn squares}.

These two situations have something in common, namely that they are both lifting conditions against a left hand square in a double pullback diagram of the form
\begin{displaymath}
  \begin{tikzcd}
    C \ar[r] \ar[d, "m"] & r^*C \ar[d] \ar[r] & C \ar[d, "m"] \\
    A \ar[r, "i"] & B \ar[r, "r"] & A
  \end{tikzcd}
\end{displaymath}
in which $m$ is a cofibration and $(i,r)$ is a deformation retract of some kind. The first step towards our definition of an effective Kan fibration is the identification of the right kind of deformation retracts. Our solution is the notion of a \emph{hyperdeformation retract}\index{hyperdeformation retract} (HDR), and to define these HDRs we use the simplicial Moore path functor defined by the first author in collaboration with Richard Garner~\cite{vdBerg-Garner}.

Once we have the concept of an HDR, we can define the \emph{mould squares}\index{mould square} as those pullback squares
\begin{displaymath}
  \begin{tikzcd}
    A' \ar[d, "m"] \ar[r, "{(i',r')}"] & B' \ar[d] \\
    A \ar[r, "{(i,r)}"] & B
  \end{tikzcd}
\end{displaymath}
in which $m$ is a cofibration, $(i,r)$ and $(i',r')$ are HDRs and the square (read from top to bottom) is what we will call a cartesian morphism of HDRs. The idea, then, is to define the effective Kan fibrations as those maps which come equipped with lifts against mould squares.

What is missing from this definition, however, are the correct compatibility conditions\index{compatibility condition}. It turns out that mould squares can be composed both horizontally and vertically (they naturally fit into a double category), and this leads to two natural compatibility conditions. In fact, there is a further ``perpendicular'' condition, because mould squares can be pulled back along morphisms of HDRs, leading to a third (triple-categorical) compatibility condition. Indeed, our main reason for introducing the concept of a mould square is that they allow us to express these compatibility conditions in such an elegant way.

\begin{textbox}{From uniform to effective Kan fibrations\label{box:fromuniformtoeffectives}}\index{uniform Kan fibration|textbf}
  One way of understanding our notion of an effective Kan fibration is as a modifcation of the notion of a uniform Kan fibration by Gambino and Sattler. They work with a category of cofibrations and an interval object $\mathbb{I}$: in the category of simplicial sets these would be a category of monomorpisms and pullback squares between them and the representable $\mathbb{I} = \Delta^1$. The interval comes equipped with two maps $\partial_0, \partial_1: 1 \to \mathbb{I}$ which yields two natural transformations $s_Y, t_Y: Y^\mathbb{I} \to Y$. A map $p: Y \to X$ is a uniform Kan fibration if $(s/t, p^\mathbb{I}): Y^\mathbb{I} \to Y \times_X X^\mathbb{I}$ has the right lifting property against the category of cofibrations. We modify this definition by replacing the path object $X^\mathbb{I}$ by the simplical Moore path object $MX$ and adding a uniformity condition demanding that lifts behave well with respect to concatenation of Moore paths, which implies that lifts of general simplicial Moore paths are determined by those of length 1. This modification can be understood as a consequence of our desire to make the definition local. Indeed, note that an $n$-simplex in $Y^\mathbb{I}$ corresponds to a prism $\Delta^n \times \mathbb{I} \to Y$. In contrast with cubical sets, the representable simplicial sets are not closed under products and therefore the domain of this map is not representable. This is ultimately what causes non-locality of the uniform Kan fibrations. However, the map $\Delta^n \times \mathbb{I} \to Y$ can be understood as a Moore path of length $n+1$, composed of Moore paths of length one of the form $\Delta^{n+1} \to Y$, reflecting the fact that the prism $\Delta^n \times \mathbb{I}$ is the union of $(n+1)$-many $(n+1)$-simplices. Therefore for effective Kan fibrations all the structure is determined by the lifts for maps with representable domain and this is what allows us to regain locality.

  We will also add a further compatibility condition for our effective Kan fibrations which relates to the fact that cofibrations are closed under composition. This additional requirement is not strictly necessary for our purposes, but it allows us to prove in Section 12 that the structure of being an effective Kan fibration is completely determined by the the lifts against the horn inclusions (or horn squares, to be more precise). In addition, as we will show in future work, it will have the consequence that the effective trivial fibrations and the effective Kan fibrations interact nicely.
  \end{textbox}

Once we have this in place, and we have checked that horn squares are mould squares, it follows immediately that effective Kan fibrations have the right lifting property against horn inclusions (it will also not be too hard to see that our effective Kan fibrations are uniform Kan fibrations in the sense of Gambino-Sattler). In fact, quite a lot of pages will be spent on proving that the lifts against the mould squares are completely determined by the lifts against the horn squares, or, in other words, that the forgetful map
\[ {\rm RLP}_{tc}(\mbox{mould squares}) \to {\rm RLP}(\mbox{horn squares}) \]
is a monomorphism of fibred structures. We will also characterise its image, which will be crucial for proving both that our notion of an effective Kan fibration is local and that it is classically correct. The paper will be consist of two parts and these two results will form the main achievements of the second part of this paper.

The first part will be devoted to proving that the effective Kan fibrations are closed under pushforward. We find it convenient to do this axiomatically, using an axiomatic setup reminiscent of the work of Orton and Pitts \cite{ortonpitts18}. The idea of Orton and Pitts (but see also \cite{Gambino-Sattler,Frumin-vdBerg}) was to develop the basic theory of the cubical sets model in the setting of a suitable category equipped with a class of cofibrations forming a dominance\index{dominance} and an interval object $\mathbb{I}$. In our setup we will keep the dominance, but replace the interval object by a Moore path functor $M$ satisfying certain equations (these can be found in an Appendix~\ref{sec:axiomsformoore} to this paper). The example we have in mind is, of course, the simplicial Moore path functor from \cite{vdBerg-Garner}. As our dominance, we take the monomorphisms in simplicial sets which are ``pointwise decidable'' (this is an additional constructive requirement that we impose on the cofibrations, which can be ignored by our classical readers). As we will show, this axiomatic setting is sufficiently powerful to define a suitable notion of mould square and effective Kan fibration, and prove that the effective Kan fibrations are closed under pushforward.

The main drawback of the notion of an effective Kan fibration might be that we are unable to constructively prove that they are closed under retracts. That is, we do not know how to effectively equip maps which are retracts of effective Kan fibration with the structure of an effective Kan fibration. In fact, we suspect that this is impossible, but we do not know this for sure.

The other drawback is that the theory is currently not complete. It should be the case that the effective Kan fibrations are the right class in an algebraic weak factorisation system and that they underlie an algebraic model structure as in~\cite{Riehl-11}. We plan to take this up in future work.

\subsection{Summary of contents} The contents of this paper are therefore as follows.

We start Part 1 with a recap of the theory of algebraic weak factorisation systems (AWFSs), a structured analogue of the notion of a weak factorisation system. In this structured notion the left maps are replaced by coalgebras for a comonad on the arrow category, while the right maps are replaced by the algebras for a monad on the arrow category. Our main reference for this theory is an important paper by Bourke and Garner \cite{Bourke-Garner}, which also explains the connection to double categories. There are two (related) points here which are perhaps worth stressing for those who are already familiar with this theory: first of all, for us the distributive law is important and we will always assume it. Secondly, we will exclusively work with the (co)algebras for the (co)monad, never with the (co)algebras for the (co)pointed endofunctor. The reason is that we will not assume that the effective Kan fibrations are closed under retracts and therefore we cannot express them as algebras for a pointed endofunctor. It also means that we cannot expect them to be cofibrantly generated by a small category: the best we can hope for is that they are cofibrantly generated by a small double category. (We believe this to be true, but we will not prove this in this paper.) In any case, the connection of algebraic weak factorisation systems to double categories is quite important for us. But to make that work, the distributive law is crucial and that is the reason we will always assume it.

We will then go on to explain how both dominances and Moore structures give rise to AWFSs (for dominances this can already be found in \cite{Bourke-Garner}). We will refer to these as the (effective cofibration, effective trivial fibration) and (HDR, naive fibration)-AWFS, respectively. Using these two ingredients we will then define the notions of mould square and effective fibration. Assuming that the Moore structure is symmetric, we will then show that these effective fibrations are closed under pushforward. An important intermediate step for this is the proof of the Frobenius property for the (HDR, naive fibration)-AWFS, which is related to an argument that can also be found in \cite{vdBerg-Garner}.

We will start the second part by showing that the category of simplicial sets can be equipped with both a dominance and a symmetric Moore structure. This will show that the theory of part 1 applies to simplicial sets. Then we will proceed to show that effective (Kan) fibrations can be completely characterised by their lifts against horn squares, which will prove both that this notion of fibred structure is local and classically correct.

To our surprise it turns out that the machinery we develop here can also be used to give effective (structured) analogues of the notions of left and right fibration in simplicial sets. Indeed, also these can be defined by a right lifting property against a class of mould squares, using the same dominance of effective cofibrations, but a different Moore structure. When our results have implications for an effective theory of left and right fibrations, we will comment on that as well.

Finally, we will finish this paper with a conclusion outlining directions for future research and four appendices. In the first appendix we give our version of the Orton-Pitts axioms. In the second appendix we compare the axioms for Moore structure with the counterpart notions of connections and diagonals in cubical sets. The third appendix proves a result on horn fillers that we need for the proofs that our different effective notions of fibration are classically correct. Finally, the fourth appendix contains Christian Sattler's proof that the notion of a \emph{uniform} Kan fibration in the category of simplicial sets is not local. 

\subsection{Acknowledgements} We thank Richard Garner for some very useful conversations on polynomial functors, which had a major influence on Section 9. We also thank Christian Sattler for allowing us to include his proof about the non-locality of the uniform Kan fibrations. Finally, we are also grateful to the referees for very helpful feedback.

\chapter{\texorpdfstring{\(\Pi\)}{Pi}-types from Moore paths}

 \section{Preliminaries}\label{sec:preliminaries}
In this section we introduce the main theoretical framework in which our theory
of effective fibrations is embedded. Abstractly put, we are studying and
constructing new notions of \emph{fibred structure} and \emph{cofibred
structure} on a category \(\E\).\index{notion of fibred structure} 

Throughout the paper, we will assume that \(\E\) is a locally cartesian closed
category with finite limits and colimits.  Certain results in this paper may
also be obtained for a wider class of categories \(\E\). For example, those
results which only use that the codomain functor is a bifibration satisfying
the Beck-Chevalley condition, hold more generally than just for locally
cartesian categories with finite colimits, because the latter notion is
stronger (\reflemm{lcc-beck-chevalley}). Yet in different sections we do resort
to the fact that \(\E\) is locally cartesian closed, and hence stating results
which combines the theory from different sections quickly becomes unwieldy when
working in full generality.  In short, the theory presented in this paper is
best understood when one is not so much worried about its level of generality.

We first recall the definition of a locally cartesian closed category.
\begin{defi}{lcc}\index{locally cartesian closed category}
  Suppose \(\E\) is a category with pullbacks. Then \(\E\) is
  \emph{locally cartesian closed} if for every arrow \(f: X \to Y\),
  the pullback functor
  \begin{equation}
    f^* : \E / Y \to \E / X
  \end{equation}
  has a right adjoint:
  \begin{equation}
    \Pi_f : \E / X \to \E / Y.
  \end{equation}
  Instead of $\Pi_f$ we will also often write $f_*$ and we refer to this
  functor as the \emph{pushforward along} $f$\index{pushforward}.
\end{defi}

\begin{nameddefi}{bifibration}{\cite{benabou-roubaud}}
 A functor \(F: \mathcal{F} \to \mathcal{E}\) is \emph{bifibration}
 \index{bifibration} when it is a Grothendieck fibration as well as an
 opfibration.
\end{nameddefi}

In this paper, most notably in Section~\ref{sec:AWFSFromM}, we use the
following important property of a locally cartesian closed category with finite
colimits.  The functor
\begin{equation}\label{eq:bifibration-condition}
 \cod : \E^{\to} \to \E
\end{equation}
from the arrow category of \(\E\) to \(\E\) which sends arrows to their
codomain is a bifibration satisfying the \emph{Beck-Chevalley
condition}\index{Beck-Chevalley condition|textbf} (see
Box~\ref{box:beck-chevalley}). Besides pullbacks (which we already have), this
also requires the existence of arbitrary pushouts, together with a
compatibility condition between them.  For~\eqref{eq:bifibration-condition},
this condition is as follows. Given a commutative cube:
\begin{equation}\label{eq:beck-chevalley-condition}
  \begin{tikzpicture}[baseline={([yshift=-.5ex]current bounding box.center)}]
    \matrix (m) [matrix of math nodes, row sep=3em,
    column sep=3em]{ 
	    |(c')| {C'} & |(d')| {D'} \\
	    |(c)| {C} & |(d)| {D}
    \\};
    \matrix (n) [matrix of math nodes, row sep=3em,
    column sep=3em, position=30:-1.7 from m]{ 
	    |(a')| {A'} & |(b')| {B'} \\
	    |(a)| {A} & |(b)| {B}
    \\};
    \begin{scope}[every node/.style={midway,auto,font=\scriptsize}]
      \path[->]
        (c) edge (d)
        (c') edge (c)
        (d') edge (d)
        (c') edge (d');
      \path[->]
        (a') edge (b')
        (a') edge (a)
        (b') edge (b)
        (a) edge (b);
      \path[->]
        (b) edge (d)
        (a) edge (c)
        (a') edge (c')
        (b') edge (d');
      \end{scope}
  \end{tikzpicture} 
\end{equation}
such that
\begin{enumerate}[(i)]
  \item The bottom square \(ABCD\) is a pullback; 
  \item The right square \(B'D'BD\) is a pushout;
  \item The back square \(A'B'AB\) is a pullback;
\end{enumerate}
then the left square \(A'C'AC\) is a pushout if and only if the front square
\(C'D'CD\) is a pullback. In many cases, we obtain this condition from the
following lemma:
\begin{lemm}{lcc-beck-chevalley}
  If \(\E\) is locally cartesian closed with finite colimits, then the codomain
  bifibration satisfies the Beck-Chevalley condition. That is, for all
  cubes~\eqref{eq:beck-chevalley-condition} satisfying (i)--(iii), the above
  compatibility condition holds.
\end{lemm}
\begin{proof}
  With pullbacks and pushouts present, it is easy to see that \(\cod\) is in
  fact a bifibration. For bifibrations, it is enough to only show one direction
  of the compatibility assertion, for which see Box~\ref{box:beck-chevalley} on
  page~\pageref{box:beck-chevalley}. Hence it suffices to show that when the
  front square is a pullback square, the left square is a pushout.  

  In a locally cartesian closed category, pullback along any \(f: C \to D\) has
  a right adjoint and therefore preserves colimits.  Since the right-hand
  square is a pushout, it is also a pushout in \(\mathcal{E}/D\) and hence
  preserved by pullback along \(f\). Hence, as the front square is in fact a
  pullback, the left square is a pushout in \(\mathcal{E}/C\). It follows that
  it is also a pushout square in \(\E\). This proves the assertion.
\end{proof}

\begin{textbox}{The Beck-Chevalley condition for bifibrations\label{box:beck-chevalley}}
The following definition of the \emph{Beck-Chevalley condition}
\index{Beck-Chevalley condition} for 
bifibrations originates from B\'{e}nabou-Roubaud~\cite{benabou-roubaud}.
For a bifibration \(\mathcal{F} \to \E\), the condition is satisfied
when, as drawn in the diagram below,
for every commutative square in the fibre
above a pullback square
 	\begin{equation}\label{eq:defn:beck-chevalley}
  \begin{tikzpicture}[baseline={([yshift=-.5ex]current bounding box.center)}]
    \matrix (m) [matrix of math nodes, row sep=3em,
    column sep=3em]{ 
	    |(ua)| {\bullet} & |(ub)| {\bullet} \\
	    |(uc)| {\bullet} & |(ud)| {\bullet}
    \\};
    \node (c) [position=-75:1.2 from uc] {$\bullet$};
    \node (d) [position=-75:1.2 from ud] {$\bullet$};
    \node (a) [position=40:0.8 from c] {$\bullet$};
    \node (b) [position=40:0.8 from d] {$\bullet$};
    \begin{scope}[every node/.style={midway,auto,font=\scriptsize}]
	    \path[-open triangle 45]
	      (ua) edge node [above] {$f'$} (ub);
      \path[open triangle 45 reversed->]
	      (ub) edge node [right] {$g$} (ud);
      \path[->]
	      (a) edge node [above] {$k'$} (b)
	      (uc) edge node [below] {$f$} (ud)
	      (a) edge node [above left]  {$l'$} (c)
	      (c) edge node [below] {$k$} (d)
	      (b) edge node [right] {$l$} (d)
	      (ua) edge node [left] {$g'$}  (uc);
      \path[-]
	      (ua) edge [dashed] (a)
	      (ub) edge [dashed] (b)
	      (uc) edge [dashed] (c)
	      (ud) edge [dashed] (d);
      \end{scope}
  \end{tikzpicture} 
\end{equation} 
such that \(f'\) is cartesian and \(g\) is cocartesian, one has that \(f\) is
cartesian if and only if \(g'\) is cocartesian. Note that in fact it is equivalent
that only one of these two directions hold: this can be seen by factorizing the
diagonal either as a precomposition with a cocartesian arrow,
or a postcomposition with a cartesian arrow.

If the bifibration comes with a choice of cartesian and cocartesian lifts in the form
of fibrewise pullback \({(-)}^*\) and pushforward \({(-)}_*\), this can be written as an
isomorphism:
\[
	l'_*k'^* \cong k^* l_*
\]
for every pullback square in the base as drawn.
\end{textbox}

Another property that we need later (specifically in the proof of
\refprop{trivFib->effTrivFib}) is that any locally cartesian closed category
with finite colimits is \emph{coherent}\index{coherent (category)}. Recall that
a \emph{coherent category} is a category where subobjects have finite unions
preserved by pullback.  For this to hold, it is enough that this category has
finite coproducts and is \emph{regular} (see
Johnstone~\cite{JOHNSTONE-ELEPHANT1}, A1.4).
\begin{lemm}{lccFinCoImpliesCoherent}
  A locally cartesian closed category \(\E\) with finite colimits
  is coherent.
\end{lemm}
\begin{proof}
  By the preceding paragraph, it is enough to show that \(\E\) is a regular
  category. This follows from Johnstone~\cite{JOHNSTONE-ELEPHANT1}, A1.3.5,
  where we use that pullback is a left adjoint and hence preserves colimits.
\end{proof}

\subsection{Fibred structure}

\begin{defi}{fibredstructure}
  Let \ct{E} be a category with finite limits and write $\ct{E}^{\to}_{\rm
  cart}$ for the category of arrows in \ct{E} and pullback squares between
  them. A presheaf \(\cat{fib}\) on $\ct{E}^{\to}_{\rm cart}$
 \[ \cat{fib}: (\ct{E}^{\to}_{\rm cart})^{\op} \to \Sets \] will also be called
 a \emph{notion of fibred structure}\index{notion of fibred structure|textbf}
 \index{fibred structure|seeonly {notion of fibred structure}}. A morphism of
 notions of fibred structure is simply of a morphism of presheaves and and two
 notions of fibred structure will be called \emph{equivalent} if they are
 naturally isomorphic as presheaves.

 Lastly, a notion of \emph{cofibred structure}\index{notion of cofibred
 structure|textbf}\index{cofibred structure|seeonly {notion of cofibred
 structure}} refers to the dual, which is the same as a presheaf
 \[ \cat{cofib}: \ct{E}^{\to}_{\rm cocart} \to \Sets \]
 where \(\ct{E}^{\to}_{\rm cocart}\) denotes the category of
 arrows in \(\E\) with pushout squares between them.
 \end{defi}

 \begin{defi}{localfibredstructure}\index{local notion of fibred structure|textbf}
 Let \(\cat{fib}\) be a notion of fibred structure on a category \ct{E}. We will call the notion of fibred structure \(\cat{fib}\) \emph{local} (or \emph{locally representable}) if the following holds for any small diagram $D: \ct{I} \to \ct{E}^{\to}_{\rm cart}$ with colimit $f$ and colimiting cocone $(\sigma_i: Di \to f \, : \, i \in {\rm Ob}(\ct{I}))$:
\begin{quote}
If we can choose fibration structures $x_i \in {\cat{fib}}(Di)$ for any $i \in {\rm Ob}(\ct{I})$ such that ${\cat{fib}}(D\alpha)(x_i) = x_j$ for any $\alpha: j \to i$ in $\ct{I}$, then there exists a unique fibration structure $x \in {\cat{fib}}(f)$ such that ${\cat{fib}}(D\sigma_i)(x) = x_i$.
\end{quote}
\end{defi}

We will mainly be interested in notions of fibred structure on presheaf
categories (in fact, on the category of simplicial sets), in which case the
notion of locality can be defined in a different way. In the following
proposition, the Yoneda embedding for a category of presheaves \(\mathcal{E}\)
on \(\mathbb{C}\) is denoted \(y_{(-)} : \mathbb{C} \to \E\).

\begin{prop}{localityforpresh}
 Suppose \ct{E} is the category of presheaves on $\mathbb{C}$, and let
 ${\cat{fib}}$ be a notion of fibred structure on \ct{E}. Then ${\cat{fib}}$ is
 local if and only if the following holds for any morphism $f: Y \to X$ in
 \ct{E}: given, for any $x \in X(C)$, a chosen fibration structure $s_x$ on a
 (chosen) pullback $f_x: Y_x \to y_C$ as in
 \begin{displaymath}
  \begin{tikzcd}
   Y_x \ar[d, "f_x"] \ar[r] & Y \ar[d, "f"] \\
    y_C  \ar[r, "x"] & X
  \end{tikzcd}
 \end{displaymath}
  such that for any $x \in X(C)$ and $\alpha: D \to C$ in $\mathbb{C}$, the fibration structure $s_x$ on $f_x$ pulls back to the one chosen on $f_{x \cdot \alpha}$ for the pullback square
 \begin{displaymath}
  \begin{tikzcd}
   Y_{x \cdot \alpha} \ar[r] \ar[d, "f_{x \cdot \alpha}"] & Y_x \ar[d, "f_x"] \\
   y_D \ar[r, "\alpha"] & y_C,
  \end{tikzcd}
 \end{displaymath}
  over $f$, then there exists a unique fibration structure $s \in {\cat{fib}}(f)$ which pulls back to $s_x$ for any pullback square of the first type.
\end{prop}
\begin{proof}
 The equivalence uses standard properties of presheaf categories: every object
 is a colimit of representables, and since pullback along \(f\) preserves
 colimits, the condition becomes a special case of
 \refdefi{localfibredstructure}.

 For the other direction, one can reduce \refdefi{localfibredstructure} to the
 special case by taking pullbacks along representables:
            \[
          \begin{tikzpicture}[baseline={([yshift=-.5ex]current bounding box.center)}]
            \matrix (m) [matrix of math nodes, row sep=2em,
            column sep=3.5em]{
              |(a)| {Y_{i,x}} & |(b)| {Y_i} & |(y)| {Y} \\
              |(c)| {y_C} & |(d)| {X_i} & |(x)| {X} \\
            };
            \node (i) at (barycentric cs:b=1,d=1,y=1,x=1) {$\overset{\sigma_i}{\Rightarrow}$};
            \begin{scope}[every node/.style={midway,auto,font=\scriptsize}]
            \path[->]
                (a) edge (a -| b.west)
                (c) edge node [below] {$x$} (c -| d.west)
                (a) edge (c)
                (b) edge node [left] {$D(i)$} (d)
                (y) edge node [right] {$f$} (x)
                (b) edge (b -| y.west)
                (d) edge (d -| x.west);
            \end{scope}\end{tikzpicture}
            \]
 Since \(\cod : \E^{\to} \to \E\) is a left adjoint, it preserves colimits, so
 every \(x': y_C \to X\) factors through some \(X_i\) since we are in a
 category of presheaves.  Hence the fibred structure determines precisely the
 input data for the special case.  The pullback property for the unique induced
 fibration follows from the uniqueness condition on each of the \(D(i)\) that
 follows from the special case.
\end{proof}

\begin{rema}{similarsources}
  A local notion of fibred structure is a structured analogue of a local class
  of maps as in \cite[Remark 4.4]{sattler18} and \cite[Definition
  3.7]{cisinski14}, for instance. An earlier structured analogue appears in
  Shulman's paper \cite[Proposition 3.18]{shulman19}. The definition we gave
  here is a special case of his, because we demand that the fibration
  structures are strictly functorial under pullback, rather than
  pseudofunctorial.
\end{rema}

A motivation behind a local notion fibred structure that is of interest to type
theory is that it allows to give a \emph{universal} fibration\index{universal
fibration}. For this to work we assume that we are given some notion of
``smallness'' in the metatheory (like being finite, or countable, or being
bound in size by some regular cardinal, or belonging to some Grothendieck
universe). If \ct{E} is the category of presheaves on a \emph{small} category
$\mathbb{C}$, then we will call a morphism of presheaves $p: Y \to X$
\emph{small} if each fibre of $p_C: Y(C) \to X(C)$ is small. A presheaf $X$
will be called small if $X \to 1$ is a small morphism of presheaves.

Using the same notation as in \refprop{localityforpresh}, we have the following
theorem:
\begin{theo}{localfsimpliesuniverse}
 Suppose \ct{E} is the category of presheaves on a small category $\mathbb{C}$,
 and let ${\cat{fib}}$ be a local notion of fibred structure on \ct{E}.

 Then, there is a small morphism of presheaves \(\pi: E \to U\) equipped with a
 unique fibration structure \(s_\pi \in \cat{fib}(\pi)\) such that every small
 \(p: Y \to X \) equipped with a structure of fibration \(s_p \in
 \cat{fib}(p)\) can be uniquely presented as a pullback of \(\pi\) with this
 structure.
\end{theo}
\begin{proof}
  The idea is to modify the Hofmann-Streicher construction as in \cite{hofmannstreicher97}. That is, we define 
  \begin{align*}
    U(C) &= \{(p, s_p) \varco p : X \to y_C 
    \text{ small, and }s_p \in \cat{fib}(p) \} \\
    E(C) &= \{(x, p, s_p) \varco p : X \to y_C 
    \text{ small, }s_p \in \cat{fib}(p) \text{ , } x \in X(C) \} \\
    \end{align*}
  and for any \(\alpha: D \to C\), maps
  \begin{align*}
    U(\alpha) : U(C) \to U(D) \\
    E(\alpha) : E(C) \to E(D) 
  \end{align*}
  by pullback:
  \begin{align*}
    U(\alpha)(p, s_p) &= ({y_\alpha}^*p, \cat{fib}({y_\alpha}^*p)(s_p))\\
    E(\alpha)(x, p, s_p) &= (X(\alpha)(x), {y_\alpha}^*p, \cat{fib}({y_\alpha}^*p)(s_p)).
  \end{align*}
  We claim that the projection \(\pi : E \to U\) has a fibration structure. Indeed,
  this follows from the local property of \(\cat{fib}\): for every pullback diagram
 \[
\begin{tikzpicture}[baseline={([yshift=-.5ex]current bounding box.center)}]
 \matrix (m) [matrix of math nodes, row sep=2em,
 column sep=3.5em]{
   |(p)| {X} & |(e)| {E} \\
   |(c)| {y_C} & |(u)| {U} \\
 };
 \begin{scope}[every node/.style={midway,auto,font=\scriptsize}]
 \path[->]
 (p) edge (e)
 (c) edge node [above] {$(p, s_p)$} (u.west |- c)
 (e) edge node [right] {$\pi$} (u)
 (p) edge node [left] {$p$} (c);
 \end{scope}\end{tikzpicture}
 \]
 \(s_p\) is a natural choice of fibration structure on \(p\) that satisfies
 the condition in \refprop{localityforpresh}. By the Yoneda lemma, every
 arrow \(u : y_C \to U\) is of this type. Hence there is a unique fibration
 structure on \(\pi\) such that for every \(p : X \to y_C\), any fibration
 structure on \(p\) can be  presented as a pullback of \(\pi\) in the way
 shown. This immediately generalises to any small morphism $p: Y \to X$, using again that \(\cat{fib}\) is local.

 The construction of $E$ and $U$ as we have just presented them does not quite work, because the action $U(\alpha)$ and $E(\alpha)$ is not strictly functorial. However, we can solve this issue by replacing the codomain fibration over \ct{E} by an equivalent strict fibration. For instance, we can use the fact that the slice category of the category of presheaves over $P$ is equivalent to the category of presheaves over $y \downarrow P$, the category of elements of $P$; this allows us to replace the codomain fibration over \ct{E} by an equivalent split fibration whose fibre over a presheaf $P$ in the base is the category of presheaves over $y \downarrow P$.
\end{proof}

\subsection{Double categories of left and right lifting structures}
\label{ssec:double-categories}
We recall the definition of a \emph{double category}:
\begin{defi}{double-category}\index{double category|textbf}
	A \emph{double category} \(\Abb\) consists of:
    \begin{enumerate}[(i)]
      \item A collection of objects together with two separate category
        (morphism) structures on it, called \emph{horizontal} and
        \emph{vertical} morphisms.\index{horizontal morphism}
        \index{vertical morphism}
  \item A special category structure whose objects are the vertical morphisms,
    and whose arrows are called \emph{squares}.\index{square (between vertical
    morphisms)} The special property is that every \emph{square} from a
    vertical morphism \(u\) to another vertical morphism \(v\) has a
    `pointwise' domain, and a `pointwise' codomain given by horizontal
    morphisms:
            \[
                f : \dom u \to \dom v \text{, } g: \cod u \to \cod v
            \]
	    Moreover, composition of squares respects composition of these horizontal
	    morphisms and identity squares have identity horizontal morphisms for their
	    pointwise domain and codomain. 
	    
	    Further, there is a `pointwise' or `vertical' composition operation
	    of squares with matching pointwise domains and codomains, which
	    extends composition of vertical morphisms.  We often think of a
	    square \(s: u\to v\) as filling in a diagram of horizontal and
	    vertical arrows:
            \[
          \begin{tikzpicture}[baseline={([yshift=-.5ex]current bounding box.center)}]
            \matrix (m) [matrix of math nodes, row sep=2em,
            column sep=3.5em]{
                |(a)| {A} & |(b)| {B} \\
                |(c)| {C} & |(d)| {D} \\
            };
	    \node (i) at (barycentric cs:a=1,b=1,c=1,d=1) {$\overset{s}{\Rightarrow}$};
            \begin{scope}[every node/.style={midway,auto,font=\scriptsize}]
            \path[->]
                (a) edge node [above] {$f$} (b)
                (c) edge node [below] {$g$} (d);
            \path[-{Stealth[open]}]
                (a) edge node [left] {$u$} (c)
                (b) edge node [right] {$v$} (d);
            \end{scope}\end{tikzpicture}
            \]
	    These diagrams can be composed both horizontally (ordinary composition of
	    squares) and vertically (pointwise composition).
    \end{enumerate}
    Alternatively, one can verify that a double category is the same thing as an
    internal category in the `category' of large categories:
	\[
\begin{tikzpicture}[baseline={([yshift=-.5ex]current bounding box.center)}]
    \matrix (m) [matrix of math nodes, row sep=2em,
    column sep=3.5em]{
	    |(c)| {\Lbb_1 \times_{\Lbb_0} \Lbb_1} &  |(a)| {\Lbb_1} 
	    & |(b)| {\Lbb_0} \\
    };
    \begin{scope}[every node/.style={midway,auto,font=\scriptsize}]
    \path[->]
	    (c) edge node [anchor=center, fill=white] {$\circ$} (a)
	    (a) edge [transform canvas={yshift=+7pt}] node [above] {$\dom$} (b)
	    (b) edge node [anchor=center, fill=white] {id} (a)
	    (a) edge [transform canvas={yshift=-7pt}] node [below] {$\cod$} (b);
    \end{scope}\end{tikzpicture}
    \]
    where `\(\circ\)' denotes vertical, or pointwise,
    composition, whilst horizontal composition is the
    composition of the two categories \(\Lbb_1\) (vertical morphisms and squares) and
    \(\Lbb_0\) (objects and horizontal morphisms).

    For double categories \Abb, \Bbb, a \emph{double functor}\index{double
    functor} \(F : \Abb \to \Bbb\) is a compatible triple of functors (which we
    can denote by the same \(F\)) between categories of horizontal morphisms,
    vertical morphisms, and squares, which in addition respects pointwise
    composition.  Here \emph{compatible} means that the image of a square as
    drawn above looks like:
    \[
          \begin{tikzpicture}[baseline={([yshift=-.5ex]current bounding box.center)}]
            \matrix (m) [matrix of math nodes, row sep=2em,
            column sep=3.5em]{
                |(a)| {F(A)} & |(b)| {F(B)} \\
                |(c)| {F(C)} & |(d)| {F(D)} \\
            };
            \begin{scope}[every node/.style={midway,auto,font=\scriptsize}]
            \path[->]
                (a) edge node [above] {$F(f)$} (b)
                (c) edge node [above] {$F(g)$} (d);
	\path[-{Stealth[open]}]
                (a) edge node [left] {$F(u)$} (c)
                (b) edge node [right] {$F(v)$} (d);
            \end{scope}\end{tikzpicture}
            \]
     Again, one can verify that a double functor between double categories is
     the same thing as an internal functor between the corresponding internal
     categories in the category of large categories.
\end{defi}
\begin{exam}{double-cat-of-arrows}
    The typical example of a double category is the category of arrows
    of any category \(\E\). The horizontal and vertical arrows are
    both given by the category structure, and squares are given
    by commutative squares. We will denote this double category
    by \(\sq{\E}\).\index{squares (double category)}
\end{exam}

In this paper, we often work with double categories over the category of squares
of some category \(\E\), i.e.\ double functors \(\Abb \to \sq{\E}\). For the next
construction, we begin with such a double functor, denoted \(I : \Lbb \to
\sq{\E}\). For a morphism \(p : Y \to X\) in \(\E\), a 
\emph{right lifting structure} with respect to \(I\) consists of:\index{right lifting structure|textbf}
\begin{enumerate}[(i)]
    \item A family \(\phi_{-,-}(-)\) of arrows in \(\E\) consisting of the following. 
        For every vertical morphism \(v\) in \(\Lbb\), and every
        commutative square in \(\E\), as in the solid part of the
        following diagram:
        \[
          \begin{tikzpicture}[baseline={([yshift=-.5ex]current bounding box.center)}]
            \matrix (m) [matrix of math nodes, row sep=2em,
            column sep=3.5em]{
                |(a)| {A} & |(b)| {Y} \\
                |(c)| {B} & |(d)| {X} \\
            };
            \begin{scope}[every node/.style={midway,auto,font=\scriptsize}]
            \path[->]
                (a) edge node [above] {$f$} (b)
                (c) edge node [above] {$g$} (d)
                (a) edge node [left] {$I(v)$} (c)
                (b) edge node [right] {$p$} (d)
                (c) edge [dashed] node [anchor=center, fill=white]
                {$\phi_{f,g}(v)$} (b);
            \end{scope}\end{tikzpicture}
            \]
            the arrow \(\phi_{f,g}(v) : B \rightarrow Y\) is a diagonal filler
            as drawn which makes the diagram commute. A commutative square like
            the above is called a \emph{lifting problem}\index{lifting problem}.
          \item  The \textbf{`horizontal' condition}\index{compatibility condition}: the compatibility condition that for every such
          \(\phi_{f,g}(v)\), and every square \(v' \rightarrow v\) whose image
          under \(I\) is given by the left-hand commutative square in the solid
          part of the diagram below:\index{horizontal condition}
       \[
          \begin{tikzpicture}[baseline={([yshift=-.5ex]current bounding box.center)}]
            \matrix (m) [matrix of math nodes, row sep=3em,
            column sep=6em]{
                |(a)| {A'} & & |(c)| {A} & |(c')| {Y} \\
                |(b)| {B'} & & |(d)| {B} & |(d')| {X} \\
            };
            \begin{scope}[every node/.style={midway,auto,font=\scriptsize}]
            \path[->]
                (a) edge node [left] {$I(v')$} (b)
                (a) edge node [above] {$I(f')$} (c)
                (b) edge node [below] {$I(g')$} (d)
                (c) edge node [pos=0.5, fill=white, anchor=center] {$I(v)$} (d)
                (c) edge node [above] {$f$} (c')
                (d) edge node [below] {$g$} (d')
                (c') edge node [right] {$p$} (d');
            \path[->]
         (b) edge [dashed] node [pos=0.3, anchor=center, fill=white]
         {$\phi_{f.I(f'), g.I(g')}(v')$} (c')
         (d) edge [dashed] node [anchor=center, fill=white] 
         {$\phi_{f, g} (v)$} (c');
 \end{scope}\end{tikzpicture}\text{,}
 \]
            the drawn morphisms make the diagram commute, i.e.
            \begin{equation*}
                \phi_{f,g}(v) . I(g') = \phi_{f. I(f'), g . I(g')} (v')
            \end{equation*}
        \item The \textbf{`vertical' condition}: the compatibility  condition that when \(v,
          w\) is a composable pair of vertical arrows, i.e.  \(\cod v = \dom
          w\), we have:\index{vertical condition}
            \begin{equation*}
                \phi_{\phi_{f,I(w).g}(v), g}(w) = \phi_{f,g}(w.v) 
            \end{equation*}
            In diagrammatic notation, this means that the two ways to fill the
            below diagram, either in two steps or in one go, are the same:
            \[
          \begin{tikzpicture}[baseline={([yshift=-.5ex]current bounding box.center)}]
            \matrix (m) [matrix of math nodes, row sep=2em,
            column sep=4.5em]{
                |(a)| {A} & |(b)| {Y} \\
                |(n)| {A'} &  {} \\ 
                |(c)| {B} & |(d)| {X} \\
            };
            \begin{scope}[every node/.style={midway,auto,font=\scriptsize}]
            \path[->]
                (a) edge node [above] {$f$} (b)
                (c) edge node [above] {$g$} (d)
                (a) edge node [left] {$I(v)$} (n)
                (n) edge node [left] {$I(w)$} (c)
                (b) edge node [right] {$p$} (d)
                (n) edge [dashed] (b)
                (c) edge [dashed] (b);
            \end{scope}\end{tikzpicture}
            \]
\end{enumerate}
When \((u,v) : p' \to p\) is a pullback square as drawn in the diagram below,
and \(\phi\) is a right lifting structure on \(p\) with respect to \(I\), the
universal property of the pullback induces a right lifting structure
\({(u,v)}^*\phi\) for \(p'\) with respect to \(I\):
\[
  \begin{tikzpicture}[baseline={([yshift=-.5ex]current bounding box.center)}]
    \matrix (m) [matrix of math nodes, row sep=3em,
    column sep=6em]{
	|(a)| {A} & |(c)| {Y'} & |(c')| {Y} \\
	|(b)| {B} & |(d)| {X'} & |(d')| {X} \\
    };
    \begin{scope}[every node/.style={midway,auto,font=\scriptsize}]
    \path[->]
	(a) edge node [left] {$I(v)$} (b)
	(a) edge node [above] {$f$} (c)
	(b) edge node [below] {$g$} (d)
	(c) edge node [pos=0.5, fill=white, anchor=center] {$p'$} (d)
	(c) edge node [above] {$u$} (c')
	(d) edge node [below] {$v$} (d')
	(c') edge node [right] {$p$} (d');
    \path[->]
	    (b) edge [dashed] node [anchor=center, fill=white] {${(u,v)}^*\phi$} (c)
 (b) edge [dashed, bend right=10] node [pos=0.8,anchor=center, fill=white] 
 {$\phi_{u.f, v.g} (v)$} (c');
\end{scope}\end{tikzpicture}\text{,}
 \]
It is easy to verify that this is a right lifting structure for \(p'\).
 This conclusion is summarised as follows:
\begin{prop}{right-lifting-is-fibred-structure}
	With this pullback action, there is, for every \(I : \Lbb \to \sq{\E}\)
	as above, a fibred structure on \(\E\) which sends each arrow to the
	set of right lifting structures with respect to \(I\) on it.
	We denote this fibred structure by the functor:
	\[
	\Drl{I}(-): {(\ct{E}^{\to}_{\rm cart})}^{op} \to \Sets
\]\(\hfill \square\)
\end{prop}

When \((q : Z \rightarrow Y, \phi)\), \((q: Y \rightarrow X, \psi)\)
are two composable arrows in \(\E\) together with a right lifting structure
with respect to \(I\) on them, there is a candidate right lifting structure on the
composition \(q.p\) defined by step-wise lifts:
\[
  \begin{tikzpicture}[baseline={([yshift=-.5ex]current bounding box.center)}]
    \matrix (m) [matrix of math nodes, row sep=2em,
    column sep=4.5em]{
        |(a)| {A} & |(b)| {Z} \\
                  & |(o)| {Y} \\ 
        |(c)| {B} & |(d)| {X} \\
    };
    \begin{scope}[every node/.style={midway,auto,font=\scriptsize}]
    \path[->]
        (a) edge node [above] {$f$} (b)
        (c) edge node [above] {$g$} (d)
        (a) edge node [left] {$I(v)$} (c)
        (b) edge node [right] {$p$} (o)
        (o) edge node [right] {$q$} (d)
        (c) edge [dashed] node [fill=white, anchor=center] 
        {$\phi_{p. f, g} (v)$} (o)
        (c) edge [bend left=10, dashed] node [fill=white, anchor=center]
        {$\psi_{f, \phi_{p. f , g}(v)}(v)$} (b);
    \end{scope}\end{tikzpicture}
            \]
To verify that this is a right-lifting structure, we have to verify
two conditions. For (ii), this is: 
\begin{align*}
    \psi_{f, \phi_{p. f , g}(v)}(v).I(g') &= \psi_{f . I(f'), 
    \phi_{p. f , g}(v) . I(g')} (v') \\
                                          &=\psi_{f . I(f'), 
                                          \phi_{p. f. I(f'), g. I(g')}(v')}(v')
\end{align*}
For condition (iii), this is better done diagrammatically:

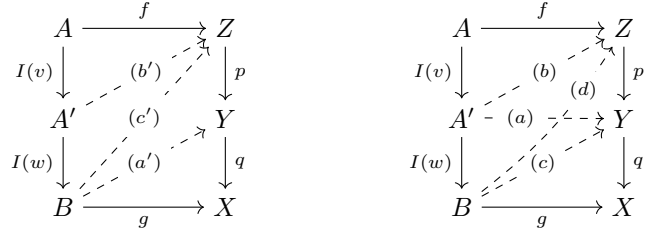
\begin{minipage}{0.4\textwidth}\centering
\[
  \begin{tikzpicture}[baseline={([yshift=-.5ex]current bounding box.center)}]
    \matrix (m) [matrix of math nodes, row sep=2em,
    column sep=4.5em]{
        |(a)| {A} & |(z)| {Z} \\
        |(a')| {A'} &|(y)|  {Y} \\ 
        |(b)| {B} & |(x)| {X} \\
    };
    \begin{scope}[every node/.style={midway,auto,font=\scriptsize}]
    \path[->]
        (a) edge node [left] {$I(v)$} (a')
        (a') edge node [left] {$I(w)$} (b)
        (z) edge node [right] {$p$} (y)
        (y) edge node [right] {$q$} (x)
        (a) edge node [above] {$f$} (z)
        (b) edge node [below] {$g$} (x)
        (b) edge [dashed] node [fill=white, anchor=center] {$(a')$} (y)
        (b) edge [dashed] node [fill=white, anchor=center] {$(c')$} (z)
        (a') edge [dashed] node [fill=white, anchor=center] {$(b')$} (z);
    \end{scope}\end{tikzpicture}
\]
\end{minipage}
\begin{minipage}{0.4\textwidth}\centering
\[
  \begin{tikzpicture}[baseline={([yshift=-.5ex]current bounding box.center)}]
    \matrix (m) [matrix of math nodes, row sep=2em,
    column sep=4.5em]{
        |(a)| {A} & |(z)| {Z} \\
        |(a')| {A'} &|(y)|  {Y} \\ 
        |(b)| {B} & |(x)| {X} \\
    };
    \begin{scope}[every node/.style={midway,auto,font=\scriptsize}]
    \path[->]
        (a) edge node [left] {$I(v)$} (a')
        (a') edge node [left] {$I(w)$} (b)
        (z) edge node [right] {$p$} (y)
        (y) edge node [right] {$q$} (x)
        (a) edge node [above] {$f$} (z)
        (b) edge node [below] {$g$} (x)
	(a') edge [dashed] node [pos=0.3, anchor=center, fill=white] {$(a)$} (y)
        (a') edge [dashed] node [anchor=center, fill=white] {$(b)$} (z)
        (b) edge [dashed] node [anchor=center, fill=white] {$(c)$} (y)
        (b) edge [bend right=10, dashed] 
        node [pos=0.75, anchor=center, fill=white] {$(d)$} (z);
    \end{scope}\end{tikzpicture}
\]
\end{minipage}

Consider the two diagrams above. On the left-hand side, 
a lift is obtained in two steps $(a')$ and $(c')$, but we immediately note
that the second step could have been done in two steps by first finding
$(b')$ and then finding the same $(c')$. On the right-hand side,
a lift is found in four steps $(a)$-$(d)$. Note that
\begin{equation*}
    (b') = \psi_{f, (a'). I(w)}(v) 
    = \psi_{f, (a)}(v) = (b) 
\end{equation*}
And similarly:
\begin{equation*}
	(c) = \phi_{p.(b), g}(w) = \phi_{(a), g} (w)
    = \phi_{\phi_{p.f, I(w) . g}(v), g}(w) = \phi_{p.f, g}(w.v) =  (a')
\end{equation*}
Therefore it follows that $(d)$ and $(c')$ are the same. So the composed
lifting structure indeed satisfies the third property above. As a consequence,
the following definition is just:
\begin{defi}{rlp-double-cat}
    Suppose \(I : \Lbb \to \sq{\E}\) is a double category over \(\sq{\E}\). We
    can define a new double category \(\Drl{I} : \Drl{\Lbb} \to \sq{\E}\) as
    follows:
    \begin{enumerate}[(i)]
	    \item Objects are objects of \(\E\), and horizontal morphisms are morphisms in 
		    \(\E\)
	    \item A vertical morphism \(Y \to X\) is a pair \((p : Y \to X , \phi)\) where
		    \(\phi\) is a right lifting structure for \(p\) with respect to \(I\);
		    that is, \(\phi \in \Drl{I}(p)\) in the notation of 
		    \refprop{right-lifting-is-fibred-structure}.
		    Composition of vertical morphisms is defined as above.
	    \item A square \((p', \phi') \to (p, \phi')\) between vertical morphisms is a
		    commutative square \(p' \to p\) in \(\E\) as on the right hand side in
		    the diagram below, such that whenever 
		    there is a lifting problem:
	          
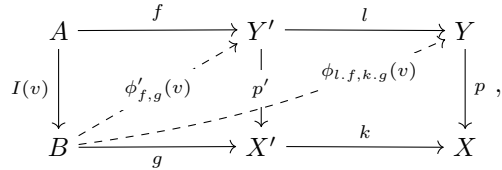
\[
		  \begin{tikzpicture}[baseline={([yshift=-.5ex]current bounding box.center)}]
		    \matrix (m) [matrix of math nodes, row sep=3em,
		    column sep=6em]{
			|(a)| {A} & |(y')| {Y'} & |(y)| {Y} \\
			|(b)| {B} & |(x')| {X'} & |(x)| {X} \\
		    };
		    \begin{scope}[every node/.style={midway,auto,font=\scriptsize}]
		    \path[->]
			(a) edge node [left] {$I(v)$} (b)
			(a) edge node [above] {$f$} (y')
			(b) edge node [below] {$g$} (x')
			(y') edge node [pos=0.5, fill=white, anchor=center] {$p'$} (x')
			(x') edge node [above] {$k$} (x)
			(y') edge node [above] {$l$} (y)
			(y) edge node [right] {$p$} (x)
		 (b) edge [dashed, bend right=8] node [pos=0.8, anchor=center, fill=white]
		 {$\phi_{l.f, k.g}(v)$} (y)
		 (b) edge [dashed] node [anchor=center, fill=white] 
		 {$\phi'_{f, g} (v)$} (y');
	 \end{scope}\end{tikzpicture}, \]
	 	     the induced diagram as drawn commutes, that is:
		     \[
			\phi_{l.f, k.g}(v) = l . \phi'_{f,g}(v).
		     \]
\end{enumerate}
Note that it needs to be checked that vertical composition of squares is compatible
with the composition operation on vertical morphisms, but this follows easily from
the definition of a square.
\end{defi}
Similarly, there is a notion of \emph{left lifting structure}\index{left
lifting structure|textbf} for an arrow \(f : A \to B\) with respect to a double
category \(J : \Rbb \to \sq{\E}\). This consists of a family of fillers for
every commutative square:
\[
\begin{tikzpicture}[baseline={([yshift=-.5ex]current bounding box.center)}]
    \matrix (m) [matrix of math nodes, row sep=2em,
    column sep=3.5em]{
	|(a)| {A} & |(b)| {Y} \\
	|(c)| {B} & |(d)| {X} \\
    };
    \begin{scope}[every node/.style={midway,auto,font=\scriptsize}]
    \path[->]
	(a) edge node [above] {$f$} (b)
	(c) edge node [above] {$g$} (d)
	(a) edge node [left] {$f$} (c)
	(b) edge node [right] {$J(v)$} (d)
	(c) edge [dashed] node [anchor=center, fill=white]
	{$\phi_{f,g}(v)$} (b);
    \end{scope}\end{tikzpicture}
\]
such that three analogous conditions (i)-(iii) hold. For the sake of brevity, we will not
repeat those here. In fact, the definition is completely dual to a right lifting
structure, as follows. Define the \emph{opposite} of a double category \(\Abb\) to be the
double category \(\Abb^{\op}\) with horizontal and vertical arrows as well as squares
reversed, e.g.  \(\sq{\E}^{\op} \cong \sq{\E^{\op}}\), with the obvious extension to double
functors. Then a left lifting structure for an arrow \(f\) with respect to \(J: \Rbb \to
\sq{\E}\) is the same thing as a right lifting structure for \(f^{\op} : B \to A\) in
\(\E^{\op}\) with respect to \(J^{\op}: \Rbb^{\op} \to \sq{\E^{\op}}\).
From \refprop{right-lifting-is-fibred-structure}, it follows that there is a functor
\[
	\Dll{J} : \E^{\to}_{\text{cocart}} \to \Sets
\]
where \(\E^{\to}_{\text{cocart}}\) is the category of arrows in \(\E\)
and pushout squares between them, which sends an arrow to the set
of left lifting structures on it. In other words, left lifting structure
is a notion of \emph{cofibred structure} (\refdefi{fibredstructure}).
The following definition follows \refdefi{rlp-double-cat}:
\begin{defi}{llp-double-cat}
    Suppose \(J : \Rbb \to \sq{\E}\) is a double category over \(\sq{\E}\). We
    can define a new double category \(\Dll{J} : \Dll{\Rbb} \to \sq{\E}\) analogous
    to \refdefi{rlp-double-cat}, but where the vertical morphisms
    have a left lifting structure with respect to \(J\).
\end{defi}

Recall that the category of arrows of a category \(\E\) and squares
between them is denoted \(\E^{\to}\).
This is the same as the category structure that exists on the vertical morphisms in
\(\sq{\E}\).  The following proposition allows us to define left and right lifting
structures for mere `categories of arrows' \(\mathcal{A} \to \E^{{\to}}\) as a special case
in \refdefi{simple-right-lifting-structure} below.  In the following proposition, the
category of (small) double categories and double functors is denoted by \DBL\@.
\begin{prop}{vertical-part-left-adj}
	The functor
	\[
		{(-)}_1 : \DBL/\sq{\E} \to \Cat / \E^{{\to}}
	\]
	which takes a double functor \(I : \Lbb \to \sq{\E}\) to the category
	of vertical arrows and squares \(I_1 : \Lbb_1 \to \E^{{\to}}\) over \(\E^{{\to}}\)
        has a fully faithful left adjoint, i.e.\ it is a coreflection.	
\end{prop}
\begin{proof}
  We will define the left adjoint
  \[
    \freedbl{(-)} : \Cat / \E^{\to} \to \DBL / \sq{\E}
  \]
  as follows.
	Suppose \(I : \mathcal{L} \to \E^{{\to}}\) is a functor. Then we can define
	the double category \(\freedbl{I} : \freedbl{\mathcal{L}} \to \sq{\E}\) as:
	\begin{enumerate}[(i)]
		\item Objects of \(\freedbl{\mathcal{L}}\) are pairs \((i, v)\) where
			\(v\) is an object of \(\mathcal{L}\) and \(i \in \{0,1\}\)	
      \item Horizontal arrows \((i, s) : (i, u) \to (j, v)\)
		require that \(i = j\) and are given by morphisms 
	      \(s : u \to v\) in \(\Lbb\)
      \item The only non-identity vertical arrows are given by
		  \[
			v : (0, v) \to (1, v)
		  \]
	  \item   For every arrow \(s : u \to v\) in \(\mathcal{L}\), there is a
		  square \(s: 1_{(i, u)} \to 1_{(i, v)}\) for every
		  \(i \in \{0,1\}\), and a square \(s : u \to v\) between the
		  non-identity morphisms of (ii). Vertical composition of
		  squares is trivial, in that \(s. s = s\).
  \end{enumerate}
Alternatively, it can be presented by the following internal category in the
category of (large) categories:
\[
\begin{tikzpicture}[baseline={([yshift=-.5ex]current bounding box.center)}]
    \matrix (m) [matrix of math nodes, row sep=2em,
    column sep=7em]{
	    |(c)| {(\mathcal{L}+\mathcal{L}) 
	    + (\mathcal{L} + \mathcal{L})} &  |(a)| {\mathcal{L}+(\mathcal{L}+\mathcal{L})} 
						  & |(b)| {\mathcal{L}+\mathcal{L}} \\
    };
    \begin{scope}[every node/.style={midway,auto,font=\scriptsize}]
    \path[->]
	    (c) edge node [above] 
	    {$[[1_{\mathcal{L}}, 1_{\mathcal{L}}], 1_{\mathcal{L}+\mathcal{L}}]$} (a)
	    (a) edge [transform canvas={yshift=+7pt}] node [above] 
	    {$[\textsf{inl}, 1_{\mathcal{L}+\mathcal{L}}]$} (b)
	    (b) edge node [anchor=center, fill=white] {$\textsf{inr}$} (a)
	    (a) edge [transform canvas={yshift=-7pt}] node [below] 
	    {$[\textsf{inr}, 1_{\mathcal{L}+\mathcal{L}}]$} (b);
    \end{scope}\end{tikzpicture}
    \]
    The functor \(\freedbl{I}\) then sends a square \(s : u \rightarrow v\) to:
\[
          \begin{tikzpicture}[baseline={([yshift=-.5ex]current bounding box.center)}]
            \matrix (m) [matrix of math nodes, row sep=2em,
            column sep=3.5em]{
                |(a)| {A} & |(b)| {B} \\
                |(c)| {C} & |(d)| {D} \\
            };
	    \node (i) at (barycentric cs:a=1,b=1,c=1,d=1) {\textit{I(s)}};
            \begin{scope}[every node/.style={midway,auto,font=\scriptsize}]
            \path[->]
  	    (a) edge (b)
                (c) edge (d)
		(a) edge node [left] {$I(u)$} (c)
		(b) edge node [right] {$I(v)$} (d);
            \end{scope}\end{tikzpicture}
            \]
It is easy to see that this construction is functorial and fully faithful. The unit is
trivial, and the counit
\(\epsilon_I : \freedbl{(I_1)} \to I\) sends vertical arrows to vertical arrows, squares to
squares, and sends objects \((0, v)\) to \(\dom v\) and \((1, v)\) to \(
\cod v\), and the same for horizontal morphisms and their pointwise domain/codomain.
It is easy to see that this constitutes a counit and that \({(-)}_1\) is a coreflection.
\end{proof}

\begin{defi}{simple-right-lifting-structure}
	Suppose \(I : \mathcal{L} \to \E^{{\to}}\) is a functor. Then
	we can define a new functor \(\rlp{I} : \rlp{\mathcal{L}} \to
	\E^{{\to}}\) as: \[
	{\left(\Drl{(\freedbl{I})}\right)}_1 : 
	{\left(\Drl{(\freedbl{\mathcal{L}})}\right)}_1 \to \E^{{\to}}.
	\]
	Essentially, objects of \(\rlp{\mathcal{\L}}\) are pairs \((p : Y \to X, \phi)\)
	where \(p\) is an arrow in \(\E\) and \(\phi\) is a right lifting structure
	with respect to the arrows in the image of \(I\), but satisfying only the
	conditions (i) and (ii) since there is no non-trivial vertical composition.
	
	For a functor \(J : \mathcal{R} \to \E^{{\to}}\), there is similarly a category
	\(\llp{J} : \llp{\mathcal{R}} \to \E^{{\to}}\) of left lifting structures satisfying
	only conditions (i) and (ii).
	
	Note that in this construction, we have explicitly forgotten composition of
	lifting structures: even for a mere functor \(I : \mathcal{L} \to \E^{{\to}}\),
	the category of right-lifting structures
	\[
		\Drl{\left(\freedbl{I}\right)} : \Drl{ \left( \freedbl{\mathcal{L}} \right)}
		\to \sq{\E}
	\]
	is a double category with non-trivial vertical composition.
\end{defi}

To conclude this section, it may be worth noting that taking double categories of
left and right lifting structures are functorial constructions, and fact adjoint ones:
\[
\begin{tikzpicture}[baseline={([yshift=-.5ex]current bounding box.center)}]
    \matrix (m) [matrix of math nodes, row sep=2em,
    column sep=3.5em]{
	    |(a)| {\left(\DBL / \sq{\E}\right)^{\op}} 
	    & |(b)| { \DBL / \sq{\E}} \\
    };
    \begin{scope}[every node/.style={midway,auto,font=\scriptsize}]
    \path[->]
	    (a) edge [color=white] node [color=black, anchor=center, fill=white] {$\bot$} (b)
	    (b) edge [transform canvas={yshift=+5pt}] node [above] {$\Dll{(-)}$} (a)
	    (a) edge [transform canvas={yshift=-5pt}] node [below] {$\Drl{(-)}$} (b);
    \end{scope}\end{tikzpicture}
\]
for which see Proposition 18 of Bourke and Garner~\cite{Bourke-Garner}. It follows
from \refprop{vertical-part-left-adj} that in that case also
\[
\begin{tikzpicture}[baseline={([yshift=-.5ex]current bounding box.center)}]
    \matrix (m) [matrix of math nodes, row sep=2em,
    column sep=3.5em]{
	    |(a)| {\left(\Cat / \E^{{\to}}\right)^{\op}} 
	    & |(b)| { \Cat / \E^{{\to}}} \\
    };
    \begin{scope}[every node/.style={midway,auto,font=\scriptsize}]
    \path[->]
	    (a) edge [color=white] node [color=black, anchor=center, fill=white] {$\bot$} (b)
	    (b) edge [transform canvas={yshift=+5pt}] node [above] {$\llp{(-)}$} (a)
	    (a) edge [transform canvas={yshift=-5pt}] node [below] {$\rlp{(-)}$} (b);
    \end{scope}\end{tikzpicture}
\]
is an adjunction.  In this paper, we make extensive use of the notions of left
and right lifting structures in the context of algebraic weak factorisation
systems, which are defined next.

\subsection{Algebraic weak factorisation systems}\index{algebraic weak factorisation system}
The main content of this section is a pair of equivalent definitions of an
\emph{algebraic weak factorisation system} (\refdefi{awfs} and
\refprop{awfs-equivalent}). The definition of algebraic weak factorisation
system (AWFS) is of foundational importance to this paper. In
Sections~\ref{sec:dominances} and \ref{sec:AWFSFromM}, we demonstrate how two
AWFSs can be constructed in a category \ct{E} equipped with a \emph{dominance}
and a \emph{Moore structure}, respectively. In the later sections, these two
are combined to present the theory of effective fibrations. For the purposes of
reading Sections~\ref{sec:dominances} and beyond, the definitions of AWFS
referred to can be taken as a starting point.

Yet, as a starting point, these definitions involve quite a lot of assumptions
on structure. The reader who would prefer to start on a basis of necessity and
find out more about the ideas behind algebraic weak factorisation systems, we
would encourage to study this section and the next in more depth. Indeed,
as a foundation of this paper, it may be worth saying a little more about 
the motivation behind this structure.

The main reference for the theory of AWFSs is the paper by Bourke and
Garner~\cite{Bourke-Garner}, which contains the most important results
off-the-shelf. Another important source is Riehl~\cite{Riehl-11}, who addresses
some aspects in more depth, such as the subtleties around the distributive law.
For the present purposes, we have already given the most important definitions
in the previous section, namely that of a (left/right) lifing structure with
respect to a double category. We now add to this the notion of a \emph{functorial
factorisation}.
\begin{defi}{functorial-factorisation}\index{functorial factorisation}
    A \emph{functorial factorisation} for a category \(\E\) is a section of the
    composition functor
  \[\E^{\to} {}_{\dom} \times_{\cod} \E^{\to} \to \E^{\to}.\]
  Spelling this out, it consists of a triple of functors \(L, R : \E^{{\to}}
  \to \E^{{\to}}\), \(E : \E^{{\to}} \to \E\), subject to two conditions.
  First, when \(f, f'\), \(h, k\) are morphisms in \(C\) with \(f'. h = k .
  f\), the following diagram commutes: 
\begin{equation}\label{eq:factorisation-functors}
  \begin{tikzpicture}[baseline={([yshift=-.5ex]current bounding box.center)}]
    \matrix (m) [matrix of math nodes, row sep=2em,
    column sep=3.5em]{
        |(a)| {A} & |(ef)| {Ef} & |(b)| {B} \\
        |(a')| {A'} & |(ef')| {Ef'} & |(b')| {B'} \\
    };
    \begin{scope}[every node/.style={midway,auto,font=\scriptsize}]
    \path[->]
        (a) edge node [left] {$h$} (a')
        (b) edge node [right] {$k$} (b')
        (a) edge node [above] {$Lf$} (ef)
        (a') edge node [below] {$Lf'$} (a' -| ef'.west)
        (ef) edge node [above] {$Rf$} (ef -| b.west)
        (ef') edge node [below] {$Rf'$} (ef' -| b'.west)
        (ef) edge node [anchor=center, fill=white] {$E(h,k)$} (ef');
    \end{scope}\end{tikzpicture}
\end{equation}
    Second, the top and bottom composites should compose to \(f\), \(f'\).
    This decomposition yields natural transformations \(\eta : 1 \Rightarrow R\),
    \(\epsilon: L \Rightarrow 1\) given (at \(f\)) by the commutative squares:
    
    \begin{minipage}{0.4\textwidth}
        \centering
    \begin{equation}\label{eq:awfs-unit}
  \begin{tikzpicture}[baseline={([yshift=-.5ex]current bounding box.center)}]
    \matrix (m) [matrix of math nodes, row sep=2em,
    column sep=3.5em]{
        |(a)| {A} & |(ef)| {Ef} \\
        |(b)| {B} & |(b2)| {B} \\
    };
    \node (i) at (barycentric cs:a=1,b=1,ef=1,b2=1) {$\overset{\eta}{\Rightarrow}$};
    \begin{scope}[every node/.style={midway,auto,font=\scriptsize}]
    \path[-]
        (b) edge [double distance=2pt] (b2);
    \path[->]
        (a) edge node [above] {$Lf$} (ef)
        (ef) edge node [right] {$Rf$} (b2)
        (a) edge node [left] {$f$} (b);
    \end{scope}\end{tikzpicture}
        \end{equation}
        \end{minipage}\begin{minipage}{0.4\textwidth}\centering
        \begin{equation}\label{eq:awfs-counit}
  \begin{tikzpicture}[baseline={([yshift=-.5ex]current bounding box.center)}]
    \matrix (m) [matrix of math nodes, row sep=2em,
    column sep=3.5em]{
        |(a)| {A} & |(a2)| {A} \\
        |(ef)| {Ef} & |(b2)| {B} \\
    };
    \node (i) at (barycentric cs:a=1,b2=1,ef=1,a2=1) {$\overset{\epsilon}{\Rightarrow}$};
    \begin{scope}[every node/.style={midway,auto,font=\scriptsize}]
    \path[-]
        (a) edge [double distance=2pt] (a2);
    \path[->]
        (a) edge node [left] {$Lf$} (ef)
        (a2) edge node [right] {$f$} (b2)
        (ef) edge node [below] {$Rf$} (b2);
    \end{scope}\end{tikzpicture}\end{equation}\end{minipage}
\end{defi}

An \emph{algebraic weak factorisation system}
\index{algebraic weak factorisation system} (AWFS) can be informally
described as a functorial factorisation into a composite of two kinds of
arrows, as follows. First, there are double categories \(I : \Lbb \to \sq{\E}\),
\(J : \Rbb \to \sq{\E}\) induced by each other in the following way:
\[
	\Dll{J} \cong I \text{ and } \Drl{I} \cong J \text{ in } \DBL/\sq{\E}.
\]
Note that also \(\Dll{\left(\Drl{I}\right)} \cong I\) and
\(\Drl{\left(\Dll{J}\right)} \cong J\). Second, it is required that the functor
\(L\) of the factorisation factors through the vertical part \(I_1 : \Lbb_1 \to
\E^{{\to}}\) (see \refprop{vertical-part-left-adj}) of
\(I\), and that \(R\) factors through the vertical part \(J_1 : \Rbb_1 \to
\E^{{\to}}\).

Taking the above as starting point, we can work towards the formal definition
of an AWFS (\refdefi{awfs}). First, observe that if the above holds, the
lifting structure \(\phi\) on the arrows of the form \(Lf\) induces a family of
maps \(\delta_f = \phi_{1_{E_f}, LLf}(RLf)\):
\begin{equation}\label{eq:comult-lift}
  \begin{tikzpicture}[baseline={([yshift=-.5ex]current bounding box.center)}]
    \matrix (m) [matrix of math nodes, row sep=2em,
    column sep=3.5em]{
        |(a)| {A} & |(elf)| {ELf} \\
        |(ef)| {Ef} & |(ef2)| {Ef} \\
    };
    \begin{scope}[every node/.style={midway,auto,font=\scriptsize}]
    \path[->]
        (ef) edge node [below] {$1_{Ef}$} (ef -| ef2.west)
        (a) edge node [above] {$LLf$} (a -| elf.west) 
        (a) edge node [left] {$Lf$} (ef)
        (elf) edge node [right] {$RLf$} (ef2)
        (ef) edge [dashed] node [anchor=center, fill=white] {$\delta_f$} (elf);
    \end{scope}\end{tikzpicture}\end{equation}
 Since \(L\) factors through \(\Lbb_1\) and \(R\) factors through
 \(\Rbb_1\), this family defines a natural transformation:
\[
	\delta: L \Rightarrow LL.
\]
Note that commutativity of the bottom triangle in \eqref{eq:comult-lift} implies
a counit law:
\[
	\epsilon_{Lf}. \delta_f = 1_{Lf}.
\]
This counit law expresses that \(Lf\) is a coalgebra for the `mere co-pointed
endofuntor' \((L, \epsilon)\). Similarly, there is a family of maps \(\mu_f :
RRf \to Rf\):
\begin{equation}\label{eq:mult-lift}
  \begin{tikzpicture}[baseline={([yshift=-.5ex]current bounding box.center)}]
    \matrix (m) [matrix of math nodes, row sep=2em,
    column sep=3.5em]{
        |(ef)| {Ef} & |(ef2)| {Ef} \\
        |(erf)| {E R f} & |(b)| {B} \\
    };
    \begin{scope}[every node/.style={midway,auto,font=\scriptsize}]
    \path[->]
        (ef) edge  node [above] {$1_{Ef}$} (ef -| ef2.west)
        (erf) edge node [below] {$RRf$} (erf -| b.west) 
        (ef) edge node [left] {$LRf$} (erf)
        (ef2) edge node [right] {$Rf$} (b)
       (erf) edge [dashed] node [anchor=center, fill=white] {$\mu_f$} (ef2);
  \end{scope}\end{tikzpicture}\end{equation}
which yields a natural transformation
\[
	\mu : RR \Rightarrow R
\]
which satisfies a unit law:
\[
	\mu_f . \eta_{Rf} = 1_{Rf}.
\]
Again, this law just expresses that \(Rf\) is an algebra for the `mere
pointed endofunctor' \((R, \eta)\). Next, there is for any vertical arrow
\(v\) in the double category \(\Rbb\)
a filler for the diagram:
\[
          \begin{tikzpicture}[baseline={([yshift=-.5ex]current bounding box.center)}]
            \matrix (m) [matrix of math nodes, row sep=3em,
            column sep=6em]{
                |(a)| {A} & |(b)| {A} \\
		|(c)| {EI(v)} & |(d)| {B} \\
            };
            \begin{scope}[every node/.style={midway,auto,font=\scriptsize}]
            \path[->]
		    (a) edge node [above] {$1_A$} (b)
		    (c) edge node [above] {$RI(v)$} (d)
		    (a) edge node [left] {$LI(v)$} (c)
		    (b) edge node [right] {$I(v)$} (d)
		    (c) edge [dashed] node [anchor=center, fill=white] {$\alpha$}
		    (b);
            \end{scope}\end{tikzpicture}
\]
since \(L\) factors through \(I_1 : \Lbb_1 \to \E^{\to}\). 
Because the top triangle commutes, this filler defines an algebra structure
\(\alpha: RI(v) \to I(v)\) for the pointed endofunctor \(R : \E^{\to} \to \E^{\to}\), 
\(\eta : 1 \Rightarrow R\). Similarly,
a vertical arrow \(u\) in the category \(\Lbb\) induces a coalgebra structure
\(\beta: I(u) \to LI(u)\) for the co-pointed endofunctor \(L : \E^{\to} \to \E^{\to}\),
\(\epsilon: L \Rightarrow 1\).

Observe that now for any commutative square 
\[
          \begin{tikzpicture}[baseline={([yshift=-.5ex]current bounding box.center)}]
            \matrix (m) [matrix of math nodes, row sep=2em,
            column sep=3.5em]{
                |(a)| {A} & |(b)| {B} \\
                |(c)| {C} & |(d)| {D} \\
            };
            \begin{scope}[every node/.style={midway,auto,font=\scriptsize}]
            \path[->]
                (a) edge node [above] {$f$} (b)
                (c) edge node [above] {$g$} (d)
		    (a) edge node [left] {$I(u)$} (c)
		    (b) edge node [right] {$J(v)$} (d);
            \end{scope}\end{tikzpicture}
\]
a filler can be given as a composite:
\begin{equation}\label{eq:filler-general-awfs}
          \begin{tikzpicture}[baseline={([yshift=-.5ex]current bounding box.center)}]
            \matrix (m) [matrix of math nodes, row sep=2em,
            column sep=3.5em]{
		    	      &                 & |(b)| {B}      & |(b')| {B} \\
		    |(a)| {A} & |(eu)| {E I(u)} & |(ev)| {E I(v)} & |(d)| {D} \\
		    |(c)| {C} & |(c')| {C} &  &  \\
            };
            \begin{scope}[every node/.style={midway,auto,font=\scriptsize}]
	    \path[->]
		    (a) edge node [pos=0.7, above] {$LI(u)$} (eu)
		    (a) edge [bend left = 10]  node [above] {$f$} (b)
		    (eu) edge node [right] {$RI(u)$} (c')
		    (c) edge node [above] {$1_C$} (c')
	            (eu) edge node [above] {$E(f,g)$} (ev)
		    (c') edge [bend right=10] node [below] {$g$} (d)
		    (b) edge node [left] {$LI(v)$} (ev)
		    (ev) edge node [above] {$RI(v)$} (d)
		    (b) edge node [above] {$1_B$} (b')
		    (c) edge [dashed] node [anchor=center, fill=white] {$\beta$} (eu)
	            (ev) edge [dashed] node [anchor=center, fill=white] {$\alpha$} (b')
		    (a) edge node [left] {$I(u)$} (c)
		    (b') edge node [right] {$I(v)$} (d);
            \end{scope}\end{tikzpicture}.
            \end{equation}

We now turn to the definition of an algebraic weak factorisation system (AWFS)
(\cite{Bourke-Garner}, section~2.2).  A consequence of this definition is that
the lifts for `canonical' diagrams, such as~\eqref{eq:mult-lift}
and~\eqref{eq:comult-lift} or more generally for any \(L\)-coalgebra and
\(R\)-algebra, agree with the general shape~\eqref{eq:filler-general-awfs}.
\begin{defi}{awfs}\index{algebraic weak factorisation system|textbf}
\index{AWFS|seeonly {algebraic weak factorisation system}}
  Suppose \((L, R, \epsilon, \eta)\) is a functorial factorisation on a
  category \(\E\) and \(\delta : L \Rightarrow LL\), \(\mu: RR \Rightarrow R\)
  are natural transformations. By conditions (i) and (ii) given below, we may
  assume that they are given by diagonals as in diagrams~\eqref{eq:comult-lift}
  and~\eqref{eq:mult-lift} (see~\cite{Bourke-Garner}, section~2.2, and the
  observations above).
  Then \((L, R, \delta, \mu, \epsilon, \eta)\) is an
  \emph{algebraic weak factorisation system} (AWFS) when the following
  conditions hold:
  \begin{enumerate}[(i)]
	\item The triple \((L, \delta, \epsilon)\) satisfies the conditions for
		a comonad structure on \(\E^{\to}\)
	\item The triple \((R, \mu, \eta)\) satisfies the conditions for a monad
		structure on \(\E^{\to}\)
	\item The commutative square
	\begin{equation}\label{eq:distr-law-square}
          \begin{tikzpicture}[baseline={([yshift=-.5ex]current bounding box.center)}]
            \matrix (m) [matrix of math nodes, row sep=2em,
            column sep=3.5em]{
                |(a)| {Ef} & |(b)| {ELf} \\
                |(c)| {ERf} & |(d)| {Ef} \\
            };
            \begin{scope}[every node/.style={midway,auto,font=\scriptsize}]
            \path[->]
                (a) edge node [above] {$\delta_f$} (b)
                (c) edge node [above] {$\mu_f$} (d)
		    (a) edge node [left] {$LRf$} (c)
		    (b) edge node [right] {$RLf$} (d);
            \end{scope}\end{tikzpicture}
        \end{equation}
	    whose diagonal is the identity \(1_{Ef} : Ef \to Ef\),
	    constitutes a \emph{distributive law} \(LR \Rightarrow RL\) of
      \(L\) over \(R\).\index{distributive law}
	    We have put some background on this in Box~\ref{box:distr-law}, but
	    we will be mostly concerned with the alternative formulation of
	    this condition formulated in \refprop{awfs-equivalent}. The significance
	    of the distributive law for an AWFS is pointed out in
	    \reftheo{distrlaw<->iso-of-cofibred-str} below. The fact that it
	    constitutes a `distributive law' in some broader sense will also
	    become apparent from the AWFS we construct in this paper.
\end{enumerate}
\end{defi}

\begin{textbox}{Distributive laws combining monads and comonads\label{box:distr-law}}
\index{distributive law|textbf}
\emph{Distributive laws} between monads were introduced by Beck
in~\cite{beck-69} as natural transformations \(\lambda : TS \Rightarrow ST\)
for monads \(S\), \(T\) on some category, subject to certain conditions
relating to the monad structures of \(S\) and \(T\). These are equivalent to a
monad structure on the composition \(TS\), subject to compatibility conditions
with \(S\) and \(T\). As remarked by Power and
Watanabe~\cite{power-watanabe-02}, there are many more distributive laws to
consider, including one for a comonad over a monad, or the other way around.
These are all different choices. What is meant by a distributive law for a
comonad \(L\) over a monad \(R\) in this paper is the dual of Definition 6.1
in~\cite{power-watanabe-02}. Spelled out, this is a natural transformation
\(\lambda : LR \Rightarrow RL\) subject to the following conditions:
\begin{equation}\label{eq:distr-law-comonad-monad}
\begin{array}{rr}
	R\delta . \lambda = \lambda L . L \lambda . \delta_R \text{ , } & R \epsilon . \lambda = \epsilon_R\\
	\mu_L . R \lambda . \lambda R = \lambda . L \mu \text{ , } & \eta_L = \lambda . L \eta.
 \end{array}
 \end{equation}
 \end{textbox}

The definition of an AWFS as given above might seem to demand an overwhelming
amount of additional structure on a functorial factorisation together with the
natural transformations \(\delta,\mu\). The following proposition shows that in
fact, the definition of an AWFS contains some redundancy and can be reduced to a
couple of equational identities. For instance, the distributive law can be
expressed using a single equational identity which combines \(\delta\) and
\(\mu\). This observation is due to Richard Garner, and since it will be quite
important for our subsequent verifications, we will refer to this equation as
the \emph{Garner equation}.
\begin{prop}{awfs-equivalent}
Suppose \((L, R, \epsilon, \eta)\) is a functorial factorisation, and suppose
\(\delta: L \Rightarrow LL\) is a natural transformation over \(\dom : \E^{\to}
\to \E\), and\footnote{These last two conditions mean that \(\delta\), \(\mu\)
	are given by arrows \(\delta_f\), \(\mu_f\) as above such that
\(\delta_f . Lf = LLf\) and \(Rf . \mu_f = RRf\)}\(\mu : RR \Rightarrow R\) is
a natural transformation over \(\cod : \E^{\to} \to \E\). Then the following
statements are equivalent:
\begin{enumerate}[(i)]
  \item The triples \((L, \delta, \epsilon)\), \((R, \mu, \eta)\) satisfy the
	  conditions for a comonad over \(\dom: \E^\to \to \E\) and a monad over
	  \(\cod: \E^\to \to \E \) respectively.
	  i.e.\ the following equations
	  are satisfied:
    \begin{equation}
        \begin{array}{rr}
        RLf. \delta_f = 1 & \mu_f . LRf = 1 \\
        E(1, Rf). \delta_f = 1 & \mu_f . E(Lf, 1) = 1 \\
        \delta_{Lf} . \delta_f = E(1, \delta_f) . \delta_f & \mu_f . \mu_{Rf} = \mu_f . E(\mu_f, 1) 
        \end{array}
    \end{equation}
    And the \emph{Garner equation}\index{Garner equation|textbf} holds:
    \begin{equation}\label{eq:garner-equation-awfs}
        \delta_f . \mu_f = \mu_{Lf} . E(\delta_f, \mu_f) . \delta_{Rf}
    \end{equation}

    \item The diagram~\eqref{eq:distr-law-square} commutes, its diagonal is the
	    identity, and constitutes a distributive
	    law for \(L\) over \(R\), in the sense that the
	    equations~\eqref{eq:distr-law-comonad-monad} are satisfied.

    \item The tuple \((L, R, \delta, \mu, \epsilon, \eta)\) is an AWFS\@.
  \end{enumerate}
\end{prop}
\begin{proof}
  The only thing to prove is the equivalence of (i) and (ii), since these two
  complement each other to an AWFS\@. We leave it to the reader to spell out
  the axioms~\eqref{eq:distr-law-comonad-monad}
  for~\(\lambda=\)\eqref{eq:distr-law-square} and conclude that these contain
  precisely the counit, unit, coassociativity and associativity conditions as
  well as the Garner equation. 
\end{proof}

\subsection{A double category of coalgebras}
This section further works out the relationship between the above definition of
an AWFS and the informal description of two double categories of left and
right arrows which determine each other. For the purposes of this paper,
the most important thing to keep is that these two double categories are
double categories of coalgebras and algebras with respect to the comonad \(L\)
and the monad \(R\). Sections \ref{sec:dominances} and \ref{sec:AWFSFromM} study
these coalgebras and algebras in depth for two different AWFSs.

We assume that \[(L, R, \epsilon, \eta, \delta, \mu)\] is an AWFS on a category
\(\E\) as in \refdefi{awfs}. We would like to define a double category
\(\coalg{L}\) whose objects are objects of \(\E\), whose vertical arrows are
coalgebras for the \emph{comonad} \((L, \delta, \epsilon)\) and whose squares
are morphisms of algebras. For this to make sense, it is needed to define a
vertical composition of coalgebras. 

We can regard \(R\) and \(L\) as either mere (co)pointed endofunctors, or
(co)monads. In both cases, (co)algebras for them can be represented as diagonal
fillers for diagrams in \(\E\). Indeed, recall that a \emph{coalgebra} \(\beta
: f \to Lf\) for the mere co-pointed endofunctor \(L\) is the same thing as a
filler for the square:\index{coalgebra (for an AWFS)}
\[
          \begin{tikzpicture}[baseline={([yshift=-.5ex]current bounding box.center)}]
            \matrix (m) [matrix of math nodes, row sep=3em,
            column sep=6em]{
                |(a)| {A} & |(ef)| {E_f} \\
        		|(b)| {B} & |(b')| {B} \\
            };
            \begin{scope}[every node/.style={midway,auto,font=\scriptsize}]
            \path[->]
		    (a) edge node [left] {$f$} (b)
		    (a) edge node [above] {$Lf$} (ef)
		    (b) edge node [below] {$1_B$} (b')
		    (ef) edge node [right] {$Rf$} (b')
		    (b) edge [dashed] node [anchor=center, fill=white] {$\beta$} (ef);
            \end{scope}\end{tikzpicture}
\]
Indeed, one can check that the co-unit condition dictates that the top arrow of
the square \(\beta: f \to Lf\) must be the identity, hence the top triangle
commutes, and that \(R_f.\beta = 1_B\), hence the bottom triangle commutes.
When we are interested in coalgebras for the \emph{comonad} \((L, \delta,
\epsilon)\), this arrow is subject to the co-associativity condition \(\delta_f
. \beta = L(\beta) . \beta\).  This boils down to the equational identity:
\begin{align}\label{eq:coass-coalgebra}
	\delta_f . \beta = E(1_A, \beta) . \beta
\end{align}

Similarly, an \emph{algebra}\index{algebra (for an AWFS)} structure \(Rf \to
f\) (for either the mere pointed endofunctor or the monad) is defined entirely
by an underlying arrow \(\beta : E_f \to \dom f \). We will hence refer to
algebra or coalgebra structures by their underlying map.  For the comonad \(L\)
and the monad \(R\), we denote the category of (co)algebras and morphisms of
(co)algebras by \(\coalgcat{L}\) and \(\algcat{R}\) respectively.  As we have
seen in~\eqref{eq:filler-general-awfs}, every \(L\)-coalgebra bears a left
lifing structure with respect to every \(R\)-algebra. This takes the form of a
functor
\[
	\Phi: \coalgcat{L} \to \llp{\algcat{R}}
\]
for which we can describe the image in the form of an extra condition on
lifting structures in the following proposition. This proposition
is the dual of Lemma 1 in Bourke-Garner~\cite{Bourke-Garner}.
\begin{prop}{coalg-lifts}
	The functor \(\Phi: \coalgcat{L} \to \llp{\algcat{R}}\) is injective on objects
	and fully faithful, and its image consists
	of those arrows \((i, \phi)\) with a left lifting structure \(\phi\) for which
	the following diagrams commute: 
	\begin{equation}\label{eq:coalgebra-lift-condition}
          \begin{tikzpicture}[baseline={([yshift=-.5ex]current bounding box.center)}]
            \matrix (m) [matrix of math nodes, row sep=4em,
            column sep=6em]{
		    |(a)| {A} & |(eh)| {Eh} & |(elh)| {ELh} \\
			      &  & |(eh')| {Eh} \\
		    |(b)| {B} & |(c')| {C}  & |(c)| {C} \\
            };
            \begin{scope}[every node/.style={midway,auto,font=\scriptsize}]
            \path[->]
		(a) edge node [left] {$i$} (b)
		(a) edge node [above] {$u$} (eh)
		(b) edge node [below] {$v$} (c')
		(c') edge node [below] {$1_C$} (c)
		(eh) edge node [above] {$\delta_h$} (elh)
		(eh) edge node [pos=0.2, right] {$Rh$} (c')
		(elh) edge node [right] {$RLh$} (eh')
		(eh') edge node [right] {$Rh$} (c)
		(b) edge [dashed] node [sloped, pos=0.5,fill=white, anchor=center]
		{$\phi_{u,v}(Rh)$} (eh)
		(b) edge [dashed] node [sloped, fill=white, anchor=center] 
		{$v' := \phi_{u, v}(Rh)$} (eh')
		(b) edge [dashed] node [sloped, fill=white, anchor=center]
		{$\phi_{\delta_h.u, v'}(RLh)$} (elh)
		;
            \end{scope}\end{tikzpicture}
\end{equation}
\end{prop}
\begin{proof}
	See Lemma 1 of Bourke-Garner~\cite{Bourke-Garner}. The proof relies on the distributive law,
    or more precisely, the Garner equation.
 \end{proof}

\begin{coro}{coalgebras-cofibred}
	Suppose we have a pushout square:
\[
  \begin{tikzpicture}[baseline={([yshift=-.5ex]current bounding box.center)}]
    \matrix (m) [matrix of math nodes, row sep=2em,
    column sep=3.5em]{
	|(a)| {A} & |(a')| {A'} \\
	|(b)| {B} & |(b')| {B'} \\
    };
    \begin{scope}[every node/.style={midway,auto,font=\scriptsize}]
    \path[->]
	    (a) edge (a')
	    (a) edge node [left] {$f$} (b)
	    (a') edge node [right] {$f'$} (b')
	    (b) edge (b');
    \end{scope}\end{tikzpicture}
\]
and suppose that \(\beta\) is a coalgebra structure on \(f\). Then there is a unique
coalgebra structure \(\beta'\) on \(f'\) which makes the diagram into a morphism of
coalgebras. Hence there is a cofibred structure (see the remarks above
\refdefi{llp-double-cat}) on \(\E\):
\begin{equation}\label{eq:coalg-cofibred}
	\coalgcofibr{L}: \E^{\to}_{\text{cocart}} \to \Sets
\end{equation}
which associates to every morphism the set of coalgebra structures on it.
\end{coro}
\begin{proof}
	There is a unique left lifting structure with respect to \(\algcat{R}\)
	on \(f'\) such that the square is a morphism of lifting structures. It
	is easy to see that this lifting structure satisfies the condition of
	\refprop{coalg-lifts}. Therefore it also defines a coalgebra structure
	on \(f'\) in such a way that the diagram is a morphism of coalgebras.
	This defines the functor~\eqref{eq:coalg-cofibred} on morphisms.
\end{proof}

With the construction introduced in \refprop{vertical-part-left-adj}, we
can improve a bit on the above Proposition, by regarding \(\Phi\) as a double functor
\begin{equation}\label{eq:llp-coalg-double}
	\Phi : \freedbl{\coalgcat{L}} \to \Dll{\freedbl{\algcat{R}}},
\end{equation}
which is injective and fully faithful on vertical morphisms.
One can check that the condition on left lifting structures of
 \refprop{coalg-lifts} is closed under composition of left lefting structures.
 So we can inherit vertical composition from this category,
 and define a double category of coalgebras as follows:\index{vertical composition!of coalgebras}

\begin{defi}{double-cat-coalg}
	Define the double category \(\coalg{L}\) as the double image of the
	functor~\eqref{eq:llp-coalg-double}. This is called the double
	category of coalgebras for the algebraic weak factorisation system
	\((L, R, \epsilon, \eta, \delta, \mu)\). Similarly, the double
	category \(\alg{R}\) of algebras is defined as the double image of the transpose
	\[
		\widetilde{\Phi} : \freedbl{\algcat{R}} \to \Drl{\freedbl{\coalgcat{L}}}.
	\]
	as per \refprop{vertical-part-left-adj}.
\end{defi}

It will be useful for the rest of this paper to record an expression for the
vertical composition of coalgebras induced by the previous definition.  For
instance, it will be used in \refcoro{vert-co-HDRs} below to show that vertical
composition of hyperdeformation retracts (HDRs) is the same as vertical
composition of coalgebras for a certain AWFS\@.
Also, we can use it to show where the distributive law for AWFSs is actually used
in the theory of Bourke and Garner~\cite{Bourke-Garner}.

So suppose that \(f : A \to B\), \(g: B \to C\) come with coalgebra structures
\(\beta: f \to Lf\), \(\gamma : g \to Lg\), and write \(h := g.f\).
By~\eqref{eq:filler-general-awfs}, both \(f\) and \(g\) have the left-lifting
property with respect to \(Rh\) -- and indeed it turns out that their 
composition coalgebra structure is given by first lifting with respect to \(f\)
and then \(g\) according to this recipe.
Spelling out~\eqref{eq:filler-general-awfs}, this filler is the diagonal in the below diagram:
\begin{equation}\label{eq:coalgebra-composition}
          \begin{tikzpicture}[baseline={([yshift=-.5ex]current bounding box.center)}]
            \matrix (m) [matrix of math nodes, row sep=2em,
            column sep=6em]{
		    |(a)| {A} &   & |(eh')| {Eh}   & |(a')| {Eh} \\
		    |(b)| {B} & |(ef)| {Ef}  & |(eh)| {ERh}   &   \\
			      & |(eg)| {Eg}  & |(erh)| {ERh} & \\
		    |(c)| {C} & |(c')| {C}  &  &  |(c2)| {C}   \\
            };
            \begin{scope}[every node/.style={midway,auto,font=\scriptsize}]
	    \path[->]
		    (a) edge node [above] {$Lh$} (eh')
		    (eh') edge node [above] {$1_{Eh}$} (a')
		    (eh') edge node [right] {$LRh$} (eh)
		    (eh) edge node [anchor=center, fill=white] {$1_{ERh}$} (erh)
		    (a) edge node [anchor=center, fill=white] {$Lf$} (ef)
		    (b) edge node [above] {$\beta$} (ef)
		    (b) edge node [anchor=center, fill=white] {$Lg$} (eg)
		    (a) edge node [left] {$f$} (b)
		    (b) edge node [left] {$g$} (c)
		    (ef) edge node [above] {$E(Lh, g)$} (eh)
		    (eh) edge node [above] {$\mu_h$} (a')
		    (c') edge node [below] {$1_C$} (c2)
		    (eg) edge node [anchor=center, fill=white] {$Rg$} (c')
		    (eg) edge node [above] {$E(\mu_h.E(Lh, g).\beta, 1_C)$} (erh)
		    (c) edge node [anchor=center, fill=white] {$\gamma$} (eg)
		    (erh) edge node [anchor=center, fill=white] {$\mu_h$} (a')
		    (erh) edge node [anchor=center, fill=white] {$RRh$} (c2)
		    (c) edge node [below] {$1_C$} (c')
		    (a') edge node [right] {$Rh$} (c2);
            \end{scope}\end{tikzpicture}.
\end{equation}
Hence the candidate coalgebra structure for the composite \(h = g.f\) is given by
\begin{align}\label{eq:coalgebra-composition-fla}
	\kappa &:= \mu_h. E(\mu_h.E(Lh, g).\beta,1_C) . \gamma : h \to Lh \\
	       &= \mu_h . E(E(1_A, g) . \beta, 1_C) . \gamma \nonumber
\end{align}
where the latter condition follows from one the unit law for \(\mu\).  The fact
that this candidate is a coalgebra structure is a result that follows from
\refprop{coalg-lifts}.

Similarly, if \(f\), \(g\) are \emph{algebras} with algebra structures
\(\beta\), \(\gamma\), their composition \(h := g.f\) has algebra structure:
\begin{equation}\label{eq:algebra-composition-fla}
	\beta . E(1_A , \gamma . E(f, Rh) . \delta_h) . \delta_h
	= \beta . E(1_A, \gamma . E(f, 1_C)) . \delta_h
\end{equation}\index{vertical composition!of algebras}

\begin{lemm}{distr-law<->morphism-of-algebras}
Suppose \((L, R, \epsilon, \eta)\) is a functorial factorisation
and \(\delta: L \Rightarrow L L\), \(\mu : RR \Rightarrow R\) are
natural transformations which make the diagrams~\eqref{eq:comult-lift}
and~\eqref{eq:mult-lift} commute, i.e.\ with \(\delta_f\), every arrow \(L f\) is a coalgebra
for the mere co-pointed endofunctor \((L, \epsilon)\) and with \(\mu_f\), every \(R f\) is an algebra
for the mere pointed endofunctor \((R, \eta)\).

Then for every \(h: A \to C\), the composition \(Rh.RLh\) has an algebra
structure for the pointed endofunctor \((R,\eta)\) given
by~\eqref{eq:algebra-composition-fla}, or explicitly:
\[
    \kappa := \mu_{Lh} . E(1_{ELh}, \mu_h . E(RLh, 1_C)) . \delta_{Rh.RLh} \thinspace: \thinspace E_{Rh . RLh} \to E_{Lh}
\]

Further, the \emph{Garner equation}~\eqref{eq:garner-equation-awfs} is satified
if and only if the following square is a morphism of algebras for the given
structures:
\[
\begin{tikzpicture}[baseline={([yshift=-.5ex]current bounding box.center)}]
            \matrix (m) [matrix of math nodes, row sep=2em,
            column sep=4em]{
		     |(eh)| {Eh} & |(elh)| {ELh} \\
		                 & |(eh')| {Eh} \\
		     |(c')| {C}  & |(c)| {C} \\
            };
            \begin{scope}[every node/.style={midway,auto,font=\scriptsize}]
            \path[->]
		(c') edge node [below] {$1_C$} (c)
		(eh) edge node [above] {$\delta_h$} (elh)
		(eh) edge node [right] {$Rh$} (c')
		(elh) edge node [right] {$RLh$} (eh')
		(eh') edge node [right] {$Rh$} (c);
            \end{scope}\end{tikzpicture}
    \]
\end{lemm}
\begin{proof}
The first claim follows from dualizing the preceding discussion on the composition of
coalgebras for the mere co-pointed endofunctors -- the filler is again the diagonal in the  
dual counterpart of~\eqref{eq:coalgebra-composition}, which is~\eqref{eq:algebra-composition-fla}.

So we focus on the second claim.  As one can readily check, this comes down to
the identity:
\begin{equation}\label{eq:algebra-morphism-distr-law}
    \kappa . E(\delta_h, 1_C) = \delta_h . \mu_h
\end{equation}
For this we have:
\begin{align*}
&\kappa . E(\delta_h, 1_C) = \\
&\mu_{Lh} . E(1_{ELh}, \mu_h . E(RLh , 1_C)) . \delta_{Rh . RLh} . E(\delta_h, 1_C) =^1\\ 
&\mu_{Lh} . E(1_{ELh}, \mu_h . E(RLh, 1_C)) . E(\delta_h , E(\delta_h, 1_C)) . \delta_{Rh} =\\
&\mu_{Lh} . E(\delta_h , \mu_h . E(RLh . \delta_h, 1_C)) . \delta_{Rh} = \\
&\mu_{Lh} . E(\delta_h, \mu_h) . \delta_{Rh}
\end{align*}
Where we used the identities
\[
E(u,v). E(s,t) = E(u.s, v.t)
\]
throughout, naturality of \(\delta\) at \(=^1\) and at the last step the identity
\(RLh . \delta_h = 1\). Hence~\eqref{eq:algebra-morphism-distr-law} precisely
states the distributive law.
\end{proof}

From the above definitions, observe that 
\begin{equation}\label{eq:coalgebra-image-llp}
	\Dll{\alg{R}} \to \Dll{(\freedbl{\algcat{R}})}
\end{equation}
is an inclusion and that \(\Phi\) factors through it via the transpose of
\[\alg{R} \hookrightarrow \Drl{(\freedbl{\coalgcat{L}})}.\]
We have the following lemma:
\begin{lemm}{double-cat-distr-law}
  Under the distributive law, the image of a vertical morphism in \(\Dll{\alg{R}}\)
  along the functor~\eqref{eq:coalgebra-image-llp} satisfies the condition of
  \refprop{coalg-lifts}. 
\end{lemm}
\begin{proof}
	It is easy to see that when \((\delta_h, 1_C)\) is an algebra morphism between
	\(Rh\) and the composed algebra, then the property of 
	 \refprop{coalg-lifts} holds. So the statement follows from 
	 \reflemm{distr-law<->morphism-of-algebras}.
\end{proof}

Since \(\Dll{\alg{R}}\) can also be taken to define a cofibred structure
\[
	\Dll{\alg{R}} : \E^{\to}_{\text{cocart}} \to \Sets \text{ , }
\]
\refprop{coalg-lifts} can be rewritten as a theorem on cofibred structures:
\begin{theo}{distrlaw<->iso-of-cofibred-str}
Suppose \((L, R, \epsilon, \eta, \delta, \mu)\) is an AWFS (\refdefi{awfs}).  
Then:
\begin{enumerate}[(i)]
\item The natural transformation between cofibred structures
\[
	 \varphi: \coalgcofibr{L} \to \Dll{\alg{R}}
 \]
 induced by \(\Phi\) is an isomorphism
 \item The natural transformation between fibred structures
	 \[
		 \widetilde{\varphi}: \algfibr{R} \to \Drl{\coalg{L}}
	 \]
	 induced by \(\widetilde{\Phi}\) is an isomorphism.
 \end{enumerate}
\end{theo}
\begin{proof}
	The two statements are dual. To prove (i), it is enough to show that
that for each morphism \(v\) in \(\E\), \(\varphi_v\) is
an isomorphism, i.e.\ that coalgebra structures on \(v\) correspond precisely to
left lifting structures on \(v\) with respect to \(\alg{R}\).  This follows from
\reflemm{double-cat-distr-law} and \refprop{coalg-lifts}.
\end{proof}

As another consequence of the above theorem, the double functor
\[
	\freedbl{\coalgcat{L}} \to \Dll{\alg{R}}
\]
through which \(\Phi\) factors is surjective and full on vertical morphisms and
squares apart from being injective and fully faithful on horizontal morphisms.
Hence this, and similarly dual reasoning, induces equivalences of images:
\[
	\coalg{L} \cong \Dll{\alg{R}} \text{ and } \alg{R} \cong \Drl{\coalg{L}}
	\text{ (over } \sq{\E} \text{)}
\]
which is the desired property of an AWFS\@. 

In the terminology of Bourke and Garner (see Section~2.8
of~\cite{Bourke-Garner}), the two double categories of coalgebras and algebras
are \emph{concrete} double categories over \(\E\):
\begin{defi}{concretedoublecat}\index{concrete double category}
  A double functor \(I: \mathbb{I} \to \sq{\E}\) is called a \emph{concrete}
  double category over \(\E\) when its functor on objects (and horizontal
  arrows) is the identity (the objects of \(\mathbb{I}\) are the objects of
  \(\E\)) and its functor on morphisms and squares (see
  \refprop{vertical-part-left-adj})
  \[
    I_1: \mathbb{I}_1 \to \E^{\to}
  \]
  is faithful.
\end{defi}

To summarise the previous, we have thus found two concrete double categories
over \(\E\) which are the categories of left or right lifting structures with
respect to each other.  These are the double categories of coalgebras and
algebras for the AWFS, respectively.

Before moving on, we address the natural question of whether there could
be a different vertical composition of algebras or coalgebras than the one
used above. As shown by Bourke and Garner in Proposition~4 of \cite{Bourke-Garner},
this is not the case, in the the sense that vertical composition of algebras,
for a certain given monad \(R\), completely determines an AWFS which induces
that composition.
\begin{prop}{alg-vertical-composition}
Suppose \(R : \E^{\to} \to \E^{\to}\) is a monad over \(\cod : \E^\to \to \E\).
Then there is a bijection between extensions of \(R\) to an AWFS and extensions
of \(\algcat{R} \to \E^\to\) to a concrete double category over \(\sq{\E}\).
Under this bijection, the vertical composition of algebras coincides with the
vertical composition induced by the AWFS\@.
\end{prop}
\begin{proof}
	See \cite{Bourke-Garner}, Proposition~4. The idea is that the unit of
	\(R\) determines \(L\), and \(\delta\) is determined by the unique
	morphism of algebras \(Rf \to Rf . RLf\) induced from \((LLf, 1) : f
	\to Rf . RLf\), since \(Rf\) has the free algebra structure on \(f\).
\end{proof}

\subsection{Cofibrant generation by a double category}
In Section~\ref{ssec:small-double-category} and further, we will 
compare different algebraic weak factorisation systems, or prove that they are
the same (e.g.,~\reftheo{HgenMoorefib}). A way to do this is to look at
generating double categories. The following is from Bourke-Garner
\cite{Bourke-Garner}:
\begin{defi}{cofibrantly-generated}\index{cofibrantly generated}
Suppose \(J : \mathbb{J} \to \sq{\E}\) is a double functor for a \emph{small}
double category \(\mathbb{J}\).  An AWFS is \emph{cofibrantly generated} by
\(\mathbb{J}\) if \(\Drl{J} \cong \alg{R}\) over \(\sq{\E}\).
\end{defi}
When \(\mathbb{J}\) is large, Bourke and Garner call this property
\emph{class-cofibrantly generated}. The dual property, when \(\Dll{I} \cong
\coalg{L}\), is called \emph{(class)-fibrantly generated}.
The conclusion reached in the previous section can be summarised as:
\begin{corollary}[\cite{Bourke-Garner}, Proposition~20]\label{coro:awfs-cof-generated}
	An AWFS is class-cofibrantly generated by its double category of coalgebras, and
	class-fibrantly generated by its double category of algebras.
\end{corollary}
We refer to Bourke-Garner to results of the type that say that under
appropriate conditions (i.e. \(\E\) locally presentable), the AWFS generated by
any small double category \(\mathbb{J} \to \sq{\E}\) exists
(\cite{Bourke-Garner}, Proposition~23). These results rely on some kind of
small object argument~\cite{Garner-09}. We will not rely on these results
for our construction of an AWFS, since we work constructively from the outset.
But they can be useful for boiling down a constructive theory to a classical
theory for comparison.

\subsection{Fibred structure revisited}\label{ssec:fibredstructurerevisited}
In this preliminary chapter we have studied collections of arrows, in a category
\(\E\), equipped with structure. These collections have taken different forms,
namely, that of:
\begin{enumerate}
  \item a \emph{notion of fibred structure} \(\cat{fib} : \E^{\to}_{\rm cart} \to \Sets\) (or \emph{co-fibred} structure);
  \item a \emph{category} \(I \to \E^\to\) over \(\E^\to\);
  \item a \emph{concrete double category} \(\mathbb{I} \to \sq{\E}\) over \(\E\);
\end{enumerate}
where we recall the definition of a \emph{concrete} double category
(\refdefi{concretedoublecat}).  We have also already seen that in many cases,
the three forms are in fact three different presentations of the same thing,
e.g.,:
\begin{enumerate}
  \item The \emph{notion of fibred structure} given by \(R\)-algebras 
    \(\algfibr{R}\);
  \item The \emph{category} of \(R\)-algebras \(\algcat{R}\);
  \item The \emph{concrete double category} of \(R\)-algebras \(\alg{R}\).
\end{enumerate}

The following definition would capture these three definitions
into one, for the given example:
\begin{defi}{discretelyfibreddoublecat}\index{discretely (co)fibred concrete double category|textbf}
  A (concrete) double category \(I: \mathbb{I} \to \sq{\E}\) over \(\sq{\E}\)
  is called \emph{discretely fibred} when for its 1-categorical restriction
  (\refprop{vertical-part-left-adj})
  \[
    I_1: \mathbb{I}_1 \to \E^{\to},
  \]
  there is, for every object \(i\) of \(\mathbb{I}_1\) and every
  cartesian square \(\alpha: f \to I_1(i)\) in \(\E^{\to}\), a unique object
  and morphism \(\alpha^*: j \to i\) in \(\mathbb{I}_1\), such that \(I_1(\alpha^*)
  = \alpha\):
\[
\begin{tikzpicture}[baseline={([yshift=-.5ex]current bounding box.center)}]
            \matrix (m) [matrix of math nodes, row sep=1.5em,
            column sep=1.5em]{
		     |(a)| {} & |(c)| {} \\
		     |(b)| {}  & |(d)| {} \\
       };
            \matrix (n) at (0,-2cm) [matrix of math nodes, row sep=1.5em,
            column sep=1.5em]{
		     |(a')| {} & |(c')| {} \\
         |(b')| {}  & |(d')| {} \\
       };
            \node (i1) at (2,0) {$\mathbb{I}_1$};
            \node (eto) at (2,-2cm) {$\E^{\to}$};
            \node (alpha) at (-0.8cm, -2cm) {$\alpha:$};
            \node (alpha) at (-0.8cm, 0) {$\alpha^*:$};
            \begin{scope}[every node/.style={midway,auto,font=\scriptsize}]
              \path[->, thick] (m) edge (n) (i1) edge node [right] {$I_1$} (eto);
              \path[->] (a') edge node [right] {$f$} (b')
                        (a') edge (c')
                        (b') edge (d')
                        (c') edge node [right] {$I_1(i)$} (d');
                      \path[->] (a) edge [dashed] (b)
                        (a) edge [dashed] (c)
                        (b) edge [dashed] (d)
                        (c) edge node [right] {$i$} (d);
            \end{scope}\end{tikzpicture}
    \]
\end{defi}
This definition has a dual definition, that of a discretely cofibred concrete
double category, which comes with an induced notion of cofibred structure.

Every discretely fibred concrete double category \(I: \mathbb{I} \to \sq{\E}\)
defines a unique notion of fibred structure \(\cat{fib}(I) : {(\E^\to_{\rm
cart})}^{\op} \to \Sets\) on \(\E\) with the action induced by the above property:
\begin{align*}
  \cat{fib}(I)(g) &= I_1^{-1}(g) \\
  \cat{fib}(I)(\alpha)(i) &= j.
\end{align*}
Further, this operation is functorial, in that a double functor \(F: \mathbb{I}
\to \mathbb{J}\) between discretely fibred concrete double categories induces
a natural transformation between fibred structures:
\[
   \cat{fib}(F): \cat{fib}(I) \to \cat{fib}(J).
\]

We have the following proposition:
\begin{prop}{fibreddoubleiso}
 Suppose \(F : \mathbb{I} \to \mathbb{J}\) is a double functor over \(\sq{\E}\)
 between discretely fibred concrete double categories \(I\) and \(J\) over \(\E\). 
 Then \(F\) is an isomorphism of double categories precisely when
 the following two conditions are satisfied:

 \begin{enumerate}[(i)]
   \item  \(F\) induces a natural isomorphism of fibred structures
 \[
   \cat{fib}(F): \cat{fib}(I) \cong \cat{fib}(J).
 \]

 \item \(F\) is full on squares\index{full on squares}, i.e., the restriction
   on morphisms and squares
 \[
    F_1 : \mathbb{I}_1 \to \mathbb{J}_1
 \]
 is full.
 \end{enumerate}
\end{prop}
\begin{proof}
  When \(F\) is invertible, the two conditions follow easily.
  For sufficiency, assume that the two conditions hold. The natural transformation
  \(\cat{fib}(F)\) has a natural inverse \({\cat{fib}(F)}^{-1}\). We define
   the functor \(G_1 : \mathbb{J}_1 \to \mathbb{I}_1\) on vertical morphisms
   by:
   \[
     G_1(j) = {\cat{fib}(F)}^{-1}_{J_1(j)}(j).
   \]
   It follows from the second condition that \(G_1\) as defined on objects
   above uniquely extends to a functor between categories over \(\E^{\to}\).
   Uniqueness follows here since \(\mathbb{I}\) and \(\mathbb{J}\) are concrete
   double categories.  Clearly, the two functors are inverses.
\end{proof}

We will use the above proposition when reasoning about different double
categories of arrows in an algebraic weak factorisation system, regarded as a
notion of fibred structure. In the latter form, it is sometimes easier to see
that they are isomorphic. To show that the corresponding double categories are
isomorphic, it is only needed to show that the isomorphism is induced by a
double functor which is full on squares. Note that the proposition clearly
has a dual, which would hold for the dual notion of discretely cofibred
concrete double categories and notions of cofibred structure.

The following result from Bourke and Garner may also be useful for the reader
interested in extending double functors between double categories of algebras
to morphisms between algebraic weak factorisation systems. As we have not
defined AWFSs as a category above, we refer to~\cite{Bourke-Garner} for the
details. We adopt the notation from the paper, \(\cat{AWFS}_{\rm (op)lax}\),
for the category of algebraic weak factorisation systems and lax (resp.\ oplax)
morphisms between them. Recall that \(\cat{DBL}\) denotes the category of
double categories and double functors between them.
\begin{namedprop}{morphismofawfsfromdoublefunctor}
  {\cite{Bourke-Garner}, Proposition~2}\index{algebraic weak factorisation system!category of AWFSs}
  The 2-functors
  \begin{align*}
    \alg{(-)} &: \cat{AWFS}_{\rm lax} \to \cat{DBL}^{\to}\\
    \coalg{(-)} &: \cat{AWFS}_{\rm oplax} \to \cat{DBL}^{\to}
  \end{align*}
 sending an AWFS to their concrete double categories of algebras,
 resp.\ coalgebras, are 2-fully faithful. \(\hfill \square\)
\end{namedprop}
The proposition says as much as that every double functor between double
categories of (co)-algebras induces a unique (op)lax morphism of algebraic weak
factorisation systems. It may be used in combination with \refprop{fibreddoubleiso}
to conclude that a double functor between categories of algebras, which is full
on squares and induces an isomorphism of notions of fibred structure, also induces
an isomorphism of algebraic weak factorisation systems.

\subsubsection{Concluding remark on notation}\label{sssec:notationofdoublecats}
As shown in the beginning of this section, we may have to deal with the same
thing in this paper, but now as a fibred structure, then as a double category,
or sometimes even just a category. To deal with this overhead, we will make
sometimes make use of the notion of a discretely (co)fibred concrete double
category. Yet there are also cases where we need to specifically look at this
structure as a mere fibred structure or mere category.  In those cases, we
denote the three types of structures as \(\cat{fib}\), \(\cat{Fib}\), and
\(\mathbb{F}\mathbf{ib}\) -- for (co)fibred structure, category and (concrete)
double category respectively. Indeed, we have already used the notation
\(\coalgcofibr{L}\), \(\coalgcat{L}\), \(\coalg{L}\), for example.

The following examples, to occur in upcoming sections, serve as a further
illustration of this notation:
\begin{itemize}
  \item The double category of effective trivial fibrations \(\mathbb{E}\cat{ffTrivFib}\), as a category \(\cat{EffTrivFib}\), and as a notion of fibred structure
    \(\cat{effTrivFib}\) (Section~\ref{sec:dominances});
  \item The double category of hyperdeformation retractions
    \(\mathbb{H}\cat{DR}\), as a category \(\cat{HDR}\), and as a notion of
    cofibred structure \(\cat{hdr}\) (Section~\ref{ssec:HDRs});
  \item The double category of naive fibrations \(\mathbb{N}\cat{Fib}\),
    as a category \(\cat{NFib}\), and as notion of fibred structure
    \(\cat{nFib}\) (Section~\ref{ssec:naivefibrations});
  \item The double category of effective fibrations \EEffRFib{}, as a category
    \EffRFib{}, and as a notion of fibred structure \effRFib{}
     (Section~\ref{sec:effective-fibration}).
\end{itemize}

 \section{Dominances}\label{sec:dominances}
\subsection{Algebraic weak factorisation systems from dominances} In this
section we show how the notion of a \emph{dominance} gives rise to an algebraic
weak factorisation system. The left class (coalgebras) of this algebraic weak
factorisation system will be shown to be the class of \emph{effective
cofibrations} defined by the dominance, while the right class (algebras) is
called the class of \emph{effective trivial fibrations}.
~\refprop{WFSfromdominance} can also be found in Bourke and
Garner~\cite{Bourke-Garner}. The rest of the section studies the (double)
category of effective cofibrations a bit more closely and in terms of
(co)fibred structure. Throughout this section, \(\E\) is a category satisfying
the conditions stated at the beginning of Section~\ref{sec:preliminaries}.
\begin{defi}{dominance} A class of monomorphisms $\Sigma$ in \(\E\) is called a 
  \emph{dominance}\index{dominance|textbf}
  \footnote{As far as the authors are aware, this terminology is due to
  Rosolini in the context of topos theory (\cite{rosolini86}, Chapter 3). In
Bourke and Garner, it is called a \emph{stable class of monics}
(\cite{Bourke-Garner}, 4.4).}
  on \(\E\) if
\begin{enumerate}
\item every isomorphism belongs to $\Sigma$ and $\Sigma$ is closed under composition.
\item every pullback of a map in $\Sigma$ again belongs to $\Sigma$.
\item the category $\Sigma_\text{cart}$ of morphisms in $\Sigma$ and pullback squares between them has a terminal object.
\end{enumerate}
A monomorphism which is an element of \(\Sigma\) will be called a
\emph{effective cofibration}\index{effective cofibration|textbf}\index{cofibration|seeonly {effective cofibration}}.
\end{defi}
Since taking the domain of a monomorphism in \(\Sigma\) has a left adjoint
\(\mathrm{id} : \E \to \Sigma_{\text{cart}}\), sending an object to the
identity on it, the domain of the terminal object in $\Sigma_{\text{cart}}$ is
the terminal object in \(\E\). We will denote this map by $\top: 1 \to \Sigma$.
The following proposition uses that \(\E\) is a locally cartesian closed
category and so admits local exponentials. Strictly speaking, we only use that
\(\top: 1 \to \Sigma\) is exponentiable.
\begin{prop}{WFSfromdominance} 
  Let \(\Sigma\) be a dominance. Then the functorial factorisation\index{functorial factorisation!from a dominance} given by
\diag{ B \ar[rr]_(.35){Lf} & & M_f = \Sigma_{a \in A} \Sigma_{\sigma \in \Sigma} B_a^\sigma \ar[rr]_(.65){Rf} & & A, }
with \(Lf(b) = (f(b), \top, \lambda x.b)\) and \(Rf(a, \sigma, \tau) = a\) can be
extended to an algebraic weak factorisation system.\index{algebraic weak factorisation system!from a dominance}
\end{prop}
\begin{proof} 
Note that $Mf$ classifies $\Sigma$-partial maps into $B$ over $A$. Let us spell out what this means. By a $\Sigma$-partial map $X \rightharpoonup B$ over $A$ we mean a pair consisting of a subobject $m: X' \to X$ with $m \in \Sigma$ (which does not depend on the choice of representative) and a map $n: X' \to B$ making
\diag{ X' \ar[r]^n \ar[d]_m & B \ar[d]^f \\
X \ar[r] & A }
commute. Note that such $\Sigma$-partial maps $X \rightharpoonup B$ over $A$ can be pulled back along arbitrary maps $Y \to X$. Saying that $Mf$ classifies $\Sigma$-partial maps into $B$ over $A$ means that any such map can be obtained by pulling back the $\Sigma$-partial map $(Lf, 1_B): M_f \rightharpoonup B$ along a unique map $X \to M_f$ over $A$.

For the monad structure, we need to define a map $\mu_f$ making
\diag{ M_{Rf} \ar[r]^{\mu_f} \ar[d]_{RRf} & M_f \ar[d]^{Rf} \\
A \ar[r]_1 & A }
commute. Maps $X \to M_{Rf}$ over $A$ correspond to diagrams of the form
\diag{ X'' \ar[r] \ar[d] & B \ar[dd]^f \\
X' \ar[d] \\
X \ar[r] & A }
with both inclusions $X'' \to X'$ and $X' \to X$ belonging to $\Sigma$ and
given by pullback along $\sigma' : X' \to \Sigma$ and $\sigma: X \to \Sigma$
respectively. By considering the composition, with a corresponding map $\sigma
\land \sigma' : X \to \Sigma$, we get a map $X \to M_f$, naturally in $X$, so
by Yoneda we obtain a map $M_{Rf} \to M_f$ as desired.  Explicitly, this map
looks like
\begin{align*}
  \mu_f &\co 
  \Sigma_{a \in A} \Sigma_{\sigma \in \Sigma} \left(
    \Sigma_{a' \in A} \Sigma_{\sigma' \in \Sigma} B_{a'}^{\sigma'}
  \right)_a^\sigma
\to \Sigma_{a \in A} \Sigma_{\sigma \in \Sigma} B_a^\sigma \\
  \mu_f &(a , \sigma , \chi) = \left( 
  a ,
  \sigma \land \sigma' ,
  \star \in \sigma \land \sigma' \mapsto (p_3 . \chi . p_1 (\star)) (p_2 (\star)
\right)
\end{align*}
Here, the proposition \(\sigma \land \sigma'\) is defined as the set of pairs
\((\star_1, \star_2)\) where \(\star_1 \in \sigma\) and \(\star_2 \in p_2 .
\chi(\star_1)\). One can now easily verify the unit law and associativity.

For the comonad structure, we need to define a map $\delta_f$ making
\diag{ B \ar[r]^1 \ar[d]_{Lf} & B \ar[d]^{LLf} \\
M_f \ar[r]_{\delta_f} & M_{Lf} }
commute. Note that
\[ M_{Lf} = \sum_{(a,\sigma,\tau) \in M_f} \sum_{\sigma' \in \Sigma} (B_{(a,\sigma,\tau)})^{\sigma'}. \]
So if $((a,\sigma,\tau),\sigma',\tau') \in M_{Lf}$ and $\star \in \sigma'$, then $(a,\sigma,\tau) = (f(b),\top,\lambda x.b)$ for $b = \tau'(\star)$; hence $\star \in \sigma$ and $\tau(\star) = \tau'(\star)$. In other words,
\[ M_{Lf} = \{ ((a \in A, \sigma \in \Sigma, \tau \in (B_a)^\sigma), \sigma' \in 
\Sigma, \tau' \in (B_a)^\sigma) \, : \, \sigma' \leq \sigma, \tau \upharpoonright \sigma' = \tau' \}. \]
So we can define a map $\delta_f: M_f \to M_{Lf}$ by sending $(a, \sigma, \tau)$ to $((a, \sigma, \tau), \sigma, \tau)$. Counit laws and coassociativity are easily verified.

The distributive law (Garner equation) can be easily verified using the
explicit notation for \(\delta_f\) and \(\mu_f\): 
\begin{align*}
  \mu_{Lf}.E(\delta_f, \mu_f) . \delta_{Rf} (a, \sigma , \chi) &= 
  \mu_{Lf}.E(\delta_f, \mu_f) . ((a, \sigma, \chi), \sigma , \chi) \\
&=\mu_{Lf} . ( \delta_f . (a, \sigma , \chi), \sigma, \mu_f . \chi ) \\
&= (\mu_f . (a, \sigma , \chi), \sigma \land \sigma', 
\star \in \sigma \land \sigma' \mapsto (p_3 . \delta_f . \chi . p_1 (\star))) \\
&= (\mu_f . (a, \sigma , \chi), \sigma \land \sigma', 
\star \in \sigma \land \sigma' \mapsto (p_3 . \chi . p_1 (\star))) \\
&= \delta_f . \mu_f (a , \sigma , \chi)
\end{align*}
\end{proof}

\begin{namedprop}{coalgebrasfordominance}{Effective cofibrations are precisely coalgebras}
A coalgebra structure for \(f: B \rightarrow A\) is unique, and it exists
if and only if \(f\) belongs to \(\Sigma\).

Hence, there is a cofibred structure 
  \begin{equation}\label{eq:dominancecofibred}
    \sigma: \E_{\mathrm{cocart}} \to \Sets
  \end{equation}
  where \(\sigma(f)\) contains a single element when \(f \in \Sigma\), and is
  empty otherwise. Moreover, there is an isomorphism of cofibred structures
  between \(\sigma\) and the cofibred structure of coalgebras.
\end{namedprop}
\begin{proof} 
 We show that every \(f: B \to A\) can be equipped with a coalgebra structure
 for the copointed endofunctor \(M\) if and only if it belongs to \(\Sigma\),
 and that the coalgebra structure is indeed unique and always satisfies
 the coassociativity condition. From this, it is easy to derive an isomorphism
 of cofibred structures in light of~\eqref{eq:dominancecofibred}.

Suppose $\gamma: A \to M_f$ is a map exhibiting $f$ as a coalgebra for the
copointed endofunctor $M$. In other words, we have $Rf.\gamma = 1_A$ and
$\gamma$ makes
\diag{ B \ar[d]_f \ar[r]^1 & B \ar[d]^{Lf} \\
A \ar[r]_{\gamma} & M_f }
commute. These data correspond to a $\Sigma$-partial map $A \rightharpoonup B$:
\diag{ A' \ar[r]^m \ar[d]_s & B \ar[d]^f \\
A \ar[r]_1 & A }
where \(s \in \Sigma\). Further, \(s\) fits in a pullback square, namely the middle
square of the following diagram:
\[
  \begin{tikzpicture}[baseline={([yshift=-.5ex]current bounding box.center)}]
    \matrix (m) [matrix of math nodes, row sep=2em,
    column sep=3.5em]{
      |(b)| {B} & |(a')| {A'} & |(b2)| {B}   & |(t)| {1} \\
      |(a)| {A} & |(a1)| {A}  & |(mf)| {M_f} & |(s)| {\Sigma} \\
    };
    \begin{scope}[every node/.style={midway,auto,font=\scriptsize}]
    \path[->]
      (b) edge node [left] {$f$} (a)
      (b) edge [dashed] node [above] {$n$} (a')
      (a') edge node [left] {$s$} (a1)
      (b2) edge node [left] {$L_f$} (mf)
      (b2) edge (t)
      (t) edge (s)
      (a') edge node [above] {$m$} (b2)
      (a1) edge node [below] {$\gamma$} (mf)
      (a) edge node [below] {$1$} (a1)
      (mf) edge (s)
      (b2) edge (mf); 
    \end{scope}\end{tikzpicture}
\]
whence we find \(n : B \to A'\) such that \(m.n = 1\).
 Because $s.n.m = f.m = s$ and $s$ is monic, we also have $n.m = 1_{A'}$. So
 $A'$ and $B$ are isomorphic over $A$. It follows $f \in \Sigma$ and that
 \(\gamma\) classifies the map $(f,1_B)$. From this it is clear that
 $\gamma$ must be unique whenever it exists and that it will always satisfy the
 coassociativity condition. It also follows, incidentally, that the square at
 the beginning of the proof is a pullback.

Conversely, if $f \in \Sigma$ we can choose $s = f$ and $m = 1$ and this gives
us the coalgebra structure for the copointed endofunctor we want.
The second part of the proposition follows immediately.
\end{proof}

Lastly, we briefly stop at algebras for the monad.  From
\reftheo{distrlaw<->iso-of-cofibred-str}, we know that the fibred structure of
algebras is isomorphic to the fibred structure of right lifting structures with
respect to the \emph{double} category of coalgebras. It remains to characterize
this category. To that end, we use the following:
\begin{lemm}{morfofcofpbks}
 If $f: B \to A$ and $f': B' \to A'$ are coalgebras, then a pair of maps $(m: B' \to B,n: A' \to A)$ making 
 \diag{ B' \ar[d]_{f'} \ar[r]^m & B \ar[d]^f \\
A' \ar[r]_n & A }
commute is a morphism of coalgebras if and only if the square is a pullback.
\end{lemm}
\begin{proof} It is not hard to check that such a pullback square constitutes a morphism of coalgebras.

For the converse, let us first make the following observation. Suppose $\gamma: A \to M_f$ is a coalgebra structure on $f$. Then $\gamma$ fits into a diagram of the form
\diag{ B \ar[d]_f \ar[r]^1 & B \ar[d]^{Lf} \ar[r] & 1 \ar[d]^{\top} \\
A \ar[r]_{\gamma} & M_f \ar[r] & \Sigma }
where $M_f \to \Sigma$ is the obvious projection. Note that the right hand square is always a pullback and that the left hand square is as well, as we saw in the previous proof. So the outer rectangle is a pullback.

So if
\diag{ B' \ar[d]_{f'} \ar[r]^m & B \ar[d]^f \\
A' \ar[r]_n & A }
is a morphism of coalgebras, then this fits into a commutative diagram of the form
\diag{ B' \ar[d]_{f'} \ar[r]^m & B \ar[d]^f \ar[r] & 1 \ar[d]^\top \\
A' \ar[r]_n & A \ar[r] & \Sigma}
in which the right hand square and the outer rectangle are pullbacks. Therefore the left hand square is a pullback as well.
\end{proof}

Hence we can deduce:
\begin{coro}{algformonadfordominance}\index{effective cofibration!double category of \(\sim\)s}
 Let \(\mathbf{\Sigma}\) be the double category whose horizontal arrows are
 arbitrary arrows from \ct{E}, whose vertical arrows are maps from $\Sigma$ and
 whose squares are pullbacks squares. Then:
 \begin{enumerate}[(i)]
   \item There is an isomorphism between the following notions of fibred structure:
         \begin{itemize}
          \item Algebras for the AWFS induced by \(\Sigma\);
          \item Right lifting structures with respect to \(\mathbf{\Sigma}\).
         \end{itemize}
  \item There is a functor
        \[
          \algcat{R} \to \rlp{\mathbf{\Sigma}}
        \]
        given on objects by (i) which is fully faithful.
  \item There is an equivalence of double categories
        \[
          \alg{R} \cong \Drl{\mathbf{\Sigma}}
        \]
        whose vertical restriction is prescribed by (ii).
    \end{enumerate}
\end{coro}
\begin{proof}
 Because $\mathbf{\Sigma}$ is, essentially, the double category of coalgebras for the comonad.
\end{proof}

\begin{defi}{trivialfibrations}
  The algebras or right lifting structure of~\refcoro{algformonadfordominance}
  are called \emph{effective trivial fibrations}\index{effective trivial fibration|textbf}\index{trivial fibration|seeonly {effective trivial fibration}}.
  The discretely fibred concrete double category of effective trivial
  fibrations is denoted \(\mathbb{E}\cat{ffTrivFib}\), but we may also refer to
  it as category (\(\cat{EffTrivFib}\)) or notion of fibred structure
  (\(\cat{effTrivFib}\)).\index{discretely (co)fibred concrete double
  category!of effective trivial fibrations}
\end{defi}

We will return to effective trivial fibrations in
Section~\ref{ssec:eff-trivial-fibrations}, where we show that they are also
effective fibrations (introduced in Section~\ref{sec:unifKanFibr}). Effective
trivial fibrations are important for developing a theory of effective Kan
fibrations in simplicial sets. For example, they allow us to relate our notion
of effective Kan fibration to other notions of fibration (see
\refcoro{comparisonwithunifKanfibrGS} and below).

 \section{AWFS from Moore structure}\label{sec:AWFSFromM}
\subsection{A category with Moore structure}
In this section we construct an algebraic weak factorisation system on a
\emph{category with Moore structure} \E. We will then study its coalgebras and
algebras more closely. We show that the structure of a coalgebra is equivalent
to the structure of a \emph{hyperdeformation retract} (HDR). Algebras, on
the other hand, will be identified as \emph{naive fibrations}.

Categories with Moore structure are a modification of the \emph{path object
categories}\index{path object category} introduced by Van den
Berg-Garner~\cite{vdBerg-Garner}. In that paper, a \emph{cloven weak
factorisation system} (see Section~3.2 loc.\ cit.) is constructed for an
arbitrary path object category.  Here, we use a new method to show that it is
also possible to construct an \emph{algebraic} weak factorisation system. To
that end, we need to modify the axioms of a Moore structure relative to the
original definition -- but note that ours is neither weaker nor stronger.

\begin{nameddefi}{moorestructure}{Moore structure}\index{Moore structure|textbf}
  \index{category with Moore structure}
  Let \(\E\) be a category with finite limits. A \emph{Moore structure} for
  \(\E\) is a pullback-preserving endofunctor \(M : \E \to \E\) together with:
  \begin{enumerate}[(i)]
    \item
  Natural transformations \(\mu, r, s, t\) with domain and codomains defined by
  the diagram:
  \[
\begin{tikzpicture}[baseline={([yshift=-.5ex]current bounding box.center)}]
    \matrix (m) [matrix of math nodes, row sep=2em,
    column sep=3.5em]{
      |(c)| {MX \times_{(t,s)} MX} &  |(a)| {MX} 
	    & |(b)| {X} \\
    };
    \begin{scope}[every node/.style={midway,auto,font=\scriptsize}]
    \path[->]
	    (c) edge node [anchor=center, fill=white] {$\mu_X$} (c -| a.west)
	    (a) edge [transform canvas={yshift=+6pt}] node [above] {$t_X$} (b)
	    (b) edge node [anchor=center, fill=white] {$r_X$} (a)
	    (a) edge [transform canvas={yshift=-5pt}] node [below] {$s_X$} (b);
    \end{scope}\end{tikzpicture}\text{,}
    \]
  Further, this data gives every object \(X\) in \(\E\) the structure of an
  internal category with arrow object \(MX\).  The object \(MX\) is to be
  interpreted as the object of \emph{paths} in \(X\) with \emph{source} and
  \emph{target} maps \(s, t\), as well as a trivial path \(r\). The object \(M1\)
  is the object of \emph{path lengths} and hence \(M! \co MX \to M1\) takes a path to
  its length.
  \item A natural transformation \(\Gamma: M \Rightarrow MM\), which turns
    \((M,s,\Gamma)\) into a comonad. The idea is that \(\Gamma_X\) sends a path
    \(\gamma\) in \(MX\) to a `path of paths' in \(MMX\), of the same length as
    \(\gamma\), which contracts \(\gamma\) to the trivial path at \(t_X(\gamma)\)
    (see Box~\ref{box:Gamma}).\index{contraction (of paths)}
  \item A natural transformation \(\alpha\) given by a family \(\alpha_X: X
    \times M1 \to MX\) called \emph{strength} (but see (2) below and
    \refdefi{pathobjcat}). This map sends a point and a length to the constant
    path of that length.
\end{enumerate}
These conditions are subject to further conditions on interactions between
them.  The full definition can be found in the Appendix, \refdefi{pathobjcat}.
There we make each of the conditions (i)--(iii) more precise and we present the
further conditions on interaction between these elements.  The most important
one is the \emph{distributive law}, which we also spell out in (2) below. 
\end{nameddefi}

\begin{textbox}{The path contraction \(\Gamma\)\label{box:Gamma}}
\begin{center}
\begin{tikzpicture}
  \draw [white, pattern = vertical lines, pattern color = lightgray] (0,0) -- (3,0) -- (0,2) -- (0,0); 
  \node (x0) at (0,0) {$x_0$};
  \node (x1) at (3,0) {$x_1$};
  \node (x1') at (0,2) {$x_1$};
  \node  (G) at (1,0.75) {$\Gamma$};
  \node (G) at (1,0.35) {$\Rightarrow$};
  \path[->, draw, thick] (x0) to node [below left] {$\gamma$} (x1);
  \path[->, draw, thick] (x0) to node [left] {$\gamma$} (x1');
  \path[->, draw, thick] (x1') to (x1);
  \path[->, draw, thick, dotted] (1.5, 0) to (1.5,1);
\end{tikzpicture}
\end{center}
The natural transformation \(\Gamma\) applied to a path \(\gamma\) can be
viewed as a contraction of \(\gamma\) to the trivial path on its end point. The
contraction has the same length as the path itself. Each transversal `fibre' of
consists of a final segment of \(\gamma\) (illustrated by the dotted arrow),
namely that which remains after one has contracted along the segment before.
This intuition is contained in the distributive law (see
\refdefi{pathobjcat}~(5) in the Appendix). The top path is obtained by taking
the pointwise target of \(\Gamma(\gamma)\), which is the constant path of the
same length as \(\gamma\).
\end{textbox}

The differences between \refdefi{moorestructure} (and
\refdefi{pathobjcat}) and the one in~\cite{vdBerg-Garner} are the following: 
\begin{enumerate}
\item The coassociativity condition:
\[ M\Gamma.\Gamma = \Gamma M.\Gamma: MX \to MMX. \]
which turns \(\Gamma\) into a comonad;
\item The distributive law:\index{distributive law!for a Moore structure}
\begin{align*}
\Gamma.\mu = &
\mu.(M \mu.\nu.(\Gamma.p_1, \alpha.(p_2, M !.p_1)), \Gamma.p_2)
: MX \times_X MX \to MMX.
\end{align*}
where we have defined the natural transformation \(\alpha : X \times M1 \to
X\) by \(\alpha_X := \theta_{MX}.\alpha_{MX,1}\) with \(\theta, \alpha_{-,-}\)
as in \refrema{onthestrength} in the Appendix. 
A diagrammatic illustration of this condition can be found in
equation~\eqref{eq:distributive-law-viz} in the same place. 
\item The `twist map' \(\tau : M \Rightarrow M\) is dropped, only to 
  return in Section~\ref{sec:frobenius} as \emph{symmetric} Moore structure.
\end{enumerate}
In addition, we will also introduce a notion of \emph{symmetric} and
\emph{two-sided} Moore structure in Section~\ref{sec:frobenius}.  In
\refexam{examplesofsymmMoorecats} in the Appendix we have given several
examples of categories with (symmetric) Moore structure. In
Chapter~\ref{ch:simplicial-sets} of this paper, we will show that the category
of simplicial sets is a category with Moore structure, for the simplicial Moore
path functor defined in
\cite{vdBerg-Garner}.

\begin{namedrema}{contractionvsconnection}{Path contraction vs.~connections}
  The `path contraction' \(\Gamma\) of \refdefi{moorestructure} can be compared
  to the notion of \emph{connection} in the category of Cubical
  Sets~\cite{CCHM17}, and more abstractly in categories with a functorial
  cylinder~\cite{Gambino-Sattler}.  In Appendix~\ref{sec:cubicalsets}, we have
  worked out that indeed, connections in cubical sets\index{cubical sets} can
  be used to define a path contraction \(\Gamma\) for a standard choice of path
  object. However, since all paths in this object have the same length, it is
  not possible to define a suitable notion of path composition for this object.
  \index{connection (cubical sets)}
\end{namedrema}

\subsection{An algebraic weak factorisation system}\index{algebraic weak factorisation system!from a Moore structure}
This subsection shows how every Moore structure bears an algebraic weak
factorisation system. A similar result, in absence of distributive laws, can
also be found in North (\cite{North}, Theorem~3.28). Here, we improve on this
by showing that our distributive law for Moore structures implies the
distributive law for an AWFS\@. In the short
subsections that follow, we address each one of the requirements of
\refprop{awfs-equivalent} (i).

\subsubsection{Functorial factorisation.}\label{sssec:functorial-factorisation}
First of all, if $f: A \to B$ is a morphism, we can factor it as
\diag{A \ar[r]^{Lf} & Ef \ar[r]^{Rf} & B }
by putting $Ef = MB \times_B A$ (pullback of $t$ and $f$), $Lf = (r.f, 1)$ and
$Rf = s.p_1$. In the obvious way \(E\), \(L\) and \(R\) extend to functors,
and the whole factorisation is readily seen to be functorial.

\subsubsection{The comonad.}
The comultiplication $\delta_f$ needs to be a mapping filling:
\diag{ A \ar[d]_{Lf = (r.f, 1)} \ar[rr]^1 & & A \ar[d]^{LLf = (r.(r.f, 1), 1)} \\
MB \times_B A \ar[rr]_(.3){\delta_f} & &
M(MB \times_B A) \times_{MB \times A} A }
where the object in the lower right-hand corner is the pullback of \(t_{MB
\times_B A}\) and \(Lf\).  Note that there is a mediating natural isomorphism
\begin{gather*}
  (MMB \times_{MB} MA) \times_{MB \times_B A} A 
  \xrightarrow{(\nu, 1)}
  M(MB \times_B A) \times_{MB \times_B A} A 
 \end{gather*}
where \(\nu\) is the natural isomorphism induced by pullback-preservation
of \(M\). Hence we can put:
\begin{equation}\label{eq:defn:comultiplication}
  \delta_f =  (\nu.(\Gamma.p_1, \alpha.(p_2, M !.p_1)), p_2).
\end{equation}
Curiously, \(\delta_f\) in no way refers to $f$. 
We can check it is well-defined.
For the part mapping into \(MMB \times_{MB} MA\), we have:
\begin{align*}
Mt.\Gamma.p_1  = & \alpha.(t,M!).p_1 \\
 = & \alpha.(t.p_1,M!.p_1) \\
 = &  \alpha.(f.p_2,M!.p_1) \\
 = & Mf.\alpha.(p_2,M!.p_1).
\end{align*}
Further, \(t.\Gamma.p_1  =  r.t.p_1 = r.f.p_2\). In addition, the square commutes, as one easily checks.
  
From here it remains to check the comonad laws given in \refprop{awfs-equivalent}.
First of all:
\begin{align*}
s.p_1.\delta_f = & (s.\Gamma.p_1, s.\alpha.(p_2, M !.p_1))\\
 = & (p_1,p_2) \\
 = & 1
\end{align*}
and
\begin{align*}
(M(s.p_1).p_1, p_2).\delta_f  = & (M s.\Gamma.p_1, p_2)\\
 = & (p_1,p_2) \\
 = & 1,
\end{align*}
and hence the counit laws are satisfied.
Second, we check coassociativity:
\begin{align*}
E(1, \delta_f) = &
(M \delta_f.p_1, p_2).\delta_f
= \\ 
= &
(\nu.(\nu.(M \Gamma.M p_1.p_1, M\alpha.\nu.(Mp_2.p_1, MM !.M p_1.p_1)), M p_2. p_1)
  , p_2).\delta_f \\
= & 
(\nu. (\nu.(M \Gamma.\Gamma.p_1, 
   M\alpha.(\alpha.(p_2, M !.p_1),MM!.\Gamma.p_1)), \\
  & \qquad M \alpha.(p_2, M !.p_1)), p_2)
\\
=^1 & (\nu.(\nu.(\Gamma.\Gamma.p_1, \Gamma.\alpha.(p_2, M !.p_1)), \alpha.(p_2, M !.p_1)), p_2) \\
=^2 & (\nu. (\Gamma.p_1, \alpha.(p_2, M !.p_1)), p_2).\delta_f\\
= & \delta_{L_f}.\delta_f 
\end{align*}
where we have used the strength of \(\Gamma\) with respect to \(\alpha\)
at \(=^1\) and the identity
\[
  M!.\nu.(\Gamma.p_1,\alpha.(p_2,M!.p_1)) = M!.p_1
\]
at \(=^2\).

\subsubsection{The monad.}\label{sssec:the-monad} The multiplication is given by:
\begin{equation}\label{eq:defn:multiplication}
  \begin{tikzpicture}[baseline={([yshift=-.5ex]current bounding box.center)}]
    \matrix (m) [matrix of math nodes, row sep=3.5em,
    column sep=4em]{ 
      |(a)| {MB \times_{(t,s.p_1)} (MB \times_B A)} & |(a2)| {MB \times_B A} \\
    	|(b1)|  {B}  & |(b2)|  {B}
    \\};
    \begin{scope}[every node/.style={midway,auto,font=\scriptsize}]
      \path[->]
        (a) edge node [above] {$\mu_f$} (a2)
        (a) edge node [left] {$s.p_1$} (b1)
        (a2) edge node [right] {$s.p_1$} (b2)
        (b1) edge node [above] {$1$} (b2);
      \end{scope}
  \end{tikzpicture}
\end{equation}
with \(\mu_f= (\mu. (p_1, p_1.p_2), p_2.p_2)\), reminding the reader that we
write the arguments to \(\mu\) in sequential order rather than order of
categorical composition (\(\mu\) is to be thought of as path concatenation).
Also note that again, the definition of the multiplication formula in no way
refers to \(f\). The monad laws are easy given that \(MB\) is an internal
category.

\subsubsection{The distributive law.}\index{distributive law}\index{Garner equation}
We should check the Garner equation~\eqref{eq:garner-equation-awfs},
given by the identity
\begin{align}\label{eq:garner-equation}
    \delta_f . \mu_f &= \mu_{Lf} . E (\delta_f, \mu_f) . \delta_{R f}
\end{align}
for every \(f\).
We unfold~\eqref{eq:garner-equation} for the definitions presented above.
We do so by (as before) explicitly by writing \(\nu\) for the (family of)
mediating isomorphisms:
\begin{equation}\label{eq:MpreservesPullback}
    \nu : M A \times_{M C} M B \rightarrow M ( A \times_C B)
\end{equation}
For this, the expression for the left-hand side 
of~\eqref{eq:garner-equation} amounts to:
\begin{align*}
&\delta_f . \mu_f &= \\ 
   & 
(\nu.(\Gamma.p_1,\alpha.(p_2, M !.p_1)), p_2).(\mu.(p_1, p_1.p_2), p_2.p_2)
   &= \\
   & 
(\nu.(\Gamma.\mu.(p_1, p_1.p_2),\alpha.(p_2.p_2, M !.\mu.(p_1.p_2, p_1))), p_2.p_2)
\end{align*}
The right-hand side amounts to:
\begin{align*}
    &\mu_{Lf} .  E (\delta_f, \mu_f) . \delta_{R f} & =  \\
    &(\mu.(p_1, p_1.p_2), p_2.p_2).\\
    & (M (\mu.(p_1, p_1.p_2), p_2.p_2).p_1, (\nu.(\Gamma.p_1,\alpha.(p_2, M !.p_1)), p_2).p_2) . \\
    & (\nu.(\Gamma.p_1,\alpha.(p_2, M !.p_1)), p_2) &= \\
    & (\mu.(p_1, p_1.p_2), p_2.p_2) .  &\\
    & (M (\mu.(p_1, p_1.p_2), p_2.p_2).\nu.(\Gamma.p_1,\alpha.(p_2, M !.p_1)), (\nu.(\Gamma.p_1,\alpha.(p_2, M !.p_1)), p_2).p_2) &= \\
    &   
(\mu.(M (\mu.(p_1, p_1.p_2), p_2.p_2).\nu.(\Gamma.p_1,\alpha.(p_2, M !.p_1)), \\
    & \qquad p_1.(\nu.(\Gamma.p_1,\alpha.(p_2, M !.p_1)), p_2).p_2), p_2.p_2)
\end{align*}
Since the second components of the two expressions are the same, we can focus on the first component. 
We proceed by reducing this first component for the right-hand side as follows:
\begin{align*}
\mu.(M (\mu.(p_1, p_1.p_2), p_2.p_2).\nu.(\Gamma.p_1,\alpha.(p_2, M !.p_1)),
  \quad\\  \qquad p_1.(\nu.(\Gamma.p_1,\alpha.(p_2, M !.p_1)), p_2).p_2) = \\
\mu.(M (\mu.(p_1, p_1.p_2), p_2.p_2).\nu.(\Gamma.p_1,\alpha.(p_2, M !.p_1)),
  \quad\\  \qquad \nu.(\Gamma.p_1,\alpha.(p_2, M !.p_1)).p_2) =^1 \\
\mu.(\nu.(M \mu.\nu.(p_1, M(p_1).p_2), M(p_2).p_2).(\Gamma.p_1,\alpha.(p_2, M !.p_1)), \quad\\
  \qquad \nu.(\Gamma.p_1,\alpha.(p_2, M !.p_1)).p_2) =\\
\mu.(\nu.(M \mu.\nu.(\Gamma.p_1, M(p_1).\alpha.(p_2, M !.p_1)), M(p_2).\alpha.(p_2, M !.p_1)), \quad\\ 
 \qquad \nu.(\Gamma.p_1,\alpha.(p_2, M !.p_1)).p_2) =^2 \\
\mu.(\nu.(M \mu.\nu.(\Gamma.p_1,\alpha.(p_1.p_2, M !.p_1)),\alpha.(p_2.p_2, M !.p_1)), \quad\\
\qquad \nu.(\Gamma.p_1.p_2,\alpha.(p_2.p_2, M !.p_1.p_2))) =^3 \\
\nu.(\mu.(M \mu.\nu.(\Gamma.p_1,\alpha.(p_1.p_2, M !.p_1)), \Gamma.p_1.p_2),
  \quad\\ 
\qquad \mu.(\alpha.(p_2.p_2, M !.p_1),\alpha.(p_2.p_2, M !.p_1.p_2))) = \\
\nu.(\Gamma.\mu.(p_1, p_1.p_2),\alpha.(p_2.p_2, M !.\mu.(p_1.p_2, p_1)))
\end{align*}
Here \(=^1\) uses the identity 
\[
\nu.(M \mu.(p_1, M p_1.p_2), M p_2.p_2) = 
M (\mu.(p_1, p_1.p_2), p_2.p_2).\nu
\]
which is just naturality of \(\nu\).  At \(=^2\) we have twice used naturality
of \(\alpha\) and absorbed a projection term in the right term of the
composite.  At that point, we are left with a nice expression of the form:
\[
    \mu . \nu \times \nu \thinspace : \thinspace  (M A  \times_{MB} MM B) \times_{A \times_B MB} (M A \times_{MB} MMB) \rightarrow (M A \times_{MB} MM B)
\]
Naturality of \(\mu\) and \(\nu\) implies that this map can be rewritten as
\begin{equation}\label{eq:mu-nu-commutes}
\nu.(\mu.(p_1.p_1, p_1.p_2), \mu.(p_2.p_1, p_2.p_2))
\end{equation}
which we have done at \(=^3\). In the resulting expression, we recognize in the
right component the composition of two constant paths, which is the constant
path on the composition, and in the left component the distributed term of the
distributive law between \(\mu\) and \(\Gamma\), which yields the desired
equality.

We can summarise the result of this section as follows:
\begin{prop}{MooreStructure->AWFS}\index{algebraic weak factorisation system!from a Moore structure}
  Suppose \(\E\) is a category with Moore structure. 
  Then the functorial factorisation\index{functorial factorisation!from a Moore structure} \((L, R, \epsilon, \eta)\) given by
	\[
\begin{tikzpicture}[baseline={([yshift=-.5ex]current bounding box.center)}]
    \matrix (m) [matrix of math nodes, row sep=2em,
    column sep=4.5em]{
	    |(a)| {A} &  |(e)| {MB \times_B A} 
	    & |(b)| {B} \\
    };
    \begin{scope}[every node/.style={midway,auto,font=\scriptsize}]
    \path[->]
      (a) edge node [above] {$Lf := (r.f, 1)$}  (e)
      (e) edge node [above] {$Rf := s.p_1$} (b);
    \end{scope}\end{tikzpicture}
    \]
    together with the natural transformations \(\delta: L \Rightarrow LL\)
    defined by~\eqref{eq:defn:comultiplication} and \(\mu: RR \Rightarrow R\)
    defined by~\eqref{eq:defn:multiplication} constitutes an algebraic weak factorisation system (AWFS)
    on \(\E\) in the sense of \refdefi{awfs}.
\(\hfill \square\)
\end{prop}

\subsection{Coalgebras}\label{ssec:HDRs}
We start with defining \emph{hyperdeformation retractions} and
\emph{hyperdeformation retracts}\index{hyperdeformation retraction|textbf}\index{HDR|seeonly {hyperdeformation retraction}}.
\begin{defi}{HDR} Let $i: A \to B$ be a map. To equip $i$ with a \emph{hyperdeformation retraction} means specifying a map $j: B \to A$ and a homotopy $H: B \to MB$ such that the following hold:
\[ j.i = 1_A, \quad s.H = 1_B, \quad t.H = i.j, \quad MH.H=\Gamma.H. \]
If such a structure can be specified for $i$, we will call it a \emph{hyperdeformation retract}, or \emph{HDR}\index{hyperdeformation retract}.
\end{defi}

The intuition behind a hyperdeformation retract \(i: A \to B\) is as follows.
First, the family of paths \(H : B \to M B\) forms a \emph{deformation
retraction} of \(B\) onto \(A\), contracting every element in \(B\) to an
element in \(A\) along the path \(H(b)\). Second, every other element in \(B\)
lying on this path \(H(b)\) retracts along the remaining segment of \(H(b)\)
defined by it. This is expressed by the last condition involving \(\Gamma\).
See Figure~\ref{fig:HDR} for an illustration of this condition.

The notion of a hyperdeformation retraction is a strengthening of the notion of
a \emph{strong} deformation retraction (see Van den
Berg-Garner~\cite{vdBerg-Garner}, Definition~6.1.3). This notion is also used
by Gambino and Sattler in the context of (algebraic) weak factorisation systems
(See~\cite{Gambino-Sattler}, Section~3 and~\cite{sattler18}) and by Cisinski in
the context of homotopical algebra (Cisinski~\cite{cisinski19}, Section~2.4). A
deformation retraction is \emph{strong} if (in the notation of \refdefi{HDR})
\(H.i = r.i\), i.e., the retraction restricted to \(A\) is trivial. Every
hyperdeformation retraction is automatically strong, as is formally shown in
the proof of \refprop{HDRiscoalgebra} below.

\begin{figure}
  \centering
\begin{tikzpicture}
  \draw[thick, lightgray, fill = lightgray] (0,0) -- (3,2) -- (3,0);
  \draw[thick] (0,0) --  (3,0);
  \node (a0) at (2.5,0) {};
  \node (b0) at (3,{2/3*(3 - 2.5)}) {};
  \node (a1) at (0.5,0) {};
  \node (b2) at (3,{2 - 1/3}) {};
  \node (b1) at (1.75, {2/3*(1.75 - 0.5)}) {};
  \draw [fill=black] (a0) circle [radius=1pt];
  \draw [fill=black] (b0) circle [radius=1pt];
  \draw [fill=black] (a1) circle [radius=1pt];
  \draw [fill=black] (b2) circle [radius=1pt];
  \draw [fill=black] (b1) circle [radius=1pt];
  \path [<-, draw, dashed] (a0) to (b0);
  \path [<-, draw, dashed] (a1) to (b1);
  \path [<-, draw, dashed] (b1) to (b2);
  \node at (b0) [right] {$b_0$};
  \node at (b2) [right] {$b_2$};
  \node at (a0) [below] {$a_0$};
  \node at (a1) [below] {$a_1$};
  \node at (b1) [below right] {$b_1$};
  \node (A) at (1.5,0) [below] {$A$};
  \node (A) at (1.8,1.7) {$B$};
\end{tikzpicture}
\caption{A \emph{hyperdeformation retract} (HDR) \(i: A \to B\).  If an element \(b_2\) is retracted along a path to an element \(a_1\), then any intermediate element \(b_1\) is retracted along the remainder of that path.\label{fig:HDR}}
\end{figure}
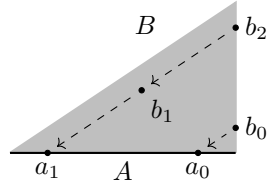

\begin{rema}{connisHDR} The maps $t$ and $\Gamma$ equip any $r: X \to MX$ with the structure of a hyperdeformation retraction.
\end{rema}

\begin{prop}{HDRiscofibredstructure}
The function which associates to every \(i: A \to B\) the set
\[\{(H, j) \varco (i, j, H) \text{ is an HDR}\}\]
can be extended to a presheaf on the category of
arrows of \(\E\) and cocartesian (pushout) squares:
\[
        \cat{hdr} \co \E^{\to}_\mathrm{cocart} \to \Sets.
\]
So HDRs form a cofibred structure on \(\E\).
\end{prop}
\begin{proof}
It is sufficient to show that HDRs are closed under pushouts in a way compatible
with composition of pushout squares. The reader is invited to do this as an exercise,
as it will also be clear from the proof of \refprop{HDRiscoalgebra}.
\end{proof}
\begin{prop}{HDRiscoalgebra}
The following notions of cofibred structure are isomorphic:
\begin{itemize}
\item Having a coalgebra structure with respect to \(L = (r.f, 1)\),
\item Having the structure of an HDR\@.
\end{itemize}
\end{prop}
\begin{proof}
Suppose \(i: A \to B\) is an arrow in \(\E\).
The map $i: A \to B$ carries a coalgebra structure if there is a map $(H, j)$ making 
\diag{ A \ar[r]^1 \ar[d]_i & A \ar[d]^{(1,r.i)} \\
B \ar[r]_{(H,j)} & MB \times_B A }
commute (which means $i.j = t.H, j.i = 1, H.i = r.i$) and such that $s.p_1.(H,j) = 1$ (that is, $s.H = 1$) and 
\(
\delta.(H, j) = (p_1.M (H, j), p_2).(H, j)
\). The latter condition means that 
\[ 
\delta.(H, j) = (\nu.(\Gamma.p_1,\alpha.(p_2, M !.p_1)), p_2).(H, j)
= (\nu.(\Gamma. H, \alpha. (j, M! . H)), j)
\]
should equal
\[ 
(M (H, j).p_1, p_2).(H, j) = (\nu.(MH.H, Mj.H), j)
\]
(using naturality of \(\nu\)). In other words, that $\alpha.(j,M!.H) =
Mj.H$ and $\Gamma.H = MH.H$. To summarise, a coalgebra structure is a
hyperdeformation retraction satisfying, additionally, $H.i = r.i$ and
$Mj.H=\alpha.(j,M!.H)$. We now show that these two conditions are always
satisfied.

First of all, we have 
\[ H.i = H.i.j.i = H.t.H.i = t.MH.H.i = t.\Gamma.H.i = r.t.H.i=r.i.j.i =r.i, \]
showing that a hyperdeformation retraction is automatically a strong deformation retraction.

Secondly, to show $Mj.H=\alpha.(j,M!.H)$, we calculate:
\begin{eqnarray*}
Mi.Mj.H & = & M(i.j).H \\
& = & M(t.H).H \\
& = & Mt.MH.H \\
& = & Mt.\Gamma.H \\
& = &\alpha.(t,M!).H \\
& = &\alpha.(t.H,M!.H) \\
& = &\alpha.(i.j,M!.H) \\
& = & Mi.\alpha.(j,M!.H).
\end{eqnarray*}
Since $Mi$ is monic (even split monic), this proves the claim.
We leave it to the reader to verify that this construction is functorial
and induces an isomorphism of cofibred structures.
\end{proof}

\begin{coro}{vert-co-HDRs}
HDRs admit a vertical composition\index{vertical composition!of HDRs}, given
by vertical composition of coalgebras for an AWFS\@.  Explicitly, the
composition of two HDRs \(i_0: A \to B\), \(i_1, B \to C\) is given by \(i_1
. i_0\) with inverse map \(j_0 . j_1\) and deformation
\begin{equation}\label{eq:composition-of-hdrs}
H_1 * H_0 := \mu.(H_1, M i_1.H_0.j_1)
\end{equation}
\end{coro}
\begin{proof}
 We only need to verify that the given formula and inverse represents composition of coalgebras,
 as defined by the formula \eqref{eq:coalgebra-composition-fla} above. Expanding this
 formula for \(h := i_1 . i_0\) yields:
 \begin{align*}
& (H_1, j_1) * ( H_0, j_0 ) \\
&= \mu_h . E(E(1_A, i_1) . (H_0, j_0), 1_C) . (H_1, j_1) \\
&= (\mu.(p_1, p_1.p_2), p_2.p_2).(M 1_C.p_1, (M i_1.p_1, p_2).(H_0, j_0).p_2).(H_1, j_1) \\
&= (\mu.(p_1, p_1.p_2), p_2.p_2).(H_1, (M i_1.H_0, j_0).j_1) \\
&= (\mu.(H_1, M i_1.H_0.j_1), j_0.j_1)
 \end{align*}
 whence the statement follows.
\end{proof}

Similarly, we have the rest of the structure of a double category, whose vertical
morphisms are HDRs. In the rest of this section, we will study this structure in more
depth and using HDRs rather than coalgebras. If $(i': A' \to B', j', H')$ and $(i: A \to
B, j, H)$ are HDRs, then a \emph{morphism of HDRs} is
defined as a pair of maps $f: A' \to A$ and $g: B' \to B$ such that
\diag{ A' \ar[d]_{i'} \ar[r]^f & A \ar[d]^i \\
B' \ar[r]_g \ar[d]_{j'} & B \ar[d]^j \\
A' \ar[r]_f & A }
commutes and $Mg.H' = H.g: B' \to MB$.

Using the dual of \refdefi{discretelyfibreddoublecat}, we can summarise this section:
\begin{coro}{hdrscoalgsdoubleiso}
  Hyperdeformation retracts define a discretely cofibred double category
  \(\mathbb{H}\cat{DR}\) over \(\E\) which is isomorphic to the double category
  of coalgebras.\index{hyperdeformation retract!double category of \(\sim\)s}\index{discretely (co)fibred concrete double category!of HDRs}
\end{coro}
\begin{proof}
  Using the dual of \refprop{fibreddoubleiso} and \refprop{HDRiscoalgebra}, it
  is enough to show that the correspondence of \refprop{HDRiscoalgebra} is
  induced by a double functor \(\coalg{L} \to \mathbb{H}\cat{DR}\) over
  \(\sq{\E}\) which is full on squares. The functor exists by construction of
  the vertical composition, and the fact that a morphism of coalgebras is a
  morphism of HDRs. To see that it is full on squares, means to see that a
  morphism of HDRs is automatically a morphism of coalgebras, which also
  follows from the above definition.
\end{proof}

\subsubsection{HDRs are comonadic}
In this subsection and onwards, we denote by \(\cat{HDR}\) the vertical part, as
in~\refprop{vertical-part-left-adj}, of the double category of HDRs, i.e.\ the
category of HDRs and morphisms of HDRs. We first consider the \emph{codomain}
functor:
\begin{equation}\label{HDRcodfunctor}
	\cod : \cat{HDR} \to \E.
\end{equation}
The following facts will be helpful in \refprop{HDRsareMcoalgebras} below.
\begin{lemm}{pbksqforsqofHDRs}
If $(i,j,H)$ is an HDR, then 
\diag{ A \ar[d]_i \ar[r]^i & B \ar[d]^H  \\
B \ar[r]_r & MB } 
is a pullback. In particular, \(i\) is the equalizer of the pair \(r, H : B \to MB\).
\end{lemm}
\begin{proof}
If $i: A \to B$ is an HDR, as witnessed by $j: B \to A$ and $H: B \to MB$, then
\diag{ A \ar[d]_i \ar[r]^i & B \ar[d]^H \ar[r]^j & A \ar[d]^i \\
B \ar[r]_r & MB \ar[r]_t & B} 
exhibits $i$ as a retract of $H$. Since $H$ is monic ($s.H = 1$),
\reflemm{swanslemma} (see below) tells us that the left hand square is a
pullback.
\end{proof}
\begin{lemm}{swanslemma} 
 If the commutative diagram \diag{ A \ar[r] \ar[d]^m & C \ar[r] \ar[d]^n & A \ar[d]^m \\
B \ar[r] & D \ar[r] & B }
exhibits $m$ as a retract of $n$ and $n$ is a monomorphism, then the left hand square is a pullback.
\end{lemm}
\begin{proof}
  This is dual to \reflemm{swanslemmadual} in the Appendix. A reference is given there. 
\end{proof}

\begin{coro}{retractscommute}
	Suppose \((i', j', H')\), \((i,j,H)\) are HDRs and \(g: B' \to B\)
	satisfies \(Mg . H' = H.g\). Then there is a unique \(f: A' \to A\)
	such that \((f, g)\) is a morphism of HDRs.
\end{coro}
\begin{proof}
The unique \(f\) is defined by virtue of \reflemm{pbksqforsqofHDRs}.
It remains to verify that square \(j' \to j\) between the retracts
commutes:
\begin{align*}
    j . g &= j . i . j . g  = j . t . H . g = t . M j. H . g \\
          &= t . M j . M g . H' = t . M (j . g) . H' \\ 
          &= j . g . t . H' = j. g. i' . j' = j . i . f . j' = f . j'
\end{align*}
So the conclusion follows.
\end{proof}

The following proposition combines the previous two observations and will be
used for developing the Frobenius construction
(\reflemm{frobeniusStableUnderComposition}) and for developing the theory of
HDRs in simplicial sets (e.g., \reftheo{HDRpreshofT}).
\begin{prop}{HDRsareMcoalgebras}
The functor~\eqref{HDRcodfunctor} is comonadic, where the corresponding
comonad is just the Moore functor \(M: \E \to \E\) with comonad structure
\((s, \Gamma)\). Specifically
\begin{equation}\label{HDRcodequivalence}
	\cod : \cat{HDR} \to \coalgcat{M},
\end{equation}
which sends \((i, j, H)\) to \((B, H : B \to MB)\), is an equivalence of categories.
\end{prop}
\begin{proof}
By \refdefi{HDR}, the functor~\eqref{HDRcodequivalence} has the correct codomain, i.e.
it maps into \(M\)-coalgebras.
By \refcoro{retractscommute}, the functor is full and faithful.
Further, \reflemm{pbksqforsqofHDRs} and its proof imply that that it is essentially
surjective. So we have an equivalence of categories.
\end{proof}

We can use the previous result to prove the following:
\begin{coro}{HDRhaspullbacks}
The category of HDRs has pullbacks.
\end{coro}
\begin{proof}
  By \refprop{HDRsareMcoalgebras}, the functor~\eqref{HDRcodfunctor} creates limits
  preserved by \(M\). So it follows because \(M\) preserves pullbacks.
\end{proof}

To end this section, we record the following fact about HDRs, which holds when
the unit \(r_X : X \to MX\) is a \emph{cartesian} natural transformation, i.e.\
all naturality squares are pullbacks.  Note this happens to be true for
simplicial sets -- but it is not an assumption of our theory.\index{cartesian natural transformation}
\begin{prop}{sqofHDRspbk}
When \(r_X : X \to MX\) is a \emph{cartesian} natural transformation,
 every morphism of HDRs is a pullback square.
\end{prop}
\begin{proof}
 If \diag{ B' \ar[d]_{i'} \ar[r]^f & B \ar[d]^i \\
 A' \ar[r]_g & A }
 is (the top part of) a morphism of HDRs, then it fits into a commutative cube:
 \diag{ & B \ar'[d]^i[dd] \ar[rr]^i & & A \ar[dd]^H \\
B' \ar[rr]^(.7){i'} \ar[dd]_{i'} \ar[ur]^f & & A' \ar[ur]^g \ar[dd]^(.3){H'} \\
& A \ar'[r]^r[rr] & & MA \\
A' \ar[rr]_r \ar[ur]^g & & MA' \ar[ur]_{Mg} }
In this cube front and back faces are pullbacks (by~\reflemm{pbksqforsqofHDRs}), as is the bottom face (because $r$ is a cartesian natural transformation). Therefore the top face is a pullback as well.
\end{proof}

\subsubsection{HDRs are bifibred}
In the previous section, we have studied the codomain functor \(\cod: \cat{HDR}
\to \E\) and obtained the main result that it is comonadic. Next, we will turn
to the \emph{domain} functor \(\dom: \cat{HDR} \to \E\). 
\begin{defi}{pbksqofHDRs}\index{cartesian morphism of HDRs}
A morphism of HDRs will be called a \emph{cartesian morphism of HDRs} if also
 the bottom part, i.e.\ the square for $j'$ and $j$, is a pullback.
\end{defi}

\begin{defi}{cocartsqofHDRs}
	A morphism of HDRs will be called a \emph{cocartesian morphism of HDRs} if the
	square for \(i'\), \(i\) and the square for \(j'\), \(j\) are pushouts.
\end{defi}

The main result of this subsection is the following:
\begin{prop}{domHDRfibration}
  The domain functor \[\dom : \cat{HDR} \to \E\] is a \emph{bifibration}
  \index{bifibration} (see Section~\ref{sec:preliminaries}), whose cartesian
  morphisms are given by \refdefi{pbksqofHDRs}, and whose cocartesian morphisms
  are given by \refdefi{cocartsqofHDRs}. Moreover, this bifibration satisfies
  the Beck-Chevalley condition (see Box~\ref{box:beck-chevalley} on
  page~\pageref{box:beck-chevalley}).\index{Beck-Chevalley condition}
\end{prop}

Before proving the proposition, we prove the following two lemmas:
\begin{lemm}{retractfibration}
    Suppose \cat{r} is the universal retract, in other words the category
    \diag{\bullet_0 \ar[r]^{i} & \bullet_1  \ar[r]^{j} & \bullet_0}
    with \(j.i = 1\). Then the functor \[\ev_0: \E^{\cat{r}} \to \E\]
    which sends a retract pair \((i,j)\) to \(\dom i\)
    is a bifibration satisfying the Beck-Chevalley condition.
\end{lemm}
\begin{proof}
	To prove that it is a fibration, suppose \(i : A \to B\), \(j: B \to
	A\) is a retract pair and suppose \(f: A' \to A\) is any morphism. Then
	we can form a double pullback 
    \diag{A' \ar[r]^{f} \ar[d]_{i':=(1,i.f)}&  A \ar[d]^{i} \\
        B' \ar[r]^{g} \ar[d]_{j'} & B \ar[d]^j \\
      A' \ar[r]^{f} & A}
      resulting in a morphism of retract pairs \((i', j') \to (i, j)\).
      It is enough to verify that this morphism is cartesian over \(f\),
      which is very easy.

      Similarly, for a pair \((i' : A' \to A, j' : B'\to B)\) and a morphism 
      \(f: A' \to A\), the double pushout diagram
    \diag{A' \ar[r]^{f} \ar[d]_{i'}&  A \ar[d]^{i} \\
	    B' \ar[r]^{g} \ar[d]_{j'} & B \ar[d]^{[f.j', 1]} \\
      A' \ar[r]^{f} & A}
      yields a retract pair \((i, [f.j', 1])\) and it is easy to see that this morphism
      of retract pairs satisfies the universal property of a cocartesian morphism over
      \(f\).

      For the Beck-Chevalley condition we prove the `\(\Rightarrow\)' direction
      of the definition explained in Box~\ref{box:beck-chevalley} on
      page~\pageref{box:beck-chevalley}. So consider the situation
      of~\eqref{eq:defn:beck-chevalley} for the case at hand, i.e.\ a diagram:
\[
  \begin{tikzpicture}[baseline={([yshift=-.5ex]current bounding box.center)}]
    \matrix (m) [matrix of math nodes, row sep=3em,
    column sep=3em]{ 
	    |(a')| {A'} & |(a)| {A} \\
	    |(b')| {B'} & |(b)| {B} \\
	    |(c')| {A'} & |(c)| {A}
    \\};
    \matrix (n) [matrix of math nodes, row sep=3em,
    column sep=3em, position=30:-1.7 from m]{ 
	    |(ba')| {C'} & |(ba)| {C} \\
	    |(bb')| {D'} & |(bb)| {D} \\
	    |(bc')| {C'} & |(bc)| {C}
    \\};
    \begin{scope}[every node/.style={midway,auto,font=\scriptsize}]
      \path[->]
	      (ba) edge (bb)
	      (ba') edge (ba)
	      (ba') edge [dashed] (bb')
	      (bb') edge [dashed] (bb)
	      (bc') edge [dashed] node [pos=0.6, anchor=center, fill=white] {$f'$} (bc)
	      (bb') edge [dashed] (bc')
	      (bb) edge [dashed] (bc);
      \path[->]
	      (a') edge node [anchor=center, fill=white] {\(f\)} (a)
	      (a') edge (b')
	      (a) edge (b)
	      (b') edge (b)
	      (c') edge node [below] {$f$} (c)
	      (b') edge (c')
	      (b) edge (c);
      \path[->]
	      (bc) edge node [below right] {$g$} (c)
	      (bc') edge node [anchor=center, fill=white] {$g'$} (c')
	      (bb) edge (b)
	      (bb') edge [dashed] (b')
	      (ba) edge node [below right] {$g$} (a)
	      (ba') edge (a');
      \end{scope}
  \end{tikzpicture} 
\]
where the bottom and top squares pullbacks, the back (vertical) squares form a
double pullback, and the right-hand vertical squares form a double pushout.  It
needs to be shown that if the front vertical squares are a double pullback,
then the left vertical squares are a double pushout. But this follows
immediately from the assumption on \(\E\) stated at the beginning of
Section~\ref{sec:preliminaries}.
\end{proof}

\begin{lemm}{onppksqofHDRs}
    Suppose we have a morphism of retract pairs \((i', j') \to (i,j)\) given by
    \(f: A' \to A\), \(g : B' \to B\). Then if this morphism is
    \begin{enumerate}[(i)]
 \item a cartesian morphism of retract pairs, and \(H : B \to M B\) gives
    \((i, j)\) the structure of an HDR, then there is a unique HDR structure
    on the pair \((i', j')\) such that the cartesian morphism is a cartesian morphism of HDRs.
 \item a cocartesian morphism of retract pairs, and \(H : B' \to M B'\) gives
    \((i', j')\) the structure of an HDR, then there is a unique HDR structure
    on the pair \((i, j)\) such that the cocartesian morphism is a cocartesian morphism of HDRs.
 \end{enumerate}
\end{lemm}
\begin{proof}
	(i): Since any HDR structure $H'$ that would be a witness to the claim makes
\diag{ B' \ar[dr]^{H'} \ar@/^/[drr]^{H.g} \ar@/_/[ddr]_{\alpha.(j',M!.H')} \\
& MB' \ar[d]^{Mj'} \ar[r]^{Mg} & MB \ar[d]^{Mj} \\
& MA' \ar[r]_{Mf} & MA }
commute (see the proof of \refprop{HDRiscoalgebra}), it must be unique because
\(M\) preserves pullback squares.  It remains to see that \(H'\) can be defined
in this way by setting \[H' = (\alpha . (j', M!.H.g) , H.g).\] To check that
$\Gamma.H' = MH'.H'$, it suffices to prove that both sides become equal upon
postcomposing with both $MMg$ and $MMj'$. But
we have:
\begin{eqnarray*} 
MMg.MH'.H' & = & M(Mg.H').H' \\
& = & M(H.g).H' \\
& = & MH.Mg.H' \\
& = & MH.H.g \\
& = & \Gamma.H.g \\
& = & \Gamma.Mg.H' \\
& = & MMg.\Gamma.H'
\end{eqnarray*}
and
\begin{align*}
MMj.\Gamma.H' & = \Gamma.Mj.H' \\
& =\Gamma.\alpha(j',M!.H.g) \\
& =M\alpha.(\alpha.(j', M!.H.g), \Gamma.M!.H.g) \\
& =M\alpha.(Mj.H', MM!.\Gamma.H.g) \\
& =M\alpha.(Mj.H', MM!.MH.H.g) \\
& =M\alpha.(Mj.H', MM!.MH.Mg.H') \\
& = M(\alpha(j,M!.H.g)).H' \\
& = M(Mg.H').H' \\
& = MMg.MH'.H'.
\end{align*}
Note we have used an identity from the proof of \refprop{HDRiscoalgebra} here.
Lastly, we verify:
\begin{align*}
    t . H' & = t . (\alpha . (j', M!.H.g) , H.g) \\
           & = (t .\alpha . (j', M!.H.g) , t.H.g) \\
           & = (j', i.j.g) \\
           & = (1, i.f).j'
\end{align*}

(ii): This is \refprop{HDRiscofibredstructure}.
\end{proof}

We can now prove the above stated proposition:
\begin{proof}[Proof of \refprop{domHDRfibration}]
By the previous two lemmas, it remains to show that when \((i_0,j_0,H_0)\),
\((i_1,j_1,H_1)\) and \((i_2, j_2, H_2)\) are HDRs such that we have
a composite
    \diag{A_2 \ar[d]_{i_2} \ar[r]^{f'} & A_1 \ar[d]_{i_1} \ar[r]^{f} & A_0
        \ar[d]_{i_0}\\
          B_2 \ar[d]_{j_2} \ar[r]^{g'} & B_1 \ar[d]_{j_1} \ar[r]^{g} & B_0
          \ar[d]_{j_0}\\
      A_2 \ar[r]^{f'} & A_1 \ar[r]^{f} & A_0}
      which is a morphism of HDRs, then:
\begin{enumerate}[(i)]
\item If the right morphism of retract pairs is a cartesian morphism of HDRs,
	the left one is automatically a morphism of HDRs
\item If the left morphism is a cocartesian morphism of HDRs, then the
	right one is automatically a morphism of HDRs.
\end{enumerate}
(i) follows again by taking projections: 
\begin{align*}
    Mg.Mg'.H_2 = M(g.g').H_2 = H_0.(g.g') = (H_0.g).g' = Mg.(H_1.g')
\end{align*}
and
\begin{align*}
    Mj_1.Mg'.H_2 & = M(f'.j_2).H_2 \\
                 & = M(f'). \alpha. (j_2, M!. H_2) \\ 
                 & =\alpha. (f'.j_2, M!. H_2) \\
                 & =\alpha. (f'.j_2, M!. H_0.g.g') \\
                 & =\alpha. (j_1 , M!. H_0.g) . g' \\
                 &=\alpha . (j_1 , M!. H_1) . g'\\
                 &= Mj_1. H_1. g'
\end{align*}

(ii) follows again from \refprop{HDRiscoalgebra}, since the property is easy to verify
for coalgebras using \refprop{coalg-lifts}.

Observe that the Beck-Chevalley condition is now inherited from 
\reflemm{retractfibration}.
\end{proof}

The following is a first consequence of the more abstract approach we have
taken so far:
\begin{coro}{pbsqofHDRalongcartsq}
    In the category of HDRs, the pullback of a cartesian morphism of HDRs
    along a morphism of HDRs exists and is a cartesian morphism of HDRs.
\end{coro}
\begin{proof}
    This is a direct consequence of the fact that \(\dom : \cat{HDR} \to \E \) is
    a Grothendieck fibration (and that \(\E\) has pullbacks).
\end{proof}

\subsection{Algebras and naive fibrations}\label{ssec:naivefibrations}
In this subsection, we will derive an alternative characterisation of the
\(R\)-algebras associated to a category with Moore structure. We will show that
\(R\)-algebras can be identified, as a notion of fibred structure, with maps
that come equipped with a property that resembles a path-lifting property.
This is analogous to the explicit characterisation of \emph{cloven
\(\mathcal{R}\)-maps} found by Van den Berg and Garner
~\cite{vdBerg-Garner} for path object categories (see
Proposition~6.1.5 loc.\ cit.).

Recall (from
Sections~\ref{sssec:functorial-factorisation},~\ref{sssec:the-monad}) that the
monad for the AWFS defined in this section is given by $R p = s.p_1$, so that
algebras are fillers:
\[\begin{tikzpicture}[baseline={([yshift=-.5ex]current bounding box.center)}]
    \matrix (m) [matrix of math nodes, row sep=3.5em,
    column sep=4em]{ 
        |(a)| {Y} & |(a2)| {Y} \\
	|(ef)|  {MX \times_{X} Y}  & |(b)|  {X}
    \\};
    \begin{scope}[every node/.style={midway,auto,font=\scriptsize}]
      \path[->]
      (a) edge node [left] {$(r.p, 1)$} (ef)
      (a) edge (a2)
      (a2) edge node [right] {$p$} (b)
      (ef) edge node [below] {$s.p_1$} (b)
      (ef) edge node [anchor=center, fill=white] {$\beta$} (a2);
      \end{scope}
  \end{tikzpicture}
\]
satisfying an additional unit and associativity condition.

Since \(R\)-algebras have more structure than cloven \(\mathcal{R}\)-maps
(namely, the unit and associativity condition), the alternative
characterisation in terms of a path lifting property can be expected to meet
more structural conditions.  Yet the idea of the correspondence, as well as the
further characterisation as \emph{naive
fibrations} originates from Van den Berg
and Garner.

We have the following definition:
\begin{defi}{transport}\index{transport structure}
Let $p: Y \to X$ be any map. To equip $p$ with \emph{transport structure} means specifying a map
\[ T: MX \times_X Y \to Y \]
where \(MX \times_X Y\) is the pullback of \(t\) and \(p\), with $p.T=s.p_1$, $T.(r.p,1) = 1$ and such that
\[
  T.(\mu.(p_1.p_1, p_2.p_1), p_2)= T.(p_1.p_1, T.(p_2.p_1, p_2)) \thinspace : \thinspace
  (MX \times_X MX) \times_X Y \rightarrow Y.
\]
where the first pullback is the pullback of \(t\) and \(s\) (the domain of \(\mu\)) and
the second of \(t.p_1\) and \(p\).
\end{defi}

The above definition is illustrated in Figure~\ref{fig:transport}. Intuitively,
an endpoint \(y\) in a space \(Y\) lying above a base space \(X\) is
\emph{transported} back along a path \(\gamma\) in the base space.
\begin{figure}
  \centering
\begin{tikzpicture}
  \path [fill=lightgray] (1.5,0) ellipse [x radius=2, y radius=1];
  \path [fill=lightgray] (1.5,3) ellipse [x radius=2.3, y radius=1];
  \node (X) at (4.5,0) {$X$};
  \node (Y) at (4.5,3) {$Y$};
  \node (y0) at (0,3) {$T(\gamma, y)$};
  \node (y1) at (3,3) {$y$};
  \node (x0) at (0,0) {$x_0$};
  \node (x1) at (3,0) {$x_1$};
  \path[->, draw] (Y) to node [left] {$p$} (X);
  \path[-, draw, dotted] (y0) to (x0);
  \path[-, draw, dashed] (y1) to (x1);
  \path[->, draw, thick] (x0) [out=45, in=225] to node [above] {$\gamma$} (x1);
\end{tikzpicture}
\caption{Transport structure for \(p\)\label{fig:transport}}
\end{figure}
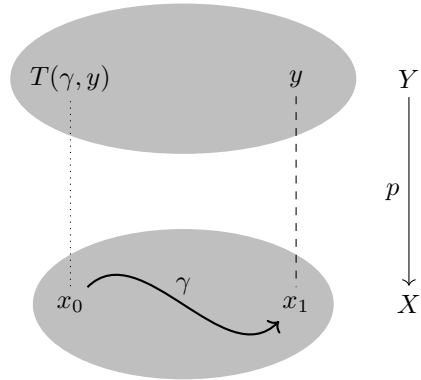

\begin{prop}{transportisfibredstructure}
The function which associates to every \(p: Y \to X\)
the set
\[
 \{T : MX \times_X Y \to Y \varco T \text{ is a transport structure for } p\}
\]
can be extended to a fibred structure
\[
(\ct{E}^{\to}_{\rm cart})^{op} \to \Sets 
\]
\end{prop}
\begin{proof}
This is easy when considering that a transport structure
amounts to the same thing as an algebra, so we leave this
as part of \refprop{transpvsalgebras}.
\end{proof}

\begin{prop}{transpvsalgebras} 
The following notions of fibred structure are isomorphic:
\begin{itemize}
     \item Transport structure, in the sense of \refdefi{transport},
     \item The structure of an algebra with respect to \(R = s.p_1\).
\end{itemize}
\end{prop}
\begin{proof} 
This is straightforward by unfolding the definition of an algebra.
\end{proof}

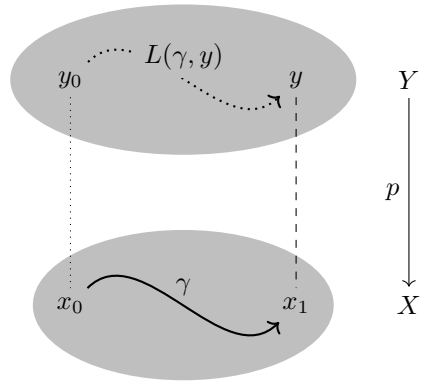
\begin{figure}
  \centering
\begin{tikzpicture}
  \path [fill=lightgray] (1.5,0) ellipse [x radius=2, y radius=1];
  \path [fill=lightgray] (1.5,3) ellipse [x radius=2.3, y radius=1];
  \node (X) at (4.5,0) {$X$};
  \node (Y) at (4.5,3) {$Y$};
  \node (y0) at (0,3) {$y_0$};
  \node (y1) at (3,3) {$y$};
  \node (x0) at (0,0) {$x_0$};
  \node (x1) at (3,0) {$x_1$};
  \path[->, draw] (Y) to node [left] {$p$} (X);
  \path[-, draw, dotted] (y0) to (x0);
  \path[-, draw, dashed] (y1) to (x1);
  \path[->, draw, thick] (x0) [out=45, in=225] to node [above] {$\gamma$} (x1);
  \path[->, draw, thick, dotted] (y0) [out=45, in=225] to node [above, fill=lightgray] {$L(\gamma, y)$} (y1);
\end{tikzpicture}
  \caption{A naive fibration \(p\)\label{fig:naive-fibration}}
\end{figure}

As remarked in the paper by Van den Berg and Garner, the lifting property of
\refdefi{transport} does not a priori specify that it is possible to lift a
path to a full path lying above it. Yet, they show that it is possible to use
the path contraction to lift entire paths `pointwise' -- and this turns out to
work also in the setting presented here. We first give a definition of this
path-lifting property for our setting, which again takes more structural axioms
than in the former paper.
\begin{defi}{strlift}
  A map $p: Y \to X$ together with an arrow \[ L: MX \times_X Y \to MY \]in
  \(\E\) is said to be a \emph{naive
  fibration}\index{naive fibration|textbf}when it satisfies the
  conditions:
\begin{enumerate}[(i)]
\item $(M p, t).L = 1$;
\item $ L.(r.p,1) = r$; 
\item $L.(\mu.(p_1.p_1, p_2.p_1), p_2) = \mu.(L.(p_1.p_1, s.L.(p_2.p_1, p_2)), L.(p_2.p_1, p_2))$;
\item $\Gamma. L = M L.p_1.\delta_p$.
\end{enumerate}
\end{defi}
A naive fibration \(p : Y \to X\)
can be thought of as a map between spaces, such that whenever one has a path
\(\gamma\) in the base \(X\) and a point \(y\) in \(Y\) which lies its end
point \(x_1 := t(\gamma)\), then one can construct a full path \(\gamma^*\) in
\(MY\) lying above \(\gamma\). This is illustrated in
Figure~\ref{fig:naive-fibration}.

Looking at Figure~\ref{fig:naive-fibration}, (i) ensures that \(L(\gamma,y)\)
lies above \(\gamma\) under \(p\); (ii) that trivial paths are lifted to
trivial paths; (iii) that lifting respects path composition; and (iv) that the
contraction of \(\gamma\) specified by \(\Gamma\) lifts, pointwise under
\(ML\), to the contraction of \(L(\gamma,y)\).  The properties (iii) and (iv)
differentiate naive fibrations from the path-lifting property established in
the Proposition~6.1.5 of Van den Berg-Garner~\cite{vdBerg-Garner} mentioned at
the beginning of this section.

\begin{prop}{nrfibisfibredstructure}
The function which associates to every \(p: Y \to X\)
the set
\[
        \{L : MX \times_X Y \to MY \varco (p,L) \text{ is a naive fibration }\}
\]
can be extended to a fibred structure
\[
        \cat{nFib} \co (\ct{E}^{\to}_{\rm cart})^{\op} \to \Sets
\]
\end{prop}
\begin{proof}
        Again, we leave this as part of their characterisation in
        \refprop{liftvstransport}.
\end{proof}

\begin{prop}{liftvstransport}\index{transport structure!implies lifting} 
Let $p: Y \to X$ be a map. If $L$ specifies a naive fibration structure
on $p$, then $T = s.L$ is a transport structure on $p$.  And if $T$ is a
transport structure on $p$, then $L = MT.p_2.\delta_p$ turns \(p\) into a naive
fibration. These operations are mutually inverse and define an isomorphism
between the following notions of fibred structure:
\begin{itemize}
        \item Transport structure,
        \item Naive fibrations.
\end{itemize}
\end{prop}
\begin{proof}
Suppose \(L\) satisfies the conditions (i)-(iv) of \refdefi{strlift}, and
let \(T = s.L\). Then \(p.T = p.s.L = s.Mp.L = s.p_1\), \(T.(r.p,1) =
s.L.(r.p, 1) = s.r = 1\), and
\begin{align*}
T.(\mu.(p_1.p_1, p_2.p_1), p_2)
 = & s.L.(\mu.(p_1.p_1, p_2.p_1), p_2) \\
 = & s.\mu.(L.(p_1.p_1, s.L.(p_2.p_1, p_2)), L.(p_2.p_1, p_2)) \\
 = & s.L.(p_1.p_1, s.L.(p_2.p_1, p_2))\\
 = & T.(p_1.p_1, T.(p_2.p_1, p_2)),
\end{align*}
so $T$ is a transport structure. In addition,
\[ MT.p_1.\delta_p = Ms.ML.p_1.\delta_p =Ms.\Gamma.L= L, \]
so $L$ can be reconstructed from $T$.

Conversely, suppose $T$ is a transport structure on $p$, and define
\[ L = MT.\nu.(\Gamma.p_1,\alpha.(p_2,M!.p_1)) \]
where \(\nu\) is given by the same mediating isomorphism as~\eqref{eq:MpreservesPullback}. This definition is illustrated in Figure~\ref{fig:liftvstransport}.
Our first aim is to show (i) -- (iv). First:
\begin{align*}
t.L   & = t.M T.\nu.(\Gamma.p_1, \alpha.(p_2, M !.p_1))\\
      & = T.t.\nu.(\Gamma.p_1, \alpha.(p_2, M !.p_1)) \\
      & = T.(t.\Gamma.p_1, t.\alpha.(p_2, M !.p_1))\\
      & = T.(r.t.p_1, p_2) \\
      & = T.(r.p.p_2, p_2) \\
      & = T.(r.p, 1).p_2 \\
      & = p_2
\end{align*}
and
\begin{align*}
Mp.L 
& = M p.M T.\nu.(\Gamma.p_1,\alpha.(p_2, M !.p_1)) \\
& = M(p.T).\nu.(\Gamma.p_1,\alpha.(p_2, M !.p_1)) \\
& = M(s.p_1).\nu.(\Gamma.p_1,\alpha.(p_2, M !.p_1)) \\
& = M s.M p_1.\nu.(\Gamma.p_1,\alpha.(p_2, M !.p_1)) \\
& = M s.p_1.(\Gamma.p_1,\alpha.(p_2, M !.p_1)) \\
& = M s.\Gamma.p_1 \\
& = p_1,
\end{align*}
and hence $(Mp,t).L = 1$.

Furthermore,
\begin{align*}
L.(r.p,1) 
& = M T.\nu.(\Gamma.p_1,\alpha.(p_2, M !.p_1)).(r.p, 1) \\
& = M T.\nu.(\Gamma.r.p,\alpha.(1, M !.r.p)) \\
& = M T.\nu.(r.r.p,\alpha.(1, r.!.p)) \\
& = M T.\nu.(r.r.p,\alpha.(1, r.!)) \\
& = M T.\nu.(r.r.p, r) \\
& = M T.r.(r.p, 1) \\
& = r.T.(r.p, 1) \\
& = r,
\end{align*}
so also the second condition for a lift is satisfied.

The following calculation shows the third condition:
\begin{align*}
L.(\mu.(p_1.p_1, p_2.p_1), p_2)
& = M T.\nu.(\Gamma.p_1,\alpha.(p_2, M !.p_1)).(\mu.(p_1.p_1, p_2.p_1), p_2) \\
& = M T.\nu.(\Gamma.\mu.(p_1.p_1, p_2.p_1),\alpha.(p_2, M !.\mu.(p_1.p_1, p_2.p_1))) \\
& =^1 M T.\nu.(\mu.(M \mu.\nu.(\Gamma.p_1.p_1,\alpha.(p_2.p_1, M !.p_1.p_1)), \Gamma.p_2.p_1), \\
&\qquad \mu.(\alpha.(p_2, M !.p_1.p_1),\alpha.(p_2, M !.p_2.p_1))) \\
& =^2 MT. \mu.(\nu.(M \mu.\nu.(\Gamma.p_1.p_1,\alpha.(p_2.p_1, M !.p_1.p_1)),\alpha.(p_2, M !.p_1.p_1)), \\
&\qquad \nu.(\Gamma.p_2.p_1,\alpha.(p_2, M !.p_2.p_1))) \\
& =^3 \mu.(M T.\nu.(M \mu.\nu.(\Gamma.p_1.p_1,\alpha.(p_2.p_1, M !.p_1.p_1)),\alpha.(p_2, M !.p_1.p_1)), \\
&\qquad M T.\nu.(\Gamma.p_2.p_1,\alpha.(p_2, M !.p_2.p_1))) \\
& =^4 \mu.(M T.\nu.(\Gamma.p_1.p_1, M T.\nu.(\alpha.(p_2.p_1, M !.p_1.p_1),\alpha.(p_2, M !.p_1.p_1))), \\ &\qquad L.(p_2.p_1, p_2)) \\
& =^5 \mu.(M T.\nu.(\Gamma.p_1.p_1,\alpha.(T.(p_2.p_1,p_2), M !.p_1.p_1)), L.(p_2.p_1, p_2)) \\
& = \mu.(M T.\nu.(\Gamma.p_1,\alpha.(p_2, M !.p_1)).(p_1.p_1, T.(p_2.p_1, p_2)), \\
& \qquad L.(p_2.p_1, p_2)) \\
& = \mu.(L.(p_1.p_1, T.(p_2.p_1, p_2)), L.(p_2.p_1, p_2)) \\
&= \mu.(L.(p_1.p_1, s.L.(p_2.p_1, p_2)), L.(p_2.p_1, p_2)) 
\end{align*}
where at \(=^1\) we have used the distributive law, at \(=^2\) we have
rewritten the equation of the form~\eqref{eq:mu-nu-commutes}, at \(=^3\) we have
used naturality of \(\mu\). At \(=^4\), we have used the definition of \(L\) and further that 
\begin{align*}
  &MT.\nu.(M\mu.\nu.(Mp_1.Mp_1,Mp_2.Mp_1),Mp_2) = \\
  &MT.\nu.(Mp_1.Mp_1, MT.\nu.(Mp_2.Mp_1,Mp_2))
\end{align*}
which is the image under \(M\) of the requirement on \(T\) with respect to \(\mu\).
The step \(=^5\) uses naturality of \(\alpha\) (for the square with \(T\)). Then
it is a matter of rewriting, and in the last step we use the equation established at the end of 
this proof.

For the fourth condition, we again calculate:
\begin{align*}
M L.\nu.(\Gamma.p_1, \alpha.(p_2, M !.p_1))
& = M M T.\nu.M (\Gamma.p_1, \alpha.(p_2, M !.p_1)).\nu.(\Gamma.p_1, \alpha.(p_2, M !.p_1)) \\
& = M M T.\nu.(M \Gamma. M p_1, \alpha.(M p_2, M M !. M p_1)).\nu.(\Gamma.p_1, \alpha.(p_2, M !.p_1)) \\
& =^1 M M T.\nu.(M \Gamma.\Gamma.p_1, \alpha.(\alpha.(p_2, M !.p_1), M M !.\Gamma.p_1 )) \\
& = M M T.\nu.(\Gamma.\Gamma.p_1, \Gamma.\alpha.(p_2, M !.p_1)) \\
& = M M T.\Gamma.\nu.(\Gamma.p_1, \alpha.(p_2, M !.p_1)) \\
& = \Gamma.M T.\nu.(\Gamma.p_1, \alpha.(p_2, M !.p_1)) \\
& =  \Gamma.L
\end{align*}
where \(=^1\) uses the axioms of the strength \(\alpha\) with respect to \(\Gamma\),
and the rest are naturality conditions.
This shows that \(L\) yields the structure of a naive fibration.

Finally, we have
\begin{align*}
s.L 
& = s.M T.\nu.(\Gamma.p_1,\alpha.(p_2, M !.p_1)) \\
& = T.s.\nu.(\Gamma.p_1,\alpha.(p_2, M !.p_1)) \\
& = T.(s.\Gamma.p_1, s.\alpha.(p_2, M !.p_1)) \\
& = T.(p_1, p_2) \\
& =  T,
\end{align*}
showing that the operations are mutually inverse.
\end{proof}
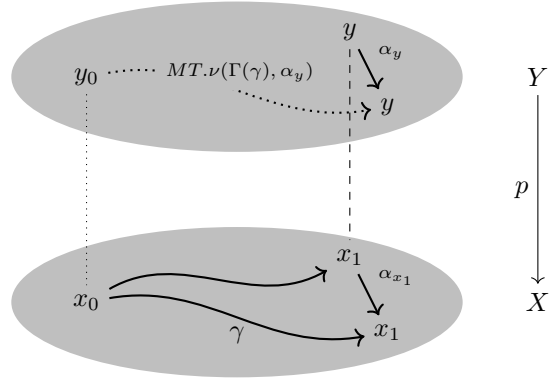
\begin{figure}
  \centering
\begin{tikzpicture}
  \path [fill=lightgray] (2,0) ellipse [x radius=3, y radius=1];
  \path [fill=lightgray] (2,3) ellipse [x radius=3, y radius=1];
  \node (X) at (6,0) {$X$};
  \node (Y) at (6,3) {$Y$};
  \node (y0) at (0,3) {$y_0$};
  \node (y1) at (4,2.6) {$y$};
  \node (y1') at (3.5,3.6) {$y$};
  \path[->, draw, thick] (y1') to node [above right] {\scriptsize $\alpha_y$} (y1);
  \path[->, draw, thick, dotted] (y0) [out=10, in=190] to node [above, fill=lightgray] {\scriptsize $MT.\nu(\Gamma(\gamma), \alpha_y)$} (y1);
  \node (x0) at (0,0) {$x_0$};
  \node (x1) at (4,-0.4) {$x_1$};
  \node (x1') at (3.5,0.6) {$x_1$};
  \path[->, draw, thick] (x0) [out=10, in=190] to node [below] {$\gamma$} (x1);
  \path[->, draw, thick] (x0) [out=30, in=210] to (x1');
  \path[->, draw, thick] (x1') to node [above right] {\scriptsize $\alpha_{x_1}$} (x1);
  \path[->, draw] (Y) to node [left] {$p$} (X);
  \path[-, draw, dotted] (y0) to (x0);
  \path[-, draw, dashed] (y1') to (x1');
\end{tikzpicture}
\caption{Deriving lifting from (pointwise) transport. The paths \(\alpha_y\) and
\(\alpha_{x_1}\) are constant paths of the same length as \(\gamma\).  The
bottom triangle is the path \(\Gamma(\gamma)\), contracting \(\gamma\) to the
trivial path on \(x_1\). Each `copy' of \(y\) on \(\alpha_y\) is
transported along a final segment of \(\gamma\) lying transversal to 
\(\Gamma(\gamma)\) (see Box~\ref{box:Gamma}).\label{fig:liftvstransport}}
\end{figure}

To conclude this section, it will be helpful to spell out the notion of
morphism between naive fibrations explicitly, where the definition is fixed by
the notion of morphism between the underlying algebras.

If \((p: E \to B, L)\), \((p : E' \to B', L')\) are naive fibrations, then a
\emph{morphism of naive fibrations} \((p, L) \to (p', L')\) is a commutative
square
\begin{equation}\label{eq:morphismOfNaiveRightFibDiag}
          \begin{tikzpicture}[baseline={([yshift=-.5ex]current bounding box.center)}]
            \matrix (m) [matrix of math nodes, row sep=2em,
            column sep=3.5em]{
                |(e')| {E'} & |(e)| {E} \\
                |(b')| {B'} & |(b)| {B} \\
            };
            \begin{scope}[every node/.style={midway,auto,font=\scriptsize}]
            \path[->]
		    (e') edge node [below] {$g$} (e)
		    (b') edge node [below] {$f$} (b)
		    (e') edge node [left] {$p'$} (b')
		    (e) edge node [right] {$p$} (b);
            \end{scope}\end{tikzpicture}
            \end{equation}
such that
\begin{equation}\label{eq:morphismOfNaiveRightFib}
	L . (g . p_1, Mf . p_2 ) = M g . L'
\end{equation}

The following corollary then summarises this section (cf.
\refcoro{hdrscoalgsdoubleiso}):
\begin{coro}{rfibCatFibredStructure}\index{naive fibration!double category of \(\sim\)s}
With the inherited vertical composition of algebras, the discretely fibred
concrete double category (over \(\E\)) of naive fibrations \(\mathbb{N}\cat{Fib}\) is
isomorphic to the double category of algebras.\index{discretely (co)fibred concrete double category!of naive fibrations}
\end{coro}
\begin{proof}
  Dual and analogous to the proof of \refcoro{hdrscoalgsdoubleiso}.
\end{proof}
As remarked in Section~\ref{sssec:notationofdoublecats}, we may also refer to
the mere category of naive fibrations \(\cat{NFib}\), or the notion of fibred
structure \(\cat{nFib}\).

 \section{The Frobenius construction}\label{sec:frobenius}
The goal of this section is to prove a \emph{Frobenius property} (see
Box~\ref{box:frobenius} below) for the AWFS constructed in the previous section
in the case that the Moore structure comes equipped with a certain additional
symmetry.\index{symmetric Moore structure}

Such a symmetry was assumed throughout by Van den Berg and
Garner~\cite{vdBerg-Garner}, who show an analogous Frobenius property for
their cloven weak factorisation systems (see Definition~3.3.3.\ loc.cit.).
Like our previous results above, the symmetry imposed here needs to
satisfy additional structure to yield an analogous result for categories with
Moore structure.

As in~\cite{vdBerg-Garner}, the symmetry comes down to the fact that it is
possible to reverse paths by means of a `twist map'\index{twist map}: a natural
transformation \(\tau : M \Rightarrow M\) subject to axioms which we will now
elaborate.  First, the natural transformation reverses paths as announced:
\begin{itemize}
  \item[(1)] For every \(X\), \(\tau_X : MX \to MX\) is an internal, idempotent
    identity-on-objects functor between the category \(MX\) (given by the Moore
    structure) and the opposite category on \(MX\). In particular \(s. \tau =
    t\) and \(t . \tau = s\).
\end{itemize}
Further, we require compatibility with \(\alpha\):
\begin{equation}\label{eq:twistalphamu}
  \tau . \alpha = \alpha(1 , \tau)
\end{equation}

So far, these requirements are the same as in~\cite{vdBerg-Garner}. Our
additional axioms can be motivated as follows. The twist map can be used to
create a new, reversed path contraction:
\begin{equation}\label{eq:Gammastar}
\Gamma^* = \tau_M . M(\tau) . \Gamma . \tau
\end{equation}

\begin{figure}[t]
  \centering 

  \begin{subfigure}[t]{0.4\textwidth}
    \centering
\begin{tikzpicture}
  \draw [white, pattern = vertical lines, pattern color = lightgray] (0,0) -- (3,0) -- (0,2) -- (0,0); 
  \node (x0) at (0,0) {$x_0$};
  \node (x1) at (3,0) {$x_1$};
  \node (x1') at (0,2) {$x_1$};
  \node  (G) at (1,0.75) {$\Gamma$};
  \node (G) at (1,0.35) {$\Rightarrow$};
  \path[->, draw, thick] (x0) to node [below left] {$\gamma$} (x1);
  \path[->, draw, thick] (x0) to node [left] {$\gamma$} (x1');
  \path[->, draw, thick] (x1') to node [above right] {$\alpha(x_1, M!(\gamma))$} (x1);
  \path[->, draw, thick, dotted] (1.5, 0) to (1.5,1);
\end{tikzpicture}
\caption{\(\Gamma\)}
\end{subfigure}
~
 \begin{subfigure}[t]{0.4\textwidth}
   \centering
\begin{tikzpicture}
  \draw [white, pattern = vertical lines, pattern color = lightgray] (0,0) -- (3,2) -- (0,2) -- (0,0); 
  \node (x0) at (0,0) {$x_0$};
  \node (x0') at (3,2) {$x_0$};
  \node (x1) at (0,2) {$x_1$};
  \node  (G) at (1,1.65) {$\Gamma$};
  \node (G) at (1,1.25) {$\Rightarrow$};
  \path[<-, draw, thick] (x0) to node [below right] {$\alpha(x_0, M!.\tau (\gamma))$} (x0');
  \path[<-, draw, thick] (x0) to node [left] {$\tau (\gamma)$} (x1);
  \path[->, draw, thick] (x1) to node [above] {$\tau (\gamma)$} (x0');
  \path[->, draw, thick, dotted] (1.5, 2) to (1.5,1);
\end{tikzpicture}
\caption{\(\Gamma . \tau (\gamma)\)}
\end{subfigure}
 
\begin{subfigure}[t]{0.4\textwidth}
  \centering
\begin{tikzpicture}
  \draw [white, pattern = vertical lines, pattern color = lightgray] (0,0) -- (3,2) -- (0,2) -- (0,0); 
  \node (x0) at (0,0) {$x_0$};
  \node (x0') at (3,2) {$x_0$};
  \node (x1) at (0,2) {$x_1$};
  \node  (G) at (1,1.65) {\scriptsize $M(\tau) . \Gamma$};
  \node (G) at (1,1.25) {$\Rightarrow$};
  \path[<-, draw, thick] (x0) to node [below right] {$\alpha(x_0, M!.\tau (\gamma))$} (x0');
  \path[->, draw, thick] (x0) to node [left] {$\gamma$} (x1);
  \path[->, draw, thick] (x1) to node [above] {$\tau (\gamma)$} (x0');
  \path[<-, draw, thick, dotted] (1.5, 2) to (1.5,1);
\end{tikzpicture}
\caption{\(M(\tau).\Gamma . \tau (\gamma)\)}
\end{subfigure}
~
  \begin{subfigure}[t]{0.4\textwidth}
    \centering
\begin{tikzpicture}
  \draw [white, pattern = vertical lines, pattern color = lightgray] (3,0) -- (0,2) -- (3,2) -- (3,0); 
  \node (x0) at (3,0) {$x_0$};
  \node (x0') at (0,2) {$x_0$};
  \node (x1) at (3,2) {$x_1$};
  \node  (G) at (2.2,1.65) {\scriptsize $\tau . M(\tau) . \Gamma$};
  \node (G) at (2.2,1.25) {$\Rightarrow$};
  \path[<-, draw, thick] (x0) to node [below left] {$\alpha(x_0, M!(\gamma))$}(x0');
  \path[->, draw, thick] (x0) to node [right] {$\gamma$} (x1);
  \path[<-, draw, thick] (x1) to node [above] {$\gamma$} (x0');
  \path[<-, draw, thick, dotted] (1.5, 2) to (1.5,1);
\end{tikzpicture}
\caption{\(\tau.M(\tau).\Gamma . \tau (\gamma)\)}
\end{subfigure}
\caption{Constructing the `cocontraction'\index{cocontraction} \(\Gamma^*\) using a twist map.\label{fig:twistingGamma}}
\end{figure}
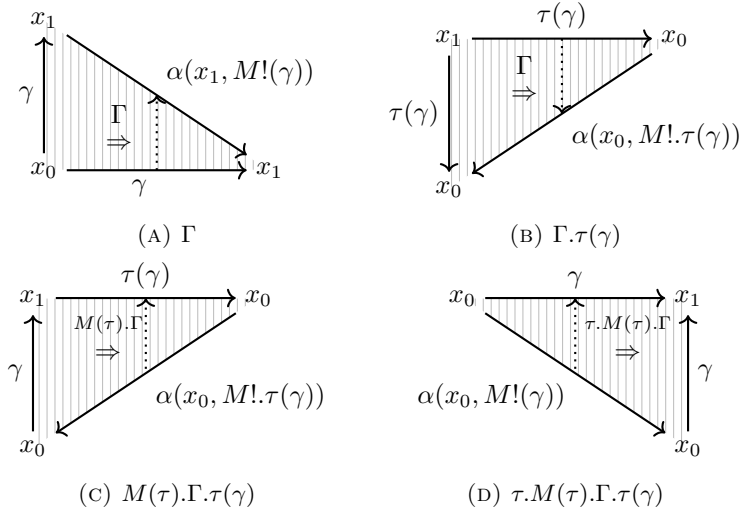

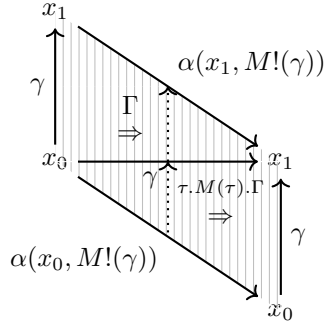
\begin{figure}
  \centering
\begin{tikzpicture}
  \draw [white, pattern = vertical lines, pattern color = lightgray] (0,2) -- (3,2) -- (0,4) -- (0,2); 
  \node (y1') at (0,4) {$x_1$};
  \node (x0) at (3,0) {$x_0$};
  \node (x0') at (0,2) {$x_0$};
  \node (x1) at (3,2) {$x_1$};
  \node  (G) at (1,2.75) {$\Gamma$};
  \node (G) at (1,2.35) {$\Rightarrow$};
  \draw [white, pattern = vertical lines, pattern color = lightgray] (3,0) -- (0,2) -- (3,2) -- (3,0); 
  \node  (G) at (2.2,1.65) {\scriptsize $\tau . M(\tau) . \Gamma$};
  \node (G) at (2.2,1.25) {$\Rightarrow$};
  \path[<-, draw, thick] (x0) to node [below left] {$\alpha(x_0, M!(\gamma))$}(x0');
  \path[->, draw, thick] (x0) to node [right] {$\gamma$} (x1);
  \path[<-, draw, thick, dotted] (1.5, 2) to (1.5,1);
  \path[->, draw, thick] (x0') to node [below left] {$\gamma$} (x1);
  \path[->, draw, thick] (x0') to node [left] {$\gamma$} (y1');
  \path[->, draw, thick] (y1') to node [above right] {$\alpha(x_1, M!(\gamma))$} (x1);
  \path[->, draw, thick, dotted] (1.5, 2) to (1.5,3);
\end{tikzpicture}
\caption{Pointwise composition of the `cocontraction' \(\Gamma^*\) with \(\Gamma\) \label{fig:Gammastarcomposition}}
\end{figure}

This is illustrated in Figure~\ref{fig:twistingGamma}.
Using~\eqref{eq:twistalphamu}, one can work out that the constant sides of the
triangle can be identified as shown. Similar, one can work out that the two
diagrams can be composed pointwise (or vertically), as illustrated in
Figure~\ref{fig:Gammastarcomposition}.  Of this pointwise composition, we will
ask that it can be identified with the constant path from \(\gamma\) to itself
of the same underlying length as itself:
\[ M\mu.(\Gamma^*, \Gamma) =\alpha.(1,M!): M \Rightarrow MM. \]
We will refer to this axiom as the \emph{sandwich equation}\index{sandwich equation},
and it is a new requirement. Last, we need an axiom that is actually unrelated
to the twist map, namely that composition is left (and right) cancellative: that
is, 
\diag{ MX_t \times_s MX \ar[r]_(.35){(1, \Delta)} & MX_t \times_s (MX_{(t,s)} \times_{(t,s)} MX) \ar@<1ex>[rr]^(.7){\mu.(p_1,p_1.p_2)} \ar@<-1ex>[rr]_(.7){\mu.(p_1,p_2.p_2)}  &  & MX }
is an equalizer (and similarly for right cancellative).\index{left
cancellative}\index{right cancellative} Since $M$ preserves pullbacks, this
remains an equalizer after applying $M$ (so we can apply left cancellation
pointwise). This axiom could have been part of the definition of Moore
structure, but it so happened that we did not need it until this section.

Although the motivation of this section is symmetric Moore structures, the
proofs presented are completely axiomatic with regard to \(\Gamma^*\) and
\(\Gamma\). Therefore we can merely assume a given natural transformation
\(\Gamma^*: M \Rightarrow MM\) with the following properties:
\begin{enumerate}
  \item[(1')] \(\Gamma^*: M \Rightarrow MM\) is a comonad satisfying the Moore
    structure axioms dual to to $\Gamma$ (with $s$ and $t$ reversed);
  \item[(2)] which satisfies the \emph{sandwich equation};
  \item[(3)] and composition is left and right cancellative.
 \end{enumerate}

This slight abstraction will be called a \emph{two-sided} Moore
structure\index{two-sided Moore structure}.  The two
definitions, of a two-sided Moore structure and a symmetric Moore structure,
are summarised in \refdefi{twosidedpathobjcat} and
\refdefi{symmetricpathobjcat}.

\subsection{Naive left fibrations}

 A two-sided Moore structure gives rise to a `dual' AWFS on \(\E\) induced by
 \(\Gamma^*\) instead of \(\Gamma\).  This yields a definition of \emph{naive
 left fibration} (cf. \refdefi{strlift}):
\begin{defi}{colift}
A map $p: Y \to X$ together with an arrow
\[ L^*: MX \times_X Y \to MY, \]
where $MX \times_X Y$ refers to the pullback of $p$ and $s$ (instead of \(t\)),
is said to be a \emph{naive left fibration}\index{naive left fibration} when it
satisfies the conditions:
\begin{enumerate}[(i)]
\item $(Mp, s).L^* = 1$;
\item $L^*.(r.p, 1) = r$;
\item $L^*.(\mu.(p_2.p_1, p_1.p_1), p_2) = \mu.(L^*.(p_1.p_1, t.L^*.(p_2.p_1, p_2)), L^*.(p_2.p_1, p_2))$;
\item $\Gamma^*. L^* = ML^*.(\Gamma^*.p_1,\alpha.(p_2,M!.p_1))$.
\end{enumerate}
\end{defi}
\begin{rema}{terminology-left-right}
The terminology of left fibrations is adopted from the corresponding notions
(due to Joyal) of left (and right) fibrations in the category of simplicial
sets, see e.g.\ \cite{Lurie}, chapter 2.  Our notion of \emph{effective}
left and right fibration (developed in Section~\ref{sec:unifKanFibr}) for
categories with Moore structure coincides with left and right fibrations in
simplicial sets in this sense.
\end{rema}
It follows from the previous sections that naive left fibrations are
\(R\)-algebras for the AWFS induced by the dual comonad \(\Gamma^*\). In
particular, \refcoro{rfibCatFibredStructure} gives:
\begin{coro}{naiveLeftFibCatFibred}
  Naive left fibrations have the structure of a discretely fibred double
  category \(\mathbb{N}\cat{LFib}\) over \(\E\), whose squares are
  given by commutative squares as in~\eqref{eq:morphismOfNaiveRightFibDiag}.\index{discretely (co)fibred concrete double category!of naive left fibrations} 
\end{coro}
Like previously, we may also refer to its structure as a category
\(\cat{NLFib}\) or notion of fibred structure \(\cat{nLFib}\).

For \emph{symmetric} Moore structure, it turns out that the notion of left
naive fibration is no different from the notion of naive fibration.  This
follows from the following proposition, which states that HDRs are the same for
the two Moore structures given by \(\Gamma\) and \(\Gamma^*\).
The second part of the following proposition is stated so as to apply
\refprop{fibreddoubleiso} immediately.
\begin{prop}{coHDRisHDR}
 In a category with symmetric Moore structure, HDRs coincide for both
 Moore structures. That is, there is a double functor over \(\sq{\E}\)
 between the two (concrete) double categories of HDRs which is an isomorphism.

 Hence, there is an induced double functor between discretely fibred concrete
 double categories over \(\E\):
 \[
   \mathbb{N}\cat{Fib} \to \mathbb{N}\cat{LFib}.
 \]
 This functor is full on squares, and induces an isomorphism between notions of
 fibred structure. Hence, this functor is an isomorphism. 
\end{prop}
\begin{proof}
For the first part, we define a double functor as follows.

If \((i : A \to B, j, H^*)\) is a `left-HDR', so for the dual structure given
by \(\Gamma^*\), then let \(H := \tau_B . H^*\). We claim that \((i, j , H)\)
is an HDR\@.  The first requirements are easy to check.  For the condition on
\(\Gamma\), we have:
\begin{align*}
       \Gamma . H &= \Gamma . \tau . H^* = M(\tau) . \tau . \Gamma^* . H^*
                  = M(\tau) . \tau . M(H^*) . H^* \\
                  &= M(\tau) . M(H^*) . \tau . H^*
                  = M(\tau. H^*) . \tau . H^* \\
                  &= M(H). H  
\end{align*}
Functoriality with
respect to squares is easy to see, and for vertical composition as
given by~\eqref{eq:composition-of-hdrs} we have: 
\begin{align*}
        (\tau . H_1) * (\tau . H_0) &= \mu . ( \tau. H_1,Mi_1 . \tau . H_0. j_1) \\
                                    &= \mu . ( \tau . H_1,\tau . Mi_1 . H_0 . j_1) \\
                                    &= \tau . \mu. ( H_1,Mi_1 . H_0 . j_1 ) \\
                                    &= \tau . (H_1 * H_0) \\
\end{align*}
Note that this proposition uses the first condition on \(\tau\) entirely.  Of
course this argument dualizes, and clearly the two operations are inverse. 

For the second part, it is a little exercise to show that the induced double
functor between the categories of algebras satisfies the conditions, where
algebra morphisms \(MX \times_X \to Y\) for a naive fibration \(p: Y \to X\)
are sent to the precomposition with with \[\tau \times_{X} Y : MX \times_{X} Y
\to MX \times_{X} Y\] (note, these are two different pullbacks, for \(s\) and
\(t\) respectively or vice-versa). To do the exercise, one can apply the
isomorphism to the HDR 
\[
  (r.p , 1) : Y \to MX \times_X Y
\]
with retraction given by the comultiplication. Then one can see what happens to
the algebra morphism, which is defined as the diagonal filler with respect to
this map.
\end{proof}

\subsection{The Frobenius construction}
The equations (2)--(3) from the beginning of this section establish a relation
between the two `dual' AWFSs on \(\E\) given by \(\Gamma\) and \(\Gamma^*\).
The following lemma exploits this in a way we will need for our formulation of
the Frobenius construction in \refprop{pbkstabilityofSDRunderfibr} below. In
the presence of a two-sided Moore structure, we will now sometimes refer to
naive fibrations as naive \emph{right} fibrations, to emphasize which one we
are talking about.
\begin{lemm}{lemma-coconnection}
        Suppose \(p : Y \to X\) has the the structure of a
        naive right fibration \(L : MX \times_X Y \to MY\),
        i.e.\ they satisfy the conditions of \refdefi{strlift}. 
        Then we also have:
\[
        \Gamma^*.L = ML.(\Gamma^*.p_1, L).
\] for the dual \(\Gamma^*\). 

Dually, a naive left fibration \((p: Y \to X, L^*)\), satisfies
\[
        \Gamma . L^* = M L^* . (\Gamma . p_1, L^*).
\]
\end{lemm}
\begin{proof}
Note that by postcomposing (iv) in \refdefi{strlift} with $Ms$ we obtain:
\[ L = Ms.ML.(\Gamma.p_1,\alpha.(p_2,M!.p_1)). \]

 By (pointwise) left cancellation and equation (iv) in \refdefi{strlift}, it suffices to prove:
\[ M\mu.(\Gamma^*.L, \Gamma.L) = 
M \mu.(M L.(\Gamma^*.p_1, L), M L.(\Gamma.p_1,\alpha.(p_2, P !.p_1)))
\]

But we have
\begin{eqnarray*}
M\mu.(\Gamma^*.L,\Gamma.L) & = & M\mu.(\Gamma^*,\Gamma).L \\
& = &\alpha.(1,M!).L \\
& = &\alpha.(L,M!.p_1),
\end{eqnarray*}
as well as
\begin{center}
\begin{displaymath}
\begin{array}{lc}
M \mu.(M L.(\Gamma^*.p_1, L), M L.(\Gamma.p_1,\alpha.(p_2, M !.p_1))) & = \\
M \mu.(M L.(\Gamma^*.p_1, M s.M L.(\Gamma.p_1,\alpha.(p_2, M !.p_1))), M L.(\Gamma.p_1,\alpha.(p_2, M !.p_1))) & = \\
M L.(M \mu.(\Gamma^*.p_1, \Gamma.p_1),\alpha.(p_2, M !.p_1)) & = \\
M L.(\alpha.(p_1, M !.p_1),\alpha.(p_2, M !.p_1))& = \\
ML.\alpha.(1,M!.p_1) & = \\
\alpha.(L,M!.p_1).
\end{array}
\end{displaymath}
\end{center}

The second statement is dual to the first.
\end{proof}

The following proposition contains our definition of the \emph{Frobenius
construction} in the current context. As a property of an AWFS resulting from a
Moore structure, the proposition is comparable to the \emph{Frobenius property}
of Van den Berg-Garner~\cite{vdBerg-Garner}. See also Box~\ref{box:frobenius}.
\begin{textbox}{The Frobenius property\label{box:frobenius}}\index{Frobenius property}
In recent literature, the Frobenius property is often formulated as `the
pullback of an \(\mathcal{L}\)-map along an \(\mathcal{R}\)-map is an
\(\mathcal{L}\)-map'~\cite{vdBerg-Garner-11}\cite{Gambino-Sattler}.
Here \(\mathcal{L}\), \(\mathcal{R}\) are the left and right classes of maps
for an algebraic or non-algebraic weak factorisation system:
\[
\begin{tikzpicture}[baseline={([yshift=-.5ex]current bounding box.center)}]
            \matrix (m) [matrix of math nodes, row sep=2em,
            column sep=3.5em]{
                |(d)| {D} & |(a)| {A} \\
                |(e)| {E} & |(b)| {B} \\
            };
            \begin{scope}[every node/.style={midway,auto,font=\scriptsize}]
            \path[{Hooks[right]}->]
                    (a) edge node [right] {$i \in \mathcal{L}$} (b)
                    (d) edge [dashed] node [right] {$\in \mathcal{L}$} (e);
            \path[->>]
                    (e) edge node [below] {$p \in \mathcal{R}$} (b)
                    (d) edge [dashed] (a);
            \end{scope}\end{tikzpicture}
\]
The property is of interest because it is closely related to
\emph{right-properness} of a model structure. This is the property that the
factorisation system of trivial cofibrations and fibrations, as part of the
model structure, satisfies the Frobenius property (see~\cite{Gambino-Sattler}).
Further, it relates to dependent products (pushforward) for a model of type
theory (as in Section~\ref{sec:Pi-types} below). The Frobenius property for
two-sided Moore structures is more abstract, since the class of right maps
\(\R\) in the above diagram is replaced by the right maps from a different, but
closely related, AWFS\@. The actual Frobenius property is extracted as a
consequence for
\emph{symmetric} Moore structures in \refcoro{symmetricPocFrobenius}.  
\end{textbox}
\begin{prop}{pbkstabilityofSDRunderfibr}{\rm (Frobenius construction)}\index{Frobenius construction}
	Suppose \((i, j, H)\) is an HDR and
\diag{ A' \ar[d]_{i'} \ar[r]^{p'} & A \ar[d]^i \\
E \ar[r]_p & B}
is a pullback square in which $p: E \to B$ is a naive left fibration.
Then $i'$ can be extended to an HDR such that the square becomes a
morphism of HDRs.
\end{prop}

\begin{rema}{notapbksqofHDRs}
 The diagram above will not be a cartesian morphism of HDRs, in general.
\end{rema}

\begin{proof}
Let $L^*$ be the naive left fibration structure on $p$. We have a map $j: B \to A$ and a homotopy $H: B \to MB$ with $j.i = 1, s.H = 1, t.H = i.j$ and $\Gamma.H = MH.H$. In addition, we have a pullback diagram of the form
\diag{ A' \ar[d]_{i'} \ar[r]^{p'} & A \ar[d]^i \\
E \ar[r]_p & B.}
Write $H' = L^*.(H.p, 1): E \to ME$.
Then $s.H' = s.L^*.(H.p, 1) = p_2.(H.p,1) = 1$ and 
\[ p.t.H' = p.t.L^*.(H.p, 1) = t.Mp.L^*.(H.p,1) = t.p_1.(H.p, 1)=t.H.p=i.j.p, \]
and therefore there is a map $j': E \to A'$ with $p'.j' = j.p$ and $i'.j' = t.H'$. We will first show that $j'.i' = 1$ and $MH'.H' = \Gamma.H'$. 

To see $j'.i' =1$, we calculate
\begin{eqnarray*}
i'.j'.i' & = & t.H'.i' \\ & = & t.L^*.(H.p, 1).i' \\ & = & t.L^*.(H.p.i',i') \\ & = & t.L^*.(H.i.p', i') \\ & = & t.L^*.(r.i.p',i') \\ & = & t.L^*.( r.p.i',i') \\ & = & t.L^*.(r.p,1).i' \\ & = & t.r.i' \\ & = & i' \\ & = & i'.1.
\end{eqnarray*}

To prove $MH'.H' = \Gamma.H'$, we compute:
\begin{eqnarray*}
MH'.H' & = & M(L^*.(H.p,1)).L^*.(H.p,1) \\ & = & ML^*.M(H.p,1).L^*.(H.p,1) \\ & = & ML^*.(MH.Mp,1).L^*.(H.p,1) \\ & = & ML^*.(MH.Mp.L^*,L^*).(H.p,1) \\ & = & ML^*.(MH.p_1,L^*).(H.p,1) \\ & = & ML^*.( MH.H.p,L^*(H.p,1)) \\ & = & ML^*.( \Gamma.H.p,L^*.(H.p,1)) \\ & = & ML^*.(\Gamma.p_1,L^*).(H.p,1) \\ & = & \Gamma.L^*.(H.p,1) \\ & = & \Gamma.H'.
\end{eqnarray*}
Here we have used the identity of \reflemm{lemma-coconnection}.

It remains to check that the square is a morphism of HDRs. However, we have $p'.j' = j.p$, by construction, and 
\begin{eqnarray*}
Mp.H' & = & Mp.L^*.(H.p,1) \\
& = & p_1.(H.p,1) \\
& = & H.p. 
\end{eqnarray*}
\end{proof}

\begin{defi}{frobeniusmorphismofHDRs}
  A \emph{Frobenius morphism}\index{Frobenius morphism of HDRs} of HDRs is a
  morphism of HDRs as constructed in the proof of
  \refprop{pbkstabilityofSDRunderfibr}. So the domain is completely determined
  by an underlying HDR \((i : A \to B, j, H)\) (the codomain) and naive left
  fibration \(p: E \to B\). It will be seen in \reflemm{pullbckstoffrob} that
  Frobenius morphisms of HDRs are stable under pullback along morphisms of
  HDRs. This fact makes the proof of the main theorem of this chapter,
  \reftheo{symmetricPocProofOfPi}, much more manageable.
\end{defi}

\begin{lemm}{frobeniusStableUnderComposition}
 The Frobenius construction of \refprop{pbkstabilityofSDRunderfibr}
 is stable under composition of naive left fibrations $p$ as well as composition of
 HDRs $i$.
 \end{lemm}
\begin{proof}
 Consider a picture as follows:
  \diag{ A_2 \ar[d]_{i_2} \ar[r]^{q_1}  & A_1 \ar[d]_{i_1} \ar[r]^{q_0} & A_0 \ar[d]^{i_0} \\
E_2 \ar[r]_{p_1} & E_1 \ar[r]_{p_0} & E_0. }
Then we have
\begin{eqnarray*}
 H_1 & = & L^*_{p_0}.(H_0.p_0,1) \\
 H_2 & = & L^*_{p_1}.(H_1.p_1,1) \\
 H_2^* & = & L^*_{p_0.p_1}.(H_0.p_0.p_1,1)
\end{eqnarray*}
In view of \refprop{HDRsareMcoalgebras}, it suffices to show $H_2^* = H_2$, which we can do as follows:
\begin{eqnarray*}
 H_2^* & = & L^*_{p_0.p_1}.(H_0.p_0.p_1,1) \\
       & = & L^*_{p_1}.(L^*_{p_0}.(p_1,H_0.p_0.p_1),1) \\
       & = & L^*_{p_1}.(L^*_{p_0}.(1,H_0.p_0).p_1),1) \\
       & = & L^*_{p_1}.(H_1.p_1, 1) \\
       & = & H_2.
\end{eqnarray*}

Now consider a picture as follows:
\diag{ E_2 \ar[r]^{p_2} \ar[d]_{i_1'} & A_2 \ar[d]^{i_1} \\
E_1 \ar[r]^{p_1} \ar[d]_{i_0'} & A_1 \ar[d]^{i_0} \\
E_0 \ar[r]_{p_0} & A_0 }.
Then we have:
\begin{eqnarray*}
H_0' & = & L^*_{p_0}.(H_0.p_0,1) \\
H_1' & = & L^*_{p_1}.(H_0.p_1,1) \\
H_2' & = & H_0' * H_1' \\
H_2^* & = & L^*_{p_0}.((H_0 * H_1).p_0, 1)
\end{eqnarray*}
and we have to compare $H_2'$ and $H_2^*$. So here we go:
\begin{eqnarray*}
 H_2^* & = & L^*_{p_0}.((H_0 * H_1).p_0,1) \\
 & = & L^*_{p_0}.(\mu.(H_0.p_0,Mi_0.H_1.j_0.p_0), 1) \\
 & = & \mu.(L^*_{p_0}.(H_0.p_0,1), L^*_{p_0}.(Mi_0.H_1.j_0.p_0, t.L^*_{p_0}(H_0.p_0,1))) \\
 & = & \mu.(H_0', L^*_{p_0}.(Mi_0.H_1.j_0.p_0,t.H_0',)) \\
 & = & \mu.(H_0', L^*_{p_0}.(Mi_0.H_1.p_1.j_0',i_0'.j_0')) \\
 & = & \mu.(H_0', Mi_0'.L^*_{p_1}(H_1.p_1,1).j_0') \\
 & = & \mu.(H_0', Mi_0'.H_1'.j_0') \\
 & = & H_0' * H_1' \\
 & = & H_2'
\end{eqnarray*}
and the proof is finished.
\end{proof}

Recall from \refdefi{frobeniusmorphismofHDRs} that we call a morphism of HDRs
arising from the Frobenius construction a \emph{Frobenius morphism} of HDRs.
The following lemma shows that Frobenius morphisms are stable under pullbacks.
As mentioned above, the proof of \reftheo{symmetricPocProofOfPi} (specifically
\reflemm{uniformkanfibrandpi}) is greatly simplified by this result.  The
property stated below is analogous to the property (of a weak
factorisation system) of being \emph{functorially Frobenius} in Van den
Berg-Garner~\cite{vdBerg-Garner}. The lemma thus shows that this property 
also holds in the setting of an algebraic weak factorisation system for a
category with (two-sided) Moore structure.
\begin{lemm}{pullbckstoffrob}{\rm (Pullback stability of Frobenius construction) }
	The Frobenius construction of \refprop{pbkstabilityofSDRunderfibr} defines a 
	\emph{functor}:
	\[
                {(-)}^*(-) : \cat{NLFib} \times_{\E} \cat{HDR} \to \cat{HDR}
	\]
	where the domain is the pullback of the domain functors to \(\E\).
	
	As a consequence, the pullback (\refcoro{HDRhaspullbacks}) in \(\cat{HDR}\) of a
	Frobenius morphism along a morphism of HDRs is again a Frobenius
 	morphism of HDRs.
\end{lemm}
\begin{proof}
	Suppose \((a, b) : i_1 \to i_0\) is a morphism of HDRs, \((q_0, p_0)
	: i_0' \to i_0\) is a Frobenius morphism of HDRs, and \((f, b) : p_1 \to p_0\) is
        a morphism of naive left fibrations, as in the solid part of
	the following diagram:
\begin{equation}\label{eq:frobenius-pullback}
  \begin{tikzpicture}[baseline={([yshift=-.5ex]current bounding box.center)}]
    \matrix (m) [matrix of math nodes, row sep=2em,
    column sep=3.5em]{
	    |(e1)| {E_1} &              & |(a1)| {A_1} & \\
		         & |(e0)| {E_0} &              & |(a0)| {A_0} \\ 
	    |(f1)| {F_1} &              & |(b1)| {B_1} & \\
			 & |(f0)| {F_0} & & |(b0)| {B_0} 
    \\};
    \begin{scope}[every node/.style={midway,auto,font=\scriptsize}]
      \path[->]
	      (e1) edge [dashed] (a1)
	(e1) edge [dashed] node [fill=white, anchor=center] {$e$} (e0)
	(e1) edge [dashed] node [left] {$i_1'$} (f1)
	(a1) edge node [above right] {$a$} (a0)
	(a1) edge node [pos=0.3, right] {$i_1$} (b1)
  (e0) edge [line width = 3pt, draw = white] (a0)
  (e0) edge node [pos=0.3, below] {$q_0$} (a0)
	(a0) edge node [right] {$i_0$} (b0)
	(f1) edge node [below left] {$f$} (f0)
	(f1) edge node [pos=0.4, below] {$p_1$} (b1)
  (e0) edge [line width = 3pt, draw = white] (f0)
	(e0) edge node [pos = 0.3, right] {$i_0'$} (f0)
	(b1) edge node [below] {$b$} (b0)
	(f0) edge node [below] {$p_0$} (b0);
      \end{scope}\end{tikzpicture}\end{equation}
      It is enough to prove that the Frobenius construction applied to the back square induces
      a unique morphism of HDRs \(i_1' \to i_0'\) on the left side of the cube.

      So all that needs to be verified is that \((e, f) : i_1' \to i_0'\)
      induced by the pullback is a morphism of HDRs. Denoting their respective
      HDR structure by \(H_1'\), \(H_0'\), and denoting the naive left fibrations by
      \((p_0,L_{p_0})\) and \((p_1, L_{p_1})\), we compute:
      \begin{align*}
	      H_0' . f &= L_{p_0}^*.( H_0.p_0,1).f = L_{p_0}^*.( H_0.p_0.f,f) \\
		       &= L_{p_0}^*.( H_0.b.p_1,f) = L_{p_0}^*.( Mb . H_1 . p_1,f)
		    \\ &= Mf . L_{p_1}^*.( H_1. p_1,1) = Mf . H_1'
      \end{align*}
      Where we have used that the bottom face is a morphism of naive left fibrations. 
      By \refcoro{retractscommute}, it follows that the left face is a morphism of
      HDRs.

      For the last statement, it is easy to see that the cube is a pullback square
      of morphisms of HDRs. 
 \end{proof}

 As a corollary, we can rephrase the Frobenius construction as follows for
 \emph{symmetric} Moore structures:
\begin{coro}{symmetricPocFrobenius}{\rm (Frobenius for symmetric Moore structures)}
In categories with symmetric Moore structure, there is a pullback functor
\[
              {(-)}^*{(-)} : \cat{NFib} \times_{\E} \cat{HDR} \to \cat{HDR}
\]
given by factoring the Frobenius construction of
\reflemm{pullbckstoffrob} through the isomorphism between naive left
and right of \refprop{coHDRisHDR}.
\end{coro}

 \section{Mould squares and effective fibrations}\label{sec:unifKanFibr}

\subsection{Mould squares}
In this section, we connect the two algebraic weak factorisation systems coming
from a dominance and a Moore structure to define the notion of \emph{effective
fibration}. We assume that \(\E\) is a finitely cocomplete,
locally cartesian closed category with Moore structure (recall the remarks in
the introduction of Section~\ref{sec:preliminaries}). We assume further
that \(\E\) comes equipped with a dominance as in Section~\ref{sec:dominances},
such that:
\begin{itemize}
  \item \(\Sigma\) is closed under binary unions of subobjects;
  \item \(\Sigma\) contains every initial arrow \(0 \to A\) in \(\E\).
\end{itemize}
As in Section~\ref{sec:dominances}, the coalgebras coming from the dominance
are \emph{effective cofibrations} and referred to as such\index{effective cofibration}.  Lastly, we
also make the combining assumption that every trivial Moore path
 \[
   r_X : X \to MX
 \]
is an effective cofibration. This assumption can be thought of as the condition
that the proposition expressing that a Moore path is trivial is cofibrant,
i.e., contained in the dominance \(\Sigma\). When \(r\) is a \emph{cartesian}
natural transformation, the condition is equivalent to saying that \(r_1 : 1 \to M1\) is an effective cofibration. That is, it is enough that the proposition expressing
that a Moore path has trivial \emph{length} is cofibrant. In the case
of simplicial sets, these last two are the case\footnote{We will define
effective cofibrations in simplicial sets as locally decidable monomorphisms. It will
follow easily that a Moore path having trivial length is
locally decidable.}.

Given a Moore structure and a dominance subject to these requirements, we have
the following definition. 
\begin{defi}{mouldsq}\index{mould square}
A \emph{mould square} is a cartesian morphism of HDRs in the fibre above an
effective cofibration, as in the following diagram:
\begin{equation}\label{eq:mould-square}
  \begin{tikzpicture}[baseline={([yshift=-.5ex]current bounding box.center)}]
    \matrix (m) [matrix of math nodes, row sep=3em,
    column sep=3em]{ 
	    |(s)| {A'} & |(b')| {B'} \\
	    |(a)| {A} & |(b)| {B}
    \\};
\begin{scope}[every node/.style={midway,auto,font=\scriptsize}]
    \path[{Hooks[right]}->]
    (s) edge node [left] {$m$} (a)
    (b') edge node [right] {$m'$} (b);
    \path[->]
    (a) edge node [above] {$i$} (b)
    (s) edge node [above] {$i'$} (b');
    \end{scope}
    \end{tikzpicture}
    \end{equation}
where \(i\), \(i'\) are HDRs and \(m\) is an effective cofibration (so \(m'\) is also an effective cofibration).
\end{defi}
A typical illustrative example of a mould square is given in
Figure~\ref{fig:mould-square}. This example also serves to explain the choice
of terminology. In the category of simplicial sets, there is a special type of
mould square that we call \emph{horn squares}\index{horn square}. These are
discussed in Section~\ref{sec:hornsquares}.

The following Lemma will be used in \refdefi{effectiveRightFibration} to define
the \emph{triple category} (see below) against which the notion of effective 
fibration will be defined. 
\begin{lemm}{mouldSqPbStable}
The pullback of a mould square along an arbitrary morphism of HDRs is
again a mould square.
\end{lemm}
\begin{proof}
This follows directly from \refcoro{pbsqofHDRalongcartsq} and the fact that effective cofibrations are stable
under pullback.
\end{proof}
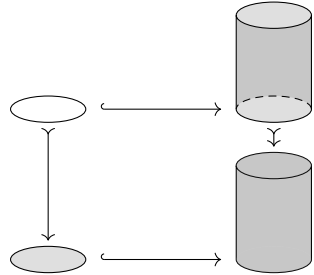
\begin{figure}
  \centering
\begin{tikzpicture}
  \def\front {
(180:5mm) coordinate (a)
-- ++(0,-12.5mm) coordinate (b)
arc (180:360:5mm and 1.75mm) coordinate (d)
-- (a -| d) coordinate (c) arc (0:-180:5mm and 1.75mm)}
  \def\back {
(180:5mm) coordinate (a)
-- ++(0,-12.5mm) coordinate (b)
arc (180:0:5mm and 1.75mm) coordinate (d)
-- (a -| d) coordinate (c) arc (0:180:5mm and 1.75mm)
}
\def\backoutline {
  (c) arc (0:180:5mm and 1.75mm)
}
\def\bottom {
   (0,-12.5mm) circle (5mm and 1.75mm);
}
\def \bottomoutline{
  (d) arc (0:180:5mm and 1.75mm);
}
\def \top {
  (0,0) circle (5mm and 1.75mm);
}
\begin{scope}[xshift = -30mm]
  \draw[yshift = 20mm] \bottom;
  \draw [fill = gray, fill opacity=.25] \bottom;
\end{scope}
\begin{scope}[yshift = 20mm]
\path [fill = gray, fill opacity = .25] \back;
\draw \backoutline;
\draw [fill = gray, fill opacity = .25] \front;
\draw [densely dashed] \bottomoutline;
\end{scope}
\path [fill = gray, fill opacity = .25] \back;
\draw \backoutline;
\draw [fill = gray, fill opacity = .25] \front;
\path [fill = gray, fill opacity = .25] \bottom;
\path [fill = gray, fill opacity = .25] \top;

\path [draw, >->] (-30mm, 5mm) -- (-30mm, -10 mm);
\path [draw, >->] (0mm, 5mm) -- (0mm, 2.5 mm);
\path [draw, {Hooks[right]}->] (-23mm, -12.5mm) -- (-7mm, -12.5 mm);
\path [draw, {Hooks[right]}->] (-23mm, 7.5 mm) -- (-7mm, 7.5 mm);
\end{tikzpicture}
\caption{An illustrative example of a mould square. The horizontal
  arrows indicate hyperdeformation retracts. On top the inclusion of a circle
  into the bottom end of a tube. On the bottom the inclusion of a disk into the
  bottom of a solid cylinder. The vertical arrows are
  cofibrations. The tube can be thought of as the upstanding walls of a 
  type of `mould' whose shape is the circle and whose base is the disk.
  The solid cylinder is the filled mould.\label{fig:mould-square}}
\end{figure}

\begin{textbox}{Triple categories\label{box:triple-categories}}
  \emph{Triple categories}\index{triple category|textbf} are not much harder to
  understand than double categories.  We can define a \emph{small} triple
  category as an internal category in the category of small double categories.
  This definition can be unfolded (and hence extended to include large triple
  categories) in the same way as in Section~\ref{ssec:double-categories}.
  Briefly, they extend double categories in that there is an additional type of
  \(1\)-dimensional morphisms, \emph{perpendicular} morphisms, and hence three
  different types of squares, between each pair of distinct \(1\)-dimensional
  morphism types.\index{perpendicular morphism} In addition, \emph{cubes} are
  morphisms between \(xy\)-squares which compose in the perpendicular
  direction, and come with additional `pointwise' compositions for the
  horizontal and vertical direction.  So a cube looks like:\index{cube (triple
  category)}
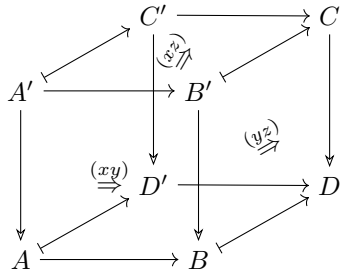
\[
  \begin{tikzpicture}[baseline={([yshift=0ex]current bounding box.center)}]
    \matrix (m) [matrix of math nodes, row sep=5em,
    column sep=5em]{ 
	    |(a')| {A'} & |(a)| {B'} \\
	    |(b')| {A} & |(b)| {B} 
    \\};
    \matrix (n) [matrix of math nodes, row sep=5em,
    column sep=5em, position=30:-1.7 from m]{ 
	    |(c')| {C'} & |(c)| {C} \\
	    |(d')| {D'} & |(d)| {D}
    \\};
    \node (i) at (barycentric cs:a'=1,b'=1,a=1,b=1) {$\overset{(xy)}{\Rightarrow}$};
    \node (i) [rotate=30] at (barycentric cs:a=1,b=1,c=1,d=1) {$\overset{(yz)}{\Rightarrow}$};
    \node (i) [rotate=60] at (barycentric cs:a'=1,a=1,c'=1,c=1) {$\overset{(xz)}{\Rightarrow}$};
    \begin{scope}[every node/.style={midway,auto,font=\scriptsize}]
    \path[->]
        (a') edge (a)
        (b') edge (b)
        (c') edge (c)
        (d') edge (d);
    \path[-{Stealth[open]}]
        (a') edge (b')
        (a) edge (b)
        (c') edge (d')
        (c) edge (d);
    \path[{Bar[]}->]
            (a') edge (c')
            (a) edge (c)
            (b') edge (d')
            (b) edge (d);
    \end{scope}
    \end{tikzpicture}\]
Note we are using the convention here that the three axes in 3-dimensional space are named as follows:
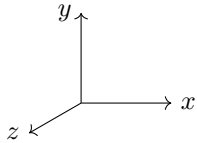
\[
\begin{tikzpicture}[baseline={([yshift=0ex]current bounding box.center)}]
  \draw [->] (0,0) -- (0:12mm) node [right] {$x$};
  \draw [->] (0,0) -- (90:12mm) node [left] {$y$};
  \draw [->] (0,0) -- (210:8mm) node [left] {$z$};
\end{tikzpicture}
\]
A cube admits composition from three different sides. The axioms guarantee that
composition of a given combination of squares and cubes is independent of the
order in which compositions are taken along the different dimensions. The standard
example of a triple category is the category \(\cube{\E}\) for a category \E, where objects are objects in \(\E\), all morphisms are given by morphisms in \E,
all squares are commutative squares, and cubes are commutative cubes.
\end{textbox}

\subsubsection{A new notion of fibred structure}
This subsection prepares the ground for \refdefi{effectiveRightFibration} in the
next subsection. We introduce a definition of a family of `fillers' for lifting
problems which deviates from the right lifting structures defined in
Section~\ref{ssec:double-categories} for double categories.\index{lifting problem!for a triple category}
Namely, we introduce the notion of a right lifting structure with respect to a
\emph{triple category}. For brevity, we have put an introduction to triple
categories in Box~\ref{box:triple-categories}. In the end, it is only the
definition of two specific triple categories (involving mould squares and
effective cofibrations) that we will need to define effective (trivial) fibrations.
The following definition spells out what it means to have a right lifting
structure with respect to a triple category over \(\cube{\E}\).
\begin{defi}{tripleCatRightLiftingStructure}\index{right lifting structure!with respect to a triple category}
        Suppose $\Lbb$ is a triple category and \(I : \Lbb \to \cube{\E}\) is a triple functor, and suppose
        \(p : Y \to X\) is a morphism in \E. Then a \emph{right lifting
        structure} for \(p\) with respect to \(I\) is a family of fillers for
        each diagram:
\[
          \begin{tikzpicture}[baseline={([yshift=-.5ex]current bounding box.center)}]
            \matrix (m) [matrix of math nodes, row sep=3em,
            column sep=6em]{
                |(a)| {A'} & & |(c)| {A} & |(c')| {Y} \\
                |(b)| {B'} & & |(d)| {B} & |(d')| {X} \\
            };
            \begin{scope}[every node/.style={midway,auto,font=\scriptsize}]
            \path[->]
                (a) edge node [left] {$I(v')$} (b)
                (a) edge node [above] {$I(f)$} (c)
                (b) edge node [below] {$I(g)$} (d)
                (c) edge node [pos=0.5, fill=white, anchor=center] {$I(v)$} (d)
                (c) edge node [above] {$a$} (c')
                (d) edge node [below] {$b$} (d')
                (c') edge node [right] {$p$} (d');
            \path[->]
         (b) edge node [pos=0.3, anchor=center, fill=white]
         {$h$} (c')
         (d) edge [dashed] node [anchor=center, fill=white] 
         {$\phi_{a, b} (f, g, h)$} (c');
 \end{scope}\end{tikzpicture}\text{,}
 \]
 where the left-hand square is the image of an \(xy\)-square (see
 Box~\ref{box:triple-categories}) under \(I\), such that the following
 compatibility conditions\index{compatibility condition} hold:
 \begin{description}
   \item[Horizontal]\index{horizontal condition} When \((f',g')\), \((f,g)\) is a horizontally composable pair of
                 \(xy\)-squares, and we are given a commutative diagram:
\[
          \begin{tikzpicture}[baseline={([yshift=-.5ex]current bounding box.center)}]
            \matrix (m) [matrix of math nodes, row sep=3em,
            column sep=6em]{
                    |(a'')| {A''} & |(a)| {A'} & |(c)| {A} & |(c')| {Y} \\
                    |(b'')| {B''} & |(b)| {B'} & |(d)| {B} & |(d')| {X} \\
            };
            \begin{scope}[every node/.style={midway,auto,font=\scriptsize}]
            \path[->]
                (a'') edge node [above] {$I(f')$} (a)
                (b'') edge node [below] {$I(g')$} (b)
                (a'') edge node [left] {$I(v'')$} (b'')
                (a) edge node [left] {$I(v')$} (b)
                (a) edge node [above] {$I(f)$} (c)
                (b) edge node [below] {$I(g)$} (d)
                (c) edge node [pos=0.7, fill=white, anchor=center] {$I(v)$} (d)
                (c) edge node [above] {$a$} (c')
                (d) edge node [below] {$b$} (d')
                (c') edge node [right] {$p$} (d')
         (b'') edge node [pos=0.5, anchor=center, fill=white]
         {$h$} (c')
         (b) edge [dashed] (c')
         (d) edge [dashed] (c');
 \end{scope}\end{tikzpicture}
 \]
 then we have:
 \[
         \phi_{a, b}(f.f', g.g', h) = \phi_{a,b}(f, g, \phi_{a.I(f),b.I(g)}(f',g',h)) 
 \]
\item[Vertical]\index{vertical condition} If we have a vertically composable pair of \(xy\)-squares and a diagram:
        \[
          \begin{tikzpicture}[baseline={([yshift=-.5ex]current bounding box.center)}]
            \matrix (m) [matrix of math nodes, row sep=3em,
            column sep=6em]{
                |(a')| {A'} & |(a)| {A} & |(y)| {Y} \\
                |(b')| {B'} &  |(b)| {B} &  \\
                |(c')| {C'} &  |(c)| {C} & |(x)| {X} \\
            };
            \begin{scope}[every node/.style={midway,auto,font=\scriptsize}]
            \path[->]
                (a') edge node [left] {$I(w')$} (b')
                (a) edge node [anchor=center, fill=white] {$I(w)$} (b)
                (b') edge node [left] {$I(v')$} (c')
                (c') edge node [below] {$I(k)$} (c)
                (a') edge node [above] {$I(f)$} (a)
                (b') edge node [below] {$I(g)$} (b)
                (b) edge node [pos=0.5, fill=white, anchor=center] {$I(v)$} (c)
                (a) edge node [above] {$a$} (y)
                (c) edge node [below] {$c$} (x)
                (y) edge node [right] {$p$} (x)
                (c') edge [bend right=5] node [pos=0.3, anchor=center, fill=white]
         {$h$} (y)
         (b) edge [dashed] (y)
         (c) edge [dashed] (y);
 \end{scope}\end{tikzpicture}\text{,}
 \]
 then we have:
 \[
         \phi_{a,c}(f, k , h) = \phi_{\phi_{a, c . I(v)} (f, g, h. I(v')) , c } (g, k , h) 
 \]
\item[Perpendicular]\index{perpendicular condition}
        For \emph{cubes}, the condition asks that for the image of a cube between \(xy\)-squares:
\[
  \begin{tikzpicture}[baseline={([yshift=0ex]current bounding box.center)}]
    \matrix (m) [matrix of math nodes, row sep=5em,
    column sep=5em]{ 
	    |(a')| {A'} & |(a)| {A} \\
	    |(b')| {B'} & |(b)| {B} 
    \\};
    \matrix (n) [matrix of math nodes, row sep=5em,
    column sep=5em, position=30:-1.7 from m]{ 
            |(c')| {C'} & |(c)| {C} & |(y)| {X} \\
            |(d')| {D'} & |(d)| {D} & |(x)| {X}
    \\};
    \begin{scope}[every node/.style={midway,auto,font=\scriptsize}]
    \path[->]
        (y) edge node [right] {$p$} (x)
        (c) edge node [above] {$c$} (y)
        (d) edge node [pos=0.3, anchor=center, fill=white] {$d$} (x)
        (d') edge node [pos=0.3, anchor=center,fill=white] {$h$} (y)
        (d) edge [dashed] (y)
        (b) edge [dashed, bend right] (y)
        (a') edge node [pos=0.3, anchor=center, fill=white] {$I(f)$} (a)
        (b') edge node [below] {$I(g)$} (b)
        (c') edge node [above] {$I(k)$} (c)
        (d') edge node [pos=0.6, anchor=center, fill=white] {$I(l)$} (d)
        (a') edge node [left] {$I(v')$}(b')
        (a) edge  node [anchor=center,fill=white, pos=0.6] {$I(v)$} (b)
        (c') edge node [anchor=center, fill=white, pos=0.3] {$I(w')$} (d')
        (c) edge node [anchor=center, fill=white, pos=0.2] {$I(w)$} (d)
        (a') edge (c')
        (a) edge node [anchor=center, fill=white] {$u'$} (c)
        (b') edge node [anchor=center, fill=white] {$r$} (d')
        (b) edge  node [anchor=center, fill=white] {$u$} (d);
    \end{scope}
    \end{tikzpicture}
\]
we have
\[
        u. \phi_{c, d}(k, l, h) = \phi_{d.u, c.u'} (f, g, h . r).
\]
\end{description}
\end{defi}

\begin{rema}{tripleCatSymmetry}
As for double categories, there is a symmetry in the definition of a triple
category, namely the choice of `top level' domain and codomain between cubes
(the same goes for squares), which could be any of the three
\(xy\),\(yz\),\(xz\) directions.  In \refdefi{tripleCatRightLiftingStructure}
it is assumed that cubes are morphisms between \(xy\)-squares, just like the
definition in Box~\ref{box:triple-categories}. The definition of lifting
structure takes this choice as a starting point. We do note that the definition
of right lifting structure is symmetric in \(x\) and \(y\), i.e., we could
swap the horizontal and vertical morphisms.
\end{rema}

\begin{rema}{mouldSquarePushout}
The fillers in \refdefi{tripleCatRightLiftingStructure} can also be given as
ordinary diagonal fillers with respect to arrows \(B' +_{A'} A \to B\) induced
by the pushout of the square. This is similar to the way generating trivial
cofibrations are defined by Gambino and Sattler~\cite{Gambino-Sattler}. Yet,
formulating the horizontal, vertical and perpendicular conditions in terms of
pushouts is very cumbersome. The contribution of mould squares (and the
surrounding triple category) is that they enable to express these conditions in
a straightforward way. This new form is used intensively in
Chapter~\ref{ch:simplicial-sets} of this paper on simplicial sets.
\end{rema}

The following definition is analogous to \refdefi{rlp-double-cat}. It is
left to the reader to spell out the details. For the notion of discretely
fibred concrete double category, see \refdefi{discretelyfibreddoublecat}.
\begin{defi}{rightLiftingTripleDoubleCat}\index{discretely (co)fibred concrete double category}
  There is a discretely fibred concrete double category over \(\E\) of right
  lifting structures with respect to a triple category.  In this category,
  squares are commutative squares satisfying the analogous compatibility
  condition (iii) of \refdefi{rlp-double-cat} with respect to the triple
  category.
\end{defi}

\subsection{Effective fibrations}\label{sec:effective-fibration}

We can now give the definition central to this paper, which combines all
previous sections.
\begin{defi}{effectiveRightFibration}\index{effective fibration|textbf}
An \emph{effective fibration} in a Moore category \(\E\) equipped with a
dominance is a morphism \(p: Y \to X\) equipped with a right lifting structure
with respect to the following triple category:
\begin{enumerate}[(i)]
\item Objects are the objects of \(\E\).
\item Horizontal morphisms are HDRs.
\item Vertical morphisms are effective cofibrations.
\item Perpendicular morphisms are morphisms in \(\E\).
\item \(xy\)-squares are mould squares.
\item \(xz\)-squares are morphisms of HDRs.
\item \(yz\)-squares are morphisms of effective cofibrations, i.e.\ pullback squares.
\item Cubes are pullback `squares' of a mould square along a morphism
        of HDRs (which always yields a mould square as per \reflemm{mouldSqPbStable}).
\end{enumerate}
Note that cubes in this triple category are unique for a given boundary of a
cube, which consists of six faces. 

We denote the discretely fibred concrete double category (over \(\E\)) of
effective fibrations by \EEffRFib{} (see
\refdefi{rightLiftingTripleDoubleCat}). We may also refer to it as a category
(\EffRFib{}) or notion of fibred structure (\effRFib) as explained in
Section~\ref{ssec:fibredstructurerevisited}.\index{discretely (co)fibred concrete double category!of effective fibrations} 
\end{defi}
\begin{figure}
  \centering
\begin{tikzpicture}
  \usepgfmodule{nonlineartransformations}
  \def\front {
(180:5mm) coordinate (a)
-- ++(0,-12.5mm) coordinate (b)
arc (180:360:5mm and 1.75mm) coordinate (d)
-- (a -| d) coordinate (c) arc (0:-180:5mm and 1.75mm)}
  \def\back {
(180:5mm) coordinate (a)
-- ++(0,-12.5mm) coordinate (b)
arc (180:0:5mm and 1.75mm) coordinate (d)
-- (a -| d) coordinate (c) arc (0:180:5mm and 1.75mm)
}
\def\backoutline {
  (c) arc (0:180:5mm and 1.75mm)
}
\def\bottom {
   (0,-12.5mm) circle (5mm and 1.75mm);
}
\def \bottomoutline{
  (d) arc (0:180:5mm and 1.75mm);
}
\def \top {
  (0,0) circle (5mm and 1.75mm);
}

\begin{scope}[xshift = 30mm, scale=0.5]
\begin{scope}[yshift = 30mm, yslant=0.5]
\draw (0, -5mm) circle (24mm and 12mm) coordinate (Y);
\node (Y) [xshift = 20mm, yshift = -2mm] {$Y$};
\path [fill = gray, fill opacity = .25] \back;
\draw \backoutline;
\draw [fill = gray, fill opacity = .25] \front;
\draw [densely dashed] \bottomoutline;
\path [fill = gray, fill opacity=.25] \bottom;
\end{scope}

\begin{scope}[yshift = -10mm]
\draw (0, -5mm) circle (20mm and 12mm) coordinate (X);
\node (X) [xshift = 20mm, yshift = -2mm] {$X$};
\path[->, draw] (Y) to node [right] {$p$} (X);
\path [fill = gray, fill opacity = .25] \back;
\draw \backoutline;
\draw [fill = gray, fill opacity = .25] \front;
\path [fill = gray, fill opacity = .25] \bottom;
\path [fill = gray, fill opacity = .25] \top;
\end{scope}
\end{scope}
\end{tikzpicture}
\caption{Illustrative example of an (effective) fibration. The mould square
  (Figure~\ref{fig:mould-square}) defines a tube with bottom in \(Y\) (a
  `mould'), lying above a solid cylinder in \(X\). The fibration structure
  prescribes a way to fill the mould, that is compatible with other fillers
(defined using a triple category).\label{fig:eff-fibration}}
\end{figure}
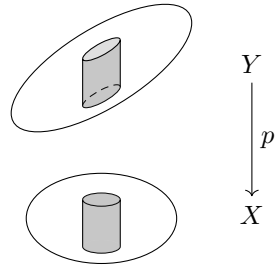

An illustration of an effective fibration, referring to the illustration of a
mould square on page~\pageref{fig:mould-square}, is given in
Figure~\ref{fig:eff-fibration}. For a concrete example of the compatibility
conditions, it is best to look at simplicial sets directly, for which see
Sections~\ref{sec:mouldsquaresinssets} and~\ref{sec:hornsquares}.  For
effective (Kan) fibrations in simplicial sets, the compatibility conditions
category ultimately come down to a certain compatibility with respect to
degeneracy maps (\reftheo{unifKanfibrinHornsquares}).

The following lemma can be used to motivate the terminology we have used for
effective fibrations in relation to naive fibrations -- we show that every
effective fibration is a naive fibration. For the notion of a discretely fibred
concrete double category, recall \refdefi{discretelyfibreddoublecat}. The lemma
implies that there is an induced natural transformation between notions of
fibred structure.
\begin{lemm}{fibrhaveliftstr}
  There is a double functor \(\EEffRFib \rightarrow \mathbb{N}\cat{Fib}\)
  between the discretely fibred concrete double categories (over \(\E\)) of
  effective fibrations and naive fibrations.
\end{lemm}
\begin{proof}
We use the assumption from the beginning of this section, that every object
is cofibrant, i.e.\ every \(0 \to A\) is contained in the dominance.
Suppose \(p: Y \to X\) is an effective fibration.
We will show that \(p\) can be equipped with a right lifting structure
with respect to the double category of HDRs. Given an HDR \(i : A \to B\),
we can define the structure as the family of fillers:
\[
          \begin{tikzpicture}[baseline={([yshift=-.5ex]current bounding box.center)}]
            \matrix (m) [matrix of math nodes, row sep=4em,
            column sep=4em]{
                |(a)| {0} & |(c)| {0} & |(c')| {Y} \\
                |(b)| {A} & |(d)| {B} & |(d')| {X} \\
            };
            \begin{scope}[every node/.style={midway,auto,font=\scriptsize}]
            \path[->]
                (a) edge (b)
                (a) edge (c)
                (b) edge node [below] {$i$} (d)
                (c) edge (d)
                (c) edge (c')
                (d) edge node [below] {$v$} (d')
                (c') edge node [right] {$p$} (d');
            \path[->]
                    (b) edge [bend left=5] node [pos=0.3, anchor=center, fill=white]
         {$u$} (c')
         (d) edge [dashed] node [anchor=center, fill=white] 
         {$\phi_{u,v}(i)$} (c');
 \end{scope}\end{tikzpicture}\text{,}
 \]
 By our assumption, the square on the left is indeed a mould square.
 It is easy to see that the horizontal condition on effective fibrations implies the
 vertical condition of a right lifting structure of Section~\ref{ssec:double-categories}.
 Similarly, the perpendicular condition implies the vertical condition. It
 follows that the family \(\phi\) gives \(p\) the structure of an \(R\)-algebra.
 It is left to the reader to verify that this association defines a 
 double functor between concrete double categories.
 \end{proof}

\begin{rema}{trivialcofibrations}
  Having defined a \emph{double} category \EEffRFib{} (from a triple category),
  there must be a corresponding \emph{double} category (or cofibred structure)
  of arrows equipped with the left lifting property with respect to this double
  category, i.e. with respect to effective fibrations. These could be
  called \emph{trivial cofibrations}.  But note that it is not straightforward
  to give a presentation of this double category starting from the triple
  category (cf.  \refrema{mouldSquarePushout}).
\end{rema}

\subsubsection{Effective trivial fibrations}\label{ssec:eff-trivial-fibrations}
Recall from Section~\ref{sec:dominances} that we named the right class with
respect to the double category \(\mathbf{\Sigma}\) of effective cofibrations
(or the dominance) \emph{effective trivial fibrations}
(\refdefi{trivialfibrations}).  Alternatively, there is a right class of arrows\index{effective trivial fibration} induced by the following triple category:
\begin{enumerate}[(i)]
\item Objects are the objects of \(\E\);
\item Horizontal and vertical morphisms are effective cofibrations;
\item Perpendicular morphisms are morphisms in \(\E\);
\item \(xy\)-squares are pullback squares of effective cofibrations;
\item \(xz\)-squares are commutative squares of effective cofibrations and morphisms in \(\E\);
\item \(yz\)-squares are pullback squares of an effective cofibration along a morphism in \(\E\);
\item Cubes are formed by pulling back an \(yz\) square along an \(xz\) square.
\end{enumerate}
For the purposes of this section, we denote the double category of arrows with
a right lifting structure with respect to this triple category by \(\mathbb{E}\cat{ffTrivFib (triple)}\). In \refprop{trivFib->effTrivFib}, we will see that this double
category is isomorphic to \(\mathbb{E}\cat{ffTrivFib}\).

First, we observe the triple category of mould squares is a \emph{sub}triple category
of this triple category, since:
\begin{lemm}{HDRinDominance}
  Every HDR is a an effective cofibration.
\end{lemm}
\begin{proof}
  This is straightforward using \reflemm{pbksqforsqofHDRs}, under the ruling
  assumption that every \(r_X : X \to MX\) is an effective cofibration.
\end{proof}

In the next proposition, we have adopted the familiar notation for
effective trivial fibrations. The proposition is stated in order to
apply \refprop{fibreddoubleiso} directly.
\begin{prop}{trivFib->effTrivFib}
  There is a double functor between discretely fibred concrete double categories
  over \(\E\)
  \[
    \mathbb{E}\cat{ffTrivFib} \rightarrow \mathbb{E}\cat{ffTrivFib (triple)} 
  \]
  which is full on squares which induces an isomorphism between notions
  of fibred structure. Hence, this functor is an isomorphism. 
\end{prop}
\begin{proof}
  The first goal is to show that every effective trivial fibration can be
  equipped with right lifting structure with respect to the above triple
  category in a functorial way.  Clearly, such a functor would be full on
  squares, since every morphism between such right lifting structures is a
  morphism between effective trivial fibrations.  It then remains to show that the
  induced natural transformation between notions of fibred structure is a
  natural isomorphism.

 Suppose \(p: Y \to X\) is an effective trivial fibration.  Every lifting
 problem with respect to an \(xy\)-square of effective cofibrations factors
 through a coproduct as follows:
  \[
          \begin{tikzpicture}[baseline={([yshift=-.5ex]current bounding box.center)}]
            \matrix (m) [matrix of math nodes, row sep=3em,
            column sep=4em]{
                |(a)| {A'} & |(c)| {B'} & |(c')| {Y} \\
                {}        &  |(e)| {A +_{A'} B} & {} \\
                |(b)| {A} & |(d)| {B} & |(d')| {X} \\
            };
            \begin{scope}[every node/.style={midway,auto,font=\scriptsize}]
            \path[->]
              (a) edge node [left] {$m$} (b)
              (e) edge (d);
            \path[->]
              (a) edge (c)
              (b) edge (e)
              (b) edge node [below] {$i$} (d)
              (c) edge (e)
              (e) edge node [anchor=center,fill=white] {$h+u$} (c')
              (c) edge (c')
              (d) edge [dashed] (c')
              (d) edge node [below] {$v$} (d')
              (c') edge node [right] {$p$} (d');
 \end{scope}\end{tikzpicture}
 \] 
Observe that \(A +_{A'} B' \hookrightarrow B\) is an effective cofibration under the
prevailing assumption that these are closed under binary unions, \(xy\) squares
are pullback squares and the fact that \(\E\) is
\emph{coherent}\index{coherent (category)} (see
\reflemm{lccFinCoImpliesCoherent}). Hence there exists a filler \(B \to Y\) as
drawn. We need to check that it satisfies the horizontal, vertical and
perpendicular conditions.  For the horizontal condition, the lifting problem
factors as follows:
\[
  \begin{tikzpicture}[baseline={([yshift=-.5ex]current bounding box.center)}]
    \matrix (m) [matrix of math nodes, row sep=3em,
    column sep=4em]{
      |(a')| {A'} & |(b')| {B'} & |(c')| {C'} & |(y)| {Y} \\
      {}   &  |(e)| {A+_{A'}B'}  &  |(f)| {A+_{B'}C'} & {} \\
      |(a)|  {A} & |(b)| {B} & |(g)| {B+_{B'}C'} & {} \\
                 &           & |(c)| {C} & |(x)| {X} \\
    };
    \begin{scope}[every node/.style={midway,auto,font=\scriptsize}]
    \node (i) at (barycentric cs:e=1,f=1,b=1,g=1) {$(*)$};
     \path[->]
       (e) edge (b)
       (f) edge (g)
       (g) edge (c)
       (a') edge (a);
      \path[->]
        (b) edge [dashed, bend right=15] (y)
        (c) edge [dashed] (y)
        (f) edge (y)
        (b') edge (e)
        (a') edge (b')
        (b') edge (c')
        (c') edge (y)
        (c') edge (f)
        (e) edge (f)
        (b) edge (g)
        (a) edge (b)
        (b) edge (c)
        (c) edge (x)
        (y) edge (x);
 \end{scope}\end{tikzpicture}
 \]
where we observe that the square \((*)\) is both a pullback and a pushout.
Hence the horizontal condition for trivial fibration applies,
and the lift \(B \to Y\) is determined by the lift \(B+_{B'}C' \to Y\).
By the vertical condition, the latter is in turn compatible with the
subsequent lift \(C \to Y\) and equal to the lift with respect to
\(A+_{B'} C' \hookrightarrow C\). It follows that the dashed fillers make
the whole diagram commute, which proves the horizontal condition.

The vertical condition can be proven similarly:
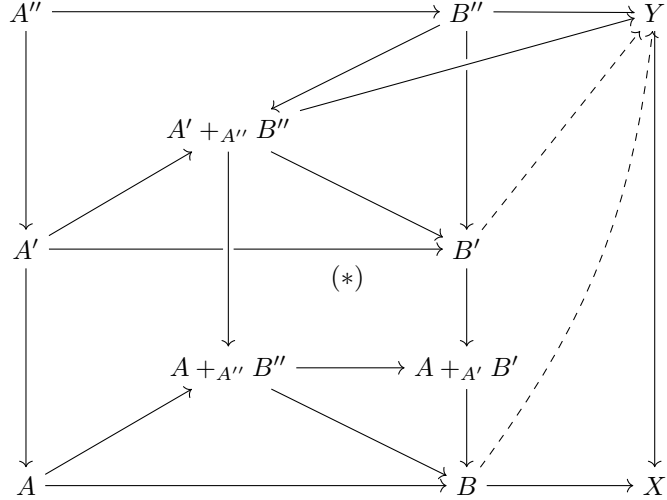
\[
 \begin{tikzpicture}[baseline={([yshift=-.5ex]current bounding box.center)}]
   \matrix (m) [matrix of math nodes, row sep=3em,
   column sep=4em]{
     |(a'')| {A''} & {} & |(b'')| {B''} & |(y)| {Y} \\
                   & |(e)| {A'+_{A''} B''} & {} & {} \\
     |(a')| {A'} &           & |(b')| {B'} & {} \\
     {}   &   |(f)| {A+_{A''}B''}      &  |(g)| {A+_{A'} B'} & {} \\
     |(a)| {A} &  {}       & |(b)| {B} & |(x)| {X} \\
   };
    \node (i) at (barycentric cs:e=1,f=1,b'=1,g=1) {$(*)$};
   \begin{scope}[every node/.style={midway,auto,font=\scriptsize},
                 on top/.style={preaction={draw=white,-,line width=#1}},
                 on top/.default=4pt]
      \path[->]
        (a'') edge (a')
        (a') edge (a)
        (e) edge (b')
        (f) edge (b)
        (f) edge (g)
        (g) edge (b)
        (b'') edge (b');
       \path[->]
         (a'') edge (b'')
         (a') edge (b')
         (a) edge (b)
         (a') edge (e)
         (b'') edge (e)
         (b') edge (g)
         (a) edge (f)
         (y) edge (x)
         (b'') edge (y)
         (b) edge (x)
         (e) edge [on top] (y)
         (e) edge [on top] (f)
         (b') edge [dashed] (y)
         (b) edge [bend right=15,dashed] (y);
 \end{scope}\end{tikzpicture}
 \]
 Again, the square \((*)\) is both pullback and pushout. The argument
 is now the same as in the horizontal case.

 The perpendicular case is easy and follows from the horizontal condition for
 effective trivial fibrations.

It is left to the reader to convince themselves that the defined lifting
structure is unique and thus induces an isomorphism of notions of fibred
structure.
\end{proof}

From the above, we obtain the following corollary.  Again, using
\refdefi{discretelyfibreddoublecat}, it follows that there is an induced
natural transformation of fibred structure:
\begin{coro}{trivFibareUniKanFib}
There is a double functor \(\mathbb{E}\cat{ffTrivFib} \rightarrow \EEffRFib\)
between discretely fibred concrete double categories (over \(\E\)).
\end{coro}
\begin{proof}
  This follows directly from the existence of a unique (concrete) triple
  functor from the triple category of mould squares to the triple category in
  this subsection, by virtue of \reflemm{HDRinDominance}.
\end{proof}

\subsection{Right and left fibrations}
Having established the definition of effective fibration, we turn to the
situations of Section~\ref{sec:frobenius}. There, we studied the AWFS for
two-sided and symmetric Moore structures.  Naturally, the two-sided setting
admits a dual notion of fibred structure which we call \emph{effective left
fibration}\index{effective left fibration} (for the terminology, recall \refrema{terminology-left-right}). In
the symmetric case, we obtain:
\begin{coro}{fibrhavecoliftstr}
Suppose the Moore structure on \(\E\) is symmetric as in as in
Section~\ref{sec:frobenius}. 
  Then there is a double functor between discretely fibred concrete double categories
  over \(\E\)
  \[
    \EEffRFib \rightarrow \mathbb{E}\cat{ffLFib}
  \]
  which is full on squares which induces an isomorphism between notions
  of fibred structure. Hence, this functor is an isomorphism. 
\end{coro}
\begin{proof}
 This essentially follows from \reflemm{fibrhaveliftstr} and \refprop{coHDRisHDR},
 but restricting to effective (left) fibrations.
\end{proof}

 \section{\texorpdfstring{\(\Pi\)}{Pi}-types}\label{sec:Pi-types}
This section contains the main result of this chapter.  Described in
more familiar terms, the result gives a constructive proof of the fact that
when \(X\) is effectively fibrant, and \(A\) is any other object, then the
exponential \(X^A\) is effectively fibrant (see
\refrema{everythingIsNaivelyFibrant}).  The statement of the proposition is a
more general, fibred version of this fact, which is a classic result on Kan
fibrations in simplicial sets (see e.g.~\cite{May-67}, Theorem~7.8).

We assume that \(\E\) is a finitely cocomplete, locally cartesian closed
category equipped with symmetric Moore structure and a dominance.  Further, it
satisfies the additional conditions from the beginning of
Section~\ref{sec:unifKanFibr}. We then have the following theorem.
\begin{theo}{symmetricPocProofOfPi}{\rm (Pushforward for effective fibrations)}
Suppose \(\E\) is given as above. Then the pushforward \(\Pi_g
f\)\index{pushforward} of an effective fibration \(f\) along an effective
fibration \(g\) is again an effective fibration. More precisely, there is a
functor
\[
        \Pi_{(-)}(-) :  \EffRFib \times_{\E} \EffRFib \to \EffRFib
\]
as in \reflemm{uniformkanfibrandpi} 
such that \(\Pi_{g}(-)\) is right adjoint to pullback along \(g\).        
\end{theo}
\begin{proof}
By precomposing the functor of
\reflemm{uniformkanfibrandpi} below with the equivalence in 
\refcoro{fibrhavecoliftstr} and the functor in 
\reflemm{fibrhaveliftstr}, the proof follows.
\end{proof}

The proof refers to the following lemma, which is stated in the generality of a
two-sided Moore structure on the category \(\E\) (see the introduction to
Section~\ref{sec:frobenius}). Apart from that, \(\E\) satisfies the same
assumptions as in the previous theorem.
\begin{lemm}{uniformkanfibrandpi}
  If $f: Y \to X$ is a naive left fibration and $g: Z \to Y$ is an effective
  fibration, then the pushforward $\Pi_f(g)$ is also an
  effective fibration.  More precisely, the pullback along a naive left
  fibration 
\(
        f^* : \EffRFib_{X} \to \EffRFib_{Y}
\)
which takes effective fibrations with codomain \(X\) to effective fibrations with
codomain \(Y\) has a right adjoint
\(
        f_* : \EffRFib_{Y} \to \EffRFib_{X}
\)
defined by a functor
\[
        \Pi_{(-)}{(-)} : \cat{NLFib} \times_{\E} \EffRFib \to \EffRFib
\]
where the domain is the pullback of the domain and codomain functors.
\end{lemm}
\begin{proof} 
  Assume \(f\) is a naive left fibration and \(g: Z \to Y\) is an effective 
fibration.  We have to show that \(\Pi_f(g)\) is an effective fibration. So
imagine we have a situation like this:
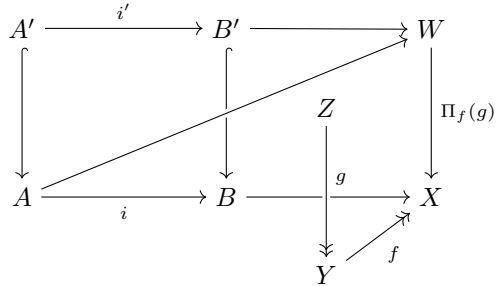
\[
          \begin{tikzpicture}[baseline={([yshift=-.5ex]current bounding box.center)}]
\matrix (m) [matrix of math nodes, row sep=5em,
column sep=6em]{
    |(a')| {A'} & |(b')| {B'} & |(pf)| {W} \\
    |(a)| {A} & |(b)| {B} & |(x)| {X} \\
};
\matrix (n) [matrix of math nodes, row sep=5em,
column sep=6em, position=210:1 from pf]{
    |(c')| {} & |(d')| {} & |(z)| {Z} \\
    |(c)| {} & |(d)| {} & |(y)| {Y} \\
};
\begin{scope}[every node/.style={midway,auto,font=\scriptsize}]
        \path[{Hooks[right]}->]
    (a') edge (a)
    (b') edge (b);
    \path[-]
    (a) edge [line width = 3pt, draw = white] (pf);
    \path[->]
    (a') edge node [above] {$i'$} (b')
    (a) edge node [below] {$i$} (b)
    (b') edge (pf)
    (b) edge (x)
    (pf) edge node [right] {$\Pi_f(g)$} (x)
    (a) edge (pf);
    \path[-]
    (z) edge [line width = 3pt, draw = white] (y);
    \path[->>]
    (y) edge node [below right] {$f$} (x)
    (z) edge node [pos=0.4, right] {$g$} (y);
 \end{scope}
 \end{tikzpicture}
 \]
in which the left hand square is a mould square.  The construction starts by
taking the pullback of \(f\) along \(A \to X\), which yields a naive left fibration, and
subsequently constructing a `Frobenius' cube like
in~\eqref{eq:frobenius-pullback}:
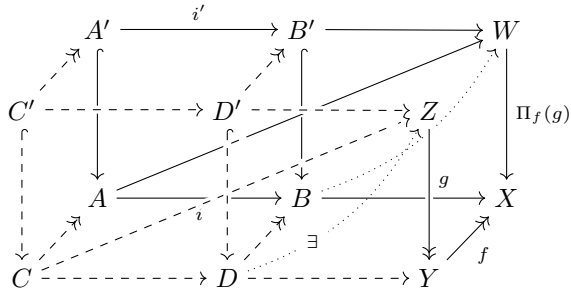
\[
          \begin{tikzpicture}[baseline={([yshift=-.5ex]current bounding box.center)}]
            \matrix (m) [matrix of math nodes, row sep=5em,
            column sep=6em]{
                |(a')| {A'} & |(b')| {B'} & |(pf)| {W} \\
                |(a)| {A} & |(b)| {B} & |(x)| {X} \\
            };
            \matrix (n) [matrix of math nodes, row sep=5em,
            column sep=6em, position=210:1 from pf]{
                |(c')| {C'} & |(d')| {D'} & |(z)| {Z} \\
                |(c)| {C} & |(d)| {D} & |(y)| {Y} \\
            };
            \begin{scope}[every node/.style={midway,auto,font=\scriptsize}]
\path[{Hooks[right]}->]
(a') edge (a)
(b') edge (b);
\path[-]
(a) edge [line width = 2pt, draw = white] (pf);
\path[->]
(a) edge node [below] {$i$} (b)
(d) edge [dotted, bend right=20] node [pos=0.3, anchor=center, fill=white] {$\exists$} (z)
(b) edge [bend right=20, dotted] (pf)
(c) edge [line width=3pt, draw=white] (z)
(c) edge [dashed] (z)
(d) edge [dashed] (y)
(c) edge [dashed] (d)
(c') edge [line width = 3pt, draw = white] (d')
(c') edge [dashed] (d')
(a') edge node [above] {$i'$} (b')
(b') edge (pf)
(b) edge (x)
(pf) edge node [right] {$\Pi_f(g)$} (x)
(a) edge (pf)
(d') edge [line width = 3pt, draw = white] (z)
(d') edge [dashed] (z);
\path[{Hooks[right]}->]
(c') edge [dashed] (c)
(d') edge [line width = 3pt, draw = white] (d)
(d') edge [dashed] (d);
\path[->>]
(y) edge node [below right] {$f$} (x)
(z) edge [line width = 3pt, draw = white] (y)
(z) edge node [pos=0.4, right] {$g$} (y)
(d) edge [dashed] (b)
(c') edge [dashed] (a')
(d') edge [dashed] (b')
(c) edge [dashed] (a);
 \end{scope}
 \end{tikzpicture}
 \]
So the front and back squares are mould squares and the bottom and top are
Frobenius morphisms of HDRs. The maps $D' \to Z$ and $C \to Z$ are induced by the 
adjunction $f^* \ladj \Pi_f$. Since $g$ is an effective fibration, there is a
map $D \to Z$ making everything commute, which we can transpose back to a map
$B \to W$ giving a filler.

It remains to check the three compatibility conditions, horizontal, vertical
and perpendicular, for this filler.  The horizontal and vertical conditions
follow fairly directly from the naturality condition for adjunctions together
with \reflemm{frobeniusStableUnderComposition}. Indeed, for the horizontal
condition:
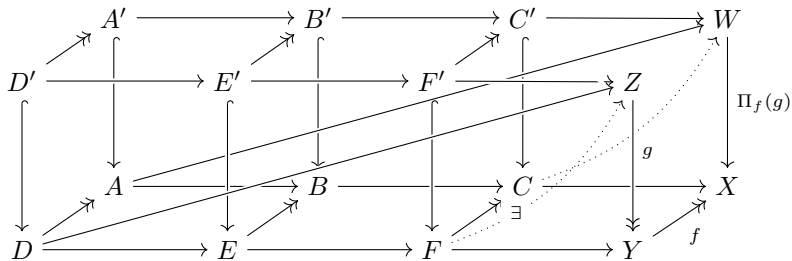
\[
          \begin{tikzpicture}[baseline={([yshift=-.5ex]current bounding box.center)}]
            \matrix (m) [matrix of math nodes, row sep=5em,
            column sep=6em]{
                    |(a')| {A'} & |(b')| {B'}& |(c')| {C'} & |(pf)| {W} \\
                    |(a)| {A} & |(b)| {B} & |(c)| {C} & |(x)| {X} \\
            };
            \matrix (n) [matrix of math nodes, row sep=5em,
            column sep=6em, position=200:1 from pf]{
                    |(d')| {D'} & |(e')| {E'} & |(f')| {F'} & |(z)| {Z} \\
                    |(d)| {D} & |(e)| {E} & |(f)| {F} & |(y)| {Y} \\
            };
            \begin{scope}[every node/.style={midway,auto,font=\scriptsize}]
                    \path[{Hooks[right]}->]
                (a') edge (a)
                (c') edge (c)
                (b') edge (b);
                \path[->>]
                (y) edge node [below right] {$f$} (x)
                (z) edge node [pos=0.4, right] {$g$} (y)
                (e) edge (b)
                (d') edge (a')
                (e') edge (b')
                (d) edge (a)
                (f) edge (c)
                (f') edge (c');
                \path[->]
                (c) edge [bend right=20, dotted] (pf)
                (a) edge [line width=1.8pt, draw=white] (pf)
                (a) edge (pf)
                (a') edge (b')
                (a) edge (b)
                (b') edge (c')
                (c') edge (pf)
                (c) edge (x)
                (b) edge (c)
                (pf) edge node [right] {$\Pi_f(g)$} (x);
                \path[->>]
                (z) edge [line width=2pt, draw=white] (y)
                (z) edge (y);
                \path[->]
                (f) edge [dotted, bend right=20] node [pos=0.3, anchor=center, fill=white] {$\exists$} (z)
                (f) edge (y)
                (e) edge (f)
                (f') edge [line width=2pt, draw=white] (z)
                (f') edge (z)
                (e') edge [line width=2pt, draw=white] (f')
                (e') edge (f')
                (d) edge (e)
                (d') edge [line width=2pt, draw=white] (e')
                (d') edge (e');
               \path[{Hooks[right]}->]
                (d') edge (d)
                (e') edge [line width=2pt, draw=white] (e)
                (e') edge (e)
                (f') edge [line width=2pt, draw=white] (f)
                (f') edge (f);
              \path[->]
                (d) edge [line width=1.8pt, draw=white] (z)
                (d) edge (z);
 \end{scope}\end{tikzpicture}
 \]
To find a map $C \to W$ we can either transpose $D \to Z$ and then push
forward in one go as $F \to Z$ and then transpose back. This should coincide
with: first obtaining $E \to Z$, transpose, then transpose back and compute
$F \to Z$ and then transpose. It is clear that this will be the same as the
other procedure.

The vertical condition is similar:
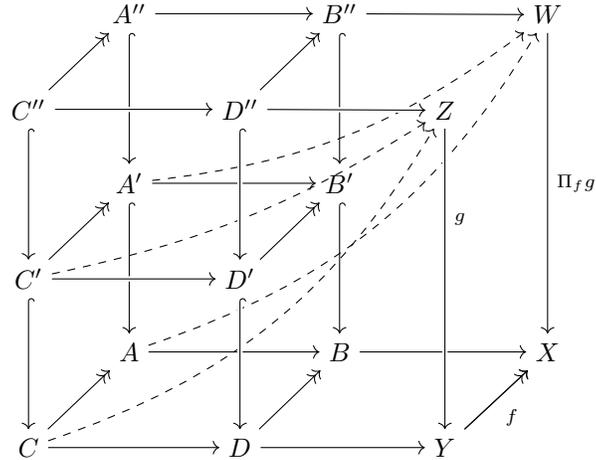
\[
\begin{tikzpicture}[baseline={([yshift=-.5ex]current bounding box.center)}]
  \matrix (m) [matrix of math nodes, row sep=5em,
  column sep=6em]{
          |(a'')| {A''} & |(b'')| {B''} & |(pf)| {W} \\
          |(a')| {A'} & |(b')| {B'} &  \\
          |(a)| {A} & |(b)| {B}  & |(x)| {X} \\
  };
  \matrix (n) [matrix of math nodes, row sep=5em,
  column sep=6em, position=220:1 from pf]{
          |(c'')| {C''} & |(d'')| {D''} & |(z)| {Z} \\
          |(c')| {C'} & |(d')| {D'} &  \\
          |(c)| {C} & |(d)| {D} &  |(y)| {Y} \\
  };
  \begin{scope}[every node/.style={midway,auto,font=\scriptsize}]
          \path[{Hooks[right]}->]
             (a'') edge (a')
             (a') edge (a)
             (b'') edge (b')
             (b') edge (b);
          \path[->>]
           (y) edge (x)
           (d) edge (b)
           (c) edge (a)
           (d') edge (b')
           (c') edge (a')
           (d'') edge (b'')
           (c'') edge (a'');
          \path[->]
          (a'') edge (b'')
          (a') edge (b')
          (a) edge (b)
          (b'') edge (pf)
          (b) edge (x)
          (a) edge [bend right=20, dashed] (pf)
          (a') edge [bend right=15, dashed] (pf)
          (c) edge [bend right=20, dashed] (z)
          (c') edge [bend right=10, dashed] (z)
          (y) edge node [below right] {$f$} (x)
          (pf) edge node [right] {$\Pi_f g$} (x);
          \path[-, line width=2pt, draw=white]
          (c'') edge (d'')
          (c') edge (d')
          (c) edge (d)
          (z) edge (y)
          (d'') edge (z)
          (d) edge (y);
          \path[->]
          (c'') edge (d'')
          (c') edge (d')
          (c) edge (d)
          (d'') edge (z)
          (z) edge node [pos=0.3, right] {$g$} (y)
          (d) edge (y);
          \path[-, line width=2pt, draw=white]
             (c'') edge (c')
             (c') edge (c)
             (d'') edge (d')
             (d') edge (d);
          \path[{Hooks[right]}->]
             (c'') edge (c')
             (c') edge (c)
             (d'') edge (d')
             (d') edge (d);
          \end{scope}
  \end{tikzpicture}
\]
Given a map \(A \to W\) one can transpose to \(C \to Z\), then push forward to \(D
\to W\) and then transpose back. Using the vertical condition on \(g\), this can also
be done in two steps. Namely, by first lifting the restricted arrow \(C' \to Z\) to
\(D' \to Z\), then transposing, and repeating the construction for the bottom cube.
This works because the square \(D'B'BD\) is a pullback, hence we return to the
same diagram (same arrows) in the two-step version.

For the perpendicular condition, it helps to reduce dimensions by one by
drawing HDRs as a point, as in the following diagram:
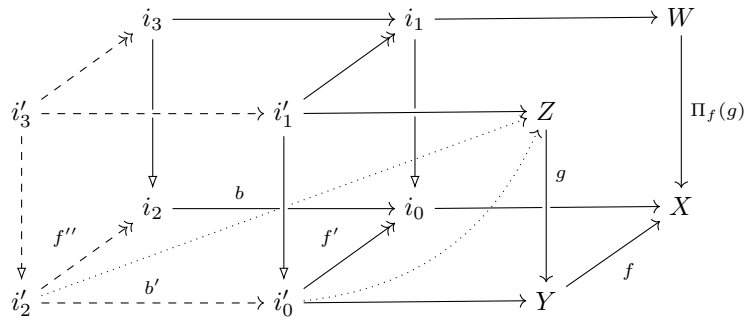
\[
  \begin{tikzpicture}[baseline={([yshift=-.5ex]current bounding box.center)}]
    \matrix (m) [matrix of math nodes, row sep=2em,
    column sep=3.5em]{
                     & |(i3)| {i_3} &  & |(i1)| {i_1} & & |(w)| {W} \\
      |(i3')| {i_3'} & & |(i1')| {i_1'} & & |(z)| {Z} & \\
                     & |(i2)| {i_2} & & |(i0)| {i_0} & & |(x)| {X} \\
      |(i2')| {i_2'} & & |(i0')| {i_0'} & & |(y)| {Y} & 
    \\};
    \begin{scope}[every node/.style={midway,auto,font=\scriptsize}]
        \path[-{Latex[open]}]
            (i1) edge (i0)
            (i3) edge (i2)
            (i3') edge [dashed] (i2');
      \path[->>]
          (i1') edge (i1)
          (i2') edge [dashed] node [above left] {$f''$} (i2)
          (i0') edge node [above left] {$f'$}  (i0)
          (i3') edge [dashed] (i3);
      \path[->]
          (i2) edge node [above, pos=0.3] {$b$} (i0)
          (i3) edge (i1)
          (i2') edge [dashed] node [above] {$b'$} (i0')
          (i0) edge (x)
          (i0') edge (y)
          (i1) edge (w)
          (y) edge node [below right] {$f$} (x)
          (w) edge node [right] {$\Pi_f(g)$} (x)
          (i2') edge [dotted] (z)
          (i0') edge [bend right, dotted] (z);
      \path[-, line width=3pt, draw=white]
          (i1') edge (z)
          (i1') edge (i0')
          (z) edge (y)
          (i3') edge (i1');
        \path[-{Latex[open]}]
          (i1') edge (i0');
      \path[->]
          (z) edge node [pos=0.3, right] {$g$} (y)
          (i1') edge (z)
          (i3') edge [dashed] (i1');
      \end{scope}\end{tikzpicture}\]
Here the open arrowtips indicate mould squares, and the double arrowtips
indicate Frobenius morphisms of HDRs. In this picture, lifting can be viewed as
completing a `partial' arrow \(i_0 \to W\) to a total arrow. Now suppose the
left square on the back is a cube between mould squares.

We can complete the diagram as drawn by forming a pullback cube in
the category \(\cat{HDR}\). Since cubes in the triple category are determined
by their boundary, the left square on the front is a cube between mould squares.
Hence the filler \(i_2' \to Z\) induced by \(g\) makes the triangle with the filler
\(i_0' \to Z\) commute. The former uniquely determines the filler \(i_2 \to W\),
and the latter uniquely determines the filler \(i_0 \to W\). So these also make
the triangle in the back commute and we are done.

Note that the crucial part of this argument is \reflemm{pullbckstoffrob} since
without it, it would not have been possible to work entirely inside the category
\cat{HDR}, and it would not have been clear that the two ways to construct the
diagram (starting with \(f'\) vs. starting with \(b\)) would yield the same
4-dimensional cube, with the same relevant HDR structures, in the underlying category.
\end{proof}

\begin{rema}{everythingIsNaivelyFibrant}
  To relate the \reftheo{symmetricPocProofOfPi} to the statement in the
  beginning of this section, consider that in a category with Moore structure,
  every object is naively fibrant. Indeed, any terminal arrow \(A \to 1\) can
  be trivially equipped with a transport structure \(A \times M1 \to A\) given
  by projection. So every \(A \to 1\) is a naive left fibration, and hence
  \[
    X^A := \Pi_{A \to 1} X \times A \to 1
  \]
  is an effective fibration whenever \(p_1 : X \times A \to X\) is.
\end{rema}

With the proof of \reftheo{symmetricPocProofOfPi}, we have concluded the first
part of this paper. In the second part, we will show that the Moore path
functor defined for simplicial sets in~\cite{vdBerg-Garner} satisfies the
conditions of a (symmetric) Moore structure, and that simplicial sets admit a
dominance of locally decidable monomorphisms. As a result, we obtain a notion
of effective fibration in simplicial sets, which is closed under pushforward.
In Section~\ref{sec:hornsquares}, we will show that effective fibrations in
simplicial sets are always Kan fibrations, and that under classical
assumptions, the notions coincide.

\chapter{Simplicial sets}\label{ch:simplicial-sets}

\section{Effective trivial Kan fibrations in simplicial sets}

Whereas in the first part we derived the existence of $\Pi$-types in an axiomatic setting based on a suitable combination of a dominance and a symmetric Moore structure, this second part will be entirely devoted to one particular example: simplicial sets. To show that the category of simplicial sets is indeed an example, we will first have to choose both a dominance and a symmetric Moore structure on simplicial sets, and the first two sections of this second part will do exactly that. Indeed, in this section we will choose a dominance and in the next section we will choose a symmetric Moore structure. After that, we will study the resulting HDRs and effective Kan fibrations in more detail. In particular, we will show that the effective Kan fibrations are a local notion of fibred structure and that in a classical metatheory the maps which can be equipped with the structure of an effective Kan fibration are precisely those maps which have the right lifting property against horn inclusions.

But before we get into that, let us first choose and study a suitable dominance on simplicial sets. As we have seen in Section 3, dominances induce AWFSs. We will call the coalgebras for the comonad of the induced AWFS \emph{effective cofibrations}, while the algebras for the monad will be called the \emph{effective trivial Kan fibrations}. The main results of this section will be that being an effective trivial Kan fibration is a local notion fibred structure, and that (in a classical metatheory) a map can be equipped with structure of an effective trivial Kan fibration if and only if it has the right lifting property against boundary inclusions.

\subsection{Effective cofibrations} Traditionally, the cofibrations in simplicial sets are simply the monomorphisms. In our constructive metatheory we believe it is important to add a decidability condition. We will further comment on this choice at the end of this section.

\begin{defi}{cofibration}
  In the category of simplicial sets we will call $m: B \to A$ an \emph{(effective) cofibration} \index{effective cofibration} if it is a locally decidable monomorphism \index{locally decidable monomorphism}: that is, each $m_n: B_n \to A_n$ is a complemented monomorphism in the subobject lattice of $A_n$. In other words, for each $a \in A_n$ we can decide whether there is an element $b \in B_n$ such that $m_n(b) = a$ or not.
\end{defi}

In the statement of the following lemma we will say that a set is finite if it can be put in bijective correspondence with some initial segment of the natural numbers. Moreover, we will say that a sieve $S$ is generated by a set of maps $I$ if $S$ is the closure of $I$ under precomposition with arbitrary maps.

\begin{lemm}{oncofsieves}
  The following are equivalent for a sieve $S \subseteq \Delta^n$:
  \begin{enumerate}
    \item[(1)] The inclusion $i: S \subseteq \Delta^n$ is an effective cofibration.
    \item[(2)] The sieve is generated by a finite set of monos $\Delta^m \to \Delta^n$.
    \item[(3)] The sieve is generated by a finite set of maps.
  \end{enumerate}
\end{lemm}
\begin{proof}
    The implications (2) $\Rightarrow$ (3) $\Rightarrow$ (1) are obvious, so it remains to show that (1) $\Rightarrow$ (2).

  But since every map in $\mathbf{\Delta}$ factors as an epi followed by a mono (in a unique way), and every epi splits, every sieve is generated by its monomorphisms. But since there are only finitely many monos with codomain $\Delta^n$, a cofibrant sieve contains only finitely many monos.
\end{proof}

\begin{defi}{cofibrantsieve} We will refer to the sieves satisfying any of the equivalent conditions in the previous lemma as the \emph{cofibrant sieves}. \index{sieve!cofibrant}
\end{defi}

Because a monomorphism in $\mathbf{\Delta}$ with codomain $[n]$ is completely determined by its image, a cofibrant sieve can be thought of as a subsimplicial complex of $\Delta^n$, that is, a collection of inhabited and decidable subsets of $\{0,1, \ldots, n \}$ which is itself closed under inhabited and decidable subsets.

\begin{theo}{cofinsimplsetsdominance}
The effective cofibrations in simplicial sets form a dominance.
\end{theo}
\begin{proof}
  The cofibrations are clearly closed under pullback and composition, so we only need to prove that there is a cofibration $1 \to \Sigma$ such that any other can be obtained as a pullback of that one in a unique way. We put
  \[ \Sigma_n := \{ \, S \subseteq \Delta^n \, : \, S \mbox{ cofibrant sieve} \, \}. \]
  (Note that $\Sigma_n$ is finite, so that $\Sigma_n$ is a set even in a predicative metatheory like {\bf CZF}.) Since cofibrant sieves are stable under pullback along $\alpha: \Delta^m \to \Delta^n$, this defines a simplicial set. That is, we define the action on $\Sigma$ by the following formula:
  \[ S \cdot \alpha = \{ \beta: [k] \to [m] \, : \, \alpha.\beta \in S \, \}. \]
  In addition, there is a natural transformation $\top: 1 \to \Sigma$ obtained by picking the maximal sieve at each level. This map classifies the cofibrations in that for any cofibration $m: B \to A$ the map $\mu: A \to \Sigma$ defined by
  \[ \mu_n(a) = \{ \alpha: [m] \to [n] \, : \, (\exists b \in B_m) \, a \cdot \alpha = m(b) \, \} \]
  turns
  \begin{displaymath}
      \begin{tikzcd}
    B \ar[d, "m"'] \ar[r] & 1 \ar[d, "\top"] \\
      A \ar[r, "\mu"'] & \Sigma
    \end{tikzcd}
  \end{displaymath}
  into a pullback. Also, the map $\mu$ is easily seen to be unique with this property.
\end{proof}

\begin{rema}{cofpropandlogicalop}
Note that the cofibrant subobjects form a sub-Heyting algebra of the full subobject lattice. In particular, the cofibrant sieves are closed under all the propositional operations: not only $\land,\top$, but also $\bot,\lor$ and $\to$. To see that they are closed under implication, for instance, note that for sieves $S, T \subseteq \Delta^n$, we have
\[ \alpha: [m] \to [n] \in (S \to T) \Longleftrightarrow (\forall \beta: [k] \to [m]) \, ( \, \alpha.\beta \in S \Rightarrow \alpha.\beta \in T \,). \]
Because maps in $\mathbf{\Delta}$ factor as an epi followed by a mono, and epis split, we only need to check the condition on the right for monos $\beta$. So if $S$ and $T$ are cofibrant, the condition on the right is decidable and $S \to T$ is cofibrant as well.
\end{rema}

\subsection{Effective trivial Kan fibrations}

Since the effective cofibrations in simplicial sets form a dominance, they are the left class in an algebraic weak factorisation system. The members of the right class will be referred to as the \emph{effective trivial Kan fibrations}\index{trivial Kan fibration!effective}. From the work by Bourke and Garner (recapitulated in Section 2), we know that these can be defined as the maps which come with a compatible system of lifts against a large double category: one where the vertical maps are the cofibrations and the squares are pullback squares. Our first goal in this subsection is to show that we can restrict attention to a small subdouble category.

Indeed, let $\mathbb{C}$ be the following small double category:
\begin{itemize}
 \item Objects are cofibrant sieves $S \subseteq \Delta^n$.
 \item Horizontal maps from $S \subseteq \Delta^n$ to $T \subseteq \Delta^m$ are maps $\alpha: \Delta^n \to \Delta^m$ such that $T \cdot \alpha= S$.
 \item Vertical maps are inclusions of cofibrant sieves $S_0 \subseteq S_1 \subseteq \Delta^n$.
 \item Squares are pullback diagrams of the form
    \begin{displaymath}
      \begin{tikzcd}
        S_0 \ar[d] \ar[r, "\alpha"] & T_0 \ar[d] \\
           S_1 \ar[r, "\alpha"'] & T_1
      \end{tikzcd}
    \end{displaymath}
  such that both horizontal maps are labelled with the same $\alpha$.
\end{itemize}
Clearly, there is an inclusion of double categories from $\mathbb{C}$ to the large double category of cofibrations.

\begin{prop}{cofsivesgentrivfib}
The following notions of fibred structure are isomorphic:
\begin{itemize}
  \item Having the right lifting property against the large double category of cofibrations (that is, to be an effective trivial Kan fibration).
  \item Having the right lifting property against the small double category $\mathbb{C}$.
\end{itemize}
In fact, if ${\rm Cof}$ is the large double category of cofibrations, then the morphism of discretely fibred concrete double categories $\Drl{\rm Cof} \to \Drl{\mathbb{C}}$ induced by the inclusion is an isomorphism.
\end{prop}
\begin{proof}
Assume $p: Y \to X$ has the right lifting property against the small double category $\mathbb{C}$, and imagine that we have a lifting problem of the form:
\begin{displaymath}
  \begin{tikzcd}
    B \ar[d, "m"'] \ar[r] & Y \ar[d, "p"] \\
    A \ar[r] \ar[ur, dotted, "l"] & X.
  \end{tikzcd}
\end{displaymath}
Suppose $a \in A_n$ is arbitrary and we pull back $m$ along $a: \Delta^n \to A$:
\begin{equation} \label{aries}
  \begin{tikzcd}
    S \ar[r] \ar[d] & B \ar[d, "m"] \ar[r] & Y \ar[d, "p"] \\
    \Delta^n \ar[r, "a"'] \ar[urr, dotted] & A \ar[r] & X.
  \end{tikzcd}
\end{equation}
Since the pullback is a cofibrant sieve, we find an element $y \in Y_n$ filling the outer rectangle, and we put $l_n(a) := y$. Note that this definition is forced, because the left hand square in the diagram \ref{aries} is a square in the large double category. Note also that $l$ is going to be a natural transformation because of the horizontal condition coming from the pullback squares in $\mathbb{C}$:
\begin{displaymath}
  \begin{tikzcd}
    S \cdot \alpha \ar[d] \ar[r] & S \ar[r] \ar[d] & B \ar[d, "m"] \ar[r] & Y \ar[d, "p"] \\
    \Delta^m \ar[r, "\alpha"'] \ar[urrr, dotted] & \Delta^n \ar[r, "a"'] \ar[urr, dotted] & A \ar[r] & X.
  \end{tikzcd}
\end{displaymath}
(Here the words horizontal and vertical condition refer to the conditions for being a right lifting structure, as can be found after \refexam{double-cat-of-arrows}.)

Next, let us check that in case $m$ is a vertical map coming from the small double category, the new lift $l$ agrees with the one coming from the fact that $p$ has the right lifting property against $\mathbb{C}$: for note that if $S \subseteq T \subseteq \Delta^n$ are cofibrant sieves and $\alpha: \Delta^m \to \Delta^n \in T$, then $T \cdot \alpha = \Delta^m$ and the left hand square in
\begin{displaymath}
  \begin{tikzcd}
    S \cdot \alpha \ar[r] \ar[d] & S \ar[r] \ar[d] & Y \ar[d, "p"] \\
    \Delta^m \ar[r] \ar[urr, dotted, "y",near start] & T \ar[r] & X
  \end{tikzcd}
\end{displaymath}
is a square in the double category $\mathbb{C}$. So both the lift $T \to Y$ we have constructed and the one coming from the fact that $p$ has the right lifting property against $\mathbb{C}$ send $\alpha$ to the lift $y$.

It is now easy to see that the constructed lifts satisfy both the horizontal and vertical conditions with respect to the large double category, thus showing that on the level of notions of fibred structure we have an isomorphism. The fullness condition for squares follows from the fact that the left hand square in diagram \ref{aries} belongs to the large double category, showing that we have an isomorphism of discretely fibred concrete double categories.
\end{proof}

We will now cut down things even further. In fact, the lifting structure against $\mathbb{C}$ is completely determined by its lifts against the boundary inclusions, as we will now show.

\begin{lemm}{liftstrdetermbyboundarymaps}
  Suppose $p: Y \to X$ has two lifting structures against the small double category $\mathbb{C}$. If these two lifting structures agree on the lifts against the boundary inclusions, then they agree on all vertical maps.
\end{lemm}
\begin{proof}
Let $S \subseteq T \subseteq \Delta^n$ be cofibrant sieves. Then this inclusion can be decomposed as
\[ S = S_0 \subseteq S_1 \subseteq S_2 \subseteq \ldots \subseteq S_k = T \subseteq \Delta^n \]
where each $S_{i+1}$ contains precisely one face of $\Delta^n$ more than $S_i$ (for $0 \leq i \lt k$). By the vertical condition, the lift against $S \subseteq T$ is completely determined by the lifts against the $S_i \subseteq S_{i+1}$. But if $\Delta^m \to S_{i+1}$ is the one face which belongs to $S_{i+1}$ but not $S_i$, then its entire boundary lies in $S_i$ and we have a pullback diagram as follows:
\begin{displaymath}
  \begin{tikzcd}
    \partial \Delta^m \ar[r] \ar[d] & S_i \ar[d] \\
         \Delta^m  \ar[r] & S_{i+1}.
  \end{tikzcd}
\end{displaymath}
Since this diagram is both a square in the double category $\mathbb{C}$ and a pushout in simplicial sets, the lift against the map on the right is completely determined by the lift against the map on the left.
\end{proof}

In the remainder of this subsection we will try to answer the following question: suppose we have a map $p: Y \to X$ and we have chosen lifts against all boundary inclusions
\begin{displaymath}
  \begin{tikzcd}
    \partial \Delta^m \ar[r] \ar[d] & Y \ar[d, "p"] \\
         \Delta^m  \ar[r] \ar[ur, "f_i", dotted] & X.
  \end{tikzcd}
\end{displaymath}
What conditions do these lifts $f_i$ have to satisfy in order for them to extend to a lifting structure against $\mathbb{C}$?

First of all, because any inclusion of sieves $S \subseteq T \subseteq \Delta^n$ can be seen as a composition of pushouts of boundary inclusions, as in the previous lemma, we can solve any lifting problem of the form
\begin{displaymath}
  \begin{tikzcd}
    S \ar[r] \ar[d] & Y \ar[d, "p"] \\
     T \ar[ur, dotted] \ar[r] & X.
  \end{tikzcd}
\end{displaymath}
The first worry is that the decomposition of the inclusion $S \subseteq T$ as a sequence of inclusions where the next sieve contains one face more than the previous is in no way unique, and it could be that the lift we construct depends on the decomposition. As a matter of fact, it does not depend on this: imagine that we choose two different decompositions of the inclusion $S \subseteq T$ and they determine lifts $f$ and $g$, respectively. Now we can prove by induction on $n \in \mathbb{N}$ that $f$ and $g$ agree on all $n$-simplices, using that they agree on their boundaries in the induction step.

The next worry is that these lifts need to satisfy both the horizontal and the vertical condition coming from $\mathbb{C}$. The vertical condition is, in fact, automatically satisfied, because of the way we constructed the lifts and the fact that the way we decompose the vertical maps in $\mathbb{C}$ is irrelevant.

Therefore we only need to consider the horizontal condition: that is, we need to determine which conditions the chosen lifts $f_i$ need to satisfy in order that in any diagram of the form
\begin{displaymath}
  \begin{tikzcd}
    \alpha^*S \ar[r] \ar[d] & S \ar[r] \ar[d] & Y \ar[d, "p"] \\
    \alpha^* T \ar[r] \ar[rru, dotted] & T \ar[r] \ar[ur, dotted] & X
  \end{tikzcd}
\end{displaymath}
with the square on the left a square in $\mathbb{C}$, the induced lifts make the resulting triangle commute. The horizontal condition can be split in two: because every map in $\mathbf{\Delta}$ is the composition of face and degeneracy maps, we only need to worry about squares where the map $\alpha$ is either a face or degeneracy map. In fact, the case where $\alpha$ is a face map is unproblematic. The reason is this: imagine that we have a decomposition
\[ S_0 \subseteq S_1 \subseteq S_2 \subseteq \ldots \subseteq S_k \subseteq \Delta^n \]
of $S \subseteq T$ and each $S_{j+1}$ contains precisely one face more than $S_j$. If we pull this back along $d_i: \Delta^{n-1} \to \Delta^n$, we get either that $d_i^* S_j = d_i^* S_{j+1}$ if the face that gets added to $S_j$ in this step does not belong to $d_i$, or that $d_i^* S_j \not= d_i^* S_{j+1}$ in case the face that gets added to $S_j$ in this step does belong to $d_i$. But in the latter case, $d_i^* S_{j+1}$ contains one face more than $d_i^* S_j$, so if we ignore all the first cases we obtain a decomposition of $d_i^*S_0 \subseteq d_i^* S_k$. If we use this decomposition to compute the lift against $d_i^*S_0 \subseteq d_i^* S_k$, then by pullback pasting
\begin{displaymath}
  \begin{tikzcd}
    \partial \Delta^m \ar[r] \ar[d] & d_i^* S_j \ar[r] \ar[d] & S_j \ar[d] \\
    \Delta^m \ar[r] & d_i^* S_{j+1} \ar[r] & S_j
  \end{tikzcd}
\end{displaymath}
it is computed in exactly the same way as the lift against $S_0 \subseteq S_k$ is computed on the simplices which belong to the $i$th face.

So the upshot of the discussion so far is that we only need to worry about the horizontal condition for squares with $\alpha$ being a degeneracy map. Here, in view of the decomposition, we can restrict attention to the situation where the map on the right in the square is an inclusion $S \subseteq T$ where $T$ contains precisely one face more than $S$. In fact, we claim that we only need to worry about the situation where the map on the right is a boundary inclusion, as in:
\begin{displaymath}
  \begin{tikzcd}
    s_i^*(\partial \Delta^n) \ar[r] \ar[d] & \partial \Delta^n \ar[r] \ar[d] & Y \ar[d, "p"] \\
    \Delta^{n+1} \ar[r, "s_i"'] \ar[urr, dotted] & \Delta^n \ar[ur, dotted] \ar[r] & X
  \end{tikzcd}
\end{displaymath}
Indeed, assume that $S \subseteq T \subseteq \Delta^n$ is an inclusion where $T$ contains precisely one more face than $S$, which happens to be $\Delta^m \to T$; also assume that we have some lifting problem of $S \subseteq T$ against $p$, and $s_i: \Delta^{n+1} \to \Delta^n$ is some degeneracy. Using that pullbacks of monos along epis exist in the simplicial category, we can create a diagram as follows:
\begin{displaymath}
  \begin{tikzcd}
    & s_i^*S \ar[dd] \ar[rr] & & S \ar[dd] \ar[rr] & & Y \ar[dd, "p"] \\
    U \ar[rr, crossing over] \ar[dd] \ar[ur] & & \partial \Delta^m \ar[ur] \\
    & s_i^* T \ar[dd] \ar[rr] & & T \ar[dd] \ar[rr] & & X \\
    \Delta^{m'} \ar[rr, crossing over] \ar[dd] \ar[ur] & & \Delta^m \ar[from = uu, crossing over] \ar[ur] \\
    & \Delta^{n+1} \ar[rr] & & \Delta^n \\
    \Delta^{m'} \ar[rr] \ar[ur] & & \Delta^m \ar[from = uu, crossing over] \ar[ur] \\
  \end{tikzcd}
\end{displaymath}
Note that in the top cube, the left, right, front and back faces are squares in the double category $\mathbb{C}$, and the right face is a pushout of simplicial sets. Therefore the left face is a pushout as well and the horizontal condition for the back face of that cube is equivalent to the horizontal condition for the front of that cube. But the map $\Delta^{m'} \to \Delta^m$ is either the identity if $i$ does not belong to the image of $\Delta^m \to \Delta^n$, or some degeneracy $s_j: \Delta^{m+1} \to \Delta^m$ if it does.

This means that we can restrict attention to the following situation:
\begin{displaymath}
  \begin{tikzcd}
    s_i^*(\partial \Delta^n) \ar[r] \ar[d] \ar[rr, "u", bend left = 20] & \partial \Delta^n \ar[r] \ar[d] & Y \ar[d, "p"] \\
    \Delta^{n+1} \ar[r, "s_i"'] \ar[urr, dotted] & \Delta^n \ar[ur, dotted, "f"'] \ar[r] & X
  \end{tikzcd}
\end{displaymath}
and let $f$ be our chosen filler. Note that $\partial \Delta^n = \bigcup d^n_k$ and the pullback of $d_k$ along $s_i$ is $d_k$ if $k \lt i$, $d_i.d_i$ if $k = i$ and $d_{k+1}$ if $k \gt i$; in other words, $s_i^*(\partial \Delta^n)$ is $\Delta^{n+1}$ with the interior and $i$th and $(i+1)$st faces missing. So to find the dotted filler in the diagram above we first need to find the filler on the faces $i$ and $i+1$. So we pull back the left hand arrow along $d_i$ and $d_{i+1}$ and choose our chosen filler, which is, actually, $f$, because $d_i.s_i = d_{i+1}.s_i = 1$. So we are left with the following filling problem:
\begin{displaymath}
  \begin{tikzcd}
    \partial \Delta^{n+1} \ar[d] \ar[r, "u \cup f \cup f"] & Y \ar[d, "p"] \\
    \Delta^{n+1} \ar[ur, dotted] \ar[r] & X.
  \end{tikzcd}
\end{displaymath}
So what we need is that the chosen solution of this problem will be $f \cdot s_i$. (Note that $f \cdot s_i$ will always be a solution. Indeed, we have $f \cdot s_i \cdot d_k = (u \cup f \cup f) \cdot d_k$ for any face $d_k$: it is true on the faces that we added ($d_i$ and $d_{i+1}$), but also on the faces that were already there, because the original picture commutes.)

So, to summarise our discussion, we have:
\begin{theo}{uniformtrivfibintermsofboundaryincl} The following notions of fibred structure are naturally isomorphic:
  \begin{itemize}
    \item To assign to each $p: Y \to X$ all effective trivial Kan fibration structures on it.
    \item To assign to each $p: Y \to X$ all systems of lifts of $p$ against boundary inclusions such that if $f$ is our chosen filler of
 \begin{displaymath}
   \begin{tikzcd}
     \partial \Delta^n \ar[r, "y"] \ar[d] & Y \ar[d, "p"] \\
     \Delta^n \ar[ur, dotted, "f"] \ar[r, "x"'] & X,
   \end{tikzcd}
 \end{displaymath}
 then $f.s_i$ is our chosen solution of the problem
 \begin{displaymath}
   \begin{tikzcd}
     \partial \Delta^{n+1} \ar[d] \ar[r, "y'"] & Y \ar[d, "p"] \\
     \Delta^{n+1} \ar[ur, dotted] \ar[r, "x.s_i"'] & X,
   \end{tikzcd}
 \end{displaymath}
 where $y'$ is the composition $s_i^*(\partial \Delta^n) \to \partial \Delta^n \to Y$ on $s_i^*(\partial \Delta^n)$ and $f$ on $d_i$ and $d_{i+1}$.
\end{itemize}
\end{theo}

\begin{rema}{evenisoofdoublecats} Clearly, both the effective trivial fibrations and the maps which carry a lifting structure against all boundary inclusions satisfying the condition in the theorem above are the vertical maps in concrete double categories. The statement of the theorem can be strengthened to say that these concrete double categories are isomorphic. Indeed, it suffices to check the fullness condition for squares for the obvious forgetful functor of double categories: but that can be shown as in \reflemm{liftstrdetermbyboundarymaps}.
\end{rema}

\subsection{Local character and classical correctness} The characterisation given in \reftheo{uniformtrivfibintermsofboundaryincl} can both be used to show that our notion of an effective trivial Kan fibration is local and that it is classically correct. Let us first discuss local character.

\begin{coro}{trivfiblocal}
 The notion of an effective trivial Kan fibration is a local notion of fibred structure.
\end{coro}
\begin{proof}
  As said, we use the characterisation in \reftheo{uniformtrivfibintermsofboundaryincl}. So assume $p: Y \to X$ is a map of simplicial sets such that any pullback of it along a map $x: \Delta^n \to X$ is an effective trivial Kan fibration. If we have a lifting problem of the form
  \begin{displaymath}
    \begin{tikzcd}
      \partial \Delta^n \ar[r, "y"] \ar[d] & Y \ar[d, "p"] \\
      \Delta^n \ar[r, "x"'] & X,
    \end{tikzcd}
  \end{displaymath}
  then we can decompose it as follows
  \begin{displaymath}
    \begin{tikzcd}
      \partial \Delta^n \ar[r, "y"] \ar[d] & Y_x \ar[r] \ar[d] & Y \ar[d, "p"] \\
      \Delta^n \ar[r,"1"'] & \Delta^n \ar[r, "x"'] & X,
    \end{tikzcd}
  \end{displaymath}
  with a pullback on the right. Since the lifting problem on the right has a solution, by assumption, we also get a filler $f$ for the outer rectangle. Since this definition is forced, we only need to check the condition for the degeneracies. So then we are looking at a situation like this:
  \begin{displaymath}
    \begin{tikzcd}
      \partial \Delta^{n+1} \ar[r] \ar[d] \ar[rr, bend left, "{y'}"]& Y_{x \cdot s_i} \ar[r] \ar[d] & Y_x \ar[r] \ar[d] & Y \ar[d, "p"] \\
      \Delta^{n+1} \ar[r, "1"] & \Delta^{n+1} \ar[r, "s_i"] & \Delta^n \ar[r, "x"] & X
    \end{tikzcd}
  \end{displaymath}
  The lift against $p$ is induced by the left hand square, but, by assumption, it is compatible with the one coming from the composition of the two squares on the left, which is $f.s_i$, as desired.
\end{proof}

It remains to check classical correctness, for which we use the following lemma, whose proof can be found in the appendix (see \refprop{degeneratehornfillerunique}).

\begin{lemm}{atmostonenondegsolution}
 A lifting problem of the form
 \begin{displaymath}
   \begin{tikzcd}
     \partial \Delta^n \ar[r] \ar[d] & X \\
     \Delta^n \ar[ur, dotted]
   \end{tikzcd}
 \end{displaymath}
 has at most one degenerate solution (that is, if both $x_0 \cdot \sigma_0$ and $x_1 \cdot \sigma_1$ fill this triangle with both $\sigma_i$ epis in $\mathbf{\Delta}$ different from the identity, then $x_0 \cdot \sigma_0 = x_1 \cdot \sigma_1$).
\end{lemm}

\begin{theo}{trivfibOKinZFC} Classically, any morphism which has the right lifting property with respect to boundary inclusions $\partial \Delta^n \subseteq \Delta^n$ can be equipped with the structure of an effective trivial Kan fibration.
\end{theo}
\begin{proof}
 Suppose $p: Y \to X$ is a map for which we have a choice of fillers $f_i$ for every lifting problem of the form
 \begin{displaymath}
   \begin{tikzcd}
     \partial \Delta^n \ar[r] \ar[d] & Y \ar[d, "p"] \\
     \Delta^n \ar[ur, dotted, "f_i"] \ar[r] & X.
  \end{tikzcd}
\end{displaymath}
(This uses the axiom of choice, depending on how one reads the assumption.) Then, using the Principle of Excluded Middle and the previous lemma, we may assume that $f_i$ is the unique degenerate solution, if it exists. Then the compatibility condition from \reftheo{uniformtrivfibintermsofboundaryincl} is automatically satisfied, because it says that under certain conditions we should choose the (unique) degenerate solution.
\end{proof}

\begin{rema}{decdegeneranyhelp} The fact that cofibrations are pointwise complemented plays an important role in the work on cubical sets in the proof of the univalence axiom (the axiomatic treatment in \cite{ortonpitts18} makes this quite clear). We expect that we will need this assumption for similar reasons and that is our main reason for including it. An additional benefit is that this ensures that, even in a predicative metatheory like ${\bf CZF}$, the effective cofibrations are a dominance and the effective trivial Kan fibrations are cofibrantly generated by a small double category, as we have seen in this section.

In their work Gambino, Henry, Sattler and Szumilo \cite{gambinohenry19,gambinoetal19,henry19} work with a stronger notion of cofibration, where the question whether an $n$-simplex is degenerate or not is decidable for those $n$-simplices outside the image of the cofibration. For us this choice would have several unfortunate consequences: one consequence will be that HDRs will not be cofibrations and effective trivial Kan fibrations might not be effective Kan fibrations (see \refcoro{trivFibareUniKanFib} and \refprop{onecofibr} below). Additionally, it will have the consequence that not every simplicial set is cofibrant.

On the other hand, one may wonder whether the previous result can be made more constructive when degeneracy is decidable (in $Y$ for instance). We fail to see how it would, and for that reason the relationship with the work of Gambino et al.~is far from clear to us.
\end{rema}

\section{Simplicial sets as a symmetric Moore category}\index{Moore structure!in simplicial sets|textbf}

The purpose of this section is to show that the category of simplicial sets can be equipped with symmetric Moore structure. As we already mentioned in the introduction, the structure that we choose was already defined in the paper by Van den Berg \& Garner \cite{vdBerg-Garner} using the notion of simplicial Moore paths. However, since the notion of a symmetric Moore structure that we work with in this paper is stronger than that of a path object category as in \cite{vdBerg-Garner}, we have to verify some additional equations. For checking that these hold, we use a new characterisation of the Moore path functor $M$, namely, as a polynomial functor. For that reason we will first give a brief introduction to the theory of polynomial functors. Then we will define $M$ as a polynomial functor and use this definition to check that it is a symmetric Moore structure in the sense of this paper. Finally, we will prove that our new definition of $M$ is equivalent to the one given in \cite{vdBerg-Garner}. We will also isolate an interesting two-sided Moore structure on simplicial sets, which will give us effective left and right fibrations.

\subsection{Polynomial yoga} We start by recapping some general facts about polynomial functors. (Some useful references are: \cite{VonGlehnPhD, AbbottAG03, Gambino-Kock}.) Throughout this section we will work in a category \ct{E} which is locally cartesian closed and has finite colimits. (We assume all this structure is chosen.)

\begin{defi}{polycat}
A \emph{polynomial}\index{polynomial} in \ct{E} is a morphism $f: B \to A$ in \ct{E}. For reasons that will become clear soon, we will also write such morphisms as $(B_a)_{a \in A}$. We will refer to $A$ as the \emph{base} and $B_a$ as a \emph{fibre}. A \emph{morphism $\alpha$ of polynomials} from $f: B \to A$ to $g: D \to C$ is a pair $(\alpha^+, \alpha^-)$ consisting of a morphism $\alpha^+: A \to C$ and a morphism $\alpha^-: A \times_C D \to B$ making
\begin{center}
\begin{tikzcd}
 B \ar[dr, "f"'] & A \times_C D \ar[d] \ar[r] \ar[l, "\alpha^-"'] & D \ar[d, "g"] \\
 & A \ar[r, "\alpha^+"'] & C
\end{tikzcd}
\end{center}
commute. We will refer to $\alpha^+$ as the \emph{positive} or \emph{forward direction} of the morphism $\alpha$, while $\alpha^-$ is its \emph{negative} or \emph{backward direction}. So, basically, a morphism $\alpha$ from $(B_a)_{a \in A}$ to $(D_c)_{c \in C})$ consists of a map $\alpha^+: A \to C$ and a family of morphisms $(\alpha^{-}_a: D_{\alpha^+(a)} \to B_a)_{a \in A}$. With this notation, composition of morphisms is given by
\begin{eqnarray*}
(\beta^+: C \to E, \beta^-_c: F_{\gamma^+(c)} \to D_c) \circ (\alpha^+: A \to C, \alpha^-_a: D_{\alpha^+(a)} \to B_a)  & = \\ (\beta^+.\alpha^+: A \to E, \alpha^-_a.\beta^-_{\alpha^+(a)}: F_{\beta^+(\alpha^+(a))} \to D_{\alpha^+(a)} \to B_a).
\end{eqnarray*}
The result is a category which we will denote ${\rm Poly}(\ct{E})$.
\end{defi}

In addition, let us write ${\rm FEnd}(\ct{E})$ for the category of fibred endofunctors on \ct{E} and fibred natural transformation between them (with respect to the codomain fibration on \ct{E}; see \cite{streicher18}). There is a functor $P: {\rm Poly}(\ct{E}) \to {\rm FEnd}(\ct{E})$ sending a polynomial $f: B \to A$ to its associated \emph{polynomial functor}\index{polynomial functor} $P_f$:
\begin{center}
\begin{tikzcd}
P_f: \ct{E} \ar[r, "B^*"] & \ct{E}/B \ar[r, "\Pi_f"] & \ct{E}/A \ar[r, "\Sigma_A"] & \ct{E}.
\end{tikzcd}
\end{center}
Written differently:
\[ P_f(X) = \sum_{a \in A} \prod_{b \in B_a} X = \sum_{a \in A} X^{B_a} =  \{ (a \in A, t: B_a \to X) \}.\]
On morphisms $\alpha: (f: B \to A) \to (g: D \to C)$, the functor $P$ acts as follows:
\[ P(\alpha)_X: P_f(X) \to P_g(X): (a \in A, t: B_a \to X) \mapsto (\alpha^+(a) \in C, t.\alpha^-_a: D_{\alpha^+(a)} \to B_a \to X). \]

Note that $P(\alpha)$ is a cartesian natural transformation (meaning: all naturality squares are pullbacks) if $\alpha^-$ is iso.

The following proposition will not be used in this paper, but explains the choice of morphisms in the category ${\rm Poly}(\ct{E})$.

\begin{prop}{polyasendofunctor}
The functor $P: {\rm Poly}(\ct{E}) \to {\rm FEnd}(\ct{E})$ is full and faithful.
\end{prop}
\begin{proof}
See \cite[Theorem 3.4]{AbbottAG03}.
\end{proof}

The category ${\rm FEnd}(\ct{E})$ has finite limits and these are inherited by ${\rm Poly}(\ct{E})$. The terminal object is the polynomial $0 \to 1$. The product of $(f: B \to A) \times (g: D \to C)$ is
\[ [1_A \times f, g \times 1_B]: A \times D + B \times C \to A \times C. \]
In other words, it has $A \times C$ as base, with fibre $D_a + B_c$ over $(a,c) \in A \times C$. The pullback of $\delta: (g: D \to C) \to (f: B \to A)$ and $\varphi: (h: F \to E) \to (f: B \to A)$ has $C \times_A E$ as base, with the fibre $P_{(c,e)}$ over $(c, e)$ being the pushout:
\begin{center}
\begin{tikzcd}
B_{\delta^+(c)} = B_{\varphi^+(e)} \ar[r, "\delta^-_c"] \ar[d, "\varphi^-_e"'] & D_c \ar[d] \\
F_e \ar[r] & P_{(c,e)}.
\end{tikzcd}
\end{center}

In addition, the category ${\rm FEnd}(\ct{E})$ carries a (non-symmetric) monoidal structure given by composition: $F \otimes G = F \circ G$. This is inherited by ${\rm Poly}(\ct{E})$ as well: indeed, it carries a monoidal structure as follows:
\[ (B_a)_{a \in A} \otimes (D_c)_{c \in C} = \{(b \in B_a, d \in D_{t(b)}) \}_{(a \in A, t: B_a \to C)}. \]
Imagine that we have a morphism $\alpha: (B_a)_{a \in A} \to (B'_{a'})_{a' \in A'}$ and a morphism $\gamma: (D_c)_{c \in C} \to (D'_{c'})_{c' \in C'}$ then $\alpha \otimes \gamma = \eta$ with:
\begin{eqnarray*}
\eta^+(a \in A, t: B_a \to C) & = & (\alpha^+(a) \in A',  \gamma^+.t.\alpha^-_a: B'_{\alpha^+(a)} \to B_a \to C \to C') \\
\eta^-_{(a \in A, t: B_a \to C)}(b' \in B'_{\alpha^+(a)}, d' \in D'_{(\gamma^+.t.\alpha^-_a)(b')}) & = & (\alpha^-_a(b') \in B_a, \gamma^-_{t(\alpha^-_a(b'))}(d') \in D_{t(\alpha^-_a(b'))} )
\end{eqnarray*}
The monoidal unit $I$ is $1 \to 1$ (corresponding to the identity functor).

\begin{prop}{tensorcommwithpbk}
The tensor $\otimes$ on ${\rm Poly}(\ct{E})$ preserves pullbacks in both coordinates.
\end{prop}
\begin{proof}
  Because pullbacks in functor categories are computed pointwise and polynomial functors preserve pullbacks.
\end{proof}

We will also be interested in comonoids for this tensor: so, this consists of an object $M$ in ${\rm Poly}(\ct{E})$ together with maps $\varepsilon: M \to I$ (the counit) and $\delta: M \to M \otimes M$ (the comultiplication) making the following
\begin{displaymath}
\begin{array}{cc}
\begin{tikzcd}
& M \ar[dr,"1_M"] \ar[dl,"1_M"'] \ar[d,"\delta"] \\
M & M \otimes M \ar[r,"1_M \otimes \varepsilon"'] \ar[l,"\varepsilon \otimes 1_M"] & M
\end{tikzcd} &
\begin{tikzcd}
M \ar[r, "\delta"] \ar[d, "\delta"'] & M \otimes M \ar[d, "\delta \otimes 1_M"] \\
M \otimes M \ar[r, "1_M \otimes \delta"'] & M \otimes M \otimes M
\end{tikzcd}
\end{array}
\end{displaymath}
commute. We will refer to such objects as \emph{polynomials comonads}.\index{polynomial comonad}

Every internal category in \ct{E} determines such a comonad. Indeed, let $\mathbb{C}$ be an internal category and ${\rm cod}: \mathbb{C}_1 \to \mathbb{C}_0$ be the codomain map. Then there is a counit:
\begin{center}
\begin{tikzcd}
\mathbb{C}_1 \ar[dr, "{\rm cod}"']  & \mathbb{C}_0 \ar[l, "\varepsilon^- = {\rm id}"'] \ar[d] \ar[r] & 1 \ar[d] \\
& \mathbb{C}_0 \ar[r, "\varepsilon^+ = !"'] & 1
\end{tikzcd}
\end{center}
and a comultiplication $(\delta^+,\delta^-): {\rm cod} \to {\rm cod} \otimes {\rm cod}$ with $\delta^+: \mathbb{C}_0 \to \sum_{C \in \mathbb{C}_0} \mathbb{C}_0^{{\rm cod}^{-1}(C)}$ given by sending an object $C$ in $\mathbb{C}$ to the pair $(C, \lambda \alpha \in {\rm cod}^{-1}(C).{\rm dom}(\alpha))$, whilst $(\delta^-)_C$ sends a pair $(\alpha: D \to C, \beta: E \to D)$ to $\alpha.\beta$.

This construction has a converse: indeed, one can show that every polynomial comonad is induced in this way by an internal category (see \cite{AhmanU17}).

Note that such a polynomial comonad is in particular a comonad (in the usual sense) on \ct{E} and that the coalgebras for this comonad are precisely the internal presheaves on $\mathbb{C}$ in \ct{E}. Note also that such a polynomial comonad is automatically strong. Indeed, because a polynomial functor preserves pullbacks, we can think of a strength on $P_{\rm cod}$ as a natural transformation $\alpha_X: X \times P_{\rm cod}(1) \to P_{\rm cod}(X)$, or, in other words, as a map of polynomials $(1: \mathbb{C}_0 \to \mathbb{C}_0) \cong (1 \to 1) \times (0 \to \mathbb{C}_0) \to ({\rm cod}: \mathbb{C}_1 \to \mathbb{C}_0)$. And there is a canonical such map:
\begin{center}
\begin{tikzcd}
\mathbb{C}_0 \ar[dr, "1"']  & \mathbb{C}_1 \ar[l, "\alpha^- = {\rm cod}"'] \ar[d, "{\rm cod}"] \ar[r, "1"] & \mathbb{C}_1 \ar[d, "{\rm cod}"] \\
& \mathbb{C}_0 \ar[r, "\alpha^+ = 1"'] & \mathbb{C}_0
\end{tikzcd}
\end{center}
One readily checks this is indeed a strength and that with respect to this strength the induced comonad is strong.

\subsection{A simplicial poset of traversals}

Let us define an internal poset $\mathbb{T}$ in simplicial sets.

The object of objects $\mathbb{T}_0$ has as its $n$-simplices the \emph{$n$-dimensional traversals}. An $n$-dimensional traversal \index{traversal|textbf} is a finite sequence of elements from $[n] \times \{ +, - \}$, that is, a function $\theta: \{ 1, \ldots, l \} \to [n] \times \{ +, - \}$ for some $l \in \mathbb{N}$ (including the empty traversal for $l = 0$). A good way to picture a traversal is as follows. An $n$-dimensional traversal is like a zigzag:
\diag{ \bullet & \ar[l]_{p_1} \bullet \ar[r]^{p_2} & \bullet \ar[r]^{p_3} & \bullet & \bullet \ar[l]_{p_4} \ar[r]^{p_5} & \bullet, }
a (possibly empty) sequence of edges pointing either to the left (-) or right (+), with a label $p_i \in [n]$. The collection of such traversals is a simplicial set: the face map $d_i$ acts on such a traversal by removing all the edges labelled with $i$ and relabelling the other edges (meaning: if an edge is labelled with $j \gt i$, replace that label by $j-1$). The degeneracy $s_i$ acts on such a traversal by duplicating edges labelled with $i$ (with the copies pointing in the same direction as the original edge) and labelling the first copy $i+1$ and the second $i$ in case the edge points to the right, and labelling the first copy $i$ and the second $i+1$ if the edge point to the left. Other edges are relabelled accordingly (meaning: if an edge was labelled $j \gt i$, then it now has the label $j+1$). In general, the action by some $\alpha: [m] \to [n]$ on such a traversal $\theta$ is given as follows: if the label of some edge is $i$, then replace it by $\# \alpha^{-1}(i)$ many edges pointing in the same direction as the original edge, labelled by the elements of $\alpha^{-1}(i)$ in decreasing order if the edge points to the right and in increasing order if the edge points to the left. In short, $\theta \cdot \alpha$ is the unique map fitting into a pullback square
\begin{displaymath}
 \begin{tikzcd}
    \{ 1, \ldots, l \} \ar[r, "\theta \cdot \alpha"] \ar[d, "v"] & {[m]} \times \{ +, - \} \ar[d, "\alpha \times 1"] \\
   \{ 1, \ldots, k \} \ar[r, "\theta"] & {[n]} \times \{ +, - \}
 \end{tikzcd}
\end{displaymath}
with ${\rm proj}_{[m]}.(\theta \cdot \alpha): \{ 1, \ldots, l \} \to [m]$ decreasing on those fibres $v^{-1}(i)$ with $\theta(i)$ positive, and increasing on those fibres $v^{-1}(i)$ with $\theta(i)$ negative.

A \emph{position} in an $n$-dimensional traversal $\theta: \{ 1, \ldots, l \} \to [n] \times \{ +, - \}$ is a choice of one of the vertices: formally, it is an element $p \in \{ 0, 1, \ldots, l \}$. The elements of $(\mathbb{T}_1)_n$ are pairs consisting of an $n$-dimensional traversal $\theta$ together with a position in this traversal (a \emph{pointed traversal}\index{pointed traversal}). The action of $\alpha$ on the traversals is as before, while it acts on the choice of vertex as follows: if $\theta' = \theta \cdot \alpha$, and $v$ is some vertex in $\theta$, then we choose that vertex in $\theta'$ which is the rightmost vertex in $\theta'$ which is either the source or target of an edge which is a copy of an edge which was to the left of $v$ (choosing the leftmost vertex if no such edge exists).

There are two maps ${\rm cod}, {\rm dom}: \mathbb{T}_1 \to \mathbb{T}_0$ with ${\rm cod}$ being the obvious forgetful map (forgetting the choice of position), while ${\rm dom}$ removes the part of the traversal \emph{before} the position. That means that we think of $\mathbb{T}$ as a simplicial poset with the final segment ordering ($\theta_0 \leq \theta_1$ if $\theta_0$ is a final segment of $\theta_1$: in that case, there is a position $p$ in $\theta_1$ such that after that point we see $\theta_0$, and $p$ can be thought of as the morphism from $\theta_0$ to $\theta_1$).

To see that this is an internal poset, note that there is a map ${\rm id}: \mathbb{T}_0 \to \mathbb{T}_1$ given by choosing the position at the start of the traversal. Finally, we need a map
\[ {\rm comp}: \mathbb{T}_1 \times_{\mathbb{T}_0} \mathbb{T}_1 \to \mathbb{T}_1.\]
That is, we start with pointed traversals $(\theta_1,p_1)$ and $(\theta_0,p_0)$ such that $\theta_0$ is the final segment we obtain from $\theta_1$ by removing everything before position $p_1$. Then comp takes $\theta_1$ with position $p_0$ (which is a position in $\theta_0$, and, because $\theta_0$ is a final segment of $\theta_1$, in $\theta_1$ as well).

In view of the correspondence between internal categories and polynomial comonads, this internal category induces a polynomial comonad, whose counit we denote $s: {\rm cod} \to I$ and whose comultiplication we call $\Gamma: {\rm cod} \to {\rm cod} \otimes {\rm cod}$. Therefore the \emph{simplicial Moore functor} $M = P_{\rm cod}$ \index{simplicial Moore functor} defined as
\[ (MX)_n = \sum_{\theta \in (\mathbb{T}_0)_n} X^{{\rm cod}^{-1}(\theta)} \] carries the structure of a strong comonad with counit $s: M \Rightarrow 1$ and comultiplication $\Gamma: M \Rightarrow MM$.

\begin{rema}{onM1andT0}
  Note that $M1 \cong \mathbb{T}_0$. In fact, the object $\mathbb{T}_0$ was introduced as $M1$ in \cite{vdBerg-Garner}.
\end{rema}

\subsection{Simplicial Moore paths}

$M$ has more structure: in fact, we have an internal category in ${\rm Poly}(\ct{E})$. To see this, note that we can also equip $\mathbb{T}_0$ with the initial segment ordering. In that case, we take the same codomain map, but now as domain map we take ${\rm dom}^*: \mathbb{T}_1 \to \mathbb{T}_0$ which removes the part of the traversal \emph{after} the chosen position. In addition, we have a map ${\rm id}^*: \mathbb{T}_0 \to \mathbb{T}_1$ which chooses the endpoint of the given traversal as its chosen position, as well as an appropriate composition
\[  {\rm comp}^*: \mathbb{T}_1 \times_{\mathbb{T}_0} \mathbb{T}_1 \to \mathbb{T}_1.\]
This means that $M$ carries a second strong comonad structure with counit $t: M \Rightarrow 1$ and comultiplication $\Gamma^*: M \Rightarrow MM$.

Note that with either ordering, the poset $\mathbb{T}$ has an initial object $0: 1 \to \mathbb{T}_0$ (the unique traversal of length 0), and with the final segment ordering, the map ${\rm id}^*: \mathbb{T}_0 \to \mathbb{T}_1$ points to the unique map from the initial traversal to the given traversal (and similarly for ${\rm id}$ and the initial segment ordering). This means that we also have a map $r: I \to {\rm cod}$ given by:
\begin{displaymath}
\begin{tikzcd}
1 \ar[dr, "1"'] & 1 \ar[d] \ar[r] \ar[l, "r^- = !"'] & \mathbb{T}_1 \ar[d, "{\rm cod}"] \\
& 1 \ar[r, "r^+ = 0"'] & \mathbb{T}_0
\end{tikzcd}
\end{displaymath}
Note that because $r^-$ iso, the natural transformation induced by $r$ is cartesian.

At this point one readily checks all the axioms for a two-sided Moore structure which do not involve the multiplication $\mu$. In fact, all of these follow simply from the fact that we are working in an internal category with an initial object 0 with the property that the only map $C \to 0$ is the identity on 0.
\begin{enumerate}
\item The equation $s.r = t.r = 1$ follows immediately from the fact that there is only polynomial map $I \to I$.
\item $r$ is strong: $\alpha.(1 \times r) = r.p_1: X \times 1 \to MX$. In this we have to compare two maps $(1 \to 1) \times (0 \to 1) \to ({\rm cod}: \mathbb{T}_1 \to \mathbb{T}_0)$. In the forwards direction they are both $0: 1 \to \mathbb{T}_0$, while in the backwards direction they are both equal as well, because they both have codomain 1.
\item $\Gamma.r = rM.r$, or: $\Gamma.r = (r \otimes 1_{\rm cod}).r$. We have $(\Gamma.r)^+ = \Gamma^+.r^+ = (0, \lambda \alpha: C \to 0.C)$, while $(r \otimes 1_{\rm cod}.r)^+ = (0, \lambda: C \to 0.0)$, which coincide, while in the backwards direction we again have to compare two maps which terminate in $1$: so these are again equal.
\item $tM.\Gamma = r.t$, or $(t \otimes 1_{\rm cod}).\Gamma = r.t$. Note $(t \otimes 1_{\rm cod}.\Gamma)^+(C) = {\rm dom}(0 \to C) = 0$, while $(r.t)^+(C) = 0$. Also, $(r.t)^-_C(\alpha: D \to 0) = !: 0 \to C$ and $((t \otimes 1_{\rm cod}).\Gamma)_C^-(\alpha: D \to 0) =  {\rm comp}(!: 0 \to C, \alpha: D \to 0)$.
\item $Mt.\Gamma = \alpha.(t,M!): {\rm cod} \to {\rm cod}$. Here $(t,M!): {\rm cod} \to (1: \mathbb{T}_0 \to \mathbb{T}_0)$ is given by $(1_{\mathbb{T}_0}, {\rm id}^*)$, so that the right hand side is $(1,\lambda \alpha: D \to C. !: 0 \to C)$. The left hand side, however, is given by $(Mt.\Gamma)^+(C) = (1_{\rm cod} \otimes t)(C, \lambda \alpha: D \to C.D) = C$, while $(Mt.\Gamma)_C^-(\alpha: D \to C) = \Gamma^-.(1_{\rm cod} \otimes t)^-(\alpha: D \to C) = {\rm comp}(\alpha: D \to C, !: 0 \to D) = !: 0 \to C$, as desired.
\item Equations similar to those in (3--5) have to be verified for $\Gamma^*$ as well: but since also with the initial segment ordering, $\mathbb{T}$ has a strong initial object, the same arguments will work.
\end{enumerate}
What is still needed, then, is to define $\mu_X: MX^t \times_X^s MX \to MX$ and to verify that it satisfies all the expected equations.

Using the formula for computing pullbacks of polynomials, we see that in order to define $\mu$ we need maps
\begin{displaymath}
 \begin{tikzcd}
  \mathbb{T}_1 \times \mathbb{T}_0 \sqcup_{\mathbb{T}_0 \times \mathbb{T}_0} \mathbb{T}_0 \times \mathbb{T}_1 \ar[dr,"{[{\rm cod} \times 1, 1 \times {\rm cod}]}"']  & \bullet \ar[r] \ar[d] \ar[l, "\mu^-"] & \mathbb{T}_1 \ar[d, "{\rm cod}"] \\
  & \mathbb{T}_0 \times \mathbb{T}_0 \ar[r,"\mu^+"] & \mathbb{T}_0
 \end{tikzcd}
\end{displaymath}
Note that the fibre over $(\theta_0,\theta_1)$ of the map on the left is the collection of positions in $\theta_0$ and $\theta_1$, with the final position in $\theta_0$ identified with the initial position in $\theta_1$. So what we can do is define $\mu^+(\theta_0,\theta_1) = \theta_0 * \theta_1$, the concatenation of the two sequences with $\theta_0$ put before $\theta_1$. Since the positions in $\theta_0 * \theta_1$ are precisely the positions in either $\theta_0$ or $\theta_1$, with the final position in $\theta_0$ coinciding with the initial position in $\theta_1$, we have a pullback square
\begin{displaymath}
 \begin{tikzcd}
  \mathbb{T}_1 \times \mathbb{T}_0 \sqcup_{\mathbb{T}_0 \times \mathbb{T}_0} \mathbb{T}_0 \times \mathbb{T}_1 \ar[d,"{[{\rm cod} \times 1, 1 \times {\rm cod}]}"']  \ar[r, "\mu^*"] & \mathbb{T}_1 \ar[d, "{\rm cod}"] \\
  \mathbb{T}_0 \times \mathbb{T}_0 \ar[r,"\mu^+"] & \mathbb{T}_0.
 \end{tikzcd}
\end{displaymath}
So we can choose $\mu^-$ to be an isomorphism and $\mu$ will be a cartesian natural transformation.

We will leave it to the reader to verify that $\mu$ is strong and combines with $r,s,t$ to yield a category structure, which is both left and right cancellative. The most difficult axioms to check are the distributive laws and the sandwich equation (see \refdefi{pathobjcat} and \refdefi{twosidedpathobjcat}, respectively), which we will discuss here in some detail, also because they were not part of \cite{vdBerg-Garner}.

\begin{lemm}{distributivelawOK}
 The distributive law \[ \Gamma.\mu = \mu.(M\mu.\nu_X.(\Gamma.p_1,\alpha_{MX}.(p_2,M!.p_1)),\Gamma.p_2): MX \times_X MX \to MMX \] holds, as does the corresponding law for $\Gamma^*$.
\end{lemm}
\begin{proof}
 We only show the distributive law for $\Gamma$ as the corresponding statement for $\Gamma^*$ is proved similarly.

 We have to compare two maps \[ {\rm cod} \times_I {\rm cod} \to {\rm cod} \otimes {\rm cod}. \] The left hand side goes via ${\rm cod}$ and in the positive direction sends $(\theta_0, \theta_1)$ to $(\theta_0 * \theta_1, \lambda p \in {\rm cod}^{-1}(\theta_0 * \theta_1).{\rm dom}(p))$, and in the negative direction sends a pair of positions $p_0$ in $\theta_0 * \theta_1$ and $p_1$ in ${\rm dom}(p_0)$ to the position corresponding to $p_1$ in either $\theta_0$ or $\theta_1$.

 Let us now try to decompose the right hand side. The map \[ \alpha M.(p_2, M!.p_1): {\rm cod} \times_I {\rm cod} \to {\rm cod} \otimes {\rm cod} \] is in the forwards direction a map $\mathbb{T}_0 \times \mathbb{T}_0 \to P_{\rm cod}(\mathbb{T}_0)$ which sends $(\theta_0, \theta_1)$ to $(\theta_0, \lambda p. \theta_1)$, while the backwards direction sends a pair of positions $(p_0 \in \theta_0, p_1 \in \theta_1)$ to $p_1$.

 Then the map $M\mu.\nu.(\Gamma.p_1,\alpha.(p_2,M!.p_1))$ can be seen as a composition:
 \begin{displaymath}
   \begin{tikzcd}
     {\rm cod} \times_I {\rm cod} \ar[r] & ({\rm cod} \otimes {\rm cod}) \times_{{\rm cod} \otimes I} ({\rm cod} \otimes {\rm cod}) \ar[d, "\cong"] \\ & {\rm cod} \otimes ({\rm cod} \times_I {\rm cod}) \ar[r, "1_{\rm cod} \otimes \mu"] & {\rm cod} \otimes {\rm cod}
   \end{tikzcd}
 \end{displaymath}
 where in the forwards directions these maps send $(\theta_0,\theta_1)$ first to \[ ((\theta_0, \lambda p \in {\rm cod}^{-1}(\theta_0). {\rm dom}(p)), (\theta_0, \lambda p \in {\rm cod}^{-1}(\theta_0).\theta_1), \] which gets rewritten to
 $(\theta_0, \lambda p \in {\rm cod}^{-1}(\theta_0).({\rm dom}(p),\theta_1),$
and then sent to \[ (\theta_0, \lambda p \in {\rm cod}^{-1}(\theta_0). {\rm dom}(p) * \theta_1). \] In the backwards direction it sends a pair of positions in $p_0$ in $\theta_0$ and $p_1$ in ${\rm dom}(p_0) * \theta_1$ first to the pair $p_0$ and the position corresponding to $p_1$ either in ${\rm dom}(p_0)$ or $\theta_1$, and then to the position corresponding to $p_1$ in $\theta_0$ if it belongs to ${\rm dom}(p_0)$ or to the position corresponding to $p_1$ in $\theta_1$ if it belongs to $\theta_1$. In short, it sends $p_0$ and $p_1$ to the position corresponding to $p_1$ in either $\theta_0$ or $\theta_1$.

In the final step we look at the whole right hand side as a composition
\begin{displaymath}
  \begin{tikzcd}
    {\rm cod} \times_1 {\rm cod} \ar[r] & ({\rm cod} \otimes {\rm cod}) \times_{1 \otimes {\rm cod}} ({\rm cod} \otimes {\rm cod}) \ar[d, "\cong"] \\ & ({\rm cod} \times_1 {\rm cod}) \otimes {\rm cod} \ar[r] & {\rm cod} \otimes {\rm cod}.
  \end{tikzcd}
\end{displaymath}
In the positive direction this takes $(\theta_0, \theta_1)$ first to \[ ((\theta_0, \lambda p \in {\rm cod}^{-1}(\theta_0). {\rm dom}(p) * \theta_1), (\theta_1, \lambda p \in {\rm cod}^{-1}(\theta_1).{\rm dom}(p))), \]
and then to $(\theta_0 * \theta_1, \lambda p \in {\rm cod}^{-1}(\theta_0 * \theta_1).{\rm dom}(p))$, as before. In the backwards direction we are given a pair consisting of a position $p_0$ in $\theta_0 * \theta_1$ and a position $p_1$ in ${\rm dom}(p_0)$ and we start by making a case distinction on whether $p_0$ lies in $\theta_0$ or $\theta_1$. If it lies in $\theta_1$, the pair gets mapped to the position $p_1$ in $\theta_1$. If it lies in $\theta_0$, the pair gets mapped to the position corresponding to $p_1$ in either $\theta_0$ or $\theta_1$. So in either case it gets mapped to the position corresponding to $p_1$ in either $\theta_0$ or $\theta_1$, as before.
\end{proof}

\begin{lemm}{faberequationOK}
The sandwich equation $M\mu.\nu.(\Gamma^*,\Gamma) = \alpha M.(1,M!): M \to MM$ holds.
\end{lemm}
\begin{proof}
 We have to compare two morphisms ${\rm cod} \to {\rm cod} \otimes {\rm cod}$. The right hand side can be seen as a composition:
 \begin{displaymath}
  \begin{tikzcd}
    {\rm cod} \ar[r] & {\rm cod} \times 1_{\mathbb{T}_0} \ar[r, "\cong"] & 1_{\mathbb{T}_0} \otimes {\rm cod} \ar[r, "\alpha \otimes 1_{\rm cod}"] & {\rm cod} \otimes {\rm cod}.
  \end{tikzcd}
 \end{displaymath}
In the positive direction these maps are the diagonal $\mathbb{T}_0 \to \mathbb{T}_0 \times \mathbb{T}_0$, a map $\mathbb{T}_0 \times \mathbb{T}_0 \to \mathbb{T}_0 \times \mathbb{T}_0$ swapping the two arguments and a map $\mathbb{T}_0 \times \mathbb{T}_0 \to P_{\rm cod}(\mathbb{T}_0)$ sending $(\theta_0,\theta_1)$ to $(\theta_0,\lambda p.\theta_1)$. In short, in the positive direction this maps sends $\theta$ to $(\theta, \lambda p.\theta)$. In the negative direction, a pair of positions $(p, p')$ in $\theta$ is sent to $p'$.

 The left hand side can be seen as a composition
  \begin{displaymath}
   \begin{tikzcd}
    {\rm cod} \ar[r] & ({\rm cod} \otimes {\rm cod}) \times_{\rm cod} ({\rm cod} \otimes {\rm cod}) \ar[d, "\cong"] \\ & {\rm cod} \otimes ({\rm cod} \times_{I} {\rm cod}) \ar[r, "1_{\rm cod} \otimes \mu"] & {\rm cod} \otimes {\rm cod}.
   \end{tikzcd}
  \end{displaymath}
  In the positive direction this first sends $\theta$ to $(\theta, {\rm dom}^*, {\rm dom})$ and then it sends $(\theta, t, t')$ to $(\theta, \lambda p.t(p)*t'(p))$. So the composition is $(\theta, \lambda p.\theta)$: the reason is that ${\rm dom}^*$ removes the part after the position, dom removes the part before the position, and * concatenates the results: so we just get the original traversal back. In the backwards direction a pair of positions $(p, p')$ is first sent to the position corresponding to $p'$ in either the part before or after $p$, and then it is sent to the corresponding position in the whole traversal. In short, it is sent to $p'$.
\end{proof}

This finishes the verification of the axioms for a two-sided Moore structure. Note that all the proofs that we have given so far would still work if we restricted the traversals in $\mathbb{T}_0$ to those which only move towards the right (that is, those traversals $\theta: \{ 1, \ldots, l \} \to [n] \times \{ +, - \}$ for which $\theta(i)$ for any $i \in \{ 1, \ldots, l \}$ is always of the form $(k,+)$ for some $k \in [n]$). We will refer to the version of $M$ that we get in this way as $M_+$. Of course, similar remarks apply if we restrict the traversals to those that only move to the left; the version of $M$ that we would have obtained in that way will be referred to as $M_-$.

\begin{theo}{twosidedpathobjcatwithM+}
  The endofunctors $M$, $M_+$ and $M_-$ equip the category of simplicial sets with three distinct two-sided Moore structures.
\end{theo}

It remains to check that $M$ equips the category of simplicial sets with the structure of a symmetric Moore category. This means that we should be able to construct a twist map\index{twist map!for simplicial Moore structure} $\tau$ if we work with two different orientations. Indeed, in that case there is a map of polynomials
\begin{displaymath}
  \begin{tikzcd}
    \mathbb{T}_1 \ar[dr, "{\rm cod}"'] & \bullet \ar[d] \ar[r] \ar[l, "\tau^-"', "\cong"] & \mathbb{T}_1 \ar[d, "{\rm cod}"] \\
    & \mathbb{T}_0 \ar[r, "\tau^+"'] & \mathbb{T}_0
  \end{tikzcd}
\end{displaymath}
with $\tau^+$ sending a traversal $\theta: \{ 1, \ldots, l \} \to [n] \times \{ +, - \}$ to a traversal $\tau^+(\theta)$ with the same length $l$ and $\tau^+(\theta)(i) = \tau(\theta(l + 1 -i))$, where $\tau(k,+) = (k,-)$ and $\tau(k,-) = (k, +)$ for any $k \in [n]$. So $\tau^+$ reverses the order and orientation of the traversal. Finally, $\tau^-$ sends a position $p$ in such a traversal to the position $l - p$. Note that $\tau^-$ is an isomorphism and $\tau$ is a cartesian natural transformation.

Most of the equations for $\tau$ are easy to verify, except perhaps for $\Gamma^*  = \tau M.M\tau.\Gamma.\tau$, which is equivalent to $\Gamma^*.\tau = \tau M.M \tau.\Gamma$, or $\Gamma^*.\tau = (\tau \otimes \tau).\Gamma$. In the positive direction we have that $\Gamma^*.\tau$ is the function sending a traversal $\theta$ to the pair where the first component is this traversal reversed, while the second component is the function which takes a position in this reversed traversal and removes the part in this reversed traversal after this position. In the backwards direction this takes a position in the reversed traversal and a position in the reversed traversal before the first position and produces the position corresponding to the second position in the original traversal. The map $\tau \otimes \tau$ is the function which in the positive direction takes a pair $(\theta, t: {\rm cod}^{-1} \to \mathbb{T}_0)$ to $(\tau^+(\theta), \tau^+ \circ t \circ \tau^-)$, so if $t$ is the domain function (removing the part of the traversal before the given position), this agrees with the other function in the positive direction. In the backwards direction, $(\tau \otimes \tau).\Gamma$ takes a position in the reversed traversal and a position in the reversed traversal before the given position, first reverses both and then takes the position corresponding to the second position in the original traversal, as before.

This means that we have the following result, as promised:
\begin{theo}{Mpathobjcatonsimpsets}
  The endofunctor $M$ equips the category of simplicial sets with the structure of a symmetric Moore category.
\end{theo}

In view of the results of Section 4.2, \reftheo{twosidedpathobjcatwithM+} implies that there are \emph{a priori} six AWFSs on the category of simplicial sets. However, up to isomorphism, there are only three, because
\begin{itemize}
\item the twist map $\tau$ induces an isomorphism between the AWFS determined by $(M_+, \Gamma_+, s)$ and the one determined by $(M_-,\Gamma_-^*, t)$.
\item the twist map $\tau$ induces an isomorphism between the AWFS determined by $(M_+, \Gamma^*_+, t)$ and the one determined by $(M_-,\Gamma_-, s)$.
\item the Moore structure determined by $M$ is symmetric, so the AWFS determined by $(M, \Gamma, s)$ is isomorphic to the one determined by $(M,\Gamma^*, t)$.
\end{itemize}

In view of this we will make the following definition:

\begin{defi}{naivefibrinsimpset}
We will refer to the algebras for the monads of the three AWFSs above as \emph{naive right fibrations}\index{naive right fibration}, \emph{naive left fibrations}\index{naive left fibration} and \emph{naive Kan fibrations}\index{naive Kan fibration}, respectively. We will refer to the coalgebras for the comonad of the third AWFS as \emph{HDRs}\index{HDR}.
\end{defi}

\subsection{Geometric realization of a traversal} In this section we have defined $M$ as the polynomial functor associated to ${\rm cod}: \mathbb{T}_1 \to \mathbb{T}_0$. In \cite{vdBerg-Garner}, the Moore path functor was defined differently (as a parametric right adjoint). The main goal of this subsection is to show that the two descriptions are equivalent. Our proof here will be fairly combinatorial; a more conceptual proof can be found in the second author's PhD thesis \cite{faber19} (see \refrema{compwithEricsPhDthesis} below).

Our combinatorial argument makes us of the ``geometric realization'' of a traversal, a construction which can already be found in \cite{vdBerg-Garner}, and will also play an important role in the later sections.

\begin{defi}{geometricrealizationoftraversal} For an element of the form $(k, \pm) \in [n] \times \{ +, - \}$, let us define: $(k,+)^s = k+1$, $(k,+)^ t = k, (k,-)^s = k, (k,-)^t = k+1$. If $\theta$ is a $n$-dimensional traversal of length $k$, then we define its \emph{geometric realization}\index{geometric realization} $\widehat{\theta}$ to be the colimit of the diagram
\begin{displaymath}
  \begin{tikzcd}
      \Delta^n \ar[dr, "d_{\theta(1)^s}"] & & \Delta^n \ar[dl, "d_{\theta(1)^t}"] \ar[dr, "d_{\theta(2)^s}"] & & \ldots \ar[dl, "d_{\theta(2)^t}"] \ar[dr, "d_{\theta(k)^s}"] & & \ar[dl, "d_{\theta(k)^t}"] \Delta^n \\
      & \Delta^{n+1} & & \Delta^{n+1} & & \Delta^{n+1}
  \end{tikzcd}
\end{displaymath}
in simplicial sets. In words: we turn an $n$-dimensional traversal into a simplicial set, by replacing its vertices by $n$-simplices and its edges by $(n+1)$-simplices, in such a way that if an $(n+1)$-simplex comes from the $i$th edge, then the $n$-simplices coming from the vertices connected by that edge are its $\theta(i)^s$-th and $\theta(i)^t$-th faces, respectively.
\end{defi}

\begin{theo}{geometricrealizinpbk} The geometric realization $\widehat{\theta}$ of an $n$-dimensional traversal $\theta$ fits into a pullback square
\begin{equation}\label{eq:geometricrealizinpbk}
  \begin{tikzcd}
    \widehat{\theta} \ar[d, "j_\theta"] \ar[r, "k_\theta"] & \mathbb{T}_1 \ar[d, "{\rm cod}"] \\
    \Delta^n \ar[r, "\theta"] & \mathbb{T}_0.
  \end{tikzcd}
\end{equation}
\end{theo}
\begin{proof}
We have to construct two maps $j_\theta: \widehat{\theta} \to \Delta^n$ and $k_\theta: \widehat{\theta} \to \mathbb{T}_1$, which we will do using that $\widehat{\theta}$ is a colimit. If $\theta(i) = (k, \pm)$, let us write $\overline{\theta(i)} = k$. Then
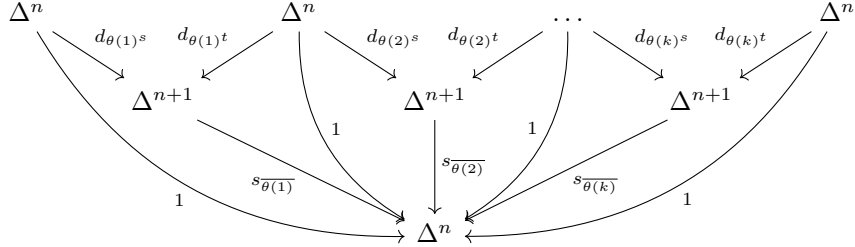
\begin{displaymath}
  \begin{tikzcd}
      \Delta^n \ar[dr, "d_{\theta(1)^s}"] \ar[dddrrr, bend right, "1"']& & \Delta^n \ar[dl, "d_{\theta(1)^t}"'] \ar[dr, "d_{\theta(2)^s}"] \ar[dddr, "1", bend right]& & \ldots \ar[dl, "d_{\theta(2)^t}"'] \ar[dr, "d_{\theta(k)^s}"] \ar[dddl, "1"', bend left]& & \ar[dl, "d_{\theta(k)^t}"'] \Delta^n \ar[dddlll, "1", bend left]\\
      & \Delta^{n+1} \ar[ddrr, "s_{\overline{\theta(1)}}"'] & & \Delta^{n+1} \ar[dd, "s_{\overline{\theta(2)}}"] & & \Delta^{n+1} \ar[ddll, "s_{\overline{\theta(k)}}"] \\ \\
      & & & \Delta^n
  \end{tikzcd}
\end{displaymath}
commutes, so determines a map $j_\theta: \widehat{\theta} \to \Delta^n$. In addition, we can construct cocone with vertex $\mathbb{T}_1$:
\begin{displaymath}
  \begin{tikzcd}
      \Delta^n \ar[dr, "d_{\theta(1)^s}"] \ar[dddrrr, bend right, "{(\theta,0)}"']& & \Delta^n \ar[dl, "d_{\theta(1)^t}"'] \ar[dr, "d_{\theta(2)^s}"] \ar[dddr, "{(\theta,1)}", bend right]& & \ldots \ar[dl, "d_{\theta(2)^t}"'] \ar[dr, "d_{\theta(k)^s}"] \ar[dddl, "{(\theta,2)}"', bend left]& & \ar[dl, "d_{\theta(k)^t}"'] \Delta^n \ar[dddlll, "{(\theta, k)}", bend left]\\
      & \Delta^{n+1} \ar[ddrr, "a_1"'] & & \Delta^{n+1} \ar[dd, "a_2"] & & \Delta^{n+1} \ar[ddll, "a_k"] \\ \\
      & & & \mathbb{T}_1
  \end{tikzcd}
\end{displaymath}
In this diagram the $n$-simplices along the top correspond to positions $p$ in $\theta$, which correspond to maps $(\theta, p): \Delta^n \to \mathbb{T}_1$. Furthermore, the $i$th $(n+1)$-simplex in the second row comes  from the $i$th edge in $\theta$ and for this $(n+1)$-simplex we choose a position in $\theta \cdot s_{\overline{\theta(i)}}$: note that that in $\theta \cdot s_{\overline{\theta(i)}}$ the edge in question gets duplicated and the position we choose is the one inbetween the two copies (we will refer to this as a ``special position''). This determines maps $a_i: \Delta^{n+1} \to \mathbb{T}_1$ which make the diagram above commute, and hence we obtain a map $k_\theta: \widehat{\theta} \to \mathbb{T}_1$. Again using that $\widehat{\theta}$ is a colimit, it is not hard to see that with the resulting maps the square in the statement of the theorem commutes. It remains to show that it is a pullback.

Imagine that we start with a map $\alpha: \Delta^m \to \Delta^n$ and a position $p$ in $\theta \cdot \alpha$. Our first task is to show there is some $m$-simplex in $\widehat{\theta}$ which gets mapped to $\alpha$ and $p$ by the maps we have just constructed. Let us partition the edges in $\theta \cdot \alpha$ by grouping together those edges which come from the same edge in $\theta$: we will call these groups ``blocks''. In other words, they are the fibres of the map $v$ as in the pullback square
\begin{displaymath}
 \begin{tikzcd}
    \{ 1, \ldots, l \} \ar[r, "\theta \cdot \alpha"] \ar[d, "v"] & {[m]} \times \{ +, - \} \ar[d, "\alpha \times 1"] \\
   \{ 1, \ldots, k \} \ar[r, "\theta"] & {[n]} \times \{ +, - \}
 \end{tikzcd}
\end{displaymath}
determining $\theta \cdot \alpha$. For the position $p$ there are now two possibilities: the first is that $p$ is the boundary between two blocks (that is, the edges to the right and left of $p$ come from different edges in $\theta$; this is meant to include the case where $p$ is one of the outer positions). In that case there is some position $q$ in $\theta$ such that $p$ is the restriction of $q$ along $\alpha$. Then we have the map $(\theta, q): \Delta^n \to \mathbb{T}_1$ corresponding to the position $q$, which we can also regard as an $n$-simplex in $\widehat{\theta}$ lying over the identity in $\Delta^n$. Restricting this one along $\alpha$, we get an element in $\widehat{\theta}$ of the form we want. The other case is that $p$ belongs to the interior of one of the blocks with the edges to the left and right mapping to the same edge with label $i$ in $\theta$. In that case we can use that if $\alpha: [m] \to [n]$ is some map in $\mathbf{\Delta}$, and $\alpha(k) = \alpha(k+1) = i$, then there is a map $\beta: [m] \to [n+1]$ such that $\alpha = s_i.\beta$ and $\beta(k) \not= \beta(k+1)$. (Indeed, we can put $\beta(j) = \alpha(j)$ if $j \leq k$ and $\beta(j) = \alpha(j) + 1$ if $j \gt k$.) Then there is some special position $q$ in  $\theta \cdot s_i$ such that $p$ is the restriction of $q$ along $\beta$. This determines a map $a_j: \Delta^{n+1} \to \mathbb{T}_1$, which we can also regard as an $(n+1)$-simplex in $\widehat{\theta}$ lying over $s_i$ in $\Delta^n$. By restricting this $(n+1)$-simplex in $\widehat{\theta}$ along $\beta$, we get an $m$-simplex of the form we want.

It remains to show that the two maps $\widehat{\theta} \to \Delta^n$ and $\widehat{\theta} \to \mathbb{T}_1$ we constructed are jointly monic. Since every element in $\widehat{\theta}$ is a restriction of some $a_i$, it suffices to prove the following statement: if $a_u$ and $a_v$ with $u \leq v$ are $(n+1)$-simplices in $\widehat{\theta}$  corresponding to $(\theta \cdot s_i, p)$  and $(\theta \cdot s_j, q)$, respectively, and $\alpha, \beta: \Delta^m \to \Delta^{n+1}$ are such that $s_i.\alpha = s_j.\beta$ and $p \cdot \alpha = q \cdot \beta$, then $a_u \cdot \alpha = a_v \cdot \beta$ in $\widehat{\theta}$.

Let us first consider the case where $i \not= j$. For convenience, we will assume that $i \lt j$, and both edges $i$ and $j$ point to the right. Then we can take the following pullback:
\begin{displaymath}
 \begin{tikzcd}
  \Delta^{n+2} \ar[d, "s_{j+1}"] \ar[r, "s_i"] & \Delta^{n+1} \ar[d, "s_j"] \\
\Delta^{n+1} \ar[r, "s_i"] & \Delta^n.
 \end{tikzcd}
\end{displaymath}
So there is some map $\gamma$ such that $\alpha = s_{j+1}.\gamma$ and $\beta = s_i.\gamma$ and we obtain the equation $p \cdot s_{j+1}. \gamma = q \cdot s_i. \gamma$. Note that both $p \cdot s_{j+1}$ and $q \cdot s_i$ are distinct positions in the same traversal and what the equation is saying is that they become identified after restricting along $\gamma$. The crucial observation is that this can only happen if $\gamma$ removes the edge to the right of $p$ and the one to the left of $q$ and everything else inbetween. In particular, $\gamma$ factors through $d_i$ and $d_{j+2}$ and we can write $\gamma = d_{j+2}.d_i.\gamma' = d_i.d_{j+1}.\gamma'$, so that $\alpha = d_i.\gamma'$ and $\beta = d_{j+1}.\gamma'$. Since $\gamma'$ must omit all the labels of edges between inbetween $u$ and $v$, the following diagram commutes:
\begin{displaymath}
\begin{tikzcd}
& & &  \Delta^m \ar[dd, "\gamma'"] \ar[ddll, "\gamma'"] \ar[ddrr, "\gamma'"] \\ \\
& \Delta^n \ar[dl, "d_{p^t_u}"] \ar[dr, "d_{p^s_{u +1}}"] & & \Delta^n \ar[dl, "d_{p^t_{u+1}}"] \ar[dr, "d_{p^s_{u+2}}"] & & \ldots \ar[dr, "d_{p^s_v}"] \ar[dl, "d_{p^t_{u+3}}"] \\
 \Delta^{n+1} & & \Delta^{n+1} & & \Delta^{n+1} & & \Delta^{n+1}
\end{tikzcd}
\end{displaymath}
Since the composite along the left is $\alpha$ and composite along the right is $\beta$, this shows that $a_u \cdot \alpha = a_v \cdot \beta$ in $\widehat{\theta}$, as desired. (There are other cases to be considered: different directions and $i \gt j$, but it all works out.)

Let us now consider the case where $i = j$, but $u \lt v$. From $s_i.\alpha = s_i.\beta$, it follows that $\alpha^{-1}\{ i, i + 1 \} = \beta^{-1} \{ i, i + 1 \} = [k, l]$ for some $k$ and $l$, whilst on inputs outside the interval $[k, l]$ the functions $\alpha$ and $\beta$ are identical. In addition, we have the equation $p \cdot \alpha = q \cdot \beta$, which implies that $\alpha$ must omit $i$ and $\beta$ must omit $i+1$ (if, for simplicity, we assume that both edges point to the right; other cases are again similar). The reason is that outside the $i$-blocks (that is, outside the pairs of consecutive edges in $\theta \cdot s_i$ of the form $v^{-1}(e)$ with $\theta(e) = (i, \pm)$) restricting along $\alpha$ and $\beta$ acts in the same way. But also on these $i$-blocks $\alpha$ and $\beta$ act in very similar ways: both replace them by strings of edges of length $l-k$ with identical labels. The only difference is that they may disagree on how to shift the special position. Hence to make the two positions coincide one therefore has to shift the chosen position to the endpoints of the blocks and to eliminate the stuff inbetween. So $\alpha = d_i.\gamma$ and $\beta = d_{i+1}.\delta$ and hence $\gamma = \delta$. Since $\gamma = \delta$ must also omit every label inbetween the edges $u$ and $v$, we can again show (as in the previous case) that $a_u \cdot \alpha = a_v \cdot \beta$ in $\widehat{\theta}$.

Finally, if $u = v$, then $s_i.\alpha = s_i.\beta$ and $p \cdot \alpha = p \cdot \beta$. The former equation again implies that $\alpha^{-1}\{ i, i + 1 \} = \beta^{-1} \{ i, i + 1 \} = [k, l]$ for some $k$ and $l$, whilst on inputs outside the interval $[k, l]$ the functions $\alpha$ and $\beta$ are identical. But then the second equation implies that also on the $i$-blocks $\alpha$ and $\beta$ act in the same way and hence $\alpha$ and $\beta$ also agree on the interval $[k,l]$. Hence $\alpha = \beta$, and $a_u \cdot \alpha = a_v \cdot \beta$.
\end{proof}

\begin{coro}{geometricrealizfunctor}
Geometric realization is part of a functor
\[ \widehat{(-)}: \int_{\Delta} \mathbb{T}_0 \to \widehat{\Delta}. \]
In fact, writing $U: \int_{\Delta} \mathbb{T}_0 \to \widehat{\Delta}$ for the functor sending $(n, \theta)$ to $\Delta_n$, we may regard the $j_\theta$ from the previous theorem as the components of a cartesian natural transformation $j: \widehat{(-)} \to U$.
\end{coro}
\begin{proof}
If $\theta$ is an $n$-dimensional traversal and $\alpha: \Delta^m \to \Delta^n$ is some map in $\Delta$, then we have two pullbacks
\begin{displaymath}
  \begin{tikzcd}
    \widehat{\theta \cdot \alpha} \ar[rr, bend left, "k_{\theta \cdot \alpha}"] \ar[r, dotted, "\widehat{\alpha}"] \ar[d, "j_{\theta \cdot \alpha}"] & \widehat{\theta} \ar[d, "j_\theta"] \ar[r, "k_\theta"] & \mathbb{T}_1 \ar[d, "{\rm cod}"] \\
    \Delta^m \ar[r, "\alpha"] & \Delta^n \ar[r, "\theta"] & \mathbb{T}_0.
  \end{tikzcd}
\end{displaymath}
Therefore there exists a dotted arrow, turning the left hand square into a pullback as well.
\end{proof}

\begin{coro}{alternativedescriptionofM}
 For the simplicial Moore path functor $M$ we have
 \[ (MX)_n \cong \sum_{\theta \in \mathbb{T}_0(n)} {\rm Hom}_{\widehat{\Delta}}(\widehat{\theta}, X). \]
 Therefore the description given in this paper is equivalent to the one given in \cite{vdBerg-Garner}.
\end{coro}
\begin{proof}
 This is immediate from the following description of polynomial functors in presheaf categories (see \cite{moerdijkpalmgren00}, for instance): if $f: B \to A$ is a morphism of presheaves over $\mathbb{C}$, then
 \[ P_f(X)(C) = \sum_{a \in A(C)} {\rm Hom}_{\widehat{\mathbb{C}}}(B_a, X), \]
 where $B_a$ is the pullback
 \begin{displaymath}
   \begin{tikzcd}
     B_a \ar[r] \ar[d] & B \ar[d, "f"] \\
     yC \ar[r, "a"'] & A.
   \end{tikzcd}
 \end{displaymath}
\end{proof}

\begin{rema}{compwithEricsPhDthesis} As said, a more abstract proof of the previous theorem
  and two corollaries can be found in the second author's PhD thesis
  \cite{faber19}. It relies on proving the categorical fact that when 
  \(P: \C^{\op} \to \Sets\) is a presheaf, \(U: \E \to \Sets\) is its category of elements,
  and \(F : \E \to [\C^{\op}, \Sets]\) is a functor with a \emph{cartesian} natural transformation
  \(\gamma : F \Rightarrow y_{U\circ (-)}\), then the square
\[\begin{tikzpicture}[baseline={([yshift=-.5ex]current bounding box.center)}]
    \matrix (m) [matrix of math nodes, row sep=2em,
    column sep=3em]{ 
      |(f)| {F(e)} & |(c)| {\operatorname{colim} F} \\
      |(y)| {y_C}  & |(p)| {P \cong \operatorname{colim} y_{U \circ (-)}}
    \\};
    \begin{scope}[every node/.style={midway,auto,font=\scriptsize}]
      \path[->]
        (f) edge (c)
        (f) edge (y)
        (c) edge (p)
        (y) edge (p);
      \end{scope}
    \end{tikzpicture}
  \]
  induced by the cocones (horizontal) and \(\gamma\) (vertical) is a pullback square.
  This square corresponds to the square~\eqref{eq:geometricrealizinpbk}
  in \reftheo{geometricrealizinpbk}.
\end{rema}

For this reason we can think of the $n$-simplices in $MX$ as pairs consisting of an $n$-dimensional traversal $\theta$ and a morphism $\pi: \widehat{\theta} \to X$ of presheaves. We think of these objects as \emph{Moore paths} in $X$. These Moore path generalise ordinary paths in $X$, that is, maps $\mathbb{I} \to X$, in the following way.

We have the 0-dimensional traversals $\iota^+ = <(0,+)>$ and $\iota^- = <(0,-)>$ which correspond to two global sections
\[ \iota^+, \iota^-: 1 \cong \Delta^0 \to \mathbb{T}_0. \]
Since $\mathbb{I} = \Delta^1$ is the geometric realization of both of these traversals, the previous theorem tells us that we have pullback squares
\begin{displaymath}
	\begin{tikzcd}
		\mathbb{I} \ar[d] \ar[r] & \mathbb{T}_1 \ar[d, "{\rm cod}"] \\
		1 \ar[r, "\iota^+ / \iota^-"'] & \mathbb{T}_0.
	\end{tikzcd}
\end{displaymath}
Regarding these pullback squares as morphisms of polynomials, we obtain two monic cartesian natural transformations $\iota^+, \iota^-: X^\mathbb{I} \to MX$. Furthermore, if we write
\[ s,t: X^\mathbb{I} \to X \]
for the maps induced by $d_1: \Delta^0 \to \Delta^1$ and $d_0: \Delta^0 \to \Delta^1$, respectively, then we have that the following diagrams serially commute:
\begin{displaymath}
	\begin{array}{cc}
		\begin{tikzcd}
			(-)^\mathbb{I} \ar[rr, "\iota^+"] \ar[dr, shift left =  1ex, "s"] \ar[dr, shift right = 1ex, "t"'] & & M \ar[dl, shift left = 1ex, "t"] \ar[dl, shift right = 1ex, "s"'] \\
			& I
		\end{tikzcd} &
		\begin{tikzcd}
		(-)^\mathbb{I} \ar[rr, "\iota^-"] \ar[dr, shift left = 1ex, "t"] \ar[dr, shift right = 1ex, "s"'] & & M \ar[dl, shift left = 1ex, "t"] \ar[dl, shift right = 1ex, "s"'] \\
		& I
		\end{tikzcd}
	\end{array}
\end{displaymath}
(So $\iota^+$ preserves source and target, while $\iota^-$ reverses them.)

\begin{rema}{onprisms} Another way of seeing that the usual path object $X^\mathbb{I}$ is a subobject of $MX$ is as follows. We can take the pullback of the square above along a map from a representable $\Delta^n \to 1$:
\begin{displaymath}
\begin{tikzcd}
\Delta^n \times \mathbb{I} \ar[r] \ar[d] & \mathbb{I} \ar[d] \ar[r] & \mathbb{T}_1 \ar[d, "{\rm cod}"] \\
\Delta^n \ar[r] & 1 \ar[r, "\iota^+ / \iota^-"'] & \mathbb{T}_0.
\end{tikzcd}
\end{displaymath}
Indeed, what this says is that $\Delta^n \times \Delta^1$ is the geometric realisation of the traversals
\[ <(n,+), (n-1,+), \ldots, (2,+), (1,+), (0,+) > \mbox{ and } <(0,-), (1,-), (2,-), \ldots, (n-1,-), (n,-) >. \]
Indeed, this reflects the well-known decomposition of the ``prism'' $\Delta^n \times \Delta^1$ as the union of $n+1$ many $(n+1)$-simplices; from our present point of view, this means that  $\Delta^n \times \Delta^1$ occurs as the geometric realisation of these traversals. From this and the description of $M$ in \refcoro{alternativedescriptionofM}, one can also see that $X^\mathbb{I}$ embeds in $MX$.
\end{rema}

\section{Hyperdeformation retracts in simplicial sets}

In the previous section we have shown that the endofunctor $M$ equips the category of simplicial sets with symmetric Moore structure. Consequently, the category of simplicial sets carries an AWFS, with the coalgebras for the comonad being called the \emph{HDRs} and the algebras for the monad being called the \emph{naive Kan fibrations}. The purpose of this section is to take a closer look at this AWFS.

By definition, the naive Kan fibrations are generated by the large double category of HDRs. One important result in this section is that they are also generated by a small (countable) double category of HDRs, and that the naive fibrations Kan form a local notion of fibred structure. It should be apparent from the proofs that similar results would be true for naive left and right fibrations as well (see \refdefi{naivefibrinsimpset}).

\subsection{HDRS are effective cofibrations} Let us start by proving that HDRs are effective cofibrations. As we have seen in \reflemm{HDRinDominance}, for this it suffices to prove the following result:

\begin{prop}{onecofibr}
The map $r_X: X \to MX$ is always an effective cofibration.
\end{prop}
\begin{proof}
In fact, since
\diag{ X \ar[r]^{!} \ar[d]_{r_X} & 1 \ar[d]^{r_1} \\
MX \ar[r]_{M!} & M1 }
is a pullback ($r$ is cartesian), it suffices to prove this statement for $X = 1$. In other words, we have to define a map $\rho: \mathbb{T}_0 \cong M1 \to \Sigma$ such that
\diag{ 1 \ar[r]^{!} \ar[d]_{r_1} & 1 \ar[d]^{\top} \\
M1 \ar[r]_\rho & \Sigma }
is a pullback. We set
\[ \rho_n(\theta) = \{ e \subseteq \{ 0,1,\ldots,n\} \, : \, e \cap {\rm Im}(\theta) = \emptyset \}. \]
(So we take those subsets $e$ for which no $i \in e$ occurs as $(i,+)$ or $(i, -)$ in the image of the traversal $\theta$. This happens precisely when the restriction of the traversal $\theta$ along $e$ is the unique traversal of length 0.) This is easily seen to be correct.
\end{proof}

\subsection{HDRs as internal presheaves} In the sequel we will often have to prove that certain maps are HDRs. It turns out that for this purpose it is often convenient to use an equivalent description of the category of HDRs and morphisms of HDRs. Indeed, we have:

\begin{theo}{HDRpreshofT}
The category of HDRs in simplicial sets is equivalent to the category of internal presheaves on $\mathbb{T}$.
\end{theo}
\begin{proof}
  This is immediate from the fact that ${\rm cod}: {\rm HDR} \to M{\rm -Coalg}$ is an equivalence (see \refprop{HDRsareMcoalgebras}) and the fact that $M$-coalgebras are equivalent to internal presheaves over $\mathbb{T}$.
\end{proof}

Let us unwind a bit more what this means and explain how one passes from HDRs to internal presheaves over $\mathbb{T}$ and back. If $(i: A \to B, j, H)$ is an HDR, then we can transpose
\[ B \to MB = \sum_{\theta \in \mathbb{T}_0} B^{{\rm cod}^{-1}(\theta)} \]
to a pair of maps $d: B \to \mathbb{T}_0$ and $\rho: B \prescript{d}{}{\times}_{\mathbb{T}_0}^{\rm cod} \mathbb{T}_1 \to B$ which satisfy the axioms for an internal presheaf over $\mathbb{T}$. Conversely, suppose one is given an internal presheaf over $\mathbb{T}_0$, that is, a pair of maps $d: B \to \mathbb{T}_0$ and $\rho: B \times_{\mathbb{T}_0} \mathbb{T}_1 \to B$ satisfying the right equations. The map $H: B \to MB$ is the transpose of $\rho$, while the inclusion $i: A \to B$ is the pullback
\begin{displaymath}
 \begin{tikzcd}
    A \ar[r, "i"] \ar[d, "!"] & B \ar[d, "d"] \\
    1 \ar[r, "0"] & \mathbb{T}_0,
 \end{tikzcd}
\end{displaymath}
so $A$ is the fibre over the initial object in $\mathbb{T}$. Finally, $t.H: B \to B$ sends an element in $B$ over $\theta$ to its restriction along the unique map $0 \to \theta$, and hence $j$ does the same.

In the sequel we will often use this internal presheaf perspective on HDRs (see in particular the proofs of \refprop{genericHDR} and \refprop{inclusionsmallHDRsintobigHDRs}). The reason for this is that presheaves are easier to manipulate than HDRs, because they get rid of the exponential in the definition of $MB$ and they (seemingly) contain less data. This becomes quite apparent when one considers the generic HDR from \refprop{genericHDR} which we will heavily exploit to construct more examples of HDRs. Therefore we will now translate the main constructions on HDRs (pullback, pushout, and vertical composition) into the language of internal presheaves.

\begin{description}
  \item[Pullback] Suppose $d: B \to \mathbb{T}_0$ and $\rho: B \times_{\mathbb{T}_0} \mathbb{T}_1 \to B$ is an internal presheaf and $A$ is the fibre over 0 (that is, the pullback as above). If we are given a map $a: A' \to A$, then we obtain a new presheaf $(B', d', \rho')$ by pullback as follows. First of all we take the pullback
  \begin{displaymath}
    \begin{tikzcd}
      B' \ar[r, "b"] \ar[d, "j'"] & B \ar[d, "j"] \\
      A' \ar[r, "a"] & A.
    \end{tikzcd}
  \end{displaymath}
  If $H: B \to MB$ is the map induced by $d,\rho$, then $H'$ is the unique map filling
  \begin{displaymath}
    \begin{tikzcd}
      B' \ar[dr, "H'", dotted] \ar[dd, "{(j',M!.H.b)}"'] \ar[rr, "b"] & & B \ar[dr, "H"] \\
      & MB' \ar[dd, "Mj'"] \ar[rr, "Mb"] & & MB \ar[dd, "Mj"] \\
      A' \times M1 \ar[dr, "\alpha"] \\
      & MA' \ar[rr, "Ma"] & & MA.
    \end{tikzcd}
  \end{displaymath}
  So $d' := M!.H'= M!.H.b=d.b: B' \to \mathbb{T}_0$. Moreover, $\rho'$ will be the unique dotted arrow filling
  \begin{displaymath}
    \begin{tikzcd}
      B' \times_{\mathbb{T}_0} \mathbb{T}_1 \ar[r, dotted, "\rho'"'] \ar[rr, bend left, "j'.\pi_1"] \ar[d,"b \times 1"] & B' \ar[d, "b"] \ar[r, "j'"'] & A' \ar[d, "a"] \\
      B \times_{\mathbb{T}_0} \mathbb{T}_1 \ar[r, "\rho"] \ar[rr, bend right, "j.\pi_1"] & B \ar[r, "j"] & A.
    \end{tikzcd}
  \end{displaymath}
  \item[Pushout] Again, suppose $d: B \to \mathbb{T}_0$ and $\rho: B \times_{\mathbb{T}_0} \mathbb{T}_1 \to B$ is an internal presheaf and $A$ is the fibre over 0. If we are given a map $a: A \to A'$, we obtain a presheaf $(B',d',\rho')$ by pushout, as follows. First, we take the pushout:
  \begin{displaymath}
    \begin{tikzcd}
      A \ar[r, "a"] \ar[d, "i"] & A' \ar[d, "i'"] \\
      B \ar[r, "b"] & B'.
    \end{tikzcd}
  \end{displaymath}
  If $H: B \to MB$ is the map induced by $d, \rho$, then $H': B' \to MB'$ is the unique dotted arrow filling
  \begin{displaymath}
    \begin{tikzcd}
      A \ar[r, "a"] \ar[d, "i"] & A' \ar[d, "i'"] \ar[dr, "i'"] \\
      B \ar[r, "b"] \ar[dr, "H"] & B' \ar[dr, "H'", dotted] & B' \ar[d, "r"] \\
      & MB \ar[r, "Mb"] & MB'.
    \end{tikzcd}
  \end{displaymath}
  This means that $d': B' \to \mathbb{T}_0$ is the unique dotted arrow filling
  \begin{displaymath}
    \begin{tikzcd}
      A \ar[r, "a"] \ar[d, "i"] & A' \ar[d, "i'"] \ar[dr, "!"] \\
      B \ar[r, "b"] \ar[drr, "d", bend right] & B' \ar[dr, "d'", dotted] & 1 \ar[d, "0"] \\
      & & \mathbb{T}_0,
    \end{tikzcd}
  \end{displaymath}
  whilst $\rho'$ is the unique map filling
  \begin{displaymath}
    \begin{tikzcd}
      A \ar[r, "a"] \ar[d, "\cong"] & A' \ar[d, "\cong"] \ar[rddd, "i'", bend left] \\
      A \times_{\mathbb{T}_0} \mathbb{T}_1 \ar[r, "a \times 1"] \ar[d, "i \times 1"] & A' \times_{\mathbb{T}_0} \mathbb{T}_1 \ar[d, "i' \times 1"]  \\
      B \times_{\mathbb{T}_0} \mathbb{T}_1 \ar[r, "b \times 1"] \ar[dr, "\rho"] & B' \times_{\mathbb{T}_0} \mathbb{T}_1 \ar[dr, "\rho'", dotted] \\
      & B \ar[r, "b"] & B'.
    \end{tikzcd}
  \end{displaymath}
  \item[Vertical composition] Suppose $(i_0: A \to B, j_0, H_0)$ and $(i_1: B \to C, j_1, H_1)$ are HDRs coming from presheaf structures $(d_0,\rho_0)$ and $(d_1, \rho_1)$. Then the vertical composition is given by \[ (i_1.i_0: A \to C, j_0.j_1, \mu.(H_1, Mi_1.H_0.j_1)). \]
  So this means we have a function $d_2: C \to \mathbb{T}_0$, which is:
  \begin{eqnarray*}
    d_2 & = & M!.\mu.(H_1, Mi_1.H_0.j_1) \\
    & = & \mu.(M!.H_1, M!.Mi_1.H_0.j_1) \\
    & = & \mu.(d_1,d_0.j_1) \\
    & = & d_1 * d_0.j_1.
  \end{eqnarray*}
  In addition, we need a morphism $\rho_2: C \times_{\mathbb{T}_0} \mathbb{T}_1 \to C$. Here the domain can also be computed in two steps by taking the following two pullbacks:
  \begin{displaymath}
    \begin{tikzcd}
       C \times_{\mathbb{T}_0} \mathbb{T}_1 \cup C \times_{\mathbb{T}_0} \mathbb{T}_1 \ar[rr, "{[\pi_1,\pi_1]}"] \ar[d, "{(\pi_2,d_0.j_1.\pi_1) \cup (d_1 \times 1)}"'] & & C \ar[d, "{(d_1, d_0.j_1)}"]  \\
      \mathbb{T}_1 \times \mathbb{T}_0 \cup \mathbb{T}_0 \times \mathbb{T}_1 \ar[d, "\mu^*"] \ar[rr, "{[{\rm cod} \times 1, 1 \times {\rm cod}]}"] & & \mathbb{T}_0 \times \mathbb{T}_0 \ar[d, "*"] \\
      \mathbb{T}_1 \ar[rr, "{\rm cod}"] & & \mathbb{T}_0.
    \end{tikzcd}
  \end{displaymath}
  Hence we can define $\rho_2$ as $[\rho_1,i_1.\rho_0.(j_1 \times 1)] : C \times_{\mathbb{T}_0} \mathbb{T}_1 \cup C \times_{\mathbb{T}_0} \mathbb{T}_1 \to C$.
\end{description}

We finish this subsection with the proof that the category of HDRs contains a ``generic'' element.
\index{generic HDR}

\begin{prop}{genericHDR} The triple $({\rm id}^*: \mathbb{T}_0 \to \mathbb{T}_1, {\rm cod}, \widehat{\rm comp})$ is an HDR, which is generic in the following sense: for any HDR $(i: A \to B, j, H)$ there exists a pullback  $(i', j', H')$ of the generic one together with a morphism of HDRs $(i',j',H') \to (i, j, H)$ which is an epimorphism on the level of presheaves.
\end{prop}
\begin{proof}
As a presheaf, the generic HDR is given by ${\rm dom}: \mathbb{T}_1 \to \mathbb{T}_0$ with ${\rm comp}: \mathbb{T}_1 \times_{\mathbb{T}_0} \mathbb{T}_1 \to \mathbb{T}_1$. Now imagine that we have some HDR, considered as a presheaf $d: B \to \mathbb{T}_0$ with $\rho: B \times_{\mathbb{T}_0} \mathbb{T}_1 \to B$. Then pulling back the generic HDR along $d$ gives as presheaf the pullback
\begin{displaymath}
 \begin{tikzcd}
   B \times_{\mathbb{T}_0} \mathbb{T}_1 \ar[d, "p_2"] \ar[r, "p_1"] & B \ar[d, "d"] \\
  \mathbb{T}_1 \ar[r, "{\rm cod}"] & \mathbb{T}_0,
 \end{tikzcd}
\end{displaymath}
together with $d' = {\rm dom}.p_2$ and $\rho'$ the unique filler of
\begin{displaymath}
  \begin{tikzcd}
    (B \times_{\mathbb{T}_0} \mathbb{T}_1) \times_{\mathbb{T}_0} \mathbb{T}_1 \ar[r, dotted, "\rho'"] \ar[rr, bend left, "p_1.\pi_1"] \ar[d,"p_2 \times 1"] & B \times_{\mathbb{T}_0} \mathbb{T}_1 \ar[d, "p_2"] \ar[r, "p_1"] & B \ar[d, "d"] \\
    \mathbb{T}_1 \times_{\mathbb{T}_0} \mathbb{T}_1 \ar[r, "{\rm comp}"] \ar[rr, bend right, "{\rm cod}.\pi_1"] & \mathbb{T}_1 \ar[r, "{\rm cod}"] & \mathbb{T}_0.
  \end{tikzcd}
\end{displaymath}
One can interpret this presheaf as follows: the category of internal presheaves has a forgetful functor to the slice category over $\mathbb{T}_0$ and this forgetful functor has a left adjoint. The presheaf ($B \times_{\mathbb{T}_0} \mathbb{T}_1, d', \rho'$) is the free presheaf on $d: B \to \mathbb{T}_0$. Therefore there is an epic morphism of presheaves
\begin{displaymath}
 \begin{tikzcd}
    B \times_{\mathbb{T}_0} \mathbb{T}_1 \ar[rr,"\rho"] \ar[dr, "d'"'] & & B \ar[dl, "d"] \\
    & \mathbb{T}_0,
 \end{tikzcd}
\end{displaymath}
namely the counit of the adjunction.
\end{proof}

\subsection{A small double category of HDRs}
\label{ssec:small-double-category} Our next goal is to show that the
naive Kan fibrations in simplicial sets are generated by a small double
category. We do this by showing that the large double category of HDRs in
simplicial sets contains a small double category such that a system of lifts
against the small double category can always be extended in a unique way to a
system of lifts again the entire double category of HDRs.

Let $\mathbb{H}$ be the following double category.
\begin{itemize}
  \item Objects are pairs $(n, \theta)$ with $n \in \mathbb{N}$ and $\theta$ an $n$-dimensional traversal.
  \item There is a unique vertical map $(n_0,\theta_0) \to (n_1,\theta_1)$ if $n_0 = n_1$ and $\theta_0$ is a final segment of $\theta_1$.
  \item A horizontal map $(m, \psi) \to (n, \theta)$ is a pair consisting of a map $\alpha: [m] \to [n]$ together with an $m$-dimensional traversal $\sigma$ such that $\psi * \sigma = \theta \cdot \alpha$. The formula for horizontal composition is $(\alpha,\sigma).(\beta,\tau) = (\alpha.\beta, \tau * (\sigma \cdot \beta))$.
  \item A square is any picture of the form
  \begin{displaymath}
    \begin{tikzcd}
      (m, \psi_0) \ar[d] \ar[r, "{(\alpha,\sigma)}"] & (n, \theta_0) \ar[d] \\
      (m, \psi_1) \ar[r, "{(\alpha, \sigma)}"] & (n, \theta_1)
    \end{tikzcd}
  \end{displaymath}
  in which the horizontal arrows have the same label.
\end{itemize}

Our first goal will be to argue that there is a double functor $\mathbb{H} \to {\rm HDR}(\widehat{\mathbf{\Delta}})$ which on the level of objects assigns to every $(n, \theta)$ the geometric realization of $\theta$. Recall from the previous section that the geometric realization of $\theta$ is, by definition, the colimit of the following diagram:
\begin{displaymath}
  \begin{tikzcd}
      \Delta^n \ar[dr, "d_{\theta(1)^s}"] & & \Delta^n \ar[dl, "d_{\theta(1)^t}"] \ar[dr, "d_{\theta(2)^s}"] & & \ldots \ar[dl, "d_{\theta(2)^t}"] \ar[dr, "d_{\theta(k)^s}"] & & \ar[dl, "d_{\theta(k)^t}"] \Delta^n \\
      & \Delta^{n+1} & & \Delta^{n+1} & & \Delta^{n+1}
  \end{tikzcd}
\end{displaymath}
From now we will denote this colimit simply as $\theta$, rather than as $\widehat{\theta}$. Note that $\theta$ comes with two maps from $\Delta^n$, corresponding to the outer maps $\Delta^n \to \Delta^{n+1}$ in this diagram. We will refer to the map $\Delta^n \to \theta$ induced by the inclusion on the left as $s_\theta$ and the map of the same shape induced by the inclusion on the right as $t_\theta$.

In fact, $s_\theta$ and $t_\theta$ occur as pullbacks of ${\rm id}$ and ${\rm id}^*$, as follows:
\begin{displaymath}
  \begin{array}{cc}
    \begin{tikzcd}
      \Delta^n \ar[d, "s_\theta"] \ar[r, "\theta"] & \mathbb{T}_0 \ar[d, "{\rm id}"] \\
      \theta \ar[d, "j_{\theta}"] \ar[r, "k_\theta"] & \mathbb{T}_1 \ar[d, "{\rm cod}"] \\
      \Delta^n \ar[r, "\theta"] & \mathbb{T}_0
    \end{tikzcd} &
    \begin{tikzcd}
      \Delta^n \ar[d, "t_\theta"] \ar[r, "\theta"] & \mathbb{T}_0 \ar[d, "{\rm id}^*"] \\
      \theta \ar[d, "j_{\theta}"] \ar[r, "k_\theta"] & \mathbb{T}_1 \ar[d, "{\rm cod}"] \\
      \Delta^n \ar[r, "\theta"] & \mathbb{T}_0
    \end{tikzcd}
  \end{array}
\end{displaymath}
Indeed, using the notation of \refcoro{geometricrealizfunctor}, we may regard $s$ and $t$ as cartesian natural transformations $U \to \widehat{(-)}$ and sections of the cartesian natural transformation $j: \widehat{(-)} \to U$. Also, since ${\rm id}^*$ is the generic HDR, we may regard $t_\theta$ as an HDR. Moreover, the picture on the right as well as the naturality squares of the natural transformation $t$ are cartesian morphisms of HDRs.

A typical vertical morphism in $\mathbb{H}$ is of the form $\psi \to \theta * \psi$. Such a morphism we can equip with an HDR-structure, because the top square in
\begin{displaymath}
  \begin{tikzcd}
    \Delta^n \ar[r, "s_\psi"] \ar[d, "t_\theta"] & \psi \ar[d, "\iota_2"] \\
    \theta \ar[r, "\iota_1"] \ar[d, "j_\theta"] & \theta * \psi \ar[d, "{[s_\psi.j_\theta, 1]}"] \\
    \Delta^n \ar[r, "s_\psi"] & \psi
  \end{tikzcd}
\end{displaymath}
is a pushout. Note that because both squares are pullbacks as well, the diagram will become a bicartesian morphism of HDRs.

This explains how we map vertical morphisms to HDRs. Let us now explain where we map the horizontal maps to. Note that the horizontal maps are of the form $(\alpha,\sigma)$ which we can write as a composition $(\alpha, <>).(1,\sigma)$, so we only need to explain where we map these composites to. However, the map $(1,\sigma): \psi \to \psi * \sigma$ is just $\iota_1$, while $(1, \alpha) = \widehat{\alpha}: \psi \cdot \alpha \to \psi$ comes from the functoriality of geometric realization (see \refcoro{geometricrealizfunctor}; from now we will also simply write $\alpha$).

To explain where we map the squares to, we use the same decomposition. Let us first look at a square coming from $(1,\sigma)$:
\begin{displaymath}
  \begin{tikzcd}
    \Delta^n \ar[d, "t_\theta"] \ar[r, "s_\psi"]  & \psi \ar[d, "\iota_2"] \ar[r, "\iota_1"] & \psi * \sigma \ar[d, "\iota_2"] \\
    \theta \ar[r, "\iota_1"] \ar[d, "j_\theta"] & \theta * \psi \ar[r, "\iota_1"] \ar[d] & \theta * \psi * \sigma \ar[d] \\
    \Delta^n \ar[r, "s_\psi"] \ar[rr, "s_{\psi * \sigma}"', bend right] & \psi \ar[r, "\iota_1"] & \psi * \sigma.
  \end{tikzcd}
\end{displaymath}
The HDR-structures on the middle and right arrow are defined by pushout from the left arrow: this automatically makes the right hand square a cocartesian morphism of HDRs, and, in particular, a morphism of HDRs. Note that, as before, all squares in the diagram above are pullbacks as well, so the morphism is actually bicartesian.

If the square comes from $(\alpha, <>)$, consider a double cube of the form:
\begin{displaymath}
  \begin{tikzcd}
    & \psi \cdot \alpha \ar[rr, "\alpha"] \ar[dd, near start, "\iota_2"] & & \psi \ar[dd, "\iota_2"] \\
    \Delta^m \ar[ur, "s_{\psi \cdot \alpha}"] \ar[rr, "\alpha", near start, crossing over] \ar[dd, "t_{\theta \cdot \alpha}"] & & \Delta^n \ar[ur, "s_\psi"]  \\
    & \theta \cdot \alpha * \psi \cdot \alpha  \ar[rr, "\alpha", near start] \ar[dd, "{[s_{\psi \cdot \alpha}.j_{\psi \cdot \alpha}, 1]}", near start] & &  \theta * \psi \ar[dd, "{[s_{\psi}.j_\theta,1]}"] \\
    \theta \cdot \alpha \ar[rr, "\alpha", crossing over, near start] \ar[dd, "j_{\theta \cdot \alpha}"] \ar[ur, "\iota_1"] & & \theta \ar[ur, "\iota_1"] \ar[from=uu, crossing over, "t_\theta", near start] \\
    & \psi \cdot \alpha \ar[rr, near start, "\alpha"] & & \psi \\
    \Delta^m \ar[rr, "\alpha"] \ar[ur, "s_{\psi \cdot \alpha}"] & & \Delta^n \ar[ur, "s_{\psi}"] \ar[from=uu, crossing over, "j_\theta", near start]
  \end{tikzcd}
\end{displaymath}
Note that the left and right hand side of this double cube are cocartesian morphisms of HDRs, while the front is a cartesian morphism (since both come with a cartesian morphism to the generic HDR). Since the bottom square is a pullback, the back is a cartesian morphism of HDRs, by Beck-Chevalley (see \refprop{domHDRfibration}).

This finishes the \emph{construction} of a potential double functor $\mathbb{H} \to {\rm HDR}({\widehat{\mathbf{\Delta}}})$: the \emph{verification} that it actually is a double functor turns out to be a lot of work and will keep us occupied for the next couple of pages.

\begin{rema}{imagesqcartesian}
  Note that all the squares in the image of this (potential) double functor are cartesian morphisms of HDRs.
\end{rema}

\begin{lemm}{horizontalcompforH}
  The potential double functor $\mathbb{H} \to {\rm HDR}(\widehat{\mathbf{\Delta}})$ just constructed preserves horizontal composition of morphisms.
\end{lemm}
\begin{proof}
To prove that our double functor preserves horizontal composition it suffices to check that
\begin{displaymath}
  \begin{tikzcd}
    \psi \cdot \alpha \ar[r, "\alpha"] \ar[d, "\iota_1"] & \psi \ar[d, "\iota_1"] \\
    \psi \cdot \alpha * \sigma \cdot \alpha \ar[r, "\alpha"] & \psi * \sigma
  \end{tikzcd}
\end{displaymath}
commutes. However, we have a commutative diagram
\begin{displaymath}
  \begin{tikzcd}
    & \Delta^n \ar[dr, "s_\sigma"] \ar[dl, "t_\psi"'] \\
    \psi \ar[dr, "\iota_1"] \ar[ddr, bend right, "j_\psi"'] & & \sigma \ar[dl, "\iota_2"] \ar[ddl, bend left, "j_\sigma"]\\
    & \psi * \sigma  \ar[d, "j_{\psi * \sigma}"] \\
    & \Delta^n
  \end{tikzcd}
\end{displaymath}
in which the square is a pushout, and by pulling it back along $\alpha: \Delta^m \to \Delta^n$ we get the commutativity of the previous square.
\end{proof}

Since our potential double functor trivially preserves horizontal identities and horizontal composition of squares, it remains to consider the vertical structure. Also here preservation of identities and vertical composition of squares will be immediate once we show vertical composition of morphisms is preserved. For that, it is convenient to use an alternative construction of the HDR-structure on $\psi \to \theta * \psi$, via the ``generic inclusion of HDRs''. Indeed, consider the following picture:
\begin{displaymath}
  \begin{tikzcd}
    \mathbb{T}_0 \ar[d, "{\rm id}^*"] & \mathbb{T}_0 \times \mathbb{T}_0 \ar[d, "{\rm id}^* \times 1"] \ar[r, "1 \times {\rm id}"] \ar[l, "\pi_1"] & \mathbb{T}_0 \times \mathbb{T}_1 \ar[d, "\iota_2"] \\
    \mathbb{T}_1 \ar[d, "{\rm cod}"] & \mathbb{T}_1 \times \mathbb{T}_0 \ar[d, "{\rm cod} \times 1"] \ar[r, "\iota_1"] \ar[l, "\pi_1"] & \mathbb{T}_1 \times \mathbb{T}_0 \cup_{\mathbb{T}_0 \times \mathbb{T}_0} \mathbb{T}_0 \times \mathbb{T}_1 \ar[d, "{[{\rm cod} \times {\rm id}, 1]}"] \\
    \mathbb{T}_0 & \mathbb{T}_0 \times \mathbb{T}_0 \ar[l, "\pi_1"] \ar[r, "1 \times {\rm id}"] & \mathbb{T}_0 \times \mathbb{T}_1
  \end{tikzcd}
\end{displaymath}
We can give $\iota_2$ the structure of an HDR, by first pulling the generic structure on ${\rm id}^*$ back along $\pi_1$ and then pushing the result forward along $1 \times {\rm id}$. By pulling this HDR back along the horizontal arrow at the centre of
\begin{displaymath}
  \begin{tikzcd}
    \Delta^n \ar[r, "{(\theta, \psi)}"] \ar[d, "s_\psi"] & \mathbb{T}_0 \times \mathbb{T}_0 \ar[d, "1 \times {\rm id}"] \\
    \psi \ar[r] \ar[d, "j_\psi"] & \mathbb{T}_0 \times \mathbb{T}_1 \ar[d, "1 \times {\rm cod}"] \\
    \Delta^n \ar[r, "{(\theta, \psi)}"] & \mathbb{T}_0 \times \mathbb{T}_0
  \end{tikzcd}
\end{displaymath}
we obtain $\psi \to \theta * \psi$. The reason is that we can apply Beck-Chevalley to the top square in the diagram above and the HDR ${\rm id}^* \times 1$ obtained by pulling the generic HDR along $\pi_1$.

Let us translate the generic inclusion in presheaf language. First, the pullback of the generic HDR along $\pi_1: \mathbb{T}_0 \times \mathbb{T}_0 \to \mathbb{T}_0$ is $\mathbb{T}_1 \times \mathbb{T}_0$ with $d = {\rm dom}.\pi_1$ and
\[ \rho = ({\rm comp}.(\pi_1.\pi_1,\pi_2), \pi_2.\pi_1): (\mathbb{T}_1 \times \mathbb{T}_0) \times_{\mathbb{T}_0} \mathbb{T}_1 \to \mathbb{T}_1 \times \mathbb{T}_0. \]
Then we need to push this forward along $1 \times {\rm id}: \mathbb{T}_0 \times \mathbb{T}_0 \to \mathbb{T}_0 \times \mathbb{T}_1$, which results in:
\begin{displaymath}
  \begin{tikzcd}
    \mathbb{T}_0 \times \mathbb{T}_0 \ar[r, "1 \times {\rm id}"] \ar[d, "{\rm id}^* \times 1"] & \mathbb{T}_0 \times \mathbb{T}_1 \ar[d, "\iota_2"] \ar[dr, "!"] \\
    \mathbb{T}_1 \times \mathbb{T}_0 \ar[r, "\iota_1"] \ar[drr, bend right, "{\rm dom}.\pi_1"] & \mathbb{T}_1 \times \mathbb{T}_0 \cup \mathbb{T}_0 \times \mathbb{T}_1 \ar[dr, dotted, "d_1"] & 1 \ar[d, "0"] \\
    & & \mathbb{T}_0
  \end{tikzcd}
\end{displaymath}
So the generic inclusion of HDRs is $\mathbb{T}_1 \times \mathbb{T}_0 \cup \mathbb{T}_0 \times \mathbb{T}_1$ with $d_1 = [{\rm dom}.\pi_1, 0.!]$. It also comes equipped with an action $\rho_1$, which we find as follows:
\begin{displaymath}
  \begin{tikzcd}
    \mathbb{T}_0 \times \mathbb{T}_0 \ar[d, "{({\rm id}^* \times 1, {\rm id}^*.0.!)}"] \ar[r, "1 \times {\rm id}"] & \mathbb{T}_0 \times \mathbb{T}_1 \ar[ddr, bend left, "\iota_2"] \ar[d, "\iota_2"] \\
    (\mathbb{T}_1 \times \mathbb{T}_0) \times_{\mathbb{T}_0} \mathbb{T}_1 \ar[dr, "\rho"] \ar[r, "\iota_1"] & (\mathbb{T}_1 \times \mathbb{T}_0) \times_{\mathbb{T}_0} \mathbb{T}_1 \cup \mathbb{T}_0 \times \mathbb{T}_1 \ar[dr, dotted, "\rho_1"] \\
    & \mathbb{T}_1 \times \mathbb{T}_0 \ar[r, "\iota_1"] & \mathbb{T}_1 \times \mathbb{T}_0 \cup \mathbb{T}_0 \times \mathbb{T}_1
  \end{tikzcd}
\end{displaymath}
Hence $\rho_1 = ({\rm comp}.(\pi_1.\pi_1,\pi_2), \pi_2.\pi_1) \cup 1$.

We will compare this with two other HDRs: the pullbacks of the generic HDR along the maps $\pi_2, *: \mathbb{T}_0 \times \mathbb{T}_0 \to \mathbb{T}_0$. In terms of presheaves, the pullback of the generic HDR along $\pi_2: \mathbb{T}_0 \times \mathbb{T}_0 \to \mathbb{T}_0$ is $d_0 = {\rm dom}.\pi_2: \mathbb{T}_0 \times \mathbb{T}_1 \to \mathbb{T}_0$, while $\rho_0$ is the unique map filling
\begin{displaymath}
  \begin{tikzcd}
    \mathbb{T}_0 \times (\mathbb{T}_1 \times_{\mathbb{T}_0} \mathbb{T}_1) \ar[d, "\pi_2"] \ar[r, dotted, "\rho_0"] \ar[rr, bend left, "{(\pi_1,{\rm cod}.\pi_1.\pi_2)}"]& \mathbb{T}_0 \times \mathbb{T}_1 \ar[r, "1 \times {\rm cod}"] \ar[d, "\pi_2"] & \mathbb{T}_0 \times \mathbb{T}_0 \ar[d, "\pi_2"] \\
    \mathbb{T}_1 \times_{\mathbb{T}_0} \mathbb{T}_1 \ar[r, "{\rm comp}"] & \mathbb{T}_1 \ar[r, "{\rm cod}"] &  \mathbb{T}_0.
  \end{tikzcd}
\end{displaymath}
Hence $\rho_0 = 1 \times {\rm comp}$.

To compute the pullback of the generic HDR along $*: \mathbb{T}_0 \times \mathbb{T}_0 \to \mathbb{T}_0$, we first compute the pullbacks
\begin{displaymath}
  \begin{tikzcd}
    \mathbb{T}_0 \times \mathbb{T}_0 \ar[d, "\iota_2.(1 \times {\rm id}^*)"] \ar[rr, "*"] & &  \mathbb{T}_0 \ar[d, "{\rm id}^*"] \\
    \mathbb{T}_1 \times \mathbb{T}_0 \cup \mathbb{T}_0 \times \mathbb{T}_1 \ar[rr, "\mu^*"] \ar[d, "{[{\rm cod} \times 1, 1 \times {\rm cod}]}"] & &  \mathbb{T}_1 \ar[d, "{\rm cod}"] \\
    \mathbb{T}_0 \times \mathbb{T}_0 \ar[rr, "*"] & & \mathbb{T}_0
  \end{tikzcd}
\end{displaymath}
Hence the presheaf we are looking for is $\mathbb{T}_1 \times \mathbb{T}_0 \cup \mathbb{T}_0 \times \mathbb{T}_1$ with $d_2 = {\rm dom}.\mu^* = [*.({\rm dom}.\pi_1, \pi_2), {\rm dom}.\pi_2]$, while $\rho_2$ is the unique arrow filling
\begin{displaymath}
  \begin{tikzcd}
     (\mathbb{T}_1 \times \mathbb{T}_0 \cup \mathbb{T}_0 \times \mathbb{T}_1)\times_{\mathbb{T}_0} \mathbb{T}_1 \ar[d, "\mu^* \times 1"] \ar[r, dotted, "\rho_2"] \ar[rrr, bend left, "{[{\rm cod} \times 1, 1 \times {\rm cod}].\pi_1}"]& \mathbb{T}_1 \times \mathbb{T}_0 \cup \mathbb{T}_0 \times \mathbb{T}_1 \ar[rr, "{[{\rm cod} \times 1, 1 \times  {\rm cod}]}"'] \ar[d, "\mu^*"] & & \mathbb{T}_0 \times \mathbb{T}_0 \ar[d, "*"] \\
    \mathbb{T}_1 \times_{\mathbb{T}_0} \mathbb{T}_1 \ar[r, "{\rm comp}"] & \mathbb{T}_1 \ar[rr, "{\rm cod}"] & & \mathbb{T}_0
  \end{tikzcd}
\end{displaymath}

\begin{lemm}{lemmaonverticcomp}
  The vertical composition of $(d_1,\rho_1)$ after $(d_0,\rho_0)$ equals $(d_2,\rho_2)$.
\end{lemm}
\begin{proof}
First of all, we have to prove that \[ d_2 = *.(d_1,d_0.j_1): \mathbb{T}_1 \times \mathbb{T}_0 \cup \mathbb{T}_0 \times \mathbb{T}_1 \to \mathbb{T}_0. \]
This holds, because:
\begin{eqnarray*}
  *.(d_1,d_0.j_1).\iota_1 & = & *.(d_1.\iota_1, d_0.j_1.\iota_1) \\
  & = & *.({\rm dom}.\pi_1, {\rm dom}.\pi_2.({\rm cod} \times {\rm id})) \\
  & = & *.({\rm dom}.\pi_1,{\rm dom}.{\rm id}.\pi_2) \\
  & = & *.({\rm dom}.\pi_1,\pi_2) \\
  & = & d_2.\iota_1
\end{eqnarray*}
and
\begin{eqnarray*}
  *.(d_1, d_0.j_1).\iota_2 & = & *.(d_1.\iota_2,d_0) \\
  & = & *.(0.!,{\rm dom}.\pi_2) \\
  & = & {\rm dom}.\pi_2 \\
  & = & d_2.\iota_2.
\end{eqnarray*}

Writing $C := \mathbb{T}_1 \times \mathbb{T}_0 \cup \mathbb{T}_0 \times \mathbb{T}_1$, this means that we can also think of the domain of $\rho_2$ as the pullback
\begin{displaymath}
  \begin{tikzcd}
    C \times_{\mathbb{T}_0} \mathbb{T}_1 \cup C \times_{\mathbb{T}_0} \mathbb{T}_1 \ar[d, "p"] \ar[rr, "q"] & & C \ar[d, "{(d_1,d_0.j_1)}"] \\
    \mathbb{T}_1 \times \mathbb{T}_0 \cup \mathbb{T}_0 \times \mathbb{T}_1 \ar[d, "\mu^*"] \ar[rr, "{[{\rm cod} \times 1, 1 \times {\rm cod}]}"] & & \mathbb{T}_0 \times \mathbb{T}_0 \ar[d,"*"] \\
    \mathbb{T}_1 \ar[rr, "{\rm cod}"] & & \mathbb{T}_0.
  \end{tikzcd}
\end{displaymath}
where
\begin{eqnarray*}
  q & := & [\pi_1,\pi_1], \\
  p & := & (\pi_2,d_0.j_1.\pi_1) \cup (d_1 \times 1).
\end{eqnarray*}
Therefore the domain of $\rho_2$ can be seen as the pushout of two pullbacks, the pullback of $d_1 = [{\rm dom}.\pi_1,0.!]: C \to \mathbb{T}_0$ and ${\rm cod}$ as well as the pullback of \[ d_0.j_1 = {\rm dom}.\pi_2.[{\rm cod} \times {\rm id},1] = [{\rm dom}.{\rm id}.\pi_2, {\rm dom}.\pi_2] = [\pi_2,{\rm dom}.\pi_2]: C \to \mathbb{T}_0 \]
and ${\rm cod}$. In these terms, $\rho_2$ is the unique arrow making
\begin{displaymath}
  \begin{tikzcd}
    C \times_{\mathbb{T}_0} \mathbb{T}_1 \cup C \times_{\mathbb{T}_0} \mathbb{T}_1 \ar[dr, dotted, "\rho_2"] \ar[drrr, bend left, "{[{\rm cod} \times 1, 1 \times {\rm cod}].q}"] \ar[ddr, bend right, "{{\rm comp}.(\mu^*.q,\mu^*.p)}"'] \\
    & C \ar[d, "\mu^*"] \ar[rr, "{[{\rm cod} \times 1, 1 \times {\rm cod}]}"] & & \mathbb{T}_0 \times \mathbb{T}_0 \ar[d, "*"] \\
    & \mathbb{T}_1 \ar[rr, "{\rm cod}"] & & \mathbb{T}_0
  \end{tikzcd}
\end{displaymath}
commute.

We have to prove $\rho_2 = [\rho_1, i_1.\rho_0.(j_1 \times 1)]$. Our strategy will be to prove that these morphisms agree on both summands, and that in turn we will do by showing that on these summands the maps become equal upon postcomposing with both $\mu^*$ and $[{\rm cod} \times 1, 1 \times {\rm cod}]$. Note that the first summand $C \prescript{d_1}{}{\times}_{\mathbb{T}_0}^{\rm cod} \mathbb{T}_1$ is isomorphic to
\[ [(\mathbb{T}_1 \times \mathbb{T}_0) \prescript{{\rm dom}.\pi_1}{}{\times}^{\rm cod}_{\mathbb{T}_0} \mathbb{T}_1] \cup [\mathbb{T}_0 \times \mathbb{T}_1], \]
and in these terms we have
\begin{align*}
  & {\rm comp}.(\mu^*.q, \mu^*.p).\iota_1 & = \\
  & {\rm comp}.(\mu^*.\pi_1,\mu^*.\iota_1.(\pi_2,d_0.j_1.\pi_1)) & = \\
  & {[{\rm comp}.(\mu^*.\iota_1.\pi_1,\mu^*.\iota_1.(\pi_2,\pi_2.\pi_1), {\rm comp}.(\mu^*.\iota_2,\mu^*.\iota_1.({\rm id}^*.0.!,{\rm dom}.\pi_2))]} & = \\
  & {[\mu^*.\iota_1.({\rm comp}.(\pi_1.\pi_1,\pi_2), \pi_2.\pi_1),\mu^*.\iota_2]} & = \\
  & \mu^*.\rho_1
\end{align*}
as well as
\begin{align*}
    & [{\rm cod} \times 1, 1 \times {\rm cod}].q.\iota_1 & = \\
   & [{\rm cod} \times 1, 1 \times {\rm cod}].\pi_1 & = \\
    & [{\rm cod} \times 1.\pi_1,1 \times {\rm cod}] & = \\
   & [({\rm cod}.\pi_1.\pi_1,\pi_2.\pi_1), 1 \times {\rm cod}] & = \\
    & [{\rm cod} \times 1.({\rm comp}.(\pi_1.\pi_1,\pi_2),\pi_2.\pi_1), 1 \times {\rm cod}] & = \\
    & [{\rm cod} \times 1, 1 \times {\rm cod}].\rho_1.
\end{align*}

We now turn to the second summand, which we can write as
\[ C \prescript{d_0.j_1}{}{\times}_{\mathbb{T}_0}^{\rm cod} \mathbb{T}_1 \cong (\mathbb{T}_1 \times \mathbb{T}_0) \times_{\mathbb{T}_0} \mathbb{T}_1 \cup (\mathbb{T}_0 \times \mathbb{T}_1) \times_{\mathbb{T}_0} \mathbb{T}_1 \cong \mathbb{T}_1 \times \mathbb{T}_1 \cup \mathbb{T}_0 \times (\mathbb{T}_1 \times_{\mathbb{T}_0} \mathbb{T}_1), \]
because $d_0.j_1 = [\pi_2,{\rm dom}.\pi_2]$. In these terms we have
\begin{align*}
  & {\rm comp}.(\mu^*.q, \mu^*.p).\iota_2 & = \\
  & {\rm comp}.(\mu^*.\pi_1,\mu^*.\iota_2.d_1 \times 1) & = \\
  & [{\rm comp}.(\mu^*.\iota_1.(\pi_1,{\rm cod}.\pi_2),\mu^*.\iota_2.({\rm dom}.\pi_1,\pi_2)),{\rm comp}.(\mu^*.\iota_2.(\pi_1,\pi_1.\pi_2), \mu^*.\iota_2.(0.!,\pi_2.\pi_2))] & = \\
  & [\mu^*.\iota_2.({\rm cod}.\pi_1,\pi_2),\mu^*.\iota_2.1 \times {\rm comp}] & =
\end{align*}
which equals:
\begin{align*}
  & \mu^*.i_1.\rho_0.(j_1 \times 1) & = \\
  & \mu^*.\iota_2.(1 \times {\rm comp}).([{\rm cod} \times {\rm id}, 1] \times 1) & = \\
  & \mu^*.\iota_2.[({\rm cod}.\pi_1, {\rm comp}.({\rm id}.{\rm cod}.\pi_2, \pi_2)),1 \times {\rm comp}] & = \\
  & \mu^*.\iota_2.[({\rm cod}.\pi_1,\pi_2), 1 \times {\rm comp}].
\end{align*}
In a similar fashion we have:
\begin{align*}
  & [{\rm cod} \times 1, 1 \times {\rm cod}].i_1.\rho_0.(j_1 \times 1) & = \\ & [{\rm cod} \times 1, 1 \times {\rm cod}].\iota_2.(1 \times {\rm comp}).([{\rm cod} \times {\rm id}, 1] \times 1) & = \\
  & 1 \times {\rm cod}.[({\rm cod}.\pi_1,\pi_2), 1 \times {\rm comp}] & = \\
  &  [({\rm cod}.\pi_1,{\rm cod}.\pi_2),(\pi_1,{\rm cod}.\pi_1.\pi_2)] & = \\
  &  [{\rm cod} \times 1, 1 \times {\rm cod}].\pi_1 & = \\
  &  [{\rm cod} \times 1, 1 \times {\rm cod}].q.\iota_2.
\end{align*}
This finishes the proof.
\end{proof}

\begin{rema}{onconnectiontodistributivelaw}
  The previous lemma is equivalent to the distributive law in the form:
  \[ \Gamma.\mu = M\mu.(\mu M.(\Gamma.p_1, Mr.p_2), \mu M.(\alpha.(p_2,M!.p_1), \Gamma.p_2)). \]
  To prove that this reformulation is equivalent to the way we usually state the distributive law, one would need an interchange law of the following form: suppose we have elements $\alpha,\beta,\gamma,\delta \in MMX$ with $t(\alpha) = s(\beta), t(\gamma) = s(\delta), Mt(\alpha) = Ms(\gamma), Mt(\beta) = Ms(\gamma)$, then:
  \[ M\mu.(\mu M(\alpha,\beta), \mu M(\gamma,\delta)) = \mu M.(M\mu.(\alpha,\gamma), M\mu.(\beta,\delta)). \]
  This interchange law may not hold in all Moore categories, but it can be shown to hold in the example at hand.

  In any case, to see the equivalence between the previous lemma and the reformulation of the distributive law, one would have to think about how one would prove the latter, and for that we have to go back to the proof of \reflemm{distributivelawOK}. The description of $\Gamma.\mu$ is still correct of course, but note that in the negative direction we can think of it as a map $C \times_{\mathbb{T}_0} \mathbb{T}_1 \to C$:
  \begin{displaymath}
    \begin{tikzcd}
      C \ar[dr] & C \times_{\mathbb{T}_0} \mathbb{T}_1 \ar[d] \ar[r] \ar[l] & \{ (\theta \in \mathbb{T}_0, t: {\rm cod}^{-1}(\theta) \to \mathbb{T}_0, p_0 \in \theta, p_1 \in t(p_0)) \} \ar[d] \\
      & \mathbb{T}_0 \times \mathbb{T}_0 \ar[r] & \{ (\theta, t: {\rm cod}^{-1}(\theta) \to \mathbb{T}_0) \}
    \end{tikzcd}
  \end{displaymath}
  where the arrow along the bottom sends $(\theta_0,\theta_1)$ to $(\theta_0 * \theta_1, d_2 = d_1.\mu^*: C \to \mathbb{T}_0)$. With some effort one can recognise the map $C \times_{\mathbb{T}_0} \mathbb{T}_1 \to C$ as $\rho_2$.

  Let us first look at the map $\mu M.(\Gamma.p_1, Mr.p_2)$, which can be thought of as a composite:
  \begin{displaymath}
    \begin{tikzcd}
      {\rm cod} \times_I {\rm cod} \ar[r] & ({\rm cod} \otimes {\rm cod}) \times_{I \otimes {\rm cod}} ({\rm cod} \otimes {\rm cod}) \ar[d, "\cong"] \\
      &  ({\rm cod} \times_I {\rm cod}) \otimes {\rm cod} \ar[r, "\mu \otimes 1_{\rm cod}"] & {\rm cod} \otimes {\rm cod}.
    \end{tikzcd}
  \end{displaymath}
  In the positive direction this sends $(\theta_0,\theta_1)$ first to $((\theta_0, \lambda p.{\rm dom}(p),(\theta_1,\lambda p.0))$ and then to what is essentially $(\theta_0 * \theta_1, d_1 = [{\rm dom}.\pi_1,0.!]: C \to \mathbb{T}_0)$. Then in the negative direction we have a map $C^{d_1} \times_{\mathbb{T}_0}^{\rm cod} \mathbb{T}_1 \to C$ which with some effort one can recognise as $\rho_1$.

  Let us now have a look at $\mu M.(\alpha.(p_2,M!.p_1), \Gamma.p_2)$. We can think of it as a morphism
  \begin{displaymath}
    \begin{tikzcd}
      {\rm cod} \times_I {\rm cod} \ar[r] & ({\rm cod} \otimes {\rm cod}) \times_{I \otimes {\rm cod}} ({\rm cod} \otimes {\rm cod}) \ar[d, "\cong"] \\
      &  ({\rm cod} \times_I {\rm cod}) \otimes {\rm cod} \ar[r, "\mu \otimes 1_{\rm cod}"] & {\rm cod} \otimes {\rm cod}.
    \end{tikzcd}
  \end{displaymath}
  In the positive direction this sends a pair $(\theta_0, \theta_1)$ first to $((\theta_0,\lambda p.\theta_1), (\theta_1,\lambda p.{\rm dom}(p))$ which then gets sent to what is essentially $(\theta_0 * \theta_1, d_0.j_1 = [\pi_2,{\rm dom}.\pi_2]: C \to \mathbb{T}_0)$. This means that in the negative direction we should have a map $C \prescript{d_0.j_1}{} {\times}_{\mathbb{T}_0}^{\rm cod} \mathbb{T}_1 \to C$: again with some effort one can recognise this as $i_1.\rho_0.(j_1 \times 1)$.

  Finally, the total right hand side is a map
  \begin{displaymath}
    \begin{tikzcd}
      {\rm cod} \times_I {\rm cod} \ar[r] & ({\rm cod} \otimes {\rm cod}) \times_{{\rm cod} \otimes I} ({\rm cod} \otimes {\rm cod}) \ar[d, "\cong"] \\
      &  {\rm cod} \otimes ({\rm cod} \times_I {\rm cod}) \ar[r, "1_{\rm cod} \otimes \mu"] & {\rm cod} \otimes {\rm cod}.
    \end{tikzcd}
  \end{displaymath}
  In the positive direction this will be a map sending $(\theta_0, \theta_1)$ to $(\theta_0 * \theta_1, *.(d_1,j_0.d_0):C \to \mathbb{T}_0)$; in the backwards direction this is a map $C \times_{\mathbb{T}_0} \mathbb{T}_1 \cup C \times_{\mathbb{T}_0} \mathbb{T}_1 \to C$ which is $[\rho_1,i_1.\rho_0.(j_1 \times 1)]$. This shows that the reformulated distributive law is equivalent to the previous lemma.
\end{rema}

\begin{prop}{inclusionsmallHDRsintobigHDRs}
  There is a double functor $\mathbb{H} \to {\rm HDR}(\widehat{\mathbf{\Delta}})$ which on the level of objects assigns to every traversal its geometric realization.
\end{prop}
\begin{proof} As said, it remains to check that vertical composition of morphisms is preserved. We start by pulling back the vertical composition from the previous lemma along $(\theta, \psi): \Delta^n \to \mathbb{T}_0 \times \mathbb{T}_0$ and obtain the following picture:
\begin{displaymath}
  \begin{tikzcd}
    \Delta^n \ar[dd, "t_{\theta * \psi}"] \ar[dr, "t_{\psi}"] \ar[rr] & & \mathbb{T}_0 \times \mathbb{T}_0  \ar[dr, "1 \times {\rm id}^*"] \ar[dd, "{\iota_2.(1 \times {\rm id}^*)}"', near start] \\
    & \psi \ar[rr, crossing over] \ar[dl, "\iota_2"] & & \mathbb{T}_0 \times \mathbb{T}_1 \ar[dl, "\iota_2"] \\
    \theta * \psi \ar[d] \ar[rr] & & \mathbb{T}_1 \times \mathbb{T}_0 \cup \mathbb{T}_0 \times \mathbb{T}_1 \ar[d, "{[{\rm cod} \times {\rm id},1]}"]  \\
    \psi \ar[rr] \ar[d] & & \mathbb{T}_0 \times \mathbb{T}_1 \ar[d, "1 \times {\rm cod}"] \\
    \Delta^n \ar[rr, "{(\theta,\psi)}"] & & \mathbb{T}_0 \times \mathbb{T}_0
  \end{tikzcd}
\end{displaymath}
Since vertical composition is preserved by pullback (see part 1) and we have a commutative triangle of HDRs on the right, the same is true on the left. Since push forward also preserves vertical composition, we can push this triangle forward along $s_\psi: \Delta^n \to \psi$ to see that our double functor does indeed preserve vertical composition.
\end{proof}

This lengthy verification was only the first step towards proving the main result of this subsection, which is:
\begin{theo}{HgenMoorefib}
The following notions of fibred structure are isomorphic:
\begin{itemize}
  \item Having the right lifting property against the large double category of HDRs in simplicial sets (that is, to be a naive Kan fibration).
  \item Having the right lifting property against the small double category $\mathbb{H}$.
\end{itemize}
In addition, the morphism of discretely fibred concrete double categories $\Drl{HDR} \to \Drl{\mathbb{H}}$ induced by the inclusion of $\mathbb{H}$ into the double category of HDRs in simplicial sets satisfies fullness on squares and is therefore an isomorphism.
\end{theo}
\begin{proof} First, we explain how one can construct a morphism of notions of fibred structure in the other direction. So imagine that we have a map $p: Y \to X$ which has the right lifting property against $\mathbb{H}$ and we are given a lifting problem
\begin{displaymath}
  \begin{tikzcd}
    A \ar[r, "g"] \ar[d, "i"] & Y \ar[d, "p"] \\
    B \ar[r, "f"] \ar[ur, "l", dotted] & X
  \end{tikzcd}
\end{displaymath}
where $(i: A \to B, j, H)$ is an HDR.

In the proof of the existence of a generic HDR, we have seen that for any HDR $(i,j,H)$ there are two morphisms of HDRs, as in
\begin{displaymath}
  \begin{tikzcd}
    \mathbb{T}_0 \ar[d, "{\rm id}^*"] & B \ar[d, "{(1,{\rm id}^*.d)}"] \ar[r, "j"] \ar[l, "d"'] & A \ar[d, "i"] \\
    \mathbb{T}_1 & B \times_{\mathbb{T}_0} \mathbb{T}_1 \ar[l, "p_2"] \ar[r, "\rho"'] & B,
  \end{tikzcd}
\end{displaymath}
where the one on the left is cartesian. So if $b \in B_n$, we can pull the middle HDR back along $b: \Delta^n \to B$ and we obtain a commutative diagram in which the left square is a morphism of HDRs, as follows:
\begin{displaymath}
  \begin{tikzcd}
    \Delta^n \ar[d, "t_\theta"] \ar[r, "j.b"] & A \ar[d, "i"] \ar[r, "g"] & Y \ar[d, "p"] \\
    \theta \ar[r, "\pi"'] \ar[urr, "L_b", dotted, near start] & B \ar[r, "f"'] & X,
  \end{tikzcd}
\end{displaymath}
if $d(b) = \theta$. By assumption, we have a dotted filler $L_b$, as indicated. Now we put $l(b) := L_b.s_\theta$.

Let us first check that this defines a natural transformation $l: B \to Y$. If we consider $b \cdot \alpha$ for some $\alpha: \Delta^m \to \Delta^n$, then we have a picture as follows:
\begin{displaymath}
  \begin{tikzcd}
    \Delta^m \ar[r, "\alpha"] \ar[d, "t_{\theta \cdot \alpha}"] & \Delta^n \ar[d, "t_\theta", near start] \ar[r, "j.b"] & A \ar[d, "i"] \ar[r, "g"] & Y \ar[d, "p"] \\
    \theta \cdot \alpha \ar[r, "\widehat{\alpha}"'] \ar[urrr, "L_{b \cdot \alpha}", dotted, near start] & \theta \ar[r, "\pi"'] \ar[urr, "L_b"', near end, dotted] & B \ar[r, "f"'] & X,
  \end{tikzcd}
\end{displaymath}
Hence \[ l(b \cdot \alpha) = L_{b \cdot \alpha}.s_{\theta \cdot \alpha} = L_b.\widehat{\alpha}.s_{\theta \cdot \alpha} = L_b.s_\theta.\alpha = l(b).\alpha, \]
which shows that $l$ is indeed a natural transformation.

Let us now check that $l$ fills the original square. Because $\pi.s_\theta = b$, we have
\[ p.l(b) = p.L_b.s_\theta = f.\pi.s_\theta = f(b), \]
hence $l$ makes the lower triangle commute. Also, if $b = i(a)$ for some $a \in A_n$, then $d(b) = \langle \rangle$ (the empty traversal), and
\[  l.i(a)=L_b.s_{\langle \rangle} = L_b.t_{\langle \rangle} = g.j(b) = g.j.i(a) = g(a). \]

It remains to check that these lifts satisfy both the horizontal and vertical compatibility conditions. We start by looking at the horizontal one. Imagine we have a commutative diagram of the form:
\begin{displaymath}
  \begin{tikzcd}
    A' \ar[r, "\alpha"] \ar[d, "i'"] & A \ar[r] \ar[d, "i", near start] & Y \ar[d, "p"] \\
    B' \ar[r, "\beta"'] \ar[urr, dotted, "l'", near start] & B \ar[r] \ar[ur, "l"', dotted] & X
  \end{tikzcd}
\end{displaymath}
where the square on the left is a morphism of HDRs and $b' \in B'_n$. Then this fits into a larger picture:
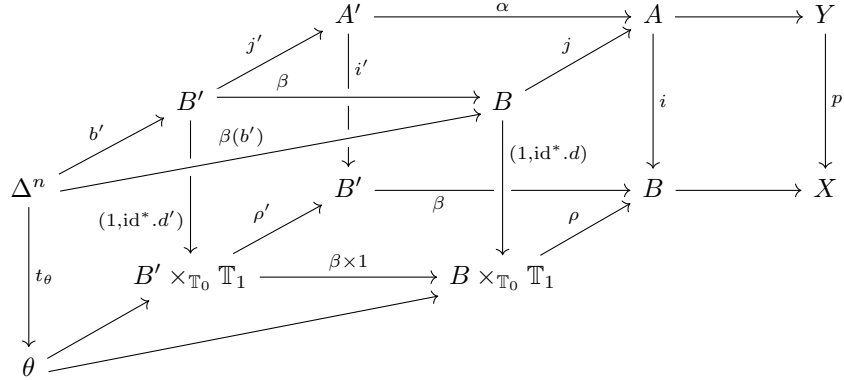
\begin{displaymath}
  \begin{tikzcd}
    & & A' \ar[rr, "\alpha"] \ar[dd, "i'", near start] & & A \ar[rr] \ar[dd, "i"] & & Y \ar[dd, "p"] \\
    & B' \ar[ur, "j'"] \ar[rr, "\beta", crossing over, near start, shift left] \ar[dd, "{(1, {\rm id}^*.d')}"', near end] & & B \ar[ur, "j"]  \\
    \Delta^n \ar[ur, "b'"] \ar[urrr,  "\beta(b')", shift right, crossing over] \ar[dd, "t_\theta"] & & B' \ar[rr, "\beta"', near start] & & B \ar[rr] & & X \\
    & B' \times_{\mathbb{T}_0} \mathbb{T}_1 \ar[rr, "\beta \times 1"] \ar[ur, "\rho'"] & & B \times_{\mathbb{T}_0} \mathbb{T}_1 \ar[ur, "\rho"] \ar[from = uu, "{(1, {\rm id}^*.d)}", near start, crossing over] \\
    \theta \ar[ur] \ar[urrr, shift right]
    \end{tikzcd}
\end{displaymath}
with $d' = d.\beta$. So from left to right we obtain a lift $\pi: \theta \to Y$ and $l'(b') = \pi.s_\theta$. But because the front face of the cube is a cartesian morphism of HDRs (as one easily checks), $l(\beta(b'))$ is computed in the same way. This shows the horizontal compatibility condition.

For checking the vertical compatibility condition, imagine that we have a lifting problem of the form:
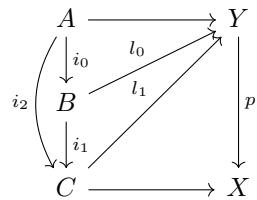
\begin{displaymath}
  \begin{tikzcd}
    A \ar[rr] \ar[d, "i_0"] \ar[dd, bend right, "i_2"'] & & Y \ar[dd, "p"] \\
    B \ar[d, "i_1"] \ar[urr, "l_0"] \\
    C \ar[rr] \ar[uurr, "l_1"] & & X
  \end{tikzcd}
\end{displaymath}
with the HDR $i_2$ the vertical composition of $i_0$ and $i_1$, $l_0$ and $l_1$ the lifts induced by $i_0$ and $i_1$, respectively, and $c \in C_n$. The aim is to show that for the induced lift $l_2: C \to Y$ induced by $i_2$ we have $l_2(c) = l_1(c)$. Using the formulas for the vertical composition of HDRs (in the language of presheaves) we obtain a commutative diagram, as follows:
\begin{displaymath}
  \begin{tikzcd}
    & \Delta^n \ar[rr, "c"] \ar[dd, "t_{\theta_0}"] & & C \ar[rr, "j_1"] \ar[dd, "{(1,{\rm id}^*.d_0.j_1)}"] & & B \ar[dd, "{(1, {\rm id}^*.d_0)}"] \ar[rr, "j_0"] & & A \ar[dd, "i_0"] \\ \\
    & \theta_0 \ar[rr] \ar[dd, "\iota_2", near start] & & C \times_{\mathbb{T}_0} \mathbb{T}_1 \ar[rr, "j_1 \times 1"] \ar[dd, "\iota_2"] & & B \times_{\mathbb{T}_0} \mathbb{T}_1 \ar[rr, "\rho_0", shift left] & & B \ar[dd, "i_1"] \\
    \Delta^n \ar[ur, "s_{\theta_0}"] \ar[rr, crossing over] \ar[dd, "t_{\theta_1}"] & & C  \ar[ur, "{(1, {\rm id}.d_0.j_1)}"] \ar[urrrrr, "j_1", crossing over, shift right] \\
    & \theta_1 * \theta_0 \ar[rr] & & C \times_{\mathbb{T}_0} \mathbb{T}_1 \cup C \times_{\mathbb{T}_0} \mathbb{T}_1 \ar[rrrr, "{\rho_2 = [\rho_1, i_1.\rho_0.(j_1 \times 1)]}", shift left] & & & & C \\
    \theta_1 \ar[ur, "\iota_1"] \ar[rr] & & C \times_{\mathbb{T}_0} \mathbb{T}_1 \ar[ur, "\iota_1"] \ar[urrrrr, "\rho_1"] \ar[from=uu, "{(1, {\rm id}^*.d_1)}", near start, crossing over, shift right]
  \end{tikzcd}
\end{displaymath}
Here $d_2(c) = \theta_1 * \theta_0$ with $\theta_0 = (d_0.j_1)(c)$ and $\theta_1 = d_1(c)$. Note that $l_2(c) = \pi_2.s_{\theta_1 * \theta_0}$ where $\pi_2: \theta_1 * \theta_0 \to Y$ is the lift coming from the lifting structure of $p$ against $\mathbb{H}$. However, $\pi_2.s_{\theta_1 * \theta_0}$ can be computed in two steps: we first compute the lift $\pi_0: \theta_0 \to Y$. Then we use $\pi_0.s_{\theta_0}$ to compute $\pi_1: \theta_1 \to Y$ and then we have $\pi_2.s_{\theta_1 * \theta_0} = \pi_1.s_{\theta_1}$. But it follows from the diagram that $\pi_0.s_{\theta_0} = l_0(j_1.c)$ and $\pi_1.s_{\theta_1} = l_1(c)$: so $l_2(c) = l_1(c)$, as desired.

We have constructed two operations between two notions of fibred structure: now it remains to show that they are mutually inverse. It is easy to see that if we start from a map having the RLP against all HDRs, then only remember the lifts against the vertical maps in $\mathbb{H}$ and then use the operation defined above to compute a lift against a general HDR, we return at our starting point. The reason is simply that the left hand square in
\begin{displaymath}
  \begin{tikzcd}
    \Delta^n \ar[d, "t_\theta"] \ar[r, "j.b"] & A \ar[d, "i"] \ar[r, "g"] & Y \ar[d, "p"] \\
    \theta \ar[r, "\pi"'] \ar[urr, "L", dotted, near start] & B \ar[r, "f"'] \ar[ur, "l"', dotted] & X,
  \end{tikzcd}
\end{displaymath}
is a morphism of HDRs, so we must have $L = l.\pi$. So if $b = \pi.s_\theta$, then $L.s_\theta = l.\pi.s_\theta = l(b)$. This argument also shows fullness on squares.

The converse turns out to be a lot harder. Suppose that we start with a map $p$ having the right lifting property against all maps in $\mathbb{H}$, and that we are given a lifting problem of $p$ against a vertical map from $\mathbb{H}$. Now we can find a solution in two different ways: first, we can use the lifting structure of $p$ directly. Alternatively, we can use that vertical maps in $\mathbb{H}$ are HDRs and use this to find a lift, following the procedure explained above. The question is: are both solutions necessarily the same? We claim that the answer is yes.

To prove the claim, it suffices to check that the lifts against traversals of length 1 are identical. The reason is that any inclusion of traversals can be written as the vertical composition of inclusions where each next traversal has length one longer than the previous. And if we have a traversal of the form $\sigma \to <(i,\pm)> * \sigma$, then there is a bicartesian square in $\mathbb{H}$ of the form:
\begin{displaymath}
  \begin{tikzcd}
    <> \ar[r] \ar[d] & \sigma \ar[d] \\
    <(i,\pm)> \ar[r] & <(i,\pm)> * \sigma
  \end{tikzcd}
\end{displaymath}
So the lift against the map on the right is completely determined by the lift against the map on the left.

So let us imagine that we have an inclusion of traversals of the form $<> \to <(i,+)>$ (we will only look at the positive case, for simplicity). Then its geometric realisation is:
\begin{displaymath}
  \begin{tikzcd}
    \Delta^n \ar[d, "d_i"] \ar[r, "{<(i,+)>}"] & \mathbb{T}_0 \ar[d, "{\rm id}^*"] \\
    \Delta^{n+1} \ar[d, "s_i"] \ar[r, "u"] & \mathbb{T}_1 \ar[d, "{\rm cod}"] \\
    \Delta^n \ar[r, "{<(i,+)>}"] & \mathbb{T}_0
  \end{tikzcd}
\end{displaymath}
where $u = k_{<i,+>}$ picks out the traversal $<(i,+)> \cdot s_i = <(i+1,+),(i,+)>$ with the special position (so the position in the middle). Note that this means that face maps $d_i: \Delta^n \to \Delta^{n+1}$ are HDRs: let us see what its HDR-structure is in presheaf language. First of all, we have $d = {\rm dom}.u = <(i,+)>: \Delta^{n+1} \to \mathbb{T}_0$. Secondly, we have to determine $\rho: \Delta^{n+1} \times_{\mathbb{T}_0} \mathbb{T}_1 \to \Delta^{n+1}$. But note that the domain of $\rho$ also arises as the centre left object in
\begin{displaymath}
  \begin{tikzcd}
    \Delta^{n+1} \ar[d, "d_i"] \ar[r, "d"] & \mathbb{T}_0 \ar[d, "{\rm id}^*"] \\
    \Delta^{n+2} \ar[r, "v"] \ar[d, "s_i"] & \mathbb{T}_1 \ar[d, "{\rm cod}"] \\
    \Delta^{n+1} \ar[r, "{d = <(i,+)>}"'] & \mathbb{T}_0
  \end{tikzcd}
\end{displaymath}
where $v$ chooses $<(i,+)> \cdot s_i$ with the special position; in other words, it is isomorphic to $\Delta^{n+2}$. This means that $\rho$ is the unique map filling
\begin{displaymath}
  \begin{tikzcd}
    \Delta^{n+2} \ar[d, "{(u \cdot s_i,v)}"'] \ar[r, "\rho", dotted] \ar[rr, bend left, "s_i.s_i"] \ar[dr, "q"] & \Delta^{n+1} \ar[d, "u"] \ar[r, "s_i"] & \Delta^n \ar[d, "{(<i,+>)}"] \\
    \mathbb{T}_1 \times_{\mathbb{T}_0} \mathbb{T}_1 \ar[r, "{\rm comp}"'] & \mathbb{T}_1 \ar[r, "{\rm cod}"'] & \mathbb{T}_0.
  \end{tikzcd}
\end{displaymath}
Note that $u \cdot s_i$ is $<(i+2,+), (i+1,+), (i,+)>$ with the position between the first and second element, and therefore $q$ is $<(i+2,+), (i+1,+), (i,+)>$ with the position between the second and third elements. We conclude that $\rho$ must be $s_{i+1}$.

What we have to prove, then, is that if we have a lifting problem of the form
\begin{displaymath}
  \begin{tikzcd}
    \Delta^n \ar[r] \ar[d, "d_i"] & Y \ar[d, "p"] \\
    \Delta^{n+1} \ar[ur, dotted] \ar[r] & X
  \end{tikzcd}
\end{displaymath}
where we think of the arrow on the left as an HDR coming from the inclusion of traversals $<> \to <(i,+)>$, then the lift coming from the fact that it is an HDR coincides with the one coming from the fact that $p$ has the right lifting property against $\mathbb{H}$. The former lift is computed by choosing an arbitrary element $\alpha \in \Delta^{n+1}$ and pulling the map on the left back along $s_i.\alpha$. But because $\Delta^{n+1}$ is representable it suffices to do this for $\alpha = 1$, which means that what we have to prove is that if we have a picture as follows:
\begin{displaymath}
  \begin{tikzcd}
    \Delta^{n+1} \ar[r, "s_i"] \ar[d, "d_i"] & \Delta^n \ar[r] \ar[d, "d_i"] & Y \ar[d, "p"] \\
    \Delta^{n+2} \ar[r, "s_{i+1}"'] \ar[urr, "l_2", dotted, near start]& \Delta^{n+1} \ar[ur, "l_1"', dotted] \ar[r] & X
  \end{tikzcd}
\end{displaymath}
then for the lifts coming from the fact that $p$ has the right lifting property against $\mathbb{H}$, we have $l_1 = l_2.d_{i+1}$. For that it suffices to prove $l_1.s_{i+1} = l_2$. (Just to be clear, the left and centre vertical arrows are to be thought of as HDRs coming from the inclusion of traversals $<> \to <(i,+)>$ in dimensions $n+1$ and $n$, respectively.)

To that end, note that we have a commutative diagram of the form:
\begin{displaymath}
 \begin{tikzcd}
   \Delta^{n+1} \ar[r, "s_i"] \ar[d, "d_i"] & \Delta^n \ar[r] \ar[dd, "d_i"] & Y \ar[dd, "p"] \\
   <(i,+)> \ar[dr, "w_2"] \ar[d, "\iota_2"] \ar[urr, dotted, "l_2"] \\
   <(i+1,+),(i,+)> \ar[r, "w"] \ar[d, "{[s_{i+1},s_i]}"] \ar[uurr, dotted, "l_3"] & <(i,+)> \ar[r] \ar[d, "s_i"] \ar[uur, dotted, "l_1"] & X \\
   \Delta^{n+1} \ar[r, "s_i"] & \Delta^n
 \end{tikzcd}
\end{displaymath}
Here we have pulled $d_i$ back along $s_i$ and then decomposed it vertically. This means that the map $w = \widehat{s_i}$ is the unique map making
\begin{displaymath}
 \begin{tikzcd}
  <(i+1,+),(i,+)> \ar[r, "w", dotted] \ar[d, "{[s_{i+1},s_i]}"] \ar[rr, bend left, "k_{<(i+1,+),(i,+)>}"] & <(i,+)> \ar[d, "s_i"] \ar[r, "k_{<(i,+)>}"'] & \mathbb{T}_1 \ar[d, "{\rm cod}"] \\
  \Delta^{n+1} \ar[r, "s_i"] & \Delta^n \ar[r, "{<(i,+)>}"] & \mathbb{T}_0
 \end{tikzcd}
\end{displaymath}
commute. Here $k_{<(i+1,+),(i,+)>} = [z_1,z_2]$ with $z_1$ choosing the position between the first and second element in $<(i+2,+), (i+1,+), (i,+)>$ and $z_2$ choosing the position between the second and third. Therefore $w = [s_i, s_{i+1}]$ and $w_2 = s_{i+1}$. This means that the top left square in the last but one diagram coincides with the left hand square in the diagram before that. Therefore also the lifts $l_1$ and $l_2$ in both diagram must coincide and since they are compatible with $l_3$, we deduce $l_1.s_{i+1} = l_2$, as desired. This completes the proof for the positive case: the negative case is similar.
\end{proof}

\subsection{Naive Kan fibrations in simplicial sets}
From the previous theorem, we get two new descriptions of the naive Kan fibrations. Both start by observing that the entire lifting structure against $\mathbb{H}$ is already determined by a subclass of the vertical maps. First of all, we can consider those inclusions $<> \to \theta$ with empty domain: any other lift is completely determined by these, because
\begin{displaymath}
 \begin{tikzcd}
   <> \ar[d, "t_\theta"] \ar[r, "s_\psi"] & \psi \ar[d, "\iota_2"] \\
   \theta \ar[r, "\iota_1"] & \theta * \psi
 \end{tikzcd}
\end{displaymath}
is a cocartesian square. Therefore the lift against the arrow on the right has to be the pushout of the lift against the arrow of the right. So one can equivalently define a naive Kan fibration structure in terms of lifts against arrows of the form $<> \to \theta$. If one does so, the horizontal compatibility condition for maps of the form $(1,\sigma)$ drops out and we are left with the horizontal compatibility condition for maps of the form $(\alpha,<>)$. In other words, we have:
\begin{coro}{nocondition4}
 The following notions of fibred structure are equivalent:
\begin{itemize}
 \item To assign to a map $p: Y \to X$ all its naive Kan fibration structures.
 \item To assign to a map $p: Y \to X$ a function which given any $n$-dimensional traversal $\theta$ and commutative square
 \begin{displaymath}
   \begin{tikzcd}
     \Delta^n \ar[d, "t_\theta"] \ar[r] & Y \ar[d, "p"] \\
     \theta \ar[r] \ar[ur, dotted] & X
   \end{tikzcd}
 \end{displaymath}
  chooses a lift $\theta \to Y$. Moreover, these chosen lifts should satisfy two conditions:
  \begin{enumerate}
   \item[(i)] If $\alpha: \Delta^m \to \Delta^n$, then the chosen lifts
    \begin{displaymath}
     \begin{tikzcd}
     \Delta^m \ar[r, "\alpha"] \ar[d, "t_{\theta \cdot \alpha}"] & \Delta^n \ar[d, "t_\theta"] \ar[r] & Y \ar[d, "p"] \\
     \theta \cdot \alpha \ar[r] \ar[urr, dotted] & \theta \ar[r] \ar[ur, dotted] & X
   \end{tikzcd}
    \end{displaymath}
  are compatible.
  \item[(ii)] If $\theta = \theta_1 * \theta_0$, then the chosen lift $l: \theta \to Y$ can be computed in two steps: we can first compute the lift
  \begin{displaymath}
     \begin{tikzcd}
     \Delta^n \ar[d, "t_\theta"] \ar[rr] & & Y \ar[d, "p"] \\
     \theta_0 \ar[r] \ar[urr, dotted, "l_0"] & \theta \ar[r] & X
   \end{tikzcd}
    \end{displaymath}
    and then compute the chosen lift $l_1$ for
  \begin{displaymath}
     \begin{tikzcd}
     \Delta^n \ar[d, "t_{\theta_1}"] \ar[r, "s_{\theta_0}"] & \theta_0 \ar[r, "l_0"]  \ar[d] & Y \ar[d, "p"] \\
     \theta_1 \ar[r] \ar[urr, "l_1", dotted] & \theta \ar[r] & X
   \end{tikzcd}
    \end{displaymath}
   and push this forward to obtain a map $l: \theta \to Y$ (so $l = [l_1,l_0]$).
  \end{enumerate}
\end{itemize}
In fact, if one sees the second notion of fibred structure can be seen as the vertical structure in a discretely fibred concrete double category in the obvious way, then this concrete double category is isomorphic to the concrete double category of naive Kan fibrations. 
\end{coro}

Using \refcoro{alternativedescriptionofM}, the second bullet in the previous condition can be seen as an externalisation of conditions (1) -- (3) in the definition of a naive fibration structure on $p$ (see \refdefi{strlift}). So the previous proposition says that there is a isomorphism of notions of fibred structure from maps carrying a naive fibration structure satisfying conditions (1) -- (4) to those which only satisfy (1) -- (3). It turns out that this isomorphism is just the forgetful map:

\begin{coro}{onweakliftstrinsimplicialsets}
 If $p: Y \to X$ is a map in simplicial sets, then any map \[ L: Y \times_X MX \to MY \] 
 which satisfies conditions (1) -- (3) for being a naive fibration structure on $p$ as in \refdefi{strlift} also satisfies the remaining fourth condition.
\end{coro}
\begin{proof}
 Let us call a map $L: Y \times_X MX \to MY$ a weak naive fibration structure if it only satisfies conditions (1) -- (3). Then we know that there is an isomorphism between the transport structures on $p$ and the weak naive fibration structures on $p$ obtained by the operations studied in this section. Let us see how we get a transport structure from a weak naive fibration structure in this way. We start by extending the weak naive fibration structure to a right lifting structure against $\mathbb{H}$, which can then be extended to all HDRs. This can be used to find the transport structure $t$ by solving the problem
  \begin{displaymath}
    \begin{tikzcd}
      Y \ar[r, "1"] \ar[d, "{(1,r.p)}"] & Y \ar[d, "p"] \\
      Y \times_X MX \ar[r, "s.p_2"] & X
    \end{tikzcd}
  \end{displaymath}
  using that the map on the left is an HDR, via
  \[ \delta_p = (\alpha.(p_1,M!.p_2), \Gamma.p_2): Y \times_X MX \to M(Y \times_X MX) \cong MY \times_{MX} MMX. \]
  This means that on an arbitrary element $(y, (\theta, \pi: \theta \to X)): \Delta^n \to Y \times_X MX$, the value $t(y,(\theta,\pi))$ is the solution to a lifting problem:
  \begin{displaymath}
    \begin{tikzcd}
      \Delta^n \ar[r, "y"] \ar[d, "t_\theta"] & Y \ar[r, "1"] \ar[d, "{(1,r.p)}"] & Y \ar[d, "p"] \\
      \theta \ar[r, "z"'] \ar[urr, dotted] & Y \times_X MX \ar[r, "s.p_2"'] & X
    \end{tikzcd}
  \end{displaymath}
  where $z$ is the transpose of $\delta_p(y,(\theta,\pi))$. But that means that $s.p_2.z$ is the transpose of $Ms.p_2.(\alpha.(p_1,M!.p_2), \Gamma.p_2)(y, (\theta,\pi)) = (\theta, \pi)$; in other words, $s.p_2.z = \pi$ and the induced lift is $L(y,(\theta, \pi))$. Therefore the induced transport structure is defined by $t(y,(\theta,\pi)) := s.L(y,(\theta,\pi))$.

  So the upshot is that $L \mapsto s.L$ is the isomorphism of notion of fibred structure from the weak naive fibration structures to the transport structures. But we have seen in \refprop{liftvstransport} that this also defines an isomorphism of notions of fibred structure between ordinary naive fibration structures and transport structures. We conclude that every weak naive fibration structure already satisfies condition (4).
\end{proof}

\begin{rema}{oncondition4forMoorefibrationsinsimplicialsets}
  We believe that the previous corollary can also be shown directly. Very roughly, the reason is the following. One can think of $\Gamma$ as being built from path composition and degeneracies, and since any weak naive fibration structure $L$ is in particular a morphism of simplicial sets, it will automatically respect degeneracies. So if $L$ respects path composition, it must also respect $\Gamma$.
\end{rema}

If $p$ is a naive Kan fibration, its lifting structure against $\mathbb{H}$ is also completely determined by its lifts against the inclusions of traversals of the form $<> \to <(i,\pm)>$. Indeed, we already used and explained this in the proof of \reftheo{HgenMoorefib}: any vertical map in $\mathbb{H}$ is a vertical composition of inclusions of traversals where the next traversal has length one more than the previous and each such inclusion is a pushout of one of the form $<> \to <(i,\pm)>$. In the remainder of this section we will determine which compatibility conditions the lifts against these maps have to satisfy in order to extend to a (unique) lifting structure against $\mathbb{H}$. This description will also allow us to prove that the notion of a being a naive Kan fibration is a local notion of fibred structure.

If we are given the lifts against the maps $<> \to <(i,\pm)>$ and we extend them to the entire double category $\mathbb{H}$ in the manner described above, then both the vertical compatibility condition as well as horizontal compatibility condition for maps of form $(1,\sigma)$ are automatically satisfied. So we only need to ensure the horizontal compatibility condition for maps of the form $(\alpha,<>)$. To ensure that, we only need to consider squares where the horizontal map $\alpha$ is either a face or degeneracy maps and the vertical maps on the right is one of the form $<> \to <(i,\pm)>$.

We obtain the following cases:
\begin{enumerate}
  \item[(i)] For the face maps, we have compatibility conditions for the case $k \lt i$ (left) and the case $k \gt i$ (right):
  \begin{displaymath}
    \begin{array}{cc}
      \begin{tikzcd}
        \Delta^{n-1} \ar[r, "d_k"] \ar[d, "d_{(i-1,\pm)^t}"] & \Delta^n \ar[d, "d_{(i,\pm)^t}"] \\
        \Delta^n \ar[d, "s_{i-1}"] \ar[r, "d_k"] & \Delta^{n+1} \ar[d, "s_i"] \\
        \Delta^{n-1} \ar[r, "d_k"] & \Delta^n
      \end{tikzcd} &
      \begin{tikzcd}
        \Delta^{n-1} \ar[r, "d_k"] \ar[d, "d_{(i,\pm)^t}"] & \Delta^n \ar[d, "d_{(i,\pm)^t}"] \\
        \Delta^n \ar[d, "s_{i}"] \ar[r, "d_{k+1}"] & \Delta^{n+1} \ar[d, "s_i"] \\
        \Delta^{n-1} \ar[r, "d_k"] & \Delta^n
      \end{tikzcd}
    \end{array}
  \end{displaymath}
  What we mean here is that we have a horizontal compatibility condition for the top squares in both diagrams below, in that if $p: Y \to X$ is a naive Kan fibration, and we have lifting problem as in
  \begin{displaymath}
    \begin{array}{cc}
      \begin{tikzcd}
        \Delta^{n-1} \ar[r, "d_k"] \ar[d, "d_{(i-1,\pm)^t}"] & \Delta^n \ar[d, "d_{(i,\pm)^t}"] \ar[r] & Y \ar[d, "p"] \\
        \Delta^n \ar[d, "s_{i-1}"] \ar[r, "d_k"] \ar[urr, dotted] & \Delta^{n+1} \ar[d, "s_i"] \ar[r] \ar[ur, dotted] & X \\
        \Delta^{n-1} \ar[r, "d_k"] & \Delta^n
      \end{tikzcd} &
      \begin{tikzcd}
        \Delta^{n-1} \ar[r, "d_k"] \ar[d, "d_{(i,\pm)^t}"] & \Delta^n \ar[d, "d_{(i,\pm)^t}"] \ar[r] & Y \ar[d, "p"] \\
        \Delta^n \ar[d, "s_{i}"] \ar[r, "d_{k+1}"] \ar[urr, dotted] & \Delta^{n+1} \ar[d, "s_i"] \ar[ur, dotted] \ar[r] & X\\
        \Delta^{n-1} \ar[r, "d_k"] & \Delta^n
      \end{tikzcd}
    \end{array}
  \end{displaymath}
  then the dotted lifts have to be compatible. (Note that there is also a case $k = i$ , but it is trivially satisfied, because in that case we get the identity inclusion $<> \to <>$ on the left.)
  \item[(ii)] For the degeneracy maps, we have compatibility condition for the case $k \lt i$ (left) and $k \gt i$ (right):
  \begin{displaymath}
    \begin{array}{cc}
      \begin{tikzcd}
        \Delta^{n-1} \ar[r, "s_k"] \ar[d, "d_{(i+1,\pm)^t}"] & \Delta^n \ar[d, "d_{(i,\pm)^t}"] \\
        \Delta^n \ar[d, "s_{i+1}"] \ar[r, "s_k"] & \Delta^{n+1} \ar[d, "s_i"] \\
        \Delta^{n-1} \ar[r, "s_k"] & \Delta^n
      \end{tikzcd} &
      \begin{tikzcd}
        \Delta^{n-1} \ar[r, "s_k"] \ar[d, "d_{(i,\pm)^t}"] & \Delta^n \ar[d, "d_{(i,\pm)^t}"] \\
        \Delta^n \ar[d, "s_{i}"] \ar[r, "s_{k+1}"] & \Delta^{n+1} \ar[d, "s_i"] \\
        \Delta^{n-1} \ar[r, "s_k"] & \Delta^n
      \end{tikzcd}
    \end{array}
  \end{displaymath}
  as in (i).
  \item[(iii)] Pulling back $<(i,\pm)>$ along $s_i$ is a rather special case, which we split in both a positive and negative case (on the left and right, respectively).
  \begin{displaymath}
    \begin{array}{cc}
      \begin{tikzcd}
        & \Delta^{n+1} \ar[r, "s_i"] \ar[d, "d_i"] & \Delta^n \ar[dd, "d_i"] \\
        \Delta^{n+1} \ar[d, "d_{i+1}"] \ar[r, "d_{i+1}"] & \Delta^{n+2} \ar[dr, "s_{i+1}"] \ar[d, "\iota_2"] \\
        \Delta^{n+2} \ar[r, "\iota_1"] & \Delta^{n+2} \cup_{\Delta^{n+1}} \Delta^{n+2} \ar[r, "{[s_i,s_{i+1}]}"'] \ar[d, "{[s_{i+1}, s_i]}"] & \Delta^{n+1} \ar[d, "s_i"] \\
        & \Delta^{n+1} \ar[r, "s_i"] & \Delta^n
      \end{tikzcd} &
      \begin{tikzcd}
        & \Delta^{n+1} \ar[r, "s_i"] \ar[d, "d_{i+2}"] & \Delta^n \ar[dd, "d_{i+1}"] \\
        \Delta^{n+1} \ar[d, "d_{i+1}"] \ar[r, "d_{i+1}"] & \Delta^{n+2} \ar[dr, "s_i"] \ar[d, "\iota_2"] \\
        \Delta^{n+2} \ar[r, "\iota_1"] & \Delta^{n+2} \cup_{\Delta^{n+1}} \Delta^{n+2} \ar[r, "{[s_{i+1}, s_i]}"'] \ar[d, "{[s_i,s_{i+1}]}"] & \Delta^{n+1} \ar[d, "s_i"] \\
        & \Delta^{n+1} \ar[r, "s_i"] & \Delta^n
      \end{tikzcd}
    \end{array}
  \end{displaymath}
  Therefore we obtain compatibility conditions as follows:
  \begin{enumerate}
    \item[(a)] In the positive case for:
    \begin{displaymath}
      \begin{array}{cc}
        \begin{tikzcd}
          \Delta^{n+1} \ar[r, "s_i"] \ar[d, "d_i"] & \Delta^n \ar[d, "d_i"] \ar[r, "y"] & Y \ar[d, "p"] \\
          \Delta^{n+2} \ar[r, "s_{i+1}"'] \ar[d, "s_i"] \ar[urr, "l_2", dotted, near start] & \Delta^{n+1} \ar[d, "s_i"] \ar[r, "x"'] \ar[ur, "l_1"', dotted] & X \\
          \Delta^{n+1} \ar[r, "s_i"] & \Delta^n
        \end{tikzcd} &
        \begin{tikzcd}
          \Delta^{n+1} \ar[d, "d_{i+1}"] \ar[rr, bend left, "l_2.d_{i+1}"] & \Delta^n \ar[d, "d_i"] \ar[r, "y"] & Y \ar[d, "p"] \\
          \Delta^{n+2} \ar[r, "s_i"'] \ar[d, "s_{i+1}"] \ar[urr, "l_3", dotted, near start] & \Delta^{n+1} \ar[d, "s_i"] \ar[r, "x"'] \ar[ur, "l_1"', dotted] & X \\
          \Delta^{n+1} \ar[r, "s_i"] & \Delta^n
        \end{tikzcd}
      \end{array}
    \end{displaymath}
    The diagram on the left expresses a compatibility condition similar to the previous ones (even in that the top left square is a morphism of HDRs: see the proof of \reftheo{HgenMoorefib}). The one on the right is different, because there is no map $\Delta^{n+1} \to \Delta^n$ making the top left hand square commute. Note that the diagram on the left implies that $l_2.d_{i+1} = l_1.s_{i+1}.d_{i+1} = l_1$, so the reference to $l_2$ in the diagram on the right can be eliminated.
    \item[(b)] In the negative case we have similar compatibility conditions:
    \begin{displaymath}
      \begin{array}{cc}
        \begin{tikzcd}
          \Delta^{n+1} \ar[r, "s_i"] \ar[d, "d_{i+2}"] & \Delta^n \ar[d, "d_{i+1}"] \ar[r, "y"] & Y \ar[d, "p"] \\
          \Delta^{n+2} \ar[r, "s_{i}"'] \ar[d, "s_{i+1}"] \ar[urr, "l_2", near start, dotted] & \Delta^{n+1} \ar[d, "s_i"] \ar[r, "x"'] \ar[ur, "l_1"', dotted] & X \\
          \Delta^{n+1} \ar[r, "s_i"] & \Delta^n
        \end{tikzcd} &
        \begin{tikzcd}
          \Delta^{n+1} \ar[d, "d_{i+1}"] \ar[rr, bend left, "l_2.d_{i+1} = l_1"] & \Delta^n \ar[d, "d_{i+1}"] \ar[r, "y"] & Y \ar[d, "p"] \\
          \Delta^{n+2} \ar[r, "s_{i+1}"'] \ar[d, "s_i"] \ar[urr, "l_3", dotted, near start] & \Delta^{n+1} \ar[d, "s_i"] \ar[r, "x"'] \ar[ur, "l_1"', dotted] & X \\
          \Delta^{n+1} \ar[r, "s_i"] & \Delta^n
        \end{tikzcd}
      \end{array}
    \end{displaymath}
  \end{enumerate}
\end{enumerate}

\begin{rema}{onhowtoreadtheselifts} Note that in this notion of fibred structure we do not just choose lifts for each commutative square with some $d_i: \Delta^n \to \Delta^{n+1}$ on the right: we are also given as input a retraction of $d_i$ (which has to be either $s_{i-1}$ or $s_i: \Delta^{n+1} \to \Delta^n$). So although the lifting problem in no way refers to this retraction, the lifting structure may choose different solutions if $d_i$ comes equipped with a different retraction. Also, the compatibility condition is formulated not for the $d_i$ as such, but for the $d_i$ together with a choice of retraction: indeed, the compatibility condition takes this choice into account in a crucial way.
\end{rema}

From this characterisation we immediately get:
\begin{coro}{Moorefibrationslocal}
 In the category of simplicial sets being a naive Kan fibration is a local notion of fibred structure.
\end{coro}

Another thing which this definition of a naive Kan fibration makes clear is that the traversals with positive and negative orientation live in parallel universes and there are no compatibility conditions relating the two. Indeed, to equip a map with the structure of a naive Kan fibration means equipping it with the structure of a naive right fibration and with the structure of a naive left fibration, with no requirements on how these two structures should relate. Put differently, we have:

\begin{coro}{allisprodleftanright}
In the category of notions of fibred structure, the notion of being a naive Kan fibration is the categorical product of the notion of being a naive right fibration and the notion of being a naive left fibration.
\end{coro}

\section{Mould squares in simplicial sets}\label{sec:mouldsquaresinssets}\index{mould square!in simplicial sets}

In this and the next section we will study \emph{effective Kan fibrations} in simplicial sets. By definition, they are those maps which have the right lifting property against the large triple category of mould squares, with mould squares coming from the simplicial Moore path functor $M$ and the pointwise decidable monomorphism as the cofibrations. The main aim of this section is to show that there is a small triple category of mould squares which generates the same class. In the next section we will use this to show that the effective Kan fibrations in simplicial sets form a local notion of fibred structure.

\begin{rema}{onleftandrightfibr}
The attentive reader will notice that results similar to the ones we derive here hold for the mould squares coming from the two other Moore structures on simplicial sets (see \reftheo{twosidedpathobjcatwithM+} and \refdefi{naivefibrinsimpset}). We will refer to the maps having the right lifting property against the triple category of mould squares coming from $(M_+, \Gamma_+, s)$ as the \emph{effective right fibrations} \index{effective right fibration} and the maps having the right lifting property against the triple category of mould squares coming from $(M_+, \Gamma^*_+, t)$ as the \emph{effective left fibrations} \index{effective left fibration}. Implicitly we will show that these are also generated by suitable small triple categories of mould squares.
\end{rema}

\subsection{Small mould squares} We will define a triple category $\mathbb{M}$ as follows.

\begin{itemize}
\item Objects are triples $(n, S, \theta)$, usually just written $(S, \theta)$, consisting of a natural number $n$, a cofibrant sieve $S \subseteq \Delta^n$ and an $n$-dimensional traversal $\theta$.
\item There is a unique horizontal morphism $(S_0, \theta_0) \to (S_1, \theta_1)$ if $S_0 = S_1$ and $\theta_0$ is a final segment of $\theta_1$.
\item There is a unique vertical morphism $(S_0, \theta_0) \to (S_1,\theta_1)$ if $\theta_0 = \theta_1$ and $S_0 \subseteq S_1$ is an inclusion of cofibrant sieves.
\item Perpendicular morphisms $(T \subseteq \Delta^m, \psi) \to (S \subseteq \Delta^n, \theta)$ are pairs $(\alpha, \sigma)$ with $\alpha: \Delta^m \to \Delta^n$ and $\sigma$ an $m$-dimensional traversal such that $\alpha^* S = T$ and $\psi * \sigma = \theta \cdot \alpha$. Perpendicular composition is given by $(\alpha,\sigma).(\beta,\tau) = (\alpha.\beta, \tau * (\sigma \cdot \beta))$, as before.
\item The triple category is codiscrete in the $xy$-plane in that whenever pairs of horizontal and vertical arrows fit together as in
\begin{displaymath}
  \begin{tikzcd}
    (S_0, \theta_0) \ar[d] \ar[r] & (S_0, \theta_1) \ar[d] \\
    (S_1, \theta_0) \ar[r] & (S_1, \theta_1),
  \end{tikzcd}
\end{displaymath}
then this is the boundary of a unique square. We will refer to such a square as a \emph{small mould square}.\index{small mould square}
\item In the $yz$- and $xz$-plane squares exist as soon as the perpendicular arrows have the same label $(\alpha,\sigma)$ (and the domains and codomains match up), and any two such which are ``parallel'' (have identical boundaries) are identical.
\item The triple category is codiscrete in the third dimension, in that any potential boundary of a cube contains a unique cube filling it.
\end{itemize}

\begin{prop}{triplefunctorfromsmalltolarge}
There is a triple functor $\mathbb{M} \to {\rm MSq}(\widehat{\Delta})$ from the triple category $\mathbb{M}$ to the large triple category of mould squares in simplicial sets.
\end{prop}
\begin{proof}
Perhaps it is good to remind the reader of the structure of the large triple category of mould squares in simplicial sets:
\begin{itemize}
  \item The objects are simplicial sets.
  \item The horizontal morphisms are HDRs.
  \item The vertical morphisms are cofibrations.
  \item The perpendicular morphisms are arbitrary maps of simplicial sets.
  \item The squares in the $xy$-plane are mould squares (morphisms of HDRs which are cartesian over a cofibration).
  \item The squares in the $xz$-plane are morphisms of HDRs.
  \item The squares in $yz$-plane are morphisms of cofibrations (that is, pullbacks).
  \item The cubes are pullback squares of HDRs (of a mould square along an arbitrary morphism of HDRs).
\end{itemize}

The idea is to send the object $(n, S, \theta)$ to the pullback $\theta \cdot S$ in simplicial sets:
\begin{displaymath}
  \begin{tikzcd}
    \theta \cdot S \ar[r] \ar[d] & \theta \ar[d, "j_\theta"] \\
    S \ar[r] & \Delta^n.
  \end{tikzcd}
\end{displaymath}
In the $x$-direction we send $(S, \theta_0) \to (S, \theta_1)$ to the HDR we obtain by pullback:
\begin{displaymath}
  \begin{tikzcd}
    S \ar[r] \ar[d] & \Delta^n \ar[d] \\
    \theta_0 \cdot S \ar[r] \ar[d] & \theta_0 \ar[d] \\
    \theta_1 \cdot S \ar[r] & \theta_1.
  \end{tikzcd}
\end{displaymath}
Note that both squares become cartesian squares of HDRs. Because pullback preserves composition of HDRs, this operation preserves composition in the $x$-direction. Similarly, in the $y$-direction we send $(S_0,\theta) \to (S_1,\theta)$ to the cofibration we obtain by pullback, as follows:
\begin{displaymath}
  \begin{tikzcd}
    \theta \cdot S_0 \ar[r] \ar[d] & \theta \cdot S_1 \ar[r] \ar[d] & \theta \ar[d] \\
    S_0 \ar[r] & S_1 \ar[r] & \Delta^n.
  \end{tikzcd}
\end{displaymath}
From this it immediately follows that the squares in the $xy$-plane are sent to mould squares in simplicial sets.

The next step should be that squares on the left are sent to morphisms of HDRs and squares on the right to morphisms of cofibrations:
\begin{displaymath}
  \begin{array}{cc}
    \begin{tikzcd}
      (T, \psi) \ar[d] \ar[r, "{(\alpha, \sigma)}"] & (S, \theta) \ar[d] \\
      (T, \psi') \ar[r, "{(\alpha,\sigma)}"] & (S, \theta')
    \end{tikzcd} &
    \begin{tikzcd}
      (T, \psi) \ar[d] \ar[r, "{(\alpha, \sigma)}"] & (S, \theta) \ar[d] \\
      (T', \psi) \ar[r, "{(\alpha,\sigma)}"] & (S', \theta)
    \end{tikzcd}
  \end{array}
\end{displaymath}
We will again split this up in the case where $\alpha = 1$ and the case where $\sigma = <>$. If $\alpha = 1$, then $S = T$, $S' = T'$, $\theta = \psi * \sigma$ and $\psi' = \tau * \psi$ and $\theta' = \tau * \theta$ for some traversal $\tau$. In this case, the square on the left is sent to pullback of the right hand square
\begin{displaymath}
  \begin{tikzcd}
    \Delta^n \ar[d, "t_\tau"] \ar[r, "s_\psi"] \ar[rr, "s_{\psi * \sigma}", bend left] & \psi \ar[d, "\iota_2"] \ar[r, "\iota_1"] & \psi * \sigma \ar[d, "\iota_2"] \\
    \tau \ar[r, "\iota_1"] & \tau * \psi \ar[r, "\iota_1"] & \tau * \psi * \sigma.
  \end{tikzcd}
\end{displaymath}
along $S \to \Delta^n$. Since pullback preserves bicartesian morphisms of HDRs (Beck-Chevalley!), the result is a bicartesian morphism of HDRs. In addition, since the outer rectangle and the right hand square in
\begin{displaymath}
  \begin{tikzcd}
    \psi \cdot S \ar[r] \ar[d] & \theta \cdot S \ar[r] \ar[d] & S \ar[d] \\
    \psi \cdot S' \ar[r] & \theta \cdot S' \ar[r] & S'
  \end{tikzcd}
\end{displaymath}
are pullbacks, the square on the left is as well. Therefore the right hand square in the earlier diagram will be sent to a morphism of cofibrations when $\alpha = 1$.

Let us now consider the case $\sigma = <>$; now $\psi = \theta \cdot \alpha$ and  $\psi' = \theta' \cdot \alpha$. Then we need to show that the front face of the bottom cube in
\begin{displaymath}
  \begin{tikzcd}
    & \Delta^m \ar[rr, "\alpha"] \ar[dd] & & \Delta^n \ar[dd] \\
    T \ar[rr, crossing over] \ar[dd] \ar[ur] & & S \ar[ur] \\
    & \psi \ar[rr] \ar[dd] & & \theta \ar[dd] \\
    \psi \cdot T \ar[rr, crossing over] \ar[dd] \ar[ur] & & \theta \cdot S \ar[from = uu, crossing over] \ar[ur] \\
    & \psi' \ar[rr] & & \theta' \\
    \psi' \cdot T \ar[rr] \ar[ur] & & \theta' \cdot S \ar[ur] \ar[from = uu, crossing over]
  \end{tikzcd}
\end{displaymath}
is a morphism of HDRs. But since all the other faces in this cube (besides the left and right one) are cartesian morphisms of HDRs, so must be the front face. Similarly, we need to show that the top of the left cube in
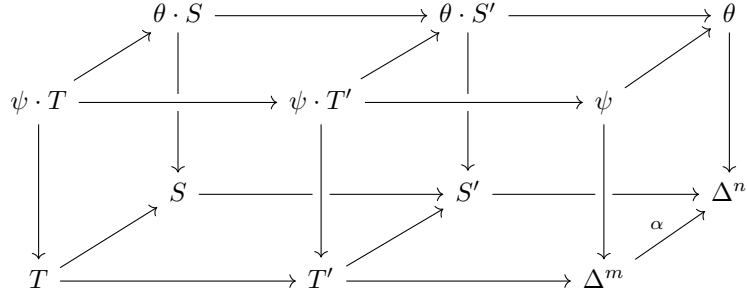
\begin{displaymath}
  \begin{tikzcd}
    & \theta \cdot S \ar[rr] \ar[dd] & & \theta \cdot S' \ar[dd] \ar[rr] & & \theta \ar[dd] \\
    \psi \cdot T \ar[dd] \ar[rr, crossing over] \ar[ur] & & \psi \cdot T' \ar[rr, crossing over] \ar[ur] & & \psi \ar[ur] \\
    & S \ar[rr] & & S' \ar[rr] & & \Delta^n \\
    T \ar[rr] \ar[ur] & & T' \ar[rr] \ar[ur] \ar[from = uu, crossing over] & & \Delta^m \ar[ur, "\alpha"] \ar[from = uu, crossing over]
  \end{tikzcd}
\end{displaymath}
is a pullback. But it is not hard to see that all faces in both cubes must be pullbacks.

From the fact that the the squares in the $yz$-plane are pullbacks, it follows from \refcoro{pbsqofHDRalongcartsq} that the cubes are sent to pullback squares of HDRs.
\end{proof}

\begin{rema}{cartmorphismsinimageofM}
Note that it follows from the proof that the morphisms of HDRs that occur as images of squares in the $xy$-plane are cartesian.
\end{rema}

\begin{theo}{unifKanfibrintermsofsmmouldsqrs}
  The following notions of fibred structure in simplicial sets are isomorphic:
  \begin{itemize}
    \item Having the right lifting property against the large triple category of mould squares (that is, to be an effective Kan fibration).
    \item To have the right lifting property against the small triple category $\mathbb{M}$.
  \end{itemize}
  Indeed, the triple functor from $\mathbb{M}$ to the large triple category of mould squares in simplicial sets induces a morphism of discretely fibred concrete double categories by right lifting properties: this induced morphism of concrete double categories satisfies fullness on squares and is therefore an isomorphism.
\end{theo}
\begin{proof}
For reasons that will become clear later, we will first prove that both notions of fibred structure are equivalent if we ignore the vertical condition on both sides (so on both sides we have lifts satisfying only the horizontal and perpendicular conditions). After we have done that, we will show that the equivalence restricts to one where on both sides the vertical condition is satisfied as well.

So suppose $p: Y \to X$ has the right lifting property against the small mould squares satisfying the horizontal and perpendicular conditions, and assume we are given a lifting problem of the form
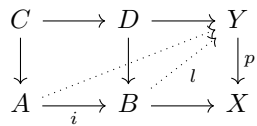
\begin{displaymath}
  \begin{tikzcd}
    C \ar[r] \ar[d] & D \ar[r] \ar[d] & Y \ar[d, "p"] \\
    A \ar[r, "i"'] \ar[urr, dotted] & B \ar[r] \ar[ur, dotted, "l"'] & X
  \end{tikzcd}
\end{displaymath}
where the square on the left is a mould square. We wish to find a map $l: B \to Y$ making everyting commute; for that, assume that we are given some $b \in B_n$. Let us write $(i: A \to B, j, H)$ for the HDR-structure on $i$. As we have seen in the previous section, we can construct a morphism of HDRs
\begin{displaymath}
  \begin{tikzcd}
    \Delta^n \ar[r] \ar[d] & A \ar[d, "i"] \\
    \theta \ar[r, "\pi"] & B
  \end{tikzcd}
\end{displaymath}
with $\pi.s_\theta = b$ and $j(b) = \theta$. By pulling back the mould square along this morphism of HDRs, we obtain a picture as follows:
\begin{displaymath}
  \begin{tikzcd}
    & C \ar[rr] \ar[dd] & & D \ar[rr] \ar[dd] & & Y \ar[dd, "p"] \\
    S \ar[ur] \ar[dd] \ar[rr, crossing over] & & \theta \cdot S \ar[ur] \\
    & A \ar[rr, "i"', near start] \ar[uurrrr, dotted] & & B \ar[rr] & & X \\
    \Delta^n \ar[ur] \ar[rr] & & \theta \ar[ur] \ar[from = uu, crossing over]
  \end{tikzcd}
\end{displaymath}
Since the mould square at the front of the cube belongs to $\mathbb{M}$, the picture induces a map $L_b: \theta \to Y$ making everything commute. We put $l(b) := L_b.s_\theta$, as in the previous section. At this point we need to verify a number of things: that this defines a natural transformation $B \to Y$, that this map fills the square and is compatible with the map $A \to Y$ that we were given. Also, we need to verify that if we choose these lifts for the mould squares, then together these lifts satisfy the horizontal and perpendicular compatibility conditions. Finally, we also need to verify fullness on squares. All of these things are just very minor extensions of results proved in the previous section, so we will omit the proofs here.

We will now show that the operation we have just defined and the one induced by the triple functor from the previous proposition are each other's inverses.
One composite is clearly the identity: if we are given a map $p: Y \to X$ which has the right lifting property against mould squares satisfying the horizontal and perpendicular conditions, restrict it to $\mathbb{M}$ and then extend it all mould squares in the manner described above, then we end up where we started. The reason is simply that the cube in the diagram above is a ``mould cube'' (belongs to the large triple category of mould squares).

The converse is the hard bit: so imagine that we have a map $p: Y \to X$ which has the right lifting property against the small mould squares satisfying the horizontal and perpendicular conditions. This means that if we have a lifting problem of the form:
\begin{displaymath}
  \begin{tikzcd}
    (S_0, \theta_0) \ar[d] \ar[r] & (S_0,\theta_1) \ar[d] \ar[r] & Y \ar[d, "p"] \\
    (S_1,\theta_0) \ar[r] & (S_1, \theta_1) \ar[r] & X,
  \end{tikzcd}
\end{displaymath}
we can solve this problem in two different ways. First of all, we can use the lifting structure of $p$ directly; but we can also observe that the square on the left is a large mould square and use the procedure outlined above to find the lift. The task is to show that both lifts are the same. Again, we argue as in the previous section, by first observing that we can reduce this problem to the situation where $\theta_0 = <>$ and $\theta_1 = <i, \pm>$. Indeed, we can write the mould square on the left as a horizontal composition of small mould squares where the traversal on the right has one entry more than the one on the left. Moreover, we have a mould cube
\begin{displaymath}
  \begin{tikzcd}
    & (S_0, \theta) \ar[dd] \ar[rr] & & (S_0,<i,\pm> * \theta) \ar[dd] \\
    (S_0, <>) \ar[ur, "{(1,\theta)}"] \ar[rr, crossing over] \ar[dd] & & (S_0, <i,\pm>) \ar[ur, "{(1,\theta)}"] \\
    & (S_1, \theta) \ar[rr] & & (S_1, <i,\pm> * \theta) \\
    (S_1,<>) \ar[rr] \ar[ur,  "{(1,\theta)}"] & & (S_1, <i,\pm>) \ar[ur,  "{(1,\theta)}"] \ar[from = uu, crossing over]
  \end{tikzcd}
\end{displaymath}
in which the top and bottom faces are cocartesian. Therefore the lift against the back is completely determined by that the lift against the front face. In fact, we can take it one step further: if $\alpha: \Delta^m \to \Delta^n \in S_1$, then
\begin{displaymath}
  \begin{tikzcd}
    & (S_0, <>) \ar[dd] \ar[rr] & & (S_0,<i, \pm>) \ar[dd] \\
    (\alpha^* S_0, <>) \ar[ur, "{(\alpha,1)}"] \ar[rr, crossing over] \ar[dd] & & (\alpha^* S_0, <i,\pm>) \ar[ur, "{(\alpha,1)}"] \\
    & (S_1, <>) \ar[rr] & & (S_1, <i,\pm>) \\
    (\Delta^m,<>) \ar[rr] \ar[ur,  "{(\alpha,1)}"] & & (\Delta^m, <i,\pm>) \ar[ur,  "{(\alpha,1)}"] \ar[from = uu, crossing over]
  \end{tikzcd}
\end{displaymath}
is a mould cube as well. This means that the lift against the back face is completely determined by the lifts against the front faces if we let $\alpha$ range over $S_1$. In short, we only have to compare lifts against small mould squares of the form:
\begin{displaymath}
  \begin{tikzcd}
    (S, <>) \ar[d] \ar[r] & (S_, <i,\pm>) \ar[d] \ar[r] & Y \ar[d, "p"] \\
    (\Delta^n,<>) \ar[r] \ar[urr, dotted] & (\Delta^n, <i, \pm>) \ar[r] & X.
  \end{tikzcd}
\end{displaymath}
But this can be argued for just as in the previous section, so we again omit the proof.

It remains to check that this equivalence of notions of fibred structure restricts to one where the vertical condition is satisfied one both sides. In fact, we only need to show that if $p: Y \to X$ comes equipped with lifts against the small mould squares (satisfying the vertical condition as well), and we extend this to all mould squares in the manner explained above, then the lifts against all the mould squares satisfy the vertical condition. Before we do that, we make the important point that in this extension of the lifting structure to all mould squares as in
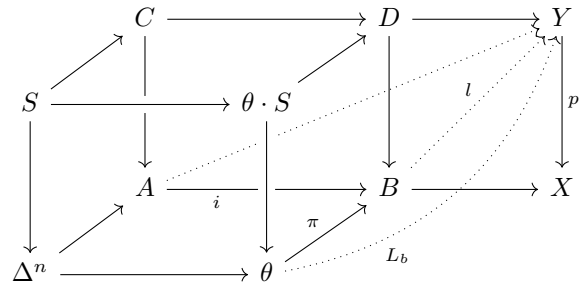
\begin{displaymath}
  \begin{tikzcd}
    & C \ar[rr] \ar[dd] & & D \ar[rr] \ar[dd] & & Y \ar[dd, "p"] \\
    S \ar[ur] \ar[dd] \ar[rr, crossing over] & & \theta \cdot S \ar[ur]  \\
    & A \ar[rr, "i"', near start] \ar[uurrrr, dotted] & & B \ar[rr] \ar[uurr, "l", dotted] & & X \\
    \Delta^n \ar[ur] \ar[rr] & & \theta \ar[ur, "\pi"] \ar[uuurrr, "L_b"', dotted, bend right, near start] \ar[from = uu, crossing over]
  \end{tikzcd}
\end{displaymath}
we must have $l.\pi = L_b$. Indeed, this follows from the fact that the two ways of computing of lifts against small mould squares coincide.

So imagine we have a lifting problem of the form:
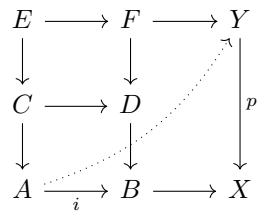
\begin{displaymath}
  \begin{tikzcd}
    E \ar[d] \ar[r] & F \ar[d] \ar[r] & Y \ar[dd, "p"] \\
    C \ar[d] \ar[r] & D \ar[d] \\
    A \ar[r, "i"'] \ar[uurr, dotted, bend right = 20] & B \ar[r] & X
  \end{tikzcd}
\end{displaymath}
in which the two squares on the left are mould squares. Imagine that we have chosen some $b \in B_n$ and constructed our morphism of HDRs from $\Delta^n \to \theta$ to $i$, as before. Then we can pull this vertical composition of mould squares back along this morphism of HDRs, and pull that back along some arbitrary morphism $\alpha \in S_1$, as follows:
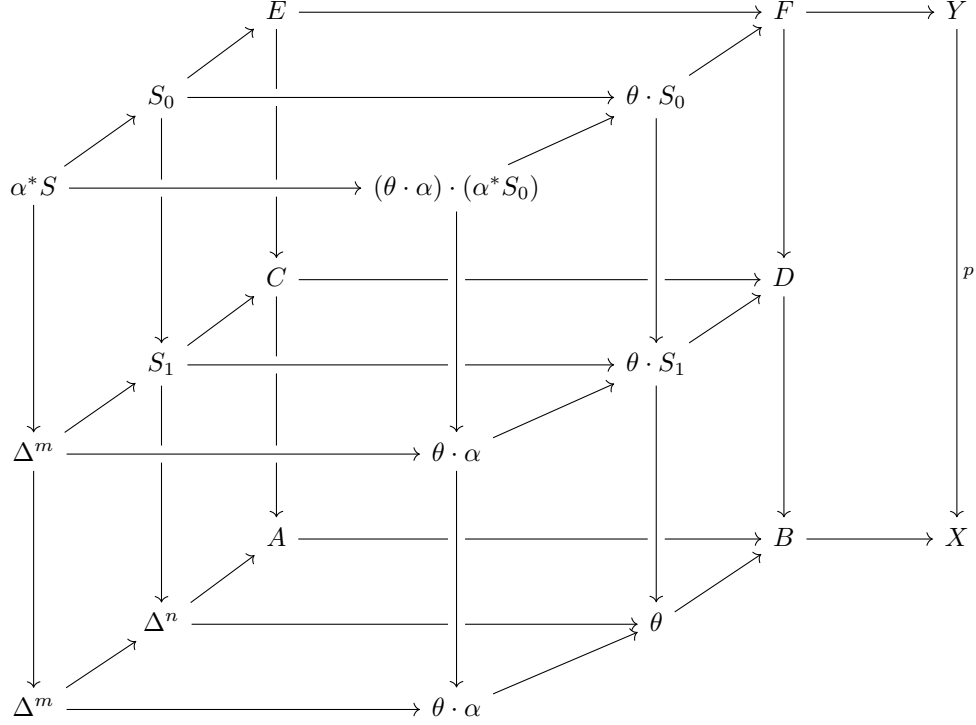
\begin{displaymath}
  \begin{tikzcd}
    & & E \ar[rrr] \ar[ddd] & & & F \ar[ddd] \ar[rr] & &  Y \ar[dddddd, "p"] \\
    & S_0 \ar[ddd] \ar[rrr, crossing over] \ar[ur] & & & \theta \cdot S_0 \ar[ur]  \\
    \alpha^* S \ar[ddd] \ar[rrr, crossing over] \ar[ur] & & & (\theta \cdot \alpha) \cdot (\alpha^* S_0) \ar[ur] \\
    & & C \ar[rrr] \ar[ddd] & & & D \ar[ddd] \\
    & S_1 \ar[rrr, crossing over] \ar[ddd] \ar[ur] & & & \theta \cdot S_1 \ar[from = uuu, crossing over] \ar[ur] \\
    \Delta^m \ar[ddd] \ar[rrr, crossing over] \ar[ur] & & & \theta \cdot \alpha \ar[ur] \ar[from = uuu, crossing over] \\
    & & A \ar[rrr] & & & B \ar[rr] & &  X \\
    & \Delta^n \ar[ur] \ar[rrr] & & & \theta \ar[ur] \ar[from = uuu, crossing over] \\
    \Delta^m \ar[ur] \ar[rrr] & & & \theta \cdot \alpha \ar[ur] \ar[from = uuu, crossing over]
  \end{tikzcd}
\end{displaymath}
What this amounts to is saying that our lift $L_b: \theta \to Y$ can be computed by first computing the map $L_{b \cdot S_1}: \theta \cdot S_1 \to Y$. But that map is completely determined by the maps $L_{b \cdot \alpha}: \theta \cdot \alpha \to Y$ with $\alpha$ ranging over $S_1$. From this the vertical condition for the large mould squares at the back follows.
\end{proof}

\subsection{Effective Kan fibrations in terms of ``filling''} If $p: Y \to X$ has the right lifting property against $\mathbb{M}$, then it comes equipped with a choice of lifts against every small mould square, where these lifts satisfy several compatibility conditions. Because of these compatibility conditions some of the lifts are completely determined by the choices we made for other lifts. What we can do is try to identify a suitable subclass and express the compatibility conditions purely in terms of lifts against elements in this smaller subclass. This is the game we have played already a number of times. For the small mould squares, we will take this to the limit in the next section, but here we can already note that the lifts general mould squares are completely determined by those of the form:
\begin{displaymath}
	\begin{tikzcd}
		(S, <>) \ar[r] \ar[d] & (S, \theta) \ar[d] \\
		(\Delta^n, <>) \ar[r] & (\Delta^n, \theta).
	\end{tikzcd}
\end{displaymath}
(Here $\Delta^n$ stands for the maximal sieve on $\Delta^n$.) Indeed, we have already implicitly argued for this in the previous proof. Indeed, if we have a small mould square of the form
\begin{displaymath}
\begin{tikzcd}
(S_0, \psi) \ar[r] \ar[d] & (S_0, \theta * \psi) \ar[d] \\
(S_1, \psi) \ar[r] & (S_1, \theta * \psi)
\end{tikzcd}
\end{displaymath}
then there is a mould cube of the form
\begin{displaymath}
\begin{tikzcd}
& (S_0, \psi) \ar[dd] \ar[rr] & & (S_0,\theta * \psi) \ar[dd] \\
(S_0, <>) \ar[ur, "{(1,\psi)}"] \ar[rr, crossing over] \ar[dd] & & (S_0, \theta) \ar[ur, "{(1,\psi)}"]  \\
& (S_1, \psi) \ar[rr] & & (S_1, \theta * \psi) \\
(S_1,<>) \ar[rr] \ar[ur,  "{(1,\psi)}"] & & (S_1, \theta) \ar[ur,  "{(1,\psi)}"] \ar[from = uu, crossing over]
\end{tikzcd}
\end{displaymath}
in which the top and bottom faces are cocartesian: therefore the lifts against the back in completely determined by the lift against the front. Furthermore, if we have a mould square as in the front of this mould cube, it occurs at the back of a mould cube
\begin{displaymath}
\begin{tikzcd}
& (S_0, <>) \ar[dd] \ar[rr] & & (S_0,\theta) \ar[dd] \\
(\alpha^* S_0, <>) \ar[ur, "{(\alpha,1)}"] \ar[rr, crossing over] \ar[dd] & & (\alpha^* S_0, \theta) \ar[ur, "{(\alpha,1)}"]  \\
& (S_1, <>) \ar[rr] & & (S_1, \theta) \\
(\Delta^m,<>) \ar[rr] \ar[ur,  "{(\alpha,1)}"] & & (\Delta^m, \theta) \ar[ur,  "{(\alpha,1)}"] \ar[from = uu, crossing over]
\end{tikzcd}
\end{displaymath}
where $\alpha: \Delta^m \to \Delta^n \in S_1$. Since the $(\alpha,1): (\Delta^m, \theta) \to (S_1,\theta)$ collectively cover $(S_1,\theta)$, any compatible system of lifts against the front squares (while $\alpha$ ranges over $S_1$) descends to a unique lift against the front. Let us call the lift against the back that we obtain in this way \emph{the induced lift}. Then we have the following result, whose proof we omit because it is a variation on a type of argument we have already seen a number of times.

\begin{prop}{mouldfibrintermsofgenmouldincl} The following notions of fibred structure are equivalent:
\begin{itemize}
	\item To assign to each map $p: Y \to X$ all its effective Kan fibration structures.
	\item To assign to each map $p: Y \to X$ all functions which given a natural number $n \in \mathbb{N}$, a cofibrant sieve $S \subseteq \Delta^n$, an $n$-dimensional traversal $\theta$ and a commutative square
\begin{displaymath}
\begin{tikzcd}
\Delta^n \cup \theta \cdot S \ar[r, "m"] \ar[d, "{[t_\theta,i_\theta]}"] & Y \ar[d, "p"] \\
\theta \ar[r, "n"] & X
\end{tikzcd}
\end{displaymath}
choose a filler $\theta \to Y$. Moreover, these chosen fillers should satisfy the following three compatibility conditions:
\begin{enumerate}
\item[(1)] for each $\alpha: \Delta^m \to \Delta^n$ the choice of filler for the composed square
\begin{displaymath}
\begin{tikzcd}
\Delta^m \cup (\theta \cdot \alpha) \cdot (\alpha^* S) \ar[r] \ar[d] & \Delta^n \cup \theta \cdot S \ar[r, "m"] \ar[d] & Y \ar[d, "p"] \\
\theta \cdot \alpha \ar[r] & \theta \ar[r, "n"] & X
\end{tikzcd}
\end{displaymath}
is the composition of $\theta \cdot \alpha \to \theta$ with the chosen filler $\theta \to Y$ for the right hand square.
\item[(2)] if $\theta = \theta_1 * \theta_0$, then the chosen filling for
\begin{displaymath}
\begin{tikzcd}
\Delta^n \cup \theta \cdot S \ar[rr, "{[y_0,[m_1,m_0]]}"] \ar[d] & & Y \ar[d, "p"] \\
\theta \ar[rr, "n"] & & X
\end{tikzcd}
\end{displaymath}
coincides with the one we obtain in the following manner. One can first compute the filler for the composed square
\begin{displaymath}
\begin{tikzcd}
\Delta^n \cup \theta_0 \cdot S \ar[r] \ar[d] \ar[rrr, bend left, "{[y_0,m_0]}"]& \Delta^n \cup\theta \cdot S \ar[rr, "{[y_0,[m_1,m_0]]}"'] \ar[d] & & Y \ar[d, "p"] \\
\theta_0 \ar[r] & \theta \ar[rr, "n"] & & X,
\end{tikzcd}
\end{displaymath}
from which we get an element $y_1: \Delta^n \to Y$ by precomposition with the source map $s_{\theta_0}: \Delta^n \to \theta_0$. Then we can compute the filler for the square
\begin{displaymath}
\begin{tikzcd}
\Delta^n \cup \theta_1 \cdot S \ar[rr, "{[y_1,m_1]}"] \ar[d] & & Y \ar[d, "p"] \\
\theta_1 \ar[r] & \theta \ar[r, "n"] & X.
\end{tikzcd}
\end{displaymath}
By amalgamating the two maps $\theta_i \to Y$ we just constructed, we obtain another map $\theta \to Y$, which is the one which should coincide with the filler for the original square.
\item[(3)] if $S_0 \subseteq S_1 \subseteq \Delta^n$ then the chosen filler for
\begin{displaymath}
 \begin{tikzcd}
   \Delta^n \cup \theta \cdot S_0 \ar[r] \ar[d, "a"'] & Y \ar[dd, "p"] \\
\Delta^n \cup \theta \cdot S_1 \ar[d, "b"']  \\
\theta \ar[r] & X
 \end{tikzcd}
\end{displaymath}
coincides with the one we obtain by first taking the induced lift $\Delta^n \cup \theta \cdot S_1 \to Y$ of $p$ against $a$ and then the chosen lift $\theta \to Y$ of $p$ against $b$.
\end{enumerate}
\end{itemize}
\end{prop}

Using the description of $M$ as a polynomial functor (see \refcoro{alternativedescriptionofM}), one can also express the second item in the previous corollary as follows: to equip a map $p: Y \to X$ with the structure of an effective Kan fibration means choosing a map
\[ L: \sum_{(y, \theta) \in Y \times_X MX} \sum_{\sigma \in \Sigma} MY_{(y, \theta)}^\sigma \to MY \]
such that:
\begin{enumerate}
\item $L$ exhibits $(t,Mp)$ as an effective trivial fibration, that is, $L$ fills
\begin{displaymath}
	\begin{tikzcd}
		MY \ar[r, "1"] \ar[d] & MY \ar[d, "{(t,Mp)}"] \\
		\sum_{(y, \theta) \in Y \times_X MX} \sum_{\sigma \in \Sigma} MY_{(y, \theta)}^\sigma \ar[r] \ar[ur, "L"] & Y \times_X MX
	\end{tikzcd}
\end{displaymath}
and is an algebra map (for the AWFS coming from the dominance).
\item $L(y, \theta_1 * \theta_0, (\sigma, \rho_1 * \rho_0)) = L(y,\theta_1, (\sigma, \rho_1)) * L(s.L(y,\theta_1,(\sigma,\rho_1)),\theta_0,(\sigma,\rho_0))$ for all generalised elements $y \in Y$, $\theta_1, \theta_0 \in MX$, $\sigma \in \Sigma$ and $\rho_1,\rho_0 \in MY^\sigma$.
\end{enumerate}

From this we immediately obtain:
\begin{coro}{uniftrivfibfromunifKanfibr}
	If $p: Y \to X$ is an effective Kan fibration, then
	\[(t, Mp): MY \to Y \times_X MX \] is an effective trivial fibration.
\end{coro}

The fact that there are cartesian natural transformations $\iota^+, \iota^-: X^\mathbb{I} \to MX$ as in the previous section, means that we have pullback squares of the form
\begin{displaymath}
	\begin{tikzcd}
		Y^\mathbb{I} \ar[d, "{(s/t,p^\mathbb{I})}"] \ar[r] & MY \ar[d, "{(t,Mp)}"] \\
		Y \times_X X^\mathbb{I} \ar[r] & Y \times_X MX.
	\end{tikzcd}
\end{displaymath}
And since effective trivial fibrations are stable under pullback, we can deduce:
\begin{coro}{comparisonwithunifKanfibrGS}
	If $p: Y \to X$ is an effective Kan fibration, then
	\[ (s/t,p^\mathbb{I}): Y^\mathbb{I} \to Y \times_X X^\mathbb{I} \]
  are effective trivial fibrations. Therefore effective Kan fibrations are uniform Kan fibrations\index{uniform Kan fibration} in the sense of~\cite{Gambino-Sattler}.
\end{coro}

In view of this result the comment in Box~\ref{box:fromuniformtoeffectives} should now make sense.

\section{Horn squares}\label{sec:hornsquares}

The purpose of this section is to show that our notion of an effective Kan fibration in simplicial sets is both local and classically correct. By the latter we mean that, in a classical metatheory, a map can be equipped with the structure of an effective Kan fibration precisely when it has the right lifting property against horn inclusions (the traditional notion of a Kan fibration). To prove both these statements we will use a characterisation of the effective Kan fibrations in terms of what we will call horn squares.

\subsection{Effective Kan fibrations in terms of horn squares} Recall that the small mould squares are the squares in the $yz$-plane in the triple category $\mathbb{M}$ (see previous section).

\begin{defi}{hornsquare} A small mould square will be called a \emph{one-step mould square}\index{one-step mould square} if in the horizontal direction the length of the traversal increases by one and in the vertical direction the sieve increases by one $m$-simplex which was not yet present, but all whose faces were. Among these one-step squares are the \emph{horn squares}\index{horn square|textbf}
\begin{displaymath}
  \begin{tikzcd}
    (\partial \Delta^n, \langle \rangle) \ar[d] \ar[r] & (\partial \Delta^n, \langle (i, \pm) \rangle ) \ar[d] \\
    (\Delta^n, \langle \rangle) \ar[r] & (\Delta^n, \langle (i, \pm) \rangle)
  \end{tikzcd}
\end{displaymath}
which start from the empty traversal in the horizontal direction and end with the maximal sieve in the vertical direction.
\end{defi}

The reason for the name horn square is the following: a lifting problem for $p$ against a horn square
\begin{displaymath}
  \begin{tikzcd}
    (\partial \Delta^n, \langle \rangle) \ar[d] \ar[r] & (\partial \Delta^n, \langle (i, \pm) \rangle ) \ar[d] \ar[r] & Y \ar[d, "p"]\\
    (\Delta^n, \langle \rangle) \ar[r] \ar[urr, dotted] & (\Delta^n, \langle (i, \pm) \rangle) \ar[r] & X
  \end{tikzcd}
\end{displaymath}
is equivalent to a lifting problem for $p$ against the map from the inscribed pushout of the left hand square to its bottom right corner:
\begin{displaymath}
  \begin{tikzcd}
    \Delta^n \cup \Delta^{n+1} \cdot \partial \Delta^n \ar[d] \ar[r] & Y \ar[d, "p"] \\
    \Delta^{n+1} \ar[r] \ar[ur, dotted] & X.
  \end{tikzcd}
\end{displaymath}
Here $\Delta^{n+1} \cdot \partial \Delta^n = \Lambda^{n+1}_{i, i+1} \cup (d_i \cap d_{i+1})$, that is, $\Delta^{n+1}$ with the interior and the $i$th and $(i+1)$st faces missing. Therefore in the previous square the map on left is the horn inclusion $\Lambda^{n+1}_{i+1} \to \Delta^{n+1}$ in the positive case and the horn inclusion $\Lambda^{n+1}_i \to \Delta^{n+1}$ in the negative case. Note that it follows from this that effective Kan fibration have the right lifting property against horn inclusions, so are Kan fibrations in the usual sense.

\begin{rema}{savingonsomebrackets}
To make notation less cluttered, we will, from now on, write the traversal $<(i,\pm)>$ as $<i, \pm>$.
\end{rema}

\begin{lemm}{hornsqrsdetermineall}
If a map is an effective Kan fibration in that it has the right lifting property against the triple category $\mathbb{M}$ of small mould squares, then this structure is completely determined by its lifts against the horn squares.
\end{lemm}
\begin{proof}
The proof combines three reductions, each of which we have already seen before.

First of all, any inclusion $S \subseteq T$ of sieves can be written as a sequence $S = S_0 \subseteq S_1 \subseteq S_2 \subseteq \ldots \subseteq S_n = T$ where at each point $S_{i+1}$ is obtained from $S_i$ by adding one new $m$-simplex whose boundary was already present in $S_i$. Therefore any small mould square can be decomposed into a grid
\begin{displaymath}
  \begin{tikzcd}
    A_{0,0} \ar[d] \ar[r] & A_{1,0} \ar[d] \ar[r] & A_{2,0} \ar[d] \ar[r] & \ldots \ar[d] \ar[r] & A_{n,0} \ar[d]  \\
    A_{0,1} \ar[d] \ar[r] & A_{1,1} \ar[d] \ar[r] & A_{2,1} \ar[d] \ar[r] & \ldots \ar[d] \ar[r] & A_{n,1} \ar[d] \\
    A_{0,2} \ar[d] \ar[r] & A_{1,2} \ar[d] \ar[r] & A_{2,2} \ar[d] \ar[r] & \ldots \ar[d] \ar[r] & A_{n,2} \ar[d] \\
    \ldots \ar[d] \ar[r] & \ldots \ar[d] \ar[r] & \ldots \ar[d] \ar[r] & \ldots \ar[d] \ar[r] & \ldots \ar[d]  \\
    A_{0,m} \ar[r] & A_{1,m} \ar[r] & A_{2,m} \ar[r] & \ldots \ar[r] & A_{n,m}
  \end{tikzcd}
\end{displaymath}
of one-step mould squares. Therefore the lifts against the one-step mould squares determine everything.

Secondly, as we have already seen in the proof of \reftheo{unifKanfibrintermsofsmmouldsqrs} the lifts against the one-step mould squares are determined by the one-step mould squares starting from the empty traversal. The reason, once again, is that there is a mould cube
\begin{displaymath}
  \begin{tikzcd}
    & (S_0, \theta) \ar[dd] \ar[rr] & & (S_0,<i,\pm> * \theta) \ar[dd] \\
    (S_0, <>) \ar[ur, "{(1,\theta)}"] \ar[rr, crossing over] \ar[dd] & & (S_0, <i,\pm>) \ar[ur, "{(1,\theta)}"]  \\
    & (S_1, \theta) \ar[rr] & & (S_1, <i,\pm> * \theta) \\
    (S_1,<>) \ar[rr] \ar[ur,  "{(1,\theta)}"] & & (S_1, <i,\pm>) \ar[ur,  "{(1,\theta)}"] \ar[from = uu, crossing over]
  \end{tikzcd}
\end{displaymath}
whose bottom face is a pushout.

Thirdly, suppose we have a one-step mould square starting from an empty traversal in the horizontal direction and suppose that in the vertical direction we have the inclusion $S \subseteq T$, where $\alpha: \Delta^m \to \Delta^n$ is the $m$-simplex that has been added to $S$ to obtain $T$. Then
\begin{displaymath}
  \begin{tikzcd}
    \partial \Delta^m \ar[d] \ar[r] & S \ar[d] \\
    \Delta^m \ar[r, "\alpha"] & T
  \end{tikzcd}
\end{displaymath}
is bicartesian and we get a small mould cube of the form:
\begin{displaymath}
  \begin{tikzcd}
    & (S, <>) \ar[dd] \ar[rr] & & (S,<i,\pm>) \ar[dd] \\
    (\partial \Delta^m, <>) \ar[ur, "{(\alpha,<>)}"] \ar[rr, crossing over] \ar[dd] & & (\partial \Delta^m, <i,\pm> \cdot \alpha) \ar[ur, "{(\alpha,<>)}"]  \\
    & (T, <>) \ar[rr] & & (T, <i,\pm>) \\
    (\Delta^m,<>) \ar[rr] \ar[ur,  "{(\alpha,<>)}"] & & (\Delta^m, <i,\pm> \cdot \alpha). \ar[ur,  "{(\alpha,<>)}"] \ar[from = uu, crossing over]
  \end{tikzcd}
\end{displaymath}
Since in this cube the right hand face is a pushout, the lift against the back face is determined by its front face. But because $\alpha$ is monic, the front face is either a horn square or trivial in the horizontal direction, depending on whether $i$ is in the image of $\alpha$ or not.
\end{proof}

The remainder of this section will almost exclusively be devoted to answering the following question: suppose we are given a map $p$ together with chosen lifts against the horn squares. Which conditions do these lifts have to satisfy in order for them to extend to a (necessarily unique) effective Kan fibration structure on $p$? We will answer this question in \reftheo{defunifkanfibrhornsqrs}.

Throughout the following discussion we assume that we have fixed a map $p: Y \to X$ together with a choice of lifts (or pushforwards) with respect to all horn squares. We have seen how these lifts can be extended first to lifts against one-step mould squares and then to lifts against small mould squares. The worry we have to address is whether the reductions in \reflemm{hornsqrsdetermineall} determine these lifts in unambiguous manner. For the one-step mould squares there is no such problem, but for the small mould squares this is far from clear. Indeed, imagine that we have a lifting problem as follows:
\begin{displaymath}
  \begin{tikzcd}
    A_{0,0} \ar[d] \ar[r] & A_{1,0} \ar[d] \ar[r] & A_{2,0} \ar[d] \ar[r] & \ldots \ar[d] \ar[r] & A_{n,0} \ar[d] \ar[r] & Y \ar[dddd]^p \\
    A_{0,1} \ar[d] \ar[r] & A_{1,1} \ar[d] \ar[r] & A_{2,1} \ar[d] \ar[r] & \ldots \ar[d] \ar[r] & A_{n,1} \ar[d] \\
    A_{0,2} \ar[d] \ar[r] & A_{1,2} \ar[d] \ar[r] & A_{2,2} \ar[d] \ar[r] & \ldots \ar[d] \ar[r] & A_{n,2} \ar[d] \\
    \ldots \ar[d] \ar[r] & \ldots \ar[d] \ar[r] & \ldots \ar[d] \ar[r] & \ldots \ar[d] \ar[r] & \ldots \ar[d]  \\
    A_{0,m} \ar[r] & A_{1,m} \ar[r] & A_{2,m} \ar[r] & \ldots \ar[r] & A_{n,m} \ar[r] & X,
  \end{tikzcd}
\end{displaymath}
in which all the little squares are one-step mould squares. We have unambiguous pushforwards for every square $(A_{i,j},A_{i+1,j},A_{j+1,i},A_{i+1,j+1})$ in that we have chosen for every pair of maps $A_{i+1,j} \to Y$ and $A_{i,j+1} \to Y$ over $X$ and under $A_{i,j}$ an extension to a map $A_{i+1,j+1} \to Y$. Then for every map $A_{0,m} \to Y$ we can build a push forward to a map $A_{n,m} \to Y$ by repeatedly taking our chosen push forwards for the one-step mould squares. The first worry is that we can travel through the grid in many different ways and that it is not immediately obvious that we will always end up with the same map $A_{n,m} \to Y$. Still, this is the case, because if both $f_{i,j}: A_{i,j} \to Y$ and $g_{i,j}: A_{i,j} \to Y$ are obtained by repeatedly taking our favourite pushforwards for these little squares, in some order, then one easily proves that $f_{i,j} = g_{i,j}$ by induction on $n = i + j$.

The second (and final) worry is that the grid decomposing a small mould square into one-step mould squares is not uniquely determined. Clearly, we have no choice in how to travel in the horizontal direction, but in the vertical direction we have some choice, coming from the following fact. If $S \subseteq T \subseteq \Delta^n$ is an inclusion of cofibrant sieves and we write it as a sequence $S = S_0 \subseteq S_1 \subseteq S_2 \subseteq \ldots \subseteq S_n = T$ of cofibrant sieves where each $S_{i+1}$ is obtained from $S_i$ by adding a single $m$-simplex whose faces belonged to $S_i$, then this sequence is far from unique. However, any two such sequences can be obtained from each other by repeatedly applying permutations of the following form: if we have such a sequence and somewhere in this sequence we have $U \subseteq V \subseteq W$ where $V$ is obtained from $U$ by adding some $k$-simplex and $W$ is obtained from $W$ by adding some $l$-simplex and the $k$-simplex is not a face of the $l$-simplex (so that the boundaries of both the $l$-simplex and $k$-simplex were already present in $U$), then we can replace this by $U \subseteq V' \subseteq W$ where $V'$ is obtained from $U$ by adding the $l$-simplex and $W$ is obtained from $V'$ by adding the $k$-simplex. Since our answer to the first worry tells us that we may always assume that the way one finds the lift against a grid as above is by computing all the lifts $A_{i,j} \to Y$ in lexicographic order, we end up with the following statement that we need to prove:

\begin{lemm}{unambiguousliftsforsmallmouldsq}
 Suppose $U,V,V',W$ are cofibrant sieves as above and we have a lifting problem of the form:
\begin{displaymath}
  \begin{tikzcd}
    (U, \theta) \ar[r] \ar[d] & (U, <i, \pm> * \theta) \ar[r] \ar[d] & Y \ar[d, "p"] \\
    (W, \theta) \ar[r] \ar[urr, dotted] & (W, <i, \pm> * \theta) \ar[r] & X
  \end{tikzcd}
\end{displaymath}
then the solutions obtained by decomposing the left hand square as in the diagram below on the left or as in the one below on the right coincide.
\begin{displaymath}
  \begin{array}{cc}
    \begin{tikzcd}
      (U, \theta) \ar[r] \ar[d] & (U, <i, \pm> * \theta) \ar[r] \ar[d] & Y \ar[dd, "p"] \\
      (V, \theta) \ar[r] \ar[d] & (V, <i,\pm> * \theta) \ar[d] \\
      (W, \theta) \ar[r] \ar[uurr, dotted, bend right = 25] & (W, <i, \pm> * \theta) \ar[r] & X
    \end{tikzcd} &
    \begin{tikzcd}
      (U, \theta) \ar[r] \ar[d] & (U, <i, \pm> * \theta) \ar[r] \ar[d] & Y \ar[dd, "p"] \\
      (V', \theta) \ar[r] \ar[d] & (V', <i,\pm> * \theta)  \ar[d] \\
      (W, \theta) \ar[r] \ar[uurr, dotted, bend right = 25] & (W, <i, \pm> * \theta) \ar[r] & X
    \end{tikzcd}
  \end{array}
\end{displaymath}
\end{lemm}
\begin{proof}
Without loss of generality we may assume that $\theta$ is the empty traversal, the reason being that the solutions for general $\theta$ are obtained by pushing forward the solutions for $\theta = <>$.

Let us write $\alpha: [k] \to [n]$ for the $k$-simplex in $V$ but not in $U$ and $\beta: [l] \to [n]$ for the $l$-simplex in $W$ but not in $V$. Then the solutions for the lifting problem in the left hand squares in both diagrams above are determined by the following mould cubes:
\begin{displaymath}
  \begin{array}{c}
    \begin{tikzcd}
      & (U, <>) \ar[dd] \ar[rr] & & (U,<i,\pm>) \ar[dd] \\
      (\partial \Delta^k, <>) \ar[ur, "{(\alpha,<>)}"] \ar[rr, crossing over] \ar[dd] & & (\partial \Delta^k, <i,\pm> \cdot \alpha) \ar[ur, "{(\alpha,<>)}"] \\
      & (V, <>) \ar[rr] & & (V, <i,\pm>) \\
      (\Delta^k,<>) \ar[rr] \ar[ur,  "{(\alpha,<>)}"] & & (\Delta^k, <i,\pm> \cdot \alpha) \ar[ur,  "{(\alpha,<>)}"] \ar[from = uu, crossing over]
    \end{tikzcd} \\
    \begin{tikzcd}
      & (U, <>) \ar[dd] \ar[rr] & & (U,<i,\pm>) \ar[dd] \\
      (\partial \Delta^l, <>) \ar[ur, "{(\beta,<>)}"] \ar[rr, crossing over] \ar[dd] & & (\partial \Delta^l, <i,\pm> \cdot \beta) \ar[ur, "{(\beta,<>)}"] \\
      & (V', <>) \ar[rr] & & (V', <i,\pm>) \\
      (\Delta^l,<>) \ar[rr] \ar[ur,  "{(\beta,<>)}"] & & (\Delta^l, <i,\pm> \cdot \beta) \ar[ur,  "{(\beta,<>)}"] \ar[from = uu, crossing over]
    \end{tikzcd} \\
    \begin{tikzcd}
      & (V, <>) \ar[dd] \ar[rr] & & (V,<i,\pm>) \ar[dd] \\
      (\partial \Delta^l, <>) \ar[ur, "{(\beta,<>)}"] \ar[rr, crossing over] \ar[dd] & & (\partial \Delta^l, <i,\pm> \cdot \beta) \ar[ur, "{(\beta,<>)}"]  \\
      & (W, <>) \ar[rr] & & (W, <i,\pm>) \\
      (\Delta^l,<>) \ar[rr] \ar[ur,  "{(\beta,<>)}"] & & (\Delta^l, <i,\pm> \cdot \beta) \ar[ur,  "{(\beta,<>)}"] \ar[from = uu, crossing over]
    \end{tikzcd}
    \end{array}
  \end{displaymath}
  \begin{displaymath}
    \begin{array}{c}
    \begin{tikzcd}
      & (V', <>) \ar[dd] \ar[rr] & & (V',<i,\pm>) \ar[dd] \\
      (\partial \Delta^k, <>) \ar[ur, "{(\alpha,<>)}"] \ar[rr, crossing over] \ar[dd] & & (\partial \Delta^k, <i,\pm> \cdot \alpha) \ar[ur, "{(\alpha,<>)}"]  \\
      & (W, <>) \ar[rr] & & (W, <i,\pm>) \\
      (\Delta^k,<>) \ar[rr] \ar[ur,  "{(\alpha,<>)}"] & & (\Delta^k, <i,\pm> \cdot \alpha) \ar[ur,  "{(\alpha,<>)}"] \ar[from = uu, crossing over]
    \end{tikzcd}
  \end{array}
\end{displaymath}
But since the second and third as well as the first and fourth have the same front face, the solutions will coincide. (What this says is that because the $k$-simplex and $l$-simplex have at most parts of their boundary in common, the the lifting problems are independent from each other and can be solved in either order.)
\end{proof}

To summarise the discussion so far, any map $p$ which has the right lifting property with respect to horn squares, has unambiguous lifts against general small mould squares. Note that these lifts will automatically satisfy the horizontal and vertical conditions for having the right lifting property against the triple category $\mathbb{M}$, because the lifts do not depend on the way we divide a small mould square into a grid of one-step mould squares or on the way we traverse that grid. That means that any requirements on the lifts against the horn squares needed for them to extend to a unique effective Kan fibration structure should come from the perpendicular condition. So we will now have a look at this condition.

\begin{rema}{onsmallmouldcubes} From now on, we will often think of the perpendicular condition on the lifts against the small mould squares in $\mathbb{M}$ as expressing as a stability condition with respect to base change or pullback. Indeed, we will often refer to it as a \emph{base change condition}. The reason is that in $\mathbb{M}$ both the morphisms of HDRs in the $yz$-plane as well as in the $xz$-plane are cartesian. This means that in a small mould cube like
  \begin{displaymath}
    \begin{tikzcd}
      & D \ar[dd] \ar[rr] & & C \ar[dd]  \\
      D' \ar[ur] \ar[dd] \ar[rr, crossing over] & & C'  \ar[ur] \\
      & B \ar[rr] & & A  \\
      B' \ar[rr] \ar[ur, "\beta"] & & A' \ar[ur] \ar[from = uu, crossing over]
    \end{tikzcd}
  \end{displaymath}
we can think of the front face of the cube as the result of pulling back the face at the back along $\beta: B' \to B$. Indeed, from now on we will often draw such a situation as follows
\begin{displaymath}
  \begin{tikzcd}
    D' \ar[r] \ar[d] & C' \ar[d] & & D \ar[d] \ar[r] & C \ar[d] \\
    B' \ar[r] \ar[rrr, bend right, dotted, "\beta"] & A' & & B \ar[r] & A.
  \end{tikzcd}
\end{displaymath}
The reader is supposed to keep in mind that there is a cube connecting the two squares, but we will only draw a dotted arrow to prevent our diagrams from becoming too cluttered.
\end{rema}

Let us call a small mould square \emph{stable} if it its induced lift is compatible with the induced lift of any base change of that same square. In fact, it will be convenient to have a relativised notion of stability. So assume $\smallmap{S}$ is a class of morphisms in $\mathbf{\Delta}$ such that if
\diag{ [n'] \ar[d]_{\beta'} \ar[r]^{\alpha'} & [m'] \ar[d]^{\beta} \\
[n] \ar[r]_{\alpha} & [m] }
is a pullback diagram in $\mathbf{\Delta}$ with $\beta$ monic, then $\alpha \in \smallmap{S}$ implies $\alpha' \in \smallmap{S}$ (think $\smallmap{S} = \{ \mbox{ face maps } \} \cup \{ \mbox{ identities} \}$ or $\smallmap{S} = \{ \mbox{ degeneracy maps } \} \cup \{ \mbox{ identities} \}$). Then a small mould square
\begin{displaymath}
  \begin{tikzcd}
     D \ar[r] \ar[d] & C \ar[d] \\
     B \ar[r] & A
  \end{tikzcd}
\end{displaymath}
will be called \emph{$\smallmap{S}$-stable} if for any commutative diagram of the form
\begin{displaymath}
  \begin{tikzcd}
    & D \ar[dd] \ar[rr] & & C \ar[rr] \ar[dd] & & Y \ar[dd, "p"] \\
    D' \ar[ur] \ar[dd] \ar[rr, crossing over] & & C' \ar[ur] \\
    & B \ar[rr] \ar[dotted, uurrrr] & & A \ar[rr] & & X \\
    B' \ar[rr] \ar[ur] & & A' \ar[ur] \ar[from = uu, crossing over]
  \end{tikzcd}
\end{displaymath}
in which the cube is a ``small'' base change cube along $(\alpha,\tau): B' \to B$ with $\alpha \in \smallmap{S}$, the induced map $A' \to Y$ can be obtained by composing the induced map $A \to Y$ with $A' \to A$. The next step is to find necessary and sufficient conditions on the fillers for the horn squares to ensure that any small mould square is $\smallmap{S}$-stable.

First of all, it is clearly necessary and sufficient if every one-step mould square is $\smallmap{S}$-stable. Indeed, if in a grid any little square is $\smallmap{S}$-stable, then so is the entire square. But note that the pullback of a one-step mould square need no longer be a one-step mould square: both in the horizontal and the vertical direction the number of steps may increase. (By the way, it may also become 0 in one of the two directions, in which case the stability condition is vacuously satisfied. So without loss of generality we may always assume this does not happen.)

\begin{lemm}{indliftswithemptytraversalnice}
Suppose $\psi = \tau * \theta$ and we have situation as follows:
\begin{displaymath}
  \begin{tikzcd}
    & (S,\theta) \ar[dd] \ar[rr] & & (S,\psi) \ar[rr] \ar[dd] & & Y \ar[dd, "p"] \\
    (S, \langle \rangle) \ar[ur] \ar[dd] \ar[rr, crossing over] & &  (S,\tau)  \ar[ur] \\
    & (T,\theta) \ar[rr] \ar[dotted, uurrrr] & & (T,\psi) \ar[rr] & & X \\
    (T, \langle \rangle) \ar[rr] \ar[ur] & & (T, \tau), \ar[ur] \ar[from = uu, crossing over]
  \end{tikzcd}
\end{displaymath}
where the cube is a base change cube along $(1,\tau): (T, \langle \rangle) \to (T, \theta)$.
Then the induced lifts $(T, \tau) \to Y$ and $(T, \psi) \to Y$ are compatible. This means that, since the bottom of the cube is a pushout, the induced lifts determine each other by composition and pushout, respectively.
\end{lemm}
\begin{proof}
We prove the statement of the lemma by induction on the length of the traversal $\tau$. Note that the case $\tau = \langle \rangle$ is vacuously true.

In case where $\tau$ has length 1, we can regard both the front and the back of the cube as a vertical composition of one-step mould squares. In that case the statement follows from the definition of the induced lifts for one-step mould squares.

Now write $\tau = \sigma * \rho$ where $\sigma$ has length 1 and consider the following situation:
\begin{displaymath}
  \begin{tikzcd}
    & (S, \langle \rangle)  \ar[d] \ar[r]  & (S, \sigma)   \ar[d]  \\
    & (T, \langle \rangle) \ar[r] \ar[dd, dotted, bend right = 45, pos =.3, "{(1,\rho)}"'] & (T, \sigma)   \\
    (S, \langle \rangle)  \ar[d] \ar[r] & (S, \rho) \ar[d] \ar[r]  & (S, \tau) \ar[d]  \\
    (T, \langle \rangle) \ar[r] \ar[dotted, bend right = 45, dd, "{(1,\theta)}"', pos =.3] & (T, \rho) \ar[dd, dotted, bend right = 45, "{(1,\theta)}"', pos =.3] \ar[r] &(T, \tau) \\
    (S, \theta) \ar[d] \ar[r] & (S, \rho * \theta) \ar[d] \ar[r] & (S, \psi) \ar[r] \ar[d] & Y \ar[d, "p"] \\
    (T, \theta) \ar[r] & (T, \rho * \theta) \ar[r] & (T, \psi) \ar[r] & X
  \end{tikzcd}
\end{displaymath}
We should imagine that we are given a map $(T,\theta) \to Y$ and we want to push it forward to a map $(T, \psi) \to Y$. Now, by induction hypothesis, the fact that the statement holds in case $\tau$ has length 1, and the earlier lemmas about grids, we can compute this as follows: take the induced lift $(T, \rho) \to Y$, push that down to map $(T, \rho * \theta) \to Y$, restrict that to a map $(T, \langle \rangle) \to Y$, take the induced lift $(T, \sigma) \to Y$ and then push that all the way down to a map $(T, \psi) \to Y$. But the latter can be done in two steps: push the map $(T, \sigma) \to Y$ down to $(T, \tau) \to Y$ and then push it further down to $(T, \psi) \to Y$. This means the end result $(T, \psi) \to Y$ coincides with taking the induced lift $(T, \tau) \to Y$ and pushing that down.
\end{proof}

\begin{lemm}{firstredforstability}
If all one-step mould squares starting from the empty traversal are $\smallmap{S}$-stable, then so are all one-step mould squares.
\end{lemm}
\begin{proof}
Imagine that we have a one-step mould square
\begin{displaymath}
  \begin{tikzcd}
    (S, \theta) \ar[d] \ar[r] & (S, \langle i, \pm \rangle * \theta) \ar[d] \\
    (T, \theta) \ar[r] & (T, \langle i, \pm \rangle * \theta)
  \end{tikzcd}
\end{displaymath}
and we pull it back along $(\alpha, \sigma)$ with $\alpha \in \smallmap{S}$. Writing $\theta \cdot \alpha = \psi * \sigma$, we get four small mould squares:
\begin{displaymath}
  \begin{tikzcd}
    (\alpha^* S, \langle \rangle) \ar[d] \ar[r] & (\alpha^*S, \langle i, \pm \rangle \cdot \alpha) \ar[d] & & (S, \langle \rangle) \ar[d] \ar[r] & (S, \langle i, \pm \rangle) \ar[d] \\
    (\alpha^* T, \langle \rangle) \ar[r] \ar[dd, dotted, bend right = 50, "{(1,\psi)}"'] \ar[rrr, dotted, bend right = 20, "{(\alpha,1)}"] & (\alpha^*T, \langle i, \pm \rangle \cdot \alpha) & & (T, \langle \rangle)  \ar[r] \ar[dd, bend right = 50, dotted, "{(1,\theta)}"'] & (T, \langle i, \pm \rangle )   \\
    (\alpha^*S, \psi) \ar[d] \ar[r] & (\alpha^*S, \langle i, \pm \rangle \cdot \alpha * \psi) \ar[d] & & (S, \theta) \ar[d] \ar[r] & (S,  \langle i, \pm \rangle * \theta) \ar[d] \\
    (\alpha^*T, \psi)  \ar[rrr, dotted, bend right, "{(\alpha,\sigma)}"] \ar[r] & (\alpha^*T, \langle i, \pm \rangle \cdot \alpha * \psi)  & & (T, \theta) \ar[r] & (T, \langle i, \pm \rangle * \theta)
  \end{tikzcd}
\end{displaymath}
Recall from \refrema{onsmallmouldcubes} that the dotted arrows indicate that the squares are connected by small mould cubes and note that the small mould cubes determined by the dotted arrows going down have pushouts at their bottom faces. We are given a map $(T, \theta) \to Y$ and asked to compare the induced maps $(T, \langle i, \pm \rangle * \theta) \to Y$ and $(\alpha^* T, \langle i, \pm \rangle \cdot \alpha * \psi) \to Y$. The previous lemma tells us that both induced maps can be computed by taking the induced maps $(\alpha^*T, \langle i, \pm \rangle \cdot \alpha) \to Y$ and $(T, \langle i, \pm \rangle) \to Y$ and then pushing these down. So if the square on the top right is $\smallmap{S}$-stable, then so is the square on the bottom right.
\end{proof}

\begin{lemm}{indliftwithmaximalsievecodomainnice} Suppose $S \subseteq T \subseteq \Delta^m$ are cofibrant sieves, $\alpha: \Delta^n \to \Delta^m$ is monic and
\begin{displaymath}
  \begin{tikzcd}
    R \ar[d] \ar[r] & S \ar[d] \\
    \Delta^n \ar[r, "\alpha"] & T
  \end{tikzcd}
\end{displaymath}
is bicartesian. Suppose moreover that we have situation as follows:
\begin{displaymath}
  \begin{tikzcd}
    & (S,\langle \rangle) \ar[dd] \ar[rr] & & (S,\theta) \ar[rr] \ar[dd] & & Y \ar[dd, "p"] \\
    (R, \langle \rangle) \ar[ur] \ar[dd] \ar[rr, crossing over] & &  (R,\theta \cdot \alpha) \ar[ur] \\
    & (T, \langle \rangle) \ar[rr] \ar[dotted, uurrrr] & & (T,\theta) \ar[rr] & & X \\
    (\Delta^n, \langle \rangle) \ar[rr] \ar[ur] & & (\Delta^n, \theta \cdot \alpha), \ar[ur] \ar[from = uu, crossing over]
  \end{tikzcd}
\end{displaymath}
where the cube is a base change cube along $(\alpha,\langle \rangle): (\Delta^n, \langle \rangle) \to (T, \langle \rangle)$. Then the induced lifts $(T, \theta) \to Y$ and $(\Delta^n, \theta \cdot \alpha) \to Y$ are compatible. This means that, since the left and right hand faces of the cube are pushouts, the induced lifts determine each other by composition and pushout, respectively.
\end{lemm}
\begin{proof}
We prove this by induction on the number $k$ of simplices in $T$ but not in $S$ (which coincides with the number of simplices in $\Delta^n$ but not in $R$). Note that the case $k = 0$ is trivial.

In case $k = 1$, we have $R = \partial \Delta^n$. We prove the desired statement by induction on the length of $\theta$. Note that because $\alpha$ is monic, $\theta \cdot \alpha$ cannot have greater length than $\theta$. The case $\theta = \langle \rangle$ is again trivial, while the case where $\theta$ has length 1 follows immediately from the way the lifts for horn squares induce lifts for one-step mould square starting from the empty traversal. Now write $\theta = \tau * \sigma$ where $\tau$ has length 1 and consider:
\begin{displaymath}
  \begin{tikzcd}
    & (R, \langle \rangle)  \ar[d] \ar[r]  & (R, \tau \cdot \alpha) \ar[d] \\
    & (\Delta^n, \langle \rangle) \ar[r] \ar[bend right = 50, dd, pos =.3, dotted, "{(1, \sigma \cdot \alpha)}"'] & (\Delta^n, \tau \cdot \alpha)  \\
    (R, \langle \rangle)  \ar[d] \ar[r] & (R, \sigma \cdot \alpha) \ar[d] \ar[r]  & (R, \theta \cdot \alpha) \ar[d]  \\
    (\Delta^n, \langle \rangle) \ar[r] \ar[bend right = 50, dd, dotted, "{(\alpha, \langle \rangle)}"'] & (\Delta^n, \sigma \cdot \alpha)  \ar[r] & (\Delta^n, \theta \cdot \alpha)  \\
    (S, \langle \rangle) \ar[d] \ar[r] & (S, \sigma) \ar[d] \ar[r] & (S, \theta) \ar[d] \ar[r] & Y \ar[d, "p"] \\
    (T, \langle \rangle) \ar[r] & (T,  \sigma) \ar[r] & (T, \theta) \ar[r] & X
  \end{tikzcd}
\end{displaymath}
We should imagine that we are given a map $(T,\langle \rangle) \to Y$ and we want to push it forward to a map $(T, \theta) \to Y$. Now, by induction hypothesis, the fact that the statement holds in case $\theta$ has length 1 (and 0), and the earlier lemmas about grids, we can compute this by taking the induced lift $(\Delta^n, \sigma \cdot \alpha) \to Y$, pushing it down to a map $(T, \sigma) \to Y$, then taking the induced lift $(\Delta^n, \tau \cdot \alpha) \to Y$ and then pushing that down to a map $(T, \theta) \to Y$. But the latter can be done in two steps: pushing it down to $(\Delta^n, \theta \cdot \alpha) \to Y$ and then pushing it further down to $(T, \theta) \to Y$. This means it coincides with taking the induced lift $(\Delta^n, \theta \cdot \alpha) \to Y$ and pushing that down.

Having proved the statement for $k = 1$, we now do the induction step. So write $S \subseteq S' \subseteq T$ where $S'$ is obtained from $S$ by adding one simplex, so that we have a picture as follows:
\begin{displaymath}
  \begin{tikzcd}
    \partial \Delta^{n'} \ar[r] \ar[d] & R \ar[r] \ar[d] & S \ar[d] \\
    \Delta^{n'} \ar[r]  \ar[dr, "{\alpha'}"'] & R' \ar[r] \ar[d]  & S' \ar[d] \\
    &  \Delta^n \ar[r, "\alpha"] & T
  \end{tikzcd}
\end{displaymath}
in which all squares are bicartesian. This gives us the following situation:
\begin{displaymath}
  \begin{tikzcd}
    (\partial \Delta^{n'}, \langle \rangle) \ar[d] \ar[r] & (\partial \Delta^{n'}, \theta \cdot \alpha \alpha') \ar[d]  \\
    (\Delta^{n'}, \langle \rangle) \ar[r] \ar[dd, dotted, bend right = 50, "{(\alpha', \langle \rangle)}"'] & (\Delta^{n'}, \theta \cdot \alpha\alpha') \\
    (R, \langle \rangle) \ar[d] \ar[r] & (R, \theta \cdot \alpha) \ar[d]  \\
    (R', \langle \rangle) \ar[d] \ar[r] & (R',\theta \cdot \alpha) \ar[d] \\
    (\Delta^n, \langle \rangle) \ar[r] \ar[bend right = 50, ddd, dotted, "{(\alpha, \langle \rangle)}"']  & (\Delta^n, \theta \cdot \alpha)  \\
    (   S, \langle \rangle) \ar[d]  \ar[r] & (S, \theta)  \ar[d]  \ar[rr] & & Y \ar[dd, "p"]  \\
    (S', \langle \rangle) \ar[d] \ar[r] & (S', \theta) \ar[d] \\
    (T, \langle \rangle) \ar[r] & (T, \theta)  \ar[rr] & & X
  \end{tikzcd}
\end{displaymath}
So if we have a map $(T, \langle \rangle) \to Y$ and we wish to push it forward to $(T, \theta) \to Y$, then we can do this in two steps: first we can compute $(S', \theta) \to Y$ and then compute $(T, \theta) \to Y$. The former we can compute by finding the lift $(\Delta^{n'}, \theta \cdot \alpha \alpha') \to Y$ and then pushing it down, which also can be done in two steps, yielding a map $(R', \theta \cdot \alpha) \to Y$ and then a map $(S', \theta) \to Y$. Given this map $(S', \theta) \to Y$, the induction hypothesis tells us that we can compute the desired map $(T, \theta) \to Y$ from the map $(R',\theta \cdot \alpha) \to Y$ we computed along the way by taking its induced lift $(\Delta^n, \theta \cdot \alpha) \to Y$ and then pushing that down. In other words, we take the induced map $(\Delta^n, \theta \cdot \alpha) \to Y$ and then push that down, thus showing the induction step.
\end{proof}

\begin{lemm}{secondredforstability}
If all horn squares are $\smallmap{S}$-stable, then so are all one-step mould squares starting from the empty traversal.
\end{lemm}
\begin{proof}
Suppose we have a one-step mould square starting from the empty traversal, like
\begin{displaymath}
  \begin{tikzcd}
    (S, \langle \rangle) \ar[d] \ar[r] & (S, \langle i, \pm \rangle) \ar[d] \\
    (T, \langle \rangle) \ar[r] & (T, \langle i, \pm \rangle)
  \end{tikzcd}
\end{displaymath}
which we want to pull it back along some map, say $(\alpha, \langle \rangle)$ with $\alpha \in \smallmap{S}$ (note that the second component has to be the empty traversal). Let $\beta: [m] \to [n]$ be the $m$-simplex which we need to add to $S$ to obtain $T$, and consider the pullback:
\begin{displaymath}
  \begin{tikzcd}
    \Delta^{m'} \ar[r, "{\alpha'}"] \ar[d, "{\beta'}"] \ar[dr, "\gamma"] & \Delta^m \ar[d, "\beta"] \\
    \Delta^{n'} \ar[r, "\alpha"] & \Delta^n
  \end{tikzcd}
\end{displaymath}
(note that $\beta$ is monic, so this pullback exists provided the images of $\alpha$ and $\beta$ have some overlap: which we may assume without loss of generality, because otherwise $\alpha^* S = \alpha^* T$ and the stability condition is trivially satisfied). Note that $\alpha' \in \smallmap{S}$.

We again get four mould squares, where the dotted arrows again indicate some base changes:
\begin{displaymath}
  \begin{tikzcd}
    (\gamma^* S, \langle \rangle) \ar[d] \ar[r] & (\gamma^*S, \langle i, \pm \rangle \cdot \gamma) \ar[d] & & (\partial \Delta^m, \langle \rangle) \ar[d] \ar[r] & (\partial \Delta^m, \langle i, \pm \rangle \cdot \beta) \ar[d] \\
    (\Delta^{m'}, \langle \rangle) \ar[r] \ar[dd, bend right = 50, dotted, "{(\beta',\langle \rangle)}"'] \ar[rrr, dotted, bend right = 20, "{(\alpha',\langle \rangle)}"] & (\Delta^{m'}, \langle i, \pm \rangle \cdot \gamma) & & (\Delta^m, \langle \rangle)  \ar[r] \ar[dd, bend right = 50, dotted, "{(\beta,\langle \rangle)}"'] & (\Delta^m, \langle i, \pm \rangle \cdot \beta)   \\
    (\alpha^*S, \langle \rangle) \ar[d] \ar[r] & (\alpha^*S, \langle i, \pm \rangle \cdot \alpha) \ar[d] & & (S, \langle \rangle) \ar[d] \ar[r] & (S, \langle i, \pm \rangle) \ar[d] \\
    (\alpha^*T, \langle \rangle)  \ar[rrr, bend right = 20, dotted, "{(\alpha,\langle \rangle)}"] \ar[r] & (\alpha^*T, \langle i, \pm \rangle \cdot \alpha)  & & (T, \langle \rangle) \ar[r] & (T,  \langle i, \pm \rangle)
  \end{tikzcd}
\end{displaymath}
Note that the base change cubes for the arrows going down have pushouts as their left and right faces. We are given a map $(T, \langle \rangle) \to Y$ and asked to compare the induced maps $(T, \langle i, \pm \rangle) \to Y$ and $(\alpha^* T,  \langle i, \pm \rangle \cdot \alpha) \to Y$. The previous lemma tells us that both induced maps can be computed by first taking the induced maps $(\Delta^{m'}, \langle i, \pm \rangle \cdot \gamma) \to Y$ and $(\Delta^{m}, \langle i, \pm \rangle \cdot \beta) \to Y$ and then pushing them down. So if the square on the top right is $\smallmap{S}$-stable, then so is the square on the bottom right.
\end{proof}

From \reflemm{firstredforstability} and \reflemm{secondredforstability} we deduce:

\begin{prop}{1sthornsqdef}
If we equip a map $p: Y \to X$ with lifts against horn squares which are $\smallmap{S}$-stable, then all the induced lifts against small mould squares will be $\smallmap{S}$-stable.
\end{prop}

What does this mean for effective Kan fibrations? To equip a map $p: Y \to X$ with the structure of an effective Kan fibration, it will be (necessary and) sufficient to find lifts against horn squares so that the induced lifts against small mould squares are $\smallmap{S}$-stable for both $\smallmap{S} = \{ \mbox{ face maps } \} \cup \{ \mbox{ identities} \}$ and $\smallmap{S} = \{ \mbox{ degeneracy maps } \} \cup \{ \mbox{ identities} \}$. So the proposition tells us that we need to find lifts against horn squares which are stable relative to both classes. But lifts against horn squares are always stable relative to the first class (faces plus identities), because $d_i^* \partial \Delta^n$ is always the maximal sieve. So we have:

\begin{theo}{defunifkanfibrhornsqrs}
  The following notions of fibred structure are isomorphic:
  \begin{itemize}
    \item Being an effective Kan fibration.
    \item To assign to each map all systems of lifts against horn squares which are stable along degeneracy maps.
  \end{itemize}
\end{theo}

\begin{rema}{isomfordfcdoublecatsagain}
  It should be clear that the second notion of fibred structure also naturally arises as the vertical maps in a discretely fibred concrete double category. Then this concrete double category is isomorphic to the one of given by the effective Kan fibrations: indeed, fullness on squares can be shown as in \reflemm{hornsqrsdetermineall}.
\end{rema}

Let us now try to unwind what that means concretely: lifts for horn squares which are stable along degeneracies map. First of all, for each $n$ there are $2(n+1)$ horn squares as follows:
\begin{displaymath}
  \begin{tikzcd}
    ( \partial \Delta^n, \langle \rangle) \ar[d] \ar[r] & (\partial \Delta^n, \langle i, \pm \rangle) \ar[d] \\
    (\Delta^n, \langle \rangle) \ar[r] & (\Delta^n, \langle i, \pm \rangle),
  \end{tikzcd}
\end{displaymath}
and these can be pulled back along $s_j: \Delta^{n+1} \to \Delta^n$. The case where $j = i$ is special and we will postpone discussion of that case.

In case $j \not= i$ we have the following cartesian morphism of HDRs:
\begin{displaymath}
  \begin{tikzcd}
    \Delta^{n+1} \ar[d, "{s_j}"] \ar[rr, "{d_{i^*}/d_{i^*+1}}"] & & \Delta^{n+2} \ar[d, "{s_{j^*}}"] \ar[r, "{s_{i^*}}"] & \Delta^{n+1} \ar[d, "{s_j}"] \\
    \Delta^n \ar[rr, "{d_i/d_{i+1}}"] & & \Delta^{n+1} \ar[r, "{s_i}"] & \Delta^n
  \end{tikzcd}
\end{displaymath}
where $i^* = i + 1$ if $j \lt i$ and $i^* = i$ if $j \gt i$, while $j^* = j$ if $j \lt i$ and $j^* = j + 1$ if $j \gt i$. This means that if we pull back the horn square above along $s_j$, then we obtain a vertical composition of three one-step mould squares, as follows:
\begin{displaymath}
  \begin{tikzcd}
    (\partial \Delta^n, \langle \rangle) \ar[r] \ar[d] & (\partial \Delta^n, \langle i, \pm \rangle) \ar[d] & (S^{n+1}_j, \langle \rangle) \ar[d] \ar[r] & (S^{n+1}_j, \langle i^*, \pm \rangle) \ar[d]  \\
    (\Delta^n, \langle \rangle) \ar[r] \ar[rr, dotted, bend right = 13, "{d_{j+1}}"'] \ar[ddddr, dotted, bend right = 80, "1"'] & (\Delta^n, \langle i, \pm \rangle) & (\Lambda^{n+1}_j, \langle \rangle) \ar[r] \ar[d] & (\Lambda^{n+1}_j, \langle i^*, \pm \rangle) \ar[d] \\
    (\partial \Delta^n, \langle \rangle) \ar[r] \ar[d] & (\partial \Delta^n, \langle i, \pm \rangle) \ar[d] & (\partial \Delta^{n+1}, \langle \rangle) \ar[d] \ar[r] & (\partial \Delta^{n+1}, \langle i^*, \pm \rangle) \ar[d] \\
    (\Delta^n, \langle \rangle) \ar[r] \ar[urr, dotted, bend right = 30, "{d_j}"'] \ar[bend right, dotted, ddr, "1"] & (\Delta^n, \langle i, \pm \rangle) & (\Delta^{n+1}, \langle \rangle) \ar[r] \ar[dotted, ddl, pos = .3, "{s_j}"'] & (\Delta^{n+1}, \langle i^*, \pm \rangle) \\
    & ( \partial \Delta^n, \langle \rangle) \ar[d] \ar[r] & (\partial \Delta^n, \langle i, \pm \rangle) \ar[d] \ar[r] & Y \ar[d, "p"] \\
    & (\Delta^n, \langle \rangle) \ar[r] & (\Delta^n, \langle i, \pm \rangle) \ar[r] & X.
  \end{tikzcd}
\end{displaymath}
Here we have used the abbreviation $S^{n+1}_j := s_j^* \partial \Delta^n = \Lambda^{n+1}_{j, j+1} \cup (d_j \cap d_{j+1})$ (that is, it is $\Delta^{n+1}$ with the interior as well as the $j$th and $(j+1)$st faces missing). Note that our recipe for computing the lifts against the first two squares in the vertical composition on the top right in the diagram tells us to solve the original lifting problem (because $s_j.d_j = s_j.d_{j+1} = 1$). In other words, we can phrase the compatibility condition\index{compatibility condition} for the case $j \not= i$ as follows: if $f: (\Delta^n, <i, \pm>) \to Y$ is our chosen solution to the lifting problem
\begin{displaymath}
  \begin{tikzcd}
    ( \partial \Delta^n, \langle \rangle) \ar[d] \ar[r] & (\partial \Delta^n, \langle i, \pm \rangle) \ar[d] \ar[r, "y"] & Y \ar[d, "p"] \\
    (\Delta^n, \langle \rangle) \ar[r] \ar[dotted, urr, "g", near start] & (\Delta^n, \langle i, \pm \rangle) \ar[r, "x"]  & X,
  \end{tikzcd}
\end{displaymath}
then our chosen solution $(\Delta^{n+1}, <i^*, \pm>) \to Y$ to the lifting problem
\begin{displaymath}
  \begin{tikzcd}
    ( \partial \Delta^{n+1}, \langle \rangle) \ar[d] \ar[r] & (\partial \Delta^{n+1}, \langle i^*, \pm \rangle) \ar[d] \ar[r, "{y'}"] & Y \ar[d, "p"] \\
    (\Delta^{n+1}, \langle \rangle) \ar[r] \ar[urr, dotted, "{g.s_j}", near start] & (\Delta^{n+1}, \langle i^*, \pm \rangle) \ar[r, "{x.s_{j^*}}"] & X
  \end{tikzcd}
\end{displaymath}
should be $f.s_{j^*}$, where $y'$ is the map which is $y \cdot s_{j^*}$ on $S^{n+1}_{j}$ and $f$ on both $d_{j}$ and $d_{j+1}$.

The case $j = i$ is special, because then the pullback of the horn square grows in a horizontal direction as well. In this case it will be convenient to treat the positive and negative case separately. So let do the positive case first. As we have seen in the proof of \reftheo{HgenMoorefib}, we have the following cartesian morphism of HDRs:
\begin{displaymath}
  \begin{tikzcd}
    \Delta^{n+1} \ar[d, "{s_i}"] \ar[r, "{d_i}"] & \Delta^{n+2} \ar[r, "{\iota_2}"] & <(i+1,+),(i,+)> \ar[d, "{[s_i,s_{i+1}]}"] \ar[rr, "{[s_{i+1},s_i]}"] & & \Delta^{n+1} \ar[d, "{s_i}"] \\
    \Delta^n \ar[rr, "{d_i}"] & & \Delta^{n+1} \ar[rr, "{s_i}"] & & \Delta^n.
  \end{tikzcd}
\end{displaymath}
This means that we have a picture as follows in which we have pulled back the horn square below along $s_i$ and decomposed the result into a grid of six one-step mould squares.
\begin{displaymath}
  \begin{tikzcd}
    (S^{n+1}_i, \langle \rangle) \ar[d] \ar[r] & (S^{n+1}_i, \langle i, + \rangle) \ar[d] \ar[r] & (S^{n+1}_i, \langle (i+1, +), (i,+) \rangle) \ar[d] \\
    (\Lambda^{n+1}_i, \langle \rangle) \ar[r] \ar[d] & (\Lambda^{n+1}_i, \langle i, + \rangle) \ar[d] \ar[r] & (\Lambda^{n+1}_i, \langle (i+1, +), (i, +) \rangle) \ar[d] \\
    (\partial \Delta^{n+1}, \langle \rangle) \ar[d] \ar[r] & (\partial \Delta^{n+1}, \langle i, + \rangle) \ar[r] \ar[d] & (\partial \Delta^{n+1}, \langle (i+1, +), (i, +) \rangle) \ar[d] \\
    (\Delta^{n+1}, \langle \rangle) \ar[r] \ar[dd, bend right = 60, dotted, "{s_i}"'] & (\Delta^{n+1}, \langle i, + \rangle) \ar[r] & (\Delta^{n+1}, \langle (i+1, +), (i, +) \rangle) \\
    ( \partial \Delta^n, \langle \rangle) \ar[d] \ar[r] & (\partial \Delta^n, \langle i, + \rangle) \ar[d] \ar[r, "y"] & Y \ar[d, "p"] \\
    (\Delta^n, \langle \rangle) \ar[r] \ar[urr, dotted, "g", near start] & (\Delta^n, \langle i, + \rangle) \ar[r, "x"] & X.
  \end{tikzcd}
\end{displaymath}
Let us first consider the left column in the grid above. Note that the if we pull back the first square along $d_{i+1}$, we get the original square back, while if we pull back the second square along $d_i$, we get a map which trivialises in the horizontal direction. For that reason the left column gives us the following compatibility condition: if $f: (\Delta^n, <i,+>) \to Y$ is our chosen solution to the lifting problem
\begin{displaymath}
  \begin{tikzcd}
    ( \partial \Delta^n, \langle \rangle) \ar[d] \ar[r] & (\partial \Delta^n, \langle i, + \rangle) \ar[d] \ar[r, "y"] & Y \ar[d, "p"] \\
    (\Delta^n, \langle \rangle) \ar[r] \ar[urr, dotted, "g", near start] & (\Delta^n, \langle i, + \rangle) \ar[r, "x"]  & X,
  \end{tikzcd}
\end{displaymath}
then our chosen solution $(\Delta^{n+1},<i,+>) \to Y$ to the lifting problem
\begin{displaymath}
  \begin{tikzcd}
    ( \partial \Delta^{n+1}, \langle \rangle) \ar[d] \ar[r] & (\partial \Delta^{n+1}, \langle i, + \rangle) \ar[d] \ar[r, "{y'}"] & Y \ar[d, "p"] \\
    (\Delta^{n+1}, \langle \rangle) \ar[r] \ar[urr, dotted, "{g.s_i}", near start] & (\Delta^{n+1}, \langle i, + \rangle) \ar[r, "{x.s_{i+1}}"]  & X
  \end{tikzcd}
\end{displaymath}
should be $f.s_{i+1}$, where $y'$ is the map which is $y \cdot s_{i+1}$ on $S^{n+1}_{i}$ and $f$ on $d_{i+1}$.

We now turn to the column on the right. Note that if we pull back the first square along $d_{i+1}$, the square trivialises in the horizontal direction, while if we pull back the second square along $d_i$, we get the original horn square back. Therefore the right hand column gives us the following compatibility condition: if $f: (\Delta^n, <i,+>) \to Y$ is our chosen solution to the lifting problem
\begin{displaymath}
  \begin{tikzcd}
    ( \partial \Delta^n, \langle \rangle) \ar[d] \ar[r] & (\partial \Delta^n, \langle i, + \rangle) \ar[d] \ar[r, "y"] & Y \ar[d, "p"] \\
    (\Delta^n, \langle \rangle) \ar[r] \ar[urr, dotted, "g", near start] & (\Delta^n, \langle i, + \rangle) \ar[r, "x"]  & X,
  \end{tikzcd}
\end{displaymath}
then our chosen solution to the lifting problem
\begin{displaymath}
  \begin{tikzcd}
    ( \partial \Delta^{n+1}, \langle \rangle) \ar[d] \ar[r] & (\partial \Delta^{n+1}, \langle i + 1, + \rangle) \ar[d] \ar[r, "{y'}"] & Y \ar[d, "p"] \\
    (\Delta^{n+1}, \langle \rangle) \ar[r] \ar[urr, dotted, "{g.s_i}", near start] & (\Delta^{n+1}, \langle i+1, + \rangle) \ar[r, "{x.s_{i}}"] & X
  \end{tikzcd}
\end{displaymath}
should be $f.s_{i}$, where $y'$ is the map which is $y \cdot s_{i}$ on $S^{n+1}_i$ and $f$ on $d_i$.

Now let us do the negative case. In that case we have following cartesian morphism of HDRs:
\begin{displaymath}
  \begin{tikzcd}
    \Delta^{n+1} \ar[d, "{s_i}"] \ar[r, "{d_{i+2}}"] & \Delta^{n+2} \ar[r, "{\iota_2}"] & <(i,-),(i+1,-)> \ar[d, "{[s_{i+1},s_{i}]}"] \ar[rr, "{[s_{i},s_{i+1}]}"] & & \Delta^{n+1} \ar[d, "{s_i}"] \\
    \Delta^n \ar[rr, "{d_{i+1}}"] & & \Delta^{n+1} \ar[rr, "{s_i}"] & & \Delta^n.
  \end{tikzcd}
\end{displaymath}
The situation we now have to look at is the one where we pull the horn square at the bottom of diagram below back along $s_i$ and decompose the result into six one-step mould squares.
\begin{displaymath}
  \begin{tikzcd}
    (S^{n+1}_i, \langle \rangle) \ar[d] \ar[r] & (S^{n+1}_i, \langle i+1, - \rangle) \ar[d] \ar[r] & (S^{n+1}_i, \langle (i, -), (i+1,-) \rangle) \ar[d] \\
    (\Lambda^{n+1}_i, \langle \rangle) \ar[r] \ar[d] & (\Lambda^{n+1}_i, \langle i+1, - \rangle) \ar[d] \ar[r] & (\Lambda^{n+1}_i, \langle (i, -), (i + 1, -) \rangle) \ar[d] \\
    (\partial \Delta^{n+1}, \langle \rangle) \ar[d] \ar[r] & (\partial \Delta^{n+1}, \langle i+1, - \rangle) \ar[r] \ar[d] & (\partial \Delta^{n+1}, \langle (i, -), (i+1, -) \rangle) \ar[d] \\
    (\Delta[n+1], \langle \rangle) \ar[r] \ar[dd, bend right = 60, dotted, "{s_i}"'] & (\Delta^{n+1}, \langle i+1, - \rangle) \ar[r] & (\Delta^{n+1}, \langle (i, -), (i+1, -) \rangle) \\
    ( \partial \Delta^n, \langle \rangle) \ar[d] \ar[r] & (\partial \Delta^n, \langle i, - \rangle) \ar[d] \ar[r] & Y \ar[d, "p"] \\
    (\Delta^n, \langle \rangle) \ar[r] & (\Delta^n, \langle i, - \rangle) \ar[r] & X.
  \end{tikzcd}
\end{displaymath}
Both columns in the grid will again determine a compatibility condition and to see what they are, we start of by considering the left hand column. Note that the pullback of the first square along $d_{i+1}$ is trivial in the horizontal direction, while if we pull back the second square along $d_i$, we get our original horn square back. So the compatibility condition becomes this: if $f: (\Delta^n, <i,->) \to Y$ is our chosen solution to the lifting problem
\begin{displaymath}
  \begin{tikzcd}
    ( \partial \Delta^n, \langle \rangle) \ar[d] \ar[r] & (\partial \Delta^n, \langle i, - \rangle) \ar[d] \ar[r, "y"] & Y \ar[d, "p"] \\
    (\Delta^n, \langle \rangle) \ar[r] \ar[urr, dotted, "g", near start] & (\Delta^n, \langle i, - \rangle) \ar[r, "x"] & X,
  \end{tikzcd}
\end{displaymath}
then our chosen solution to the lifting problem
\begin{displaymath}
  \begin{tikzcd}
    ( \partial \Delta^{n+1}, \langle \rangle) \ar[d] \ar[r] & (\partial \Delta^{n+1}, \langle i+1, - \rangle) \ar[d] \ar[r, "{y'}"] & Y \ar[d, "p"] \\
    (\Delta^{n+1}, \langle \rangle) \ar[r] \ar[urr, dotted, "{g.s_i}", near start] & (\Delta^{n+1}, \langle i+1, - \rangle) \ar[r, "{x.s_{i}}"] & X
  \end{tikzcd}
\end{displaymath}
should be $f.s_{i}$, where $y'$ is the map which is $y \cdot s_{i}$ on $S^{n+1}_{i}$ and $f$ on $d_{i}$.

Finally, if we consider the right hand column, then the pullback of the first square along $d_{i+1}$ gives us the original horn square back, while the pullback of the second square along $d_i$ trivialises in the horizontal diagram. Therefore this column yields the following compatibility condition: if $f$ is our chosen solution to the lifting problem
\begin{displaymath}
  \begin{tikzcd}
    ( \partial \Delta^n, \langle \rangle) \ar[d] \ar[r] & (\partial \Delta^n, \langle i, - \rangle) \ar[d] \ar[r, "y"] & Y \ar[d, "p"] \\
    (\Delta^n, \langle \rangle) \ar[r] \ar[urr, dotted, "g", near start] & (\Delta^n, \langle i, - \rangle) \ar[r, "x"] & X,
  \end{tikzcd}
\end{displaymath}
then our chosen solution to the lifting problem
\begin{displaymath}
  \begin{tikzcd}
    ( \partial \Delta^{n+1}, \langle \rangle) \ar[d] \ar[r] & (\partial \Delta^{n+1}, \langle i , - \rangle) \ar[d] \ar[r, "{y'}"] & Y \ar[d, "p"] \\
    (\Delta^{n+1}, \langle \rangle) \ar[r] \ar[urr, dotted, "{g.s_i}", near start] & (\Delta^{n+1}, \langle i, - \rangle) \ar[r, "{x.s_{i+1}}"] & X
  \end{tikzcd}
\end{displaymath}
should be $f.s_{i+1}$, where $y'$ is the map which is $y \cdot s_{i+1}$ on $S^{n+1}_i$ and $f$ on $d_{i+1}$.

To summarise the entire discussion, let us write for each $n \in \mathbb{N}$:
\begin{eqnarray*}
 \mathcal{A}_n & = & \big\{ (i,j, i+1, j) \, : \, i, j \leq n, j \lt i \big\} \\
& \cup & \big\{ (i,j,i,j+1) \, : \, i, j \leq n, j \gt i \big \} \\
& \cup & \big\{ (i,i,i,i+1) \, : \, i \leq n \big\} \\
& \cup & \big\{ (i,i,i+1,i) \, : \, i \leq n \big\}
\end{eqnarray*}

\begin{theo}{unifKanfibrinHornsquares}
  The following notions of fibred structure are isomorphic:
  \begin{itemize}
    \item To be an effective Kan fibration.
    \item To assign to a map $p: Y \to X$ lifts against horn squares, in such a way that for any $n \in \mathbb{N}$, $(i, j, i^*, j^*) \in \mathcal{A}_n$ and $\pm \in \{ +, - \}$: if $f$ is our chosen solution to the lifting problem
    \begin{displaymath}
      \begin{tikzcd}
        ( \partial \Delta^n, \langle \rangle) \ar[d] \ar[r] & (\partial \Delta^n, \langle i, \pm \rangle) \ar[d] \ar[r, "y"] & Y \ar[d, "p"] \\
        (\Delta^n, \langle \rangle) \ar[r] \ar[urr, dotted, "g", near start] & (\Delta^n, \langle i, \pm \rangle) \ar[r, "x"] & X,
      \end{tikzcd}
    \end{displaymath}
    then our chosen solution to the lifting problem
    \begin{displaymath}
      \begin{tikzcd}
        ( \partial \Delta^{n+1}, \langle \rangle) \ar[d] \ar[r] & (\partial \Delta^{n+1}, \langle i^*, \pm \rangle) \ar[d] \ar[r, "{y'}"] & Y \ar[d, "p"] \\
        (\Delta^{n+1}, \langle \rangle) \ar[r] \ar[urr, dotted, "{g.s_j}", near start] & (\Delta^{n+1}, \langle i^*, \pm \rangle) \ar[r, "{x.s_{j^*}}"] & X
      \end{tikzcd}
    \end{displaymath}
    should be $f.s_{j^*}$, where $y'$ is the map which is $y \cdot s_{j^*}$ on $S^{n+1}_{j}$ and $f$ on the faces $d_k$ with $k \in \{ j, j+1 \} - \{ i^* \}$.
  \end{itemize}
\end{theo}

\begin{rema}{hornsqrsnothorns}
The second bullet in the theorem above is really a lifting condition against horn squares, not horns inclusions. We have seen at the beginning of the section that every horn inclusion is induced by some horn square, but for inner horns this horn square is not unique (it is for outer horns). Indeed, if $\Lambda^n_i$ is an inner horn, then it is induced by two horn squares, one coming from $(i, +)$ and one coming from $(i-1,-)$, and our notion of an effective Kan fibration may choose different lifts for these two horn squares.
\end{rema}

\begin{rema}{strengtheningtoisoofdfconcretedoublecat}
  The isomorphism above can again be upgraded to one of discretely fibred concrete double categories in a manner which we will not make precise here.
\end{rema}

A result similar to \reftheo{unifKanfibrinHornsquares} holds for effective right (left) fibrations: this notion of fibred structure is equivalent to having the right lifting property with respect to horn squares with positive (negative) orientation, with the lifts satisfying the compatibility in the second item of that theorem. In fact, since there are no compatibility conditions relating the horn squares with different polarity, we obtain:

\begin{prop}{catprodeffrighthandeffleftiseffKan}
In the category of notions of fibred structure being an effective Kan fibration is the categorical product of being a effective left fibration and being an effective right fibration.
\end{prop}

\subsection{Local character and classical correctness} From the characterisation of effective Kan fibrations in \reftheo{unifKanfibrinHornsquares} we can deduce that our notion of being an effective Kan fibration is both local and classically correct.

\begin{coro}{unifKanfibrlocal}\index{local notion of fibred structure}
The notion of being an effective Kan fibration is a local notion of fibred structure.
\end{coro}
\begin{proof}
Suppose $p: Y \to X$ is a map and for every pullback of $p$ along a map $x: \Delta^n \to X$ we have a stable choice of a structure as in \reftheo{unifKanfibrinHornsquares}. Then, if we are given a lifting problem as in
\begin{displaymath}
  \begin{tikzcd}
    ( \partial \Delta^n, \langle \rangle) \ar[d] \ar[r] & (\partial \Delta^n, \langle i, \pm \rangle) \ar[d] \ar[r, "y"] & Y \ar[d, "p"] \\
    (\Delta^n, \langle \rangle) \ar[r] \ar[urr, dotted, "g", near start] & (\Delta^n, \langle i, \pm \rangle) \ar[r, "x"] & X,
  \end{tikzcd}
\end{displaymath}
then we may pull $p$ back along $x$ and we get:
\begin{displaymath}
  \begin{tikzcd}
    ( \partial \Delta^n, \langle \rangle) \ar[d] \ar[r] & (\partial \Delta^n, \langle i, \pm \rangle) \ar[d] \ar[r, "y"] & Y_x \ar[d, "{x^* p}"] \ar[r] & Y \ar[d, "p"] \\
    (\Delta^n, \langle \rangle) \ar[r] \ar[urr, dotted, "g", near start] & (\Delta^n, \langle i, \pm \rangle) \ar[r, "1"] & \Delta^{n+1} \ar[r, "x"] & X.
  \end{tikzcd}
\end{displaymath}
So, using the lifting structure of $x^*p$, we obtain a map $(\Delta^n, <i, \pm>) \to Y_x$ which we may compose with $Y_x \to Y$. In this way we obtain a lift against $p$. We still have to check that such lifts satisfy the condition in \reftheo{unifKanfibrinHornsquares}.

So imagine that we wish to solve
\begin{displaymath}
  \begin{tikzcd}
    ( \partial \Delta^{n+1}, \langle \rangle) \ar[d] \ar[r] & (\partial \Delta^{n+1}, \langle i^*, \pm \rangle) \ar[d] \ar[r, "{y'}"] & Y \ar[d, "p"] \\
    (\Delta^{n+1}, \langle \rangle) \ar[r] \ar[urr, dotted, "{g.s_j}", near start] & (\Delta^{n+1}, \langle i^*, \pm \rangle) \ar[r, "{x.s_{j^*}}"] & X
  \end{tikzcd}
\end{displaymath}
with $(i, j, i^*, j^*) \in \mathcal{A}_n$ and where $y'$ is the map which is $y \cdot s_{j^*}$ on $S^{n+1}_{j}$ and $f$ on the faces $d_k$ with $k \in \{ j, j+1 \} - \{ i^* \}$. The recipe we were given is that we write this as
\begin{displaymath}
  \begin{tikzcd}
    ( \partial \Delta^{n+1}, \langle \rangle) \ar[d] \ar[r] & (\partial \Delta^{n+1}, \langle i^*, \pm \rangle) \ar[d] \ar[r, "{y'}"] & Y_{x.s_{j^*}} \ar[r] \ar[d, "{(x.s_{j^*})^*p}"] & Y \ar[d, "p"] \\
    (\Delta^{n+1}, \langle \rangle) \ar[r] \ar[urr, dotted, "{g.s_j}", near start] & (\Delta^{n+1}, \langle i^*, \pm \rangle) \ar[r, "1"] & \Delta^{n+2} \ar[r, "{x.s_{j^*}}"] & X,
  \end{tikzcd}
\end{displaymath}
find the induced lift $(\Delta^{n+1}, <i^*, +>) \to Y_{x.s_{j^*}}$ and compose with $Y_{x.s_{j^*}} \to Y$. But we may write the pullback in the previous diagram as the composition of two pullbacks, as follows:
\begin{displaymath}
  \begin{tikzcd}
    ( \partial \Delta^{n+1}, \langle \rangle) \ar[d] \ar[r] & (\partial \Delta^{n+1}, \langle i^*, \pm \rangle) \ar[d] \ar[r, "{y'}"] & Y_{x.s_{j^*}} \ar[r] \ar[d, , "{(x.s_{j^*})^*p}"] & Y_x \ar[r] \ar[d, "{x^*p}"] & Y \ar[d, "p"] \\
    (\Delta^{n+1}, \langle \rangle) \ar[r] \ar[urr, dotted, "{g.s_j}", near start] & (\Delta^{n+1}, \langle i^*, \pm \rangle) \ar[r, "1"] & \Delta^{n+2} \ar[r, "{s_{j^*}}"] & \Delta^{n+1}\ar[r, "x"] & X,
  \end{tikzcd}
\end{displaymath}
By our stability assumption, this means that the composition of the induced lift $(\Delta^{n+1}, <i^*, +>) \to Y_{x.s_{j^*}}$ with $Y_{x.s_{j^*}} \to Y_x$ is the induced lift $(\Delta^{n+1}, <i^*, +>) \to Y_x$ against $x^* p$. But the latter is $f.s_{j^*}$, because $x^*p$ has lifts satisfying the condition in \reftheo{unifKanfibrinHornsquares}. This means that $p$ has lifts satisfying that condition as well, finishing the proof.
\end{proof}

\begin{coro}{KanfibruniformKanfibrclassically}
In a classical metatheory, every map which has the right lifting property against horns (a Kan fibration in the usual sense) can be equipped with the structure of an effective Kan fibration.
\end{coro}
\begin{proof}
Suppose that we have a map which has the right lifting property against all horns. Because a lifting problem against a horn has at most one degenerate solution (see \refprop{degeneratehornfillerunique}), we may always choose the degenerate solution if it exists. In that case our lifts will satisfy the condition in \reftheo{unifKanfibrinHornsquares}, because it says that under certain circumstances we should choose a degenerate solution. But by always choosing the unique degenerate solution (if it exists), this will automatically be satisfied.
\end{proof}

\begin{rema}{onleftandrightagain} 
        We again have similar results for effective left and right fibrations.
        Indeed, proofs which are almost identical to the ones of
        \refcoro{unifKanfibrlocal} and
        \refcoro{KanfibruniformKanfibrclassically} yield:
\begin{itemize}
  \item Being an effective right (left) fibration is a local notion of fibred structure.
  \item In a classical metatheory, a map can be equipped with the structure of
          an effective right (left) fibration if and only if it has the right
          lifting property against horn inclusions $\Lambda_i^n \to \Delta^n$
          with $i \not= 0$ (with $i \not= n$), that is, if and only if it is a
          right (left) fibration in the usual sense.
\end{itemize}
\end{rema}

In addition, the definition of effective Kan fibrations in terms of horn squares allows us to prove that they have the following properties:

\begin{coro}{additionalpropofeffKanfibr} Effective Kan fibrations have the following properties:
  \begin{enumerate}
      \item If each of $f_i: A_i \to B_i $ is effectively Kan, then so is their sum $\coprod f_i: \coprod A_i \to \coprod_{i \in I} B_i$.
      \item If each of $f_i: A_i \to B$ is effectively Kan, then so is their copairing $\coprod f_i: \coprod A_i \to B$.
      \item Discrete presheaves, such as the terminal object and the natural numbers object $\mathbb{N}$, are effectively Kan.
  \end{enumerate}
\end{coro}
\begin{proof}
  Points (1) and (2) follow from \reftheo{unifKanfibrinHornsquares} in combination with the fact that any map $\Delta^n \to \coprod_{i \in I} A_i$ factors through one of the coproduct inclusions $A_j \to \coprod_{i \in I} A_i$.

  Point (3) follows immediately from point (2) since discrete presheaves are coproducts $\coprod_{x \in X} 1$ of the terminal object.
\end{proof}

\section{Conclusion}

In this final section we would like to take stock of the properties of effective Kan fibrations that we have established and outline some directions for future research.

\subsection{Properties of effective Kan fibrations} The properties of effective Kan fibrations that we have established in this paper are:
\begin{enumerate}
    \item Effective Kan fibrations are closed under composition and pullback.
    \item Isomorphisms are effective Kan fibrations.
    \item If $f: Z \to Y$ and $g: Y \to X$ are effective Kan fibrations, then so is $g_*f$, the pushforward\index{pushforward} of $f$ along $g$.
    
    \item The notion of an effective Kan fibration is local and hence universal effective Kan fibrations exist.
    \item The notion of an effective Kan fibration is classically correct in that, in a classical metatheory, a map of simplicial sets can be equipped with the structure of an effective Kan fibration precisely when it has the right lifting property against horn inclusions.
    \item If each of $f_i: A_i \to B$ is effectively Kan, then so is their copairing $\coprod f_i: \coprod A_i \to B$. In particular, discrete presheaves, such as the terminal object and the natural numbers object $\mathbb{N}$, are effectively Kan.
\end{enumerate}   
Properties (1) and (2) are immediate from the fact that effective Kan fibration are defined via a lifting property against a triple category. Property (3) was \reftheo{symmetricPocProofOfPi}. Property (4) was \refcoro{unifKanfibrlocal} in combination with \reftheo{localfsimpliesuniverse}. Property (5) was \refcoro{KanfibruniformKanfibrclassically}. Property (6) was \refcoro{additionalpropofeffKanfibr}.

The main disadvantage of the notion of an effective Kan fibration is that it is not clear (constructively!) that they are closed under retracts. That is, it is unclear how from a diagram
\begin{displaymath}
    \begin{tikzcd}
        \bullet \ar[r] \ar[d, "f"] & \bullet \ar[r] \ar[d, "g"] & \bullet \ar[d, "f"] \\
        \bullet \ar[r] & \bullet \ar[r] & \bullet
    \end{tikzcd}
\end{displaymath}
showing that $f$ is a retract of $g$ and an effective Kan fibration structure on $g$, one is supposed to construct an effective Kan fibration structure on $f$. (Of course, one can in a classical metatheory, because the notion of an effective Kan fibration is classically correct and the usual Kan fibrations are closed under retracts.) This leads to:
\begin{quote}
{\bf Open question:} Can it be proved in a constructive metatheory that the effective Kan fibrations are closed under retracts?
\end{quote}

\subsection{Directions for future research} In future work we plan to show that the effective Kan fibrations are the right class in an algebraic weak factorisation system. At the point of writing we have a classical proof of this fact; we are working on transforming this into a constructive one.

This would also be the first step towards a constructive proof showing that the effective Kan fibrations and the effective trivial Kan fibration yield an algebraic model structure (as in \cite{Riehl-11}) on the category of simplicial sets.

Ultimately we hope to build a model of homotopy type theory in the category of simplicial sets based on our notion of an effective Kan fibration. Here the main challenge will be to construct univalent universes. Other structure that one might hope to be present are certain inductive types, like W-types, as in \cite{vdbergmoerdijk15}. All of this is left to future work.

\printindex

\begin{appendices}

\section{Axioms}\label{sec:axiomsformoore}

In this appendix we will collect the axioms for a Moore category and a dominance that play a role in this paper. The reader can think of these as our version of the Orton-Pitts axioms \cite{ortonpitts18} (see also \cite{Gambino-Sattler,Frumin-vdBerg}).

\subsection{Moore structure} Our first ingredient is a suitable notion of Moore paths.

\begin{defi}{pathobjcat} Let \ct{E} be a category with finite limits. A \emph{Moore structure}~\index{Moore structure} on \ct{E} consists of the following data:
\begin{enumerate}
 \item We have a pullback-preserving endofunctor $M$ on \ct{E} together with natural transformations $r: 1_\ct{E} \to M$, $s, t: M \to 1_\ct{E}$, and $\mu: M_t \times_s M \to M$ turning every object $X$ in \ct{E} in the object of objects of an internal category, with $MX$ as the object of arrows. Note the order in which $\mu$ takes its arguments: it is not in the way categorical composition is usually written. The reason is that we think of $\mu$ as concatenation of paths rather than as categorical composition and we write it as such.
 \item There is a natural transformation $\Gamma: M \to MM$ making $(M,s,\Gamma)$ into a comonad.
 \item There is a strength $\alpha_{X,Y}: X \times MY \to M(X \times Y)$, that is, $\alpha$ is a natural transformation making
\begin{displaymath}
 \begin{array}{c}
 \begin{tikzcd}
 (X \times Y) \times MZ \ar[d, "\cong"] \ar[rr, "\alpha_{X \times Y,Z}"] &  & M((X \times Y) \times Z) \ar[d, "\cong"] \\
 X \times (Y \times MZ) \ar[r, "1_X \times \alpha_{Y,Z}"'] & X \times M(Y \times Z) \ar[r, "\alpha_{X, Y \times Z}"'] & M(X \times (Y \times Z))
 \end{tikzcd} \\
 \begin{tikzcd}
  X \times MY \ar[r, "\alpha_{X,Y}"] \ar[dr, "p_2"'] & M(X \times Y) \ar[d, "Mp_2"] \\
   & MY
 \end{tikzcd}
 \end{array}
\end{displaymath}
 commute. In addition, all the previous structure is strong, so the following diagrams commute as well:
  \begin{displaymath}
  \begin{array}{cc}
   \begin{tikzcd}
    X \times MY \ar[r, "\alpha"] \ar[d, "1 \times s"', shift right = 4ex] \ar[d, "1 \times t", shift left = 2ex] & M(X \times Y) \ar[d, "s"', shift right = 2ex] \ar[d, "t", shift left = 2ex] \\
    X \times Y \ar[r, "1"] \ar[u, "1 \times r"] & X \times Y \ar[u, "r"]
   \end{tikzcd} &
   \begin{tikzcd}
    X \times MY_t \times_s MY \ar[d, "1 \times \mu"] \ar[rrr, "{(\alpha.(p_1,p_2),\alpha.(p_1,p_3))}"] & & & M(X \times Y)_t \times_s M(X \times Y) \ar[d, "\mu"] \\ X \times MY \ar[rrr, "\alpha"] & & & M(X \times Y)
   \end{tikzcd}
  \end{array}
  \end{displaymath}
  \begin{displaymath}
   \begin{tikzcd}
    X \times MY \ar[rrr, "\alpha"] \ar[d, "1 \times \Gamma"] & & & M(X \times Y) \ar[d, "\Gamma"] \\
   X \times MMY \ar[rr, "\alpha"'] & & M(X \times MY) \ar[r, "M\alpha"'] &  MM(X \times Y).
   \end{tikzcd}
  \end{displaymath}
 \item We have the following axioms for the connection $\Gamma$ (interaction with $r,t$):
  \[ \Gamma.r = rM.r, \qquad tM.\Gamma = r.t, \qquad Mt.\Gamma = \theta_X.\alpha_{X,1}.(t,M!), \]
    with $\theta_X$ being the iso $Mp_1: M(X \times 1) \to MX$.
  \item And, finally, we have the following \emph{distributive law}\index{distributive law} (interaction between $\Gamma$ and $\mu$):
    \begin{align*}
      \Gamma.\mu &=
  \mu.(M\mu.\nu.(\Gamma.p_1,\theta_{MX}.\alpha_{MX,1}.(p_2,M!.p_1)),\Gamma.p_2) \\
                 &\thinspace : \thinspace
      MX \times_X MX \to MMX
\end{align*}
   with $\nu$ being the natural transformation
   (in this case $MMX \times_{MX} MMX \to M(MX \times_X MX)$) induced by
   preservation of pullbacks. This condition can be visualized as follows. When
   \(p, q \in MX\) are composable Moore paths as in the left-hand size of the diagram, then
   \(\Gamma. \mu (p, q)\) is defined by \(\Gamma(p)\) and \(\Gamma(q)\) in the following way:
   \begin{equation}\label{eq:distributive-law-viz}
    \begin{tikzpicture}[xscale=1.2,baseline={([yshift=-.5ex]current bounding box.center)}]
      \node (z) at (0,2) {$z$};
      \node (y) at (0,1) {$y$};
      \node (x) at (0,0) {$x$};
      \node (y') at (1,0) {$y$};
      \node (z') at (2,0) {$z$};
      \node (s) at (0,-0.6) {};
      \node (t) at (2,-0.6) {};
      \path [->,draw] (x) -- node [left] {$p$} (y);
      \path [->,draw] (y) -- node [left] {$q$} (z);
      \path [->,draw] (x) -- (y');
      \path [->,draw] (y') -- (z');
      \path [->,draw] (z) -- (z');
      \path [->,draw] (s) -- node [below=5pt] {$\Gamma . \mu(p,q)$} (t);
    \end{tikzpicture}
    \quad= \qquad
    \begin{tikzpicture}[xscale=1.3,baseline={([yshift=-.5ex]current bounding box.center)}]
      \node (z) at (0,2) {$z$};
      \node (y) at (0,1) {$y$};
      \node (x) at (0,0) {$x$};
      \node (y') at (1,0) {$y$};
      \node (z') at (2,0) {$z$};
      \node (z'') at (1,1) {$z$};
      \node (s2) at (1,-0.8) {};
      \node (t2) at (2,-0.8) {};
      \node (s1) at (0,-1.1) {};
      \node (t1) at (1,-1.1) {};
      \node (s3) at (0,-0.4) {};
      \node (t3) at (1,-0.4) {};
      \path [->,draw] (x) -- node [left] {$p$} (y);
      \path [->,draw] (y) -- node [left] {$q$} (z);
      \path [->,draw] (x) -- (y');
      \path [->,draw] (y') -- (z');
      \path [->,draw] (y) -- (y');
      \path [->,draw] (y') -- (z'');
      \path [->,draw] (z) -- (z'');
      \path [->,draw] (z'') -- (z');
      \path [->,draw] (s2) -- node [below=5pt] {$\Gamma(q)$} (t2);
      \path [->,draw] (s1) -- node [below] {$\Gamma(p)$} (t1);
      \path [->,draw] (s3) -- node [below] {$\alpha(q, M!p)$} (t3);
    \end{tikzpicture}
\end{equation}
\end{enumerate}
Whenver a category \ct{E} is equipped with structure thus described, we call
\ct{E} a \emph{category with Moore structure}, or a \emph{Moore category} for
short. \index{category with Moore structure} \index{Moore category}
\end{defi}

\begin{rema}{compwithpathobjcat} The notion of a path object category from \cite{vdBerg-Garner} can be obtained from this by dropping the coassociativity axiom for $\Gamma$ as well as the distributive law, whilst adding a symmetry $\tau$ (see below).
\end{rema}

\begin{rema}{onthestrength}
As observed in \cite{vdBerg-Garner}, the fact that $M$ preserves pullbacks means that the entire strength is determined by the maps
\begin{displaymath}
 \begin{tikzcd}
  \alpha_X := X \times M1 \ar[r, "\alpha_{X,1}"] & M(X \times 1) \ar[r, "\theta_X"] & MX.
 \end{tikzcd}
\end{displaymath}
The reason for this is that the outer rectangle and right hand square in
 \begin{displaymath}
   \begin{tikzcd}
   X \times MY \ar[r, "\alpha_{X,Y}"'] \ar[d, "1_X \times M!"'] \ar[rr, bend left = 20, "p_2"] & M(X \times Y) \ar[d, "Mp_1"] \ar[r, "Mp_2"'] & MY \ar[d, "M!"] \\
   X \times M1 \ar[r, "\alpha_X"] \ar[rr, bend right = 20, "p_2"'] & MX \ar[r, "M!"] & M1
   \end{tikzcd}
   \end{displaymath}
are pullbacks. And, if we wish, axioms (4-6) can also be formulated as follows: there is a natural transformation \[ \alpha_X: X \times M1 \to MX \]
with $M!.\alpha_X = p_2: X \times M1 \to M1$, and, in addition, the following diagrams commute:
 \begin{displaymath}
  \begin{array}{cc}
   \begin{tikzcd}
    X \times M1 \ar[r, "\alpha"] \ar[d, "1 \times s"', shift right = 4ex] \ar[d, "1 \times t", shift left = 2ex] & MX \ar[d, "s"', shift right = 2ex] \ar[d, "t", shift left = 2ex] \\
    X \times 1 \ar[r, "\cong", "p_1"'] \ar[u, "1 \times r"] & X \ar[u, "r"]
   \end{tikzcd} &
   \begin{tikzcd}
    X \times M1 \times M1 \ar[d, "1 \times \mu"] \ar[rrr, "{(\alpha.(p_1,p_2),\alpha.(p_1,p_3))}"] & & & MX \times_X MX \ar[d, "\mu"] \\ X \times M1 \ar[rrr, "\alpha"] & & & MX
   \end{tikzcd}
  \end{array}
  \end{displaymath}
  \begin{displaymath}
   \begin{tikzcd}
    X \times M1 \ar[rrrr, "\alpha"] \ar[d, "1 \times \Gamma"] & & & & MX \ar[d, "\Gamma"] \\
   X \times MM1 \ar[rr, "{(\alpha.(p_1,M!.p_2),p_2)}"'] & & MX \times_{M1} MM1 \ar[r,"\cong"] & M(X \times M1) \ar[r, "M\alpha"'] &  MMX.
   \end{tikzcd}
  \end{displaymath}
Finally, we have the following axioms for the interaction between the connection $\Gamma$ and the category structure:
\begin{eqnarray*}
 \Gamma.r & = & rM.r, \\
 tM a.\Gamma & = & r.t, \\
 Mt.\Gamma & = & \alpha.(t,M!), \\
 \Gamma.\mu & = & \mu.(M\mu.\nu_X.(\Gamma.p_1,\alpha_{MX}.(p_2,M!.p_1)),\Gamma.p_2).
\end{eqnarray*}
\end{rema}

\begin{defi}{twosidedpathobjcat}
We will call a Moore structure \emph{two-sided}\index{two-sided Moore
structure|textbf} if it also comes equipped with a map $\Gamma^*: M \to MM$
turning $(M, t, \Gamma^*)$ into a strong comonad, and such that the following
equations hold:
\begin{equation} \label{fiets}
  \begin{array}{rcl}
 \Gamma^*.r & = & rM.r, \\
 s.\Gamma^* & = & r.s, \\
 Ms.\Gamma^* & = & \alpha.(s,M!), \\
 \Gamma^*.\mu & = &\mu.(\Gamma^*.p_1,M\mu.\nu.(\alpha_M.(p_1,M!.p_2),\Gamma^*.p_2)).
\end{array}
\end{equation}
This has the effect that if we switch $s$ and $t$ and define $\mu^* := \mu.(p_2,p_1)$, then we get a second Moore structure. We will also require that $\mu$ is both left and right cancellative, and that we have the sandwich equation:\index{sandwich equation}
\[ M\mu.\nu.(\Gamma^*,\Gamma) = \alpha.(1,M!): M \to MM \]
(which also implies $M\mu^*.\nu.(\Gamma,\Gamma^*) = \alpha.(1,M!)$).
\end{defi}

\begin{defi}{symmetricpathobjcat}
A two-sided Moore structure will be called \emph{symmetric}\index{symmetric
Moore structure|textbf} if it also comes equipped with a natural transformation
$\tau: M \to M$ (sometimes referred to as a `twist map'\index{twist
map|textbf}) such that
\begin{eqnarray*}
\tau.\tau & = & 1, \\
\tau.r & = & r, \\
s.\tau & = & t, \\
t.\tau & = & s, \\
\Gamma^* & = & \tau M.M\tau.\Gamma.\tau,
\end{eqnarray*}
while also the following diagrams commute:
\begin{displaymath}
 \begin{array}{cc}
  \begin{tikzcd}
   X \times M1 \ar[r,"\alpha"] \ar[d, "1 \times \tau"'] & MX \ar[d, "\tau"] \\
   X \times M1 \ar[r, "\alpha"] & MX
  \end{tikzcd} &
  \begin{tikzcd}
   MX_t \times_s MX \ar[r, "\mu"] \ar[d,"{(\tau.p_2,\tau.p_1)}"'] & MX \ar[d,"\tau"] \\
   MX_t \times_s MX \ar[r, "\mu"] & MX
  \end{tikzcd}
 \end{array}
\end{displaymath}
\end{defi}

\begin{rema}{fietsequationsfollow}
 There is a bit of redundancy in the previous definition, in that $\Gamma^*  = \tau M.M\tau.\Gamma.\tau$ implies the equations (\ref{fiets}) above.
\end{rema}

\begin{exam}{examplesofsymmMoorecats} The following examples from \cite{vdBerg-Garner} all satisfy the axioms for a symmetric Moore category.
  \begin{enumerate}
    \item The category of topological spaces with
    \[ MX = \sum_{r \in \mathbb{R}_{\geq 0}} X^{[0,r]}, \]
    the space of Moore paths.
    \item The category of small groupoids with
    \[ MX = X^{\mathbb{I}}, \]
    where $\mathbb{I}$ is the interval groupoid containing two objects and one arrow $x \to y$ for any pair of objects $(x,y)$. In fact, this also defines a symmetric Moore structure on the category of small categories.
    \item The category of chain complexes over a ring $R$.
  \end{enumerate}
  For more details we refer to \cite[Section 5]{vdBerg-Garner}.
\end{exam}

\subsection{Dominance} The second ingredient is a dominance\cite{rosolini86}.

\begin{defi}{dominanceagain}
  A \emph{dominance}\index{dominance} on a category \ct{E} is a class of monomorphism $\Sigma$ in \ct{E} satisfying the following three properties:
  \begin{enumerate}
    \item every isomorphism is in $\Sigma$ and $\Sigma$ is closed under composition.
    \item every pullback of a map in $\Sigma$ again belongs to $\Sigma$.
    \item the category $\Sigma_{cart}$ of maps in $\Sigma$ and pullback squares between them has a terminal object $1 \to \Sigma$.
  \end{enumerate}
\end{defi}

For some of our arguments it will be convenient to assume the following two additional axioms:
\begin{enumerate}
  \item The elements in $\Sigma$ are closed under finite unions; that is, $0 \to X$ always belongs to $\Sigma$ and whenever $A \to X$ and $B \to X$ belong to $\Sigma$, then so does $A \cup B \to X$.
  \item The morphism $r_X: X \to MX$ belongs to $\Sigma$ for any object $X$.
\end{enumerate}

\section{Cubical sets}\label{sec:cubicalsets}\index{connection (cubical sets)|textbf}
In this section, we show that the category of \emph{cubical sets}
(\cite{CCHM17}, also~\cite{Gambino-Sattler}) can be equipped with parts of the
definition of a Moore structure. This makes it possible to highlight the
relationship between the notion of Moore structure introduced in this paper and
the related notion of \emph{connections} in cubical sets.  The conclusion is
that, for the standard choice of path object, all structure except from path
composition is present. The path contraction \(\Gamma\) corresponds to the
notion of \emph{connection} in cubical sets.

The \emph{category of cubes} \(\mathcal{C}\) has as objects finite subsets
\(I\) of a fixed and countable set of \emph{names} \(\{x,y,z,\dots\}\), and
morphisms \(I \to J\) are set maps \(J \to \mathsf{dM}(I)\) into the \emph{free
de Morgan algebra} generated by \(I\).  The free de Morgan algebra is the
bounded distributive lattice generated by the set of constants and binary
operations \(\{0, 1, \land, \lor\}\), and the involutive function \(\neg\)
representing negation, which needs to satisfy de Morgan's laws.

The category of cubical sets\index{cubical sets|textbf} is the category of presheaves
on \(\mathcal{C}\).  A cubical set \(X\) can be regarded as a set
\(X(x_1,\dots,x_n)\) for every given tuple of variables \(x_1, \dots, x_n\),
whose elements support several substitution operations, e.g.:
\begin{equation}
  x_i = 0 \colon X(x_1, \dots, x_n) \to X(x_1, \dots, \widehat{x_i}, \dots, x_n)
\end{equation}
The operations \(\land, \lor\) give rise to \emph{connections}, while the
operation \(\neg\) gives rise to a \emph{symmetry}. Therefore this category is
sometimes referred to as cubical sets with connections and symmetries, to
distinguish it from possible other (more restricted) definitions. For an
introduction to cubical sets, we refer to~\cite{CCHM17}.

An element \(\gamma \in X(y_1,\dots,y_n, x)\) can be regarded as a `path'
between \(X(x = 0)(\gamma)\) and \(X(x = 1)(\gamma)\). This suggests to define
a `path object functor' \(M\) as follows:
\begin{align}\label{eq:cubicalmoorepath}
  MX(I) &:= X(I \cup \{x\}) \text{ where } x \notin I \\
  M\eta (I) &:= \eta_{I \cup \{x\}}
\end{align}
with source, target, and identity maps induced by:
\begin{align*}
  s_X(I) := X(x = 0) &\colon X(I \cup \{x\}) \to X(I) \\
  t_X(I) := X(x = 1) &\colon X(I \cup \{x\}) \to X(I) \\
  r_X(I) := X(\id) &\colon X(I) \to X(I \cup \{x\})
\end{align*}

Using connections, we can define a contraction \(\Gamma\):
\begin{equation*}
  \Gamma_X(I) := X(x = x \lor y) \colon X(I \cup \{x\}) \to X(I \cup \{x\} \cup \{y\}).
\end{equation*}

We can readily verify that \(M\) is a comonad:
\begin{align*}
  s_{MX}.\Gamma_X &= X(y = 0) . X(x = x \lor y) = 1_{MX} \\
  M(s_X). \Gamma_X &= M(X(x = 0)) . X(x = x \lor y) = X(y = 0). X(x = x \lor y) \\
                   &= 1_{MX} \\
  \Gamma_{MX}.\Gamma_X &= X(y = y \lor z) . X(x = x \lor y) \\
                       &=  M(\Gamma_X) . \Gamma_X
\end{align*}

Similarly, it is easy to define a strength and show that requirement (4) of \refdefi{pathobjcat} holds.

Further, the \(\land\) operation would define a dual \(\Gamma^*\) as in the
definition of two-sided Moore structure (\refdefi{twosidedpathobjcat}), whereas
negation \(\neg\) can be used to define a `twist map' as in
\refdefi{symmetricpathobjcat}. The fact that the \(\Gamma^*\) derived from the
twist map is the same as the former follows from the fact that negation is required
to be an involution, i.e., \(\neg \neg x = x\).

All that is missing is a definition of multiplication, or path concatenation:
\begin{equation}
  \mu_X(I) : X(I \times \{x\}) {}_{x = 1} \times_{x = 0} X(I \times \{x\}) 
  \to X(I \times \{x\})
\end{equation}
To define this, it would be needed to `compose' a pair of paths \(\gamma,
\gamma'\) which satisfy \(X(x = 1)(\gamma) = X(x = 0)(\gamma')\).  Yet, given
that all paths have only single length, this is not straightforward.

\section{Degenerate horn fillers are unique}

The purpose of this appendix is to show that horn filling problems have at most one degenerate filler, in the following sense:

\begin{prop}{degeneratehornfillerunique}
If both $x_0 \cdot \sigma_0$ and $x_1 \cdot \sigma_1$ are fillers for
\begin{displaymath}
\begin{tikzcd}
  \Lambda_i^n \ar[r] \ar[d] & X \\
  \Delta^n \ar[ur, dotted]
\end{tikzcd}
\end{displaymath}
where $\sigma_0: \Delta^n \to \Delta^k$ and $\sigma_1: \Delta^n \to \Delta^l$ are epimorphisms in $\mathbf{\Delta}$ different from the identity, then $x_0 \cdot \sigma_0 = x_1 \cdot \sigma_1$.
\end{prop}

The proof strategy that we will follow here was suggested to us by Christian Sattler. The (constructive) argument relies on the following lemma (see \cite[Lemma 5.6]{vdbergmoerdijk15}):

\begin{lemm}{swanslemmadual} Suppose that we have a diagram of the form
  \begin{displaymath}
    \begin{tikzcd}
      \bullet \ar[d, "f"] \ar[r] & \bullet \ar[d, "g"] \ar[r] & \bullet \ar[d, "f"] \\
      \bullet \ar[r] & \bullet \ar[r] & \bullet
    \end{tikzcd}
  \end{displaymath}
  exhibiting $f$ as a retract of $g$, while $g$ has a section. Then the right hand square is an absolute pushout.
\end{lemm}

\begin{proof} (Of \refprop{degeneratehornfillerunique}.) It suffices to consider the case where $\sigma_0 = s_i$ and $\sigma_1 = s_j$ with $i \lt j$.

Because $i \not= j + 1$ in at least one of the following diagrams the dotted arrow exists:
  \begin{displaymath}
  \begin{array}{cc}
    \begin{tikzcd}
      & \Lambda^n_k \ar[d] \\
      \Delta^{n-1} \ar[d, "s_{j-1}"] \ar[r, "d_{i}"]  \ar[ur, dotted, bend left]     & \Delta^n \ar[d, "s_j"] \ar[r, "s_i"] & \Delta^{n-1} \ar[d, "s_{j-1}"] \\
      \Delta^{n-2} \ar[r, "d_{i}"] & \Delta^{n-1} \ar[r, "s_{i}"] & \Delta^{n-2}.
    \end{tikzcd} &
  \begin{tikzcd}
     & \Lambda^n_k \ar[d] \\
     \Delta^{n-1} \ar[d, "s_i"] \ar[r, "d_{j+1}"]  \ar[ur, dotted, bend left]    & \Delta^n \ar[d, "s_i"] \ar[r, "s_j"] & \Delta^{n-1} \ar[d, "s_i"] \\
     \Delta^{n-2} \ar[r, "d_{j}"]" & \Delta^{n-1} \ar[r, "s_{j-1}"] & \Delta^{n-2}
   \end{tikzcd}
 \end{array}
 \end{displaymath}
In either case, the previous lemma implies that both the inner and outer square in
\begin{displaymath}
\begin{tikzcd}
  \Lambda^n_k \ar[dr] \ar[drr, bend left] \ar[ddr, bend right] \\
        & \Delta^n \ar[d, "s_i"] \ar[r, "s_j"] & \Delta^{n-1} \ar[d, "s_i"] \\
   & \Delta^{n-1} \ar[r, "s_{j-1}"] & \Delta^{n-2}
\end{tikzcd}
\end{displaymath}
are pushouts, from which the proposition follows.
\end{proof}

\section{Uniform Kan fibrations}\label{sec:ongambinosattler}

This section is devoted to a proof by Christian Sattler \cite{sattler18b} showing that the notion of a uniform Kan fibration as used in \cite{Gambino-Sattler} is not a local notion of fibred structure. In this connection it is interesting to note that their definition is the same as the one used by Coquand and others \cite{CCHM17} for the case of cubical sets, where this notion of fibred structure is local.

Their definition makes use of the \emph{pushout product} (or \emph{Leibniz product}) $f \hat{\otimes} g$ \index{pushout product} \index{Leibniz product} of two maps $f: A \to B$ and $g: C \to D$, which is the inscribed map from the pushout as in
\begin{displaymath}
    \begin{tikzcd}
        A \times C \ar[r, "f \times 1"] \ar[d, "1 \times g"] & B \times C \ar[d, "\iota_2"] \ar[ddr, bend left, "1 \times g"] \ar[d] \\
        A \times D \ar[drr, bend right, "f \times 1"] \ar[r, "\iota_1"] & A \times D \sqcup_{A \times C} B \times C \ar[dr, dotted, "f \hat{\otimes} g"] \\
        & & B \times D.
    \end{tikzcd}
\end{displaymath}
For the cofibrations they take the same class as maps as us (or as in the standard theory): the (pointwise decidable) monomorphisms. The generating trivial cofibrations are then the maps $m \hat{\otimes} \delta_i$ where $m$ is cofibration and $\delta_i: 1 \to \mathbb{I}$ is one of the end point inclusions. It is a classical result \cite[Chap. IV, Sec. 2]{gabrielzisman67} that these maps, like the horn inclusions, generate the Kan fibrations. In fact, this class would already be generated by the maps $m \hat{\otimes} d_i$ where $d_i$ is an end point inclusion and $m$ is a cofibrant sieve $S \subseteq \Delta^n$ (see \emph{loc.cit.}).

The innovation of Gambino and Sattler is to add a uniformity condition to this definition, where this uniformity condition was suggested by the work of Coquand and others. The cofibrations can be seen as the objects in a category where the morphisms $m \to n$ are pullback squares
\begin{displaymath}
    \begin{tikzcd}
        A \ar[r] \ar[d, "m"] & C \ar[d, "n"] \\
        B \ar[r] & D
    \end{tikzcd}
\end{displaymath}
with composition and identities as in the arrow category on simplicial sets. If we have such a pullback square and $\delta_i: 1 \to \mathbb{I}$ is an endpoint inclusion, then we obtain another pullback square
\begin{displaymath}
    \begin{tikzcd}
        B \times \{ i \} \cup A \times \mathbb{I} \ar[r] \ar[d, "m \hat{\otimes} d_i"] & D \times \{ i \} \cup C \times \mathbb{I} \ar[d, "n \hat{\otimes} d_i"] \\
        B \times \mathbb{I} \ar[r] & D \times \mathbb{I}.
    \end{tikzcd}
\end{displaymath}
This suggest the following notion of fibred structure on simplicial sets:

\begin{defi}{uniformKanfibration} \cite{Gambino-Sattler}
    A \emph{uniform Kan fibration structure} \index{uniform Kan fibration} on a map $p: Y \to X$ of simplicial sets is a function which, given a cofibration $m: A \to B$, an endpoint inclusion $\delta_i: 1 \to \mathbb{I}$ and a solid commutative square
    \begin{displaymath}
        \begin{tikzcd}
            B \times \{ i \} \cup A \times \mathbb{I} \ar[r] \ar[d, "m \hat{\otimes} \delta_i"'] & Y \ar[d, "p"] \\
            B \times \mathbb{I} \ar[r] \ar[ur, dotted] & X
        \end{tikzcd}
    \end{displaymath}
    chooses a dotted filler as shown. These chosen fillers are supposed to satisfy the following \emph{uniformity condition}: for any pullback square
    \begin{displaymath}
        \begin{tikzcd}
            A \ar[r] \ar[d, "m"] & C \ar[d, "n"] \\
            B \ar[r] & D,
        \end{tikzcd}
    \end{displaymath}
    the chosen fillers should make the inscribed triangle in
    \begin{displaymath}
        \begin{tikzcd}
            B \times \{ i \} \cup A \times \mathbb{I} \ar[r] \ar[d, "m \hat{\otimes} d_i"] & D \times \{ i \} \cup C \times \mathbb{I} \ar[d, "n \hat{\otimes} d_i"] \ar[r] & Y \ar[d, "p" ]\\
            B \times \mathbb{I} \ar[r] \ar[urr, dotted] & D \times \mathbb{I} \ar[r] \ar[ur, dotted] & X
        \end{tikzcd}
    \end{displaymath}
    commute.
\end{defi}  

In \cite{Gambino-Sattler} Gambino and Sattler show that their notion of uniform Kan fibration satisfies the following properties:
\begin{enumerate}
    \item[(1)] One can show, constructively, that the uniform Kan fibrations are closed under pushforward: so if $f$ and $g$ are uniform Kan fibration, then so is $\Pi_f(g)$.
    \item[(2)] Their notion is classically correct, in that one can show in a classical metatheory that every map which has the right lifting property against horn inclusions can be equipped with the structure of a uniform Kan fibration.
\end{enumerate}   
In this way they circumvent the BCP-obstruction as discussed in the introduction. What is left open, however, is whether it is also local. It turns out that it is not and in this appendix we will give Christian Sattler's proof of this fact. (It should be noted that the proof makes use of classical logic.

\begin{rema}{howdoesthatlooktous?} In terms of mould squares the uniform Kan fibrations can be understood as follows. Note that $p: Y \to X$ having the right lifting property against a map $m \hat{\otimes} \delta_i$ means that given any solid diagram as in
    \begin{displaymath}
        \begin{tikzcd}
            A \times \{ i \} \ar[r] \ar[d] & A \times \mathbb{I} \ar[d] \ar[r] & Y \ar[d, "p"] \\
            B \times \{ i \} \ar[r] \ar[urr, dotted] & B \times \mathbb{I} \ar[r] \ar[ur, dotted] & X
        \end{tikzcd}
    \end{displaymath}
    and dotted arrow $B \times \{ i \} \to Y$ we can find a dotted arrow (pushforward) $B \times \mathbb{I} \to Y$ as shown. Since $\delta_i: 1 \to \mathbb{I}$ are HDRs (they are maps of the form $\Delta^n \to \widehat{\theta}$ with $\theta$ being the 1-dimensional traversal $<0, \pm>$), the square on the left in the diagram above is a mould square. The uniformity condition says the following: suppose $m \to n$ is a pullback square and we are given a solid diagram
    \begin{displaymath}
        \begin{tikzcd}
            & C \times \{ i \} \ar[dd] \ar[rr] & & C \times \mathbb{I} \ar[rr] \ar[dd] &  & Y \ar[dd, "p"] \\
            A \times \{ i \} \ar[rr] \ar[dd] \ar[ur] &  & A \times \mathbb{I} \ar[dd] \ar[ur] \\ & D \times \{ i \} \ar[rr] & & D \times \mathbb{I} \ar[rr] & & X \\
            B \times \{ i \} \ar[ur] \ar[rr] & &  B \times \mathbb{I} \ar[ur]
        \end{tikzcd}
    \end{displaymath}
    Then any arrow $D \times \{ i \} \to Y$ making everything commute induces two pushforwards: one of the form $D \times \mathbb{I} \to Y$ using that the back of the cube is a mould square and one of the form $B \times \mathbb{I} \to Y$ using that the front of the cube is a mould square. We demand that these commute via the map $B \times \mathbb{I} \to D \times \mathbb{I}$. Since the cube in the diagram above is a mould cube (a cube in the triple category of mould squares), we see that this uniformity condition is also part of our definition of an effective Kan fibration. In particular, effective Kan fibrations are uniform Kan fibrations; or, more precisely, there is a morphism of notions of fibred structure from the effective Kan fibrations to the uniform Kan fibrations. So what the notion of an effective Kan fibration adds to that of a uniform Kan fibration is the following:
    \begin{enumerate}
        \item[(1)] Instead of requiring the existence of pushforwards for HDRs of the form $B \times \{ i \} \to B \times \mathbb{I}$ only, we demand their existence for all HDRs.
        \item[(2)] Since HDRs can be composed, mould squares can be composed horizontally leading to an additional horizontal compatibility condition.
        \item[(3)] Since cofibrations can be composed, mould squares can be composed vertically leading to an additional vertical compatibility condition.
        \item[(4)] There is more general notion of mould cube leading to a more general perpendicular compatibility condition.
    \end{enumerate}   
    As we have seen, the result is a local notion of fibred structure.
    \end{rema} 

In order to show that the uniform Kan fibrations are not a local notion of fibred structure, we make a number of useful observations. First of all, in the definition of a uniform Kan fibration we can restrict attention to cofibrations of the form $S \subseteq \Delta^n$ and pullback squares of the form:
\begin{displaymath}
    \begin{tikzcd}
        \alpha^* S \ar[d] \ar[r] & S \ar[d] \\
        \Delta^m \ar[r, "\alpha"] & \Delta^n.
    \end{tikzcd}
\end{displaymath}
To be a bit more precise: if we would define another notion of fibred structure by assiging to each map $p: Y \to X$ all the solutions to lifting problems against maps of the form $m \hat{\otimes} \delta_i$ where $m$ is a cofibrant sieve satisfying the uniformity condition, then the morphism of notions of fibred structure from the uniform Kan fibrations to this notion of fibred structure induced by the embedding of the full subcategory of cofibrant sieves into the category of al cofibrations and pullback squares between them, is an isomorphism. This can be using methods which are very similar to those used in Section 8. 

What this implies is that there is also a more ``internal'' way of stating what a uniform Kan fibration is. If $p: Y \to X$ is a map, then we can define simplicial sets:
\begin{eqnarray*}
    Probl(p)_n & = & \{ (\pi: \Delta^n \times \mathbb{I} \to X, S \subseteq \Delta^n, y \in Y_n, \rho: S \times \mathbb{I} \to Y) \, : \\
    & & \, \pi(-, i) = p(y), p \circ \rho = \pi \upharpoonright S \times \mathbb{I}, \rho(-,i) = y \upharpoonright S \} \\
    Sol(p)_n & = & \{ (\pi, S, y, \rho, \lambda: \Delta^n \times \mathbb{I} \to Y) \, : \, \\
    & & 
    (\pi,S,y,\rho) \in Probl(p)_n, p \circ \lambda = \pi, \lambda(-,i) = y, \lambda \upharpoonright S \times \mathbb{I} = \rho \}
\end{eqnarray*}
There is a forgetful map of simplicial maps $U_p: Sol(p) \to Probl(p)$. Saying that $p$ has right lifting property against all maps of the form $m \hat{\otimes} \delta_i$ (with $m$ a cofibrant sieve) is the same thing as saying that this map $U_p$ is epi. For $p$ to be a uniform Kan fibration means that this map has a section: indeed, a uniform Kan fibration structure on $p$ can be understood as a section of this map.

At this point it is good to remember that in a classical metatheory sections of a map $f: B \to A$ of simplicial sets can be constructed by induction on the dimension. That is, we construct maps $s_n: A_n \to B_n$ by induction on $n$, where we make a case distinction on whether an element $a \in A_n$ is degenerate or not. Indeed, if $a \in A_n$ is degenerate, then $a = a' \cdot \sigma$ for some unique non-degenerate $a'$ and epi $\sigma: [n] \to [k]$. This means that the value $s_n(a)$ is determined by $s_k(a')$. If $a$ is non-degenerate, then any face of $s(a)$ is already determined $s_{n-1}$; this means that $s(a)$ has to get the correct boundary as determined by $s_{n-1}$ and be such that $p(s(a)) = a$. However, beyond these requirements we are free to choose $s(a)$ in any manner we like; meaning that any such choice will yield a map of simplicial sets. In particular, if we choose our solutions following this recipe the lifts will automatically satisfy the uniformity condition.

Note that a notion of fibred structure ${\rm Fib}$ could fail to be local in two different ways:
\begin{enumerate}
    \item[(1)] There is a map $p: Y \to X$ and distinct elements $a,b \in {\rm Fib}(p)$ such that for any pullback square $\sigma$ of the form 
    \begin{displaymath}
        \begin{tikzcd}
            Y_x \ar[r] \ar[d] & Y \ar[d, "p"] \\
            \Delta^n \ar[r, "x"] & X  
        \end{tikzcd}
    \end{displaymath}
    we have ${\rm Fib}(\sigma)(a) = {\rm Fib}(\sigma)(b)$.
    \item[(2)] There is a map $p: Y \to X$ as well as a function $a$ which assigns to each $x \in X_n$ an element $a_x \in {\rm Fib}(Y_x \to \Delta^n)$ in such a way that for any pullback square $\tau$ of the form
    \begin{displaymath}
        \begin{tikzcd}
            Y_{x \cdot \alpha} \ar[r] \ar[d] & Y_x \ar[d] \\
            \Delta^m \ar[r, "\alpha"] & \Delta^n
        \end{tikzcd}
    \end{displaymath}
    we have ${\rm Fib}(\tau)(a_x) = a_{x \cdot \alpha}$, without there being an element $a \in {\rm Fib}(p)$ such that for any pullback square $\sigma$ as in (1) we have ${\rm Fib}(\sigma)(a) = a_x$.
\end{enumerate}  
It turns out that the notion of a uniform Kan fibration fails to be local in both ways.

\begin{prop}{unifKanfibrnotlocal1}
    The notion of a uniform Kan fibration fails to be local for the reason explained in (1). 
\end{prop}
\begin{proof}
    Let $X = \Delta^1 \times \Delta^1$ and $Y = {\rm Cosk}^1(X \cup \{ f \})$; that is, $Y$ is the 1-coskeleton of $X \cup \{ f \}$ where $f$ is an additional copy of the edge $e: (1,0) \to (1,1)$. The map $p: Y \to X$ is the unique extension of the identity on $X$. 

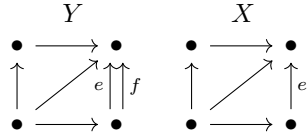
    \begin{displaymath}
        \begin{array}{ccc}
            Y & X \\
            \begin{tikzcd}
                \bullet \ar[r] & \bullet \\
                \bullet \ar[u] \ar[ur] \ar[r] & \bullet \ar[u, shift left, "e"] \ar[u, shift right, "f"']
            \end{tikzcd}  & 
            \begin{tikzcd}
                \bullet \ar[r] & \bullet \\
                \bullet \ar[u] \ar[ur] \ar[r] & \bullet \ar[u, "e"']
            \end{tikzcd}
        \end{array}
    \end{displaymath}

    We wish to construct two uniform Kan fibration structures on $p$, which we do following the recipe outlined above. So suppose we have a lifting problem as follows:
    \begin{displaymath}
        \begin{tikzcd}
            \Delta^n \times \{ i \} \cup S \times \mathbb{I} \ar[r] \ar[d] & Y \ar[d, "p"] \\
            \Delta^n \times \mathbb{I} \ar[r] \ar[ur, dotted] & X.
        \end{tikzcd}
    \end{displaymath}
    Note for any $x \in X_n$ we have:
    \begin{eqnarray*}
        p_n^{-1}(x) & = & \left\{ \begin{array}{ll}
            \{ e, f \} & \mbox{if } x = e \\
            \{ \bullet \} & \mbox{else}
        \end{array} \right. 
    \end{eqnarray*}
    Therefore the number of solutions to a lifting problem as the one above is typically 1, except when the image of the map $\Delta^n \times \mathbb{I} \to X$ includes the edge $e$, while the image of the map $\Delta^n \times \{ i \} \cup S \times \mathbb{I} \to X$ excludes it: in this case the number of solutions is 2, where each such is completely determined by the lift of the edge $e$ (whether it is $e$ or $f$).

    So to define our uniform Kan fibration structures $a, b$ we assume we are given a lifting problem as above where the image of the map $\Delta^n \times \mathbb{I} \to X$ includes the edge $e$, while the image of the map $\Delta^n \times \{ i \} \cup S \times \mathbb{I} \to X$ excludes it. We do this by making a case distinction on whether the
    map $\Delta^n \times \mathbb{I} \to X$ factors through $u$ where $u: \Delta^2 \to X$ if the unique inclusion whose image includes $e$. If it does, both $a$ and $b$ choose the lift which map $e$ to $e$; if it does not, $a$ lifts $e$ to $e$, while $b$ lifts $e$ to $f$.

    First of all, note that $a$ and $b$ are distinct, because they provide different solutions to problem of filling the ``open box'' 
    \begin{displaymath}
        \begin{tikzcd}
            \Delta^1 \times \{ 0 \} \cup \{ 0, 1 \} \times \mathbb{I} \ar[r] \ar[d] & Y \ar[d, "p"] \\
            \Delta^1 \times \mathbb{I} \ar[r] \ar[ur, dotted] & X
        \end{tikzcd}
    \end{displaymath}
    where the map at the bottom is the identity.

    It remains to check that for any $x: \Delta^n \to X$ both $a$ and $b$ induce the same uniform Kan fibration structure on $Y_x \to \Delta^n$. If $e$ is not in the image of $x$, then this is clear. Otherwise there is a map $\sigma: \Delta^n \to \Delta^2$ such that $u\sigma = x$. Thus, it suffices to check that $a$ and $b$ pull back to the same element in ${\rm Fib}(Y_u \to \Delta^2)$. But this is clear by construction.
\end{proof}

\begin{prop}{unifKanfibrnotlocal2}
    The notion of a uniform Kan fibration fails to be local for the reason explained in (2). 
\end{prop}
\begin{proof}
    Let $E$ and $F$ be 3-simplices glued together along inclusions $\Delta^1 \times \Delta^1 \to S$ and $\Delta^1 \times \Delta^1 \to T$ to yield a simplicial set $X$. We name an edge $e: (1,0) \to (1,1)$. For $Y$ we again take the 1-coskeleton of $X \cup \{ f \}$ where $f$ is an additional copy of the edge $e$ and $p: Y \to X$ is the unique map extending the identity on $X$. Let $Y_E$ and $Y_F$ be the restriction of $Y$ to $E$ and $F$, respectively.
    \begin{displaymath}
        \begin{array}{cc}
            \begin{tikzcd}
                Y_E \ar[d] \ar[r] & Y \ar[d, "p"] \\
                E \ar[r] & X 
            \end{tikzcd} & 
            \begin{tikzcd}
                Y_F \ar[d] \ar[r] & Y \ar[d, "p"] \\
                F \ar[r] & X 
            \end{tikzcd}
        \end{array}
    \end{displaymath}
    We will define uniform Kan fibration structures $a$ and $b$ on $Y_E \to E$ and $Y_F \to F$, respectively, which restrict to the same uniform Kan fibration structure on any triangle in their overlap $\Delta^1 \times \Delta^1 \to X$. In doing so we will prove that any pullback of $p$ along a representable carries a uniform Kan fibration structure, where these uniform Kan fibration structures are compatible as made precise in (2). We build these uniform Kan fibration structures as before, where for $Y_E \to E$ the fibration structure $a$ will always lift $e$ to $e$ if possible. For $Y_F \to F$ the fibration structure $b$ will always lift $e$ to $f$ if possible; however, if the lifting problem factors through a triangle of $\Delta^1 \times \Delta^1 \to X$ the fibration structure $b$ will lift $e$ to $e$. In this way we ensure that the uniform Kan fibration structures on $Y_E \to E$ and $Y_F \to F$ restrict to the same uniform Kan fibration structure on any triangle of their overlap $\Delta^1 \times \Delta^1 \to X$. 

    However, there is no uniform Kan fibration structure on $p$ which restrict to both $a$ and $b$. Indeed, there is open box filling problem on the overlap $\Delta^1 \times \Delta^1 \to X$ on which $a$ and $b$ want to lift $e$ to $e$ and $f$, respectively.
\end{proof}

\end{appendices}

\printbibliography{}

\end{document}